\documentclass[12pt]{article}
\usepackage{amsmath,amssymb,amsthm,fullpage,url,xypic}

\usepackage{tikz}
\usetikzlibrary{arrows,decorations.pathmorphing,decorations.pathreplacing,positioning,shapes.geometric,shapes.misc,decorations.markings,decorations.fractals,calc,patterns}

\tikzset{>=stealth',
        cvertex/.style={circle,draw=black,inner sep=1pt,outer sep=3pt},
        vertex/.style={circle,fill=black,inner sep=1pt,outer sep=3pt},
        star/.style={circle,fill=yellow,inner sep=0.75pt,outer sep=0.75pt},
        tvertex/.style={inner sep=1pt,font=\scriptsize},
        gap/.style={inner sep=0.5pt,fill=white}}

\theoremstyle{definition}
\newtheorem{thm}{Theorem}
\newtheorem{theorem}[thm]{Theorem}
\newtheorem{lemma}[thm]{Lemma}
\newtheorem{proposition}[thm]{Proposition}
\newtheorem{corollary}[thm]{Corollary}
\newtheorem{definition}[thm]{Definition}

\newtheorem{example}[thm]{Example}
\newtheorem{exercise}[thm]{Exercise}
\newtheorem{question}[thm]{Question}
\newtheorem{conjecture}[thm]{Conjecture}
\newtheorem{remark}[thm]{Remark}

\numberwithin{thm}{subsection}
\numberwithin{figure}{subsection}
\numberwithin{equation}{section}

\newcommand{\cyc}{{\operatorname{cyc}}}

\newcommand{\obar}{{\operatorname{bar}}}
\newcommand{\MCE}{\operatorname{\mathsf{MCE}}}
\newcommand{\MC}{\operatorname{MC}}
\newcommand{\HKR}{\operatorname{\mathsf{HKR}}}
\newcommand{\Lie}{\operatorname{Lie}}
\newcommand{\ad}{\operatorname{ad}}
\newcommand{\Ad}{\operatorname{Ad}}
\newcommand{\Ass}{\operatorname{Ass}}

\newcommand{\Vect}{\operatorname{Vect}}

\newcommand{\vol}{\mathsf{vol}}
\newcommand{\Span}{\operatorname{Span}}
\newcommand{\bk}{\mathbf{k}}
\newcommand{\Id}{\operatorname{Id}}

\newcommand{\tr}{\operatorname{tr}}

\newcommand{\HH}{\mathsf{HH}}

\newcommand{\symm}{\text{symm}}

\newcommand{\HP}{\mathsf{HP}}

\newcommand{\bR}{\mathbf{R}}
\newcommand{\bQ}{\mathbf{Q}}
\newcommand{\bZ}{\mathbf{Z}}

\newcommand{\bP}{\mathbf{P}}
\newcommand{\iso}{{\;\stackrel{_\sim}{\to}\;}}
\newcommand{\bC}{\mathbf{C}}
\newcommand{\bG}{\mathbf{G}}
\newcommand{\Nil}{\operatorname{Nil}}

\newcommand{\cO}{\mathcal{O}}

\newcommand{\caD}{\mathcal{D}}

\newcommand{\mfh}{\mathfrak{h}}

\newcommand{\im}{\operatorname{im}}
\newcommand{\Hom}{\operatorname{Hom}}
\newcommand{\RHom}{\operatorname{RHom}}
\newcommand{\sign}{\operatorname{sign}}

\newcommand{\End}{\operatorname{End}}
\newcommand{\poly}{{\operatorname{poly}}}
\newcommand{\rk}{\operatorname{rk}}
\newcommand{\Diff}{\operatorname{Diff}}
\newcommand{\Der}{\operatorname{Der}}
\newcommand{\Inn}{\operatorname{Inn}}
\newcommand{\Out}{\operatorname{Out}}
\newcommand{\Aut}{\operatorname{Aut}}
\newcommand{\Ind}{\operatorname{Ind}}
\newcommand{\Rep}{\operatorname{Rep}}
\newcommand{\Ext}{\operatorname{Ext}}
\newcommand{\Tor}{\operatorname{Tor}}
\newcommand{\op}{{\operatorname{op}}}

\newcommand{\mfg}{\mathfrak{g}}

\newcommand{\SL}{\mathsf{SL}}
\newcommand{\GL}{\mathsf{GL}}
\newcommand{\mfgl}{\mathfrak{gl}}
\newcommand{\mfsl}{\mathfrak{sl}}
\newcommand{\Sp}{\mathsf{Sp}}
\newcommand{\gr}{\operatorname{\mathsf{gr}}}
\newcommand{\Spec}{\operatorname{\mathsf{Spec}}}
\newcommand{\Spf}{\operatorname{\mathsf{Spf}}}
\newcommand{\Weyl}{\mathsf{Weyl}}
\newcommand{\onto}{\twoheadrightarrow}
\newcommand{\into}{\hookrightarrow}
\newcommand{\Sym}{\operatorname{\mathsf{Sym}}}

\newcommand{\bA}{\mathbf{A}}

\begin{document}
\title{Deformations of algebras in noncommutative geometry}
\author{Travis Schedler}
\date{\today}
\maketitle

\begin{abstract}
These are significantly expanded lecture notes for the author's
minicourse at MSRI in June 2012, as published in the MSRI lecture note series, with some minor additional corrections.  In these notes, following, e.g.,
\cite{Eti-enadt,Kform,EGdelpezzo}, we give an example-motivated review
of the deformation theory of associative algebras in terms of the
Hochschild cochain complex as well as quantization of Poisson
structures, and Kontsevich's formality theorem in the smooth setting.
We then discuss quantization and deformation via Calabi-Yau algebras
and potentials.
Examples discussed include Weyl algebras, enveloping algebras of Lie
algebras, symplectic reflection algebras, quasihomogeneous isolated
hypersurface singularities (including du Val singularities), and
Calabi-Yau algebras.
\end{abstract}

The exercises are a great place to learn the material more
detail. There are detailed solutions provided, which the reader is
encouraged to consult if stuck.  There are some \emph{starred} (parts
of) exercises which are quite difficult, so the reader can feel free
to skip these (or just glance at them).

There are a lot of remarks, not all of which are essential; so
many of them can be skipped on a first reading.

We will work throughout over a field $\bk$. A lot of the time we will
need it to have characteristic zero; feel free to assume this always.


\subsection*{Acknowledgements}
These notes are based on my lectures for MSRI's 2012 summer graduate
workshop on noncommutative algebraic geometry. I am grateful to MSRI
and the organizers of the Spring 2013 MSRI program on noncommutative
algebraic geometry and representation theory for the opportunity to
give these lectures; to my fellow instructors and scientific
organizers Gwyn Bellamy, Dan Rogalski, and Michael Wemyss for their
help and support; to the excellent graduate students who attended the
workshop for their interest, excellent questions, and corrections; and
to Chris Marshall and the MSRI staff for organizing the workshop. I am
grateful to Daniel Kaplan and Michael Wong for carefully studying
these notes and providing many corrections.

\section*{Introduction}
Deformation theory is ubiquitous in mathematics: given any sort of
structure it is a natural (and often deep and interesting) question to
determine its deformations. In geometry there are several types of
deformations one can consider. The most obvious is actual
deformations, such as a family of varieties (or manifolds), $X_t$,
parameterized by $t \in \bR$ or $\bC$.  Many times this is either too
difficult to study or there are not enough actual deformations (which
are not isomorphic to the original variety $X$), so it makes sense to
consider \emph{infinitesimal} deformations: this is a family of
structures $X_t$ where $t \in \bC[\varepsilon]/(\varepsilon^2)$, i.e.,
the type of family which can be obtained from an actual family by
taking the tangent space to the deformation.  Sometimes, but not
always, infinitesimal deformations can be extended to higher-order
deformations, i.e., one can extend the family to a family where $t \in
\bC[\varepsilon]/(\varepsilon^k)$ for some $k \geq 2$.  Sometimes
these extensions exist to all orders. A \emph{formal deformation} is
the same thing as a compatible family of such extensions for all $k
\geq 2$, i.e., such that restriction from order $k$ to order $j < k$
recovers the deformation at order $j$.

In commutative algebraic geometry, affine varieties correspond to
commutative algebras (which are finitely generated and have no
nilpotent elements): they are of the form $\Spec A$ for $A$
commutative.  In ``noncommutative affine algebraic geometry,''
therefore, it makes sense to study deformations of associative
algebras.  This is the main subject of these notes.  As we will see,
such deformations arise from and have applications to a wide variety
of subjects in representation theory (of algebras, Lie algebras, and
Lie groups), differential operators and D-modules, quantization,
rational homotopy theory, Calabi-Yau algebras (which are a
noncommutative generalization of affine Calabi-Yau varieties), and
many other subjects.  Moreover, many of the important examples of
noncommutative algebras studied in representation theory, such as
symplectic reflection algebras (the subject of Bellamy's chapter)
and many noncommutative projective spaces
(the subject of Rogalski's chapter)
arise in this way. The
noncommutative resolutions studied in
Wemyss's chapter
can
also be deformed, and the resulting deformation theory should be
closely related to that of commutative resolutions.

Of particular interest is the study of noncommutative deformations of
commutative algebras.  These are called quantizations. Whenever one
has such a deformation, the first-order part of the deformation (i.e.,
the derivative of the deformation) recovers a Lie bracket on the
commutative algebra, which is a derivation in each component.  A
commutative algebra together with such a bracket is called a Poisson
algebra. Its spectrum is called an (affine) Poisson variety.  By
convention, a quantization is an associative deformation of a
commutative algebra equipped with a fixed Poisson bracket, i.e., a
noncommutative deformation which, to first order, recovers the Poisson
bracket.  Poisson brackets are very old and appeared already in
classical physics (particularly Hamiltonian mechanics) and this notion
of quantization is already used in the original formulation of quantum
mechanics.  In spite of the old history, quantization has attracted a
lot of recent attention in both mathematics and physics.

One of the most important questions about quantization, which is a
central topic of these notes, is of their existence: given a Poisson
algebra, does there exist a quantization, and can one construct it
explicitly? In the most nondegenerate case, the Poisson variety is a
smooth symplectic variety; in this case, the analogous problem for
$C^\infty$ manifolds was answered in the affirmative in
\cite{DWL-espfdPLaasm}, and an important explicit construction was
given in \cite{Fed-sgcdq}.  In the general case of affine algebraic
Poisson varieties, the answer is negative; see Mathieu's example in
Remark \ref{r:mathieu} below.  A major breakthrough occurred in 1997
with Kontsevich's proof that, for arbitrary smooth $C^\infty$
manifolds, and for real algebraic affine space $\bR^n$, all Poisson
structures can be quantized. In fact, Kontsevich constructed a natural
(i.e., functorial) quantization, and indicated how to extend it to
general smooth affine (or suitable nonaffine) varieties; the details
of this extension and a study of the obstructions for nonaffine
varieties were first completed by Yekutieli \cite{Yek-dqag}, see also
\cite{VdB-gdqac}, but there have been a large body of refinements to
the result, e.g., in \cite{DTT-hgahcraf} for the affine setting, and
in \cite{CV-HcAc,CV-GfGl} for a sheaf version of the global setting.

More recently, the study of Calabi-Yau algebras, mathematically
pioneered by Ginzburg \cite{GinzCY}, has become extremely
interesting. This is a subject of overlap of all the chapters of this
book, since many of the algebras studied in all chapters are
Calabi-Yau, including many of the regular algebras studied in
Rogalski's chapter,
all of the symplectic reflection algebras
studied in Bellamy's chapter,
and all of the noncommutative
crepant resolutions of Gorenstein singularities studied in
Wemyss's chapter.
In the commutative case, Calabi-Yau
algebras are merely rings of functions on affine Calabi-Yau varieties;
the noncommutative generalization is much more interesting, but shares
some of the same properties.  Beginning with a Calabi-Yau variety, one
can consider Calabi-Yau deformations (see, e.g., \cite{VdBV-CYdnch}
for a study of their moduli).  In \cite{EGdelpezzo}, this was applied
to the quantization of del Pezzo surfaces: Etingof and Ginzburg first
reduced the problem to the affine surface obtained by deleting an
elliptic curve; this affine surface embeds into $\bC^3$.  They then
deformed the ambient smooth Calabi-Yau variety $\bC^3$, together with
the hypersurface.  One advantage of this is the fact that many
Calabi-Yau algebras can be defined only by a single noncommutative
polynomial, called the (super)potential, such that the relations for
the algebra are obtained by differentiating the potential (see, e.g.,
\cite{GinzCY,BSW}).  (It was in fact conjectured that all Calabi-Yau
algebras are obtained in this way; this was proved for graded algebras
\cite{Boc-gcyad3} and more generally for completed algebras
\cite{VdB-cyas}, but is false in general \cite{Dav-sam}.)  Thus the
deformations of $\bC^3$ studied in \cite{EGdelpezzo} are very explicit
and given by deforming the potential function.  To quantize the
original affine surface, one then takes a quotient of such a deformed
algebra by a central element. We end this chapter by explaining this
beautiful construction.

Deformations of algebras are closely related to a lot of other
subjects we are not able to discuss here.  Notably, this includes the
mathematical theory of quantum groups, pioneered in the 1980s by
Drinfeld, Jimbo, and others: this is the analogue for groups of the
latter, where one quantizes Poisson-Lie groups rather than Poisson
varieties.  Mathematically, this means one deforms Hopf algebras
rather than associative algebras, and one begins with the commutative
Hopf algebras of functions on a group.  One can moreover consider
actions of such quantum groups on noncommutative spaces, which has
attracted recent attention in, e.g., \cite{EW-sHacd,CWWZ-Hafra}.
Quantum groups also have close relationships to the formality theorem:
Etingof and Kazhdan proved in \cite{EK} that all Lie bialgebras can be
quantized to a quantum group (a group analogue of Kontsevich's
existence theorem), and Tamarkin gave a new proof of Kontsevich's
formality theorem for $\bR^n$ which used this result; in fact, this
result was stronger, as it takes into account the cup product
structure on polyvector fields $\wedge_{\cO(\bR^n)}^\bullet
\Vect(\bR^n)$, i.e., its full differential graded Gerstenhaber algebra
structure, rather than merely considering its differential graded Lie
algebra structure.  Moreover, using \cite{Tam-Fcold}, Tamarkin's proof
works over any field of characteristic zero, and requires only a
rational Drinfeld associator rather than the Etingof-Kazhdan theorem;
in this form the latter theorem also follows as a consequence
\cite{Tam-qLbfold}.


A rough outline of these notes is as follows. In Section
\ref{s:mot-ex}, we will survey some of the most basic and interesting
examples of deformations of associative algebras. This serves not
merely as a motivation, but also begins the study of the theory and
important concepts and constructions.  The reader should have in mind
these examples while reading the remainder of the text.  In
particular, we will consider Weyl algebras and algebras of
differential operators, universal enveloping algebras of Lie algebras,
quantizations of the nilpotent cone, and we will conclude by
explaining the Beilinson-Bernstein localization theorem, which relates
all of these examples and is one of the cornerstones of geometric
representation theory.

In Section \ref{s:kont-state}, we will define the notions of formal
deformations, Poisson structures, and deformation quantization.  We
culminate with the statement of Kontsevich's theorem (and its
refinements) on deformation quantization of smooth Poisson manifolds
and smooth affine Poisson varieties. We will come back to this in
Section \ref{s:dgla}.

In Section \ref{s:hoch}, we begin a systematic study of deformation
theory of algebras, focusing on their Hochschild cohomology.  This
allows us to classify infinitesimal deformations and the obstructions
to second-order deformations, as well as some theory of deforming
their modules.

In Section \ref{s:dgla} we pass from the Hochschild cohomology to the
Hochschild cochain complex, which is the structure of a differential
graded Lie algebra.  This allows us to classify formal deformations.
We then return to the subject of Kontsevich's theorem, and explain how
it follows from his more refined statement on formality of the
Hochschild cochain complex as a differential graded Lie algebra.

Finally, in Section \ref{s:cy} we discuss Calabi-Yau algebras, which
is a subject that connects all of the chapters of the book.  This came up
already in previous sections because our main examples up to this
point are all (twisted) Calabi-Yau, including Weyl algebras, universal
enveloping algebras (as well as many of their central reductions), and
symplectic reflection algebras.  In this section, we explain how to
define and deform Calabi-Yau algebras using potentials. We then
apply this to quantization of hypersurfaces in $\bC^3$, following
\cite{EGdelpezzo}.

We stress that these notes only scratch the surface of the theory
of deformations of associative algebras.  Many subjects are not
discussed, such as Gerstenhaber and Schack's detailed study via
Hochschild cohomology and their cocycles (via lifting one by one from
$k$-th to $(k+1)$-st order deformations); other important subjects are
mentioned only in the exercises, such as the Koszul deformation
principle (Theorem \ref{t:kdp}).

\section{Motivating examples}\label{s:mot-ex}
In this section, we begin with the definitions of graded associative
algebras and filtered deformations. We proceed with the fundamental
examples of Weyl algebras and universal enveloping algebras of Lie
algebras, which we define and discuss, along with the invariant
subalgebras of Weyl algebras. We then introduce the concept of Poisson
algebras, which one obtains from a filtered 
deformation of a commutative algebra, such as in the previous cases,
and define a filtered quantization, which is a filtered deformation
whose associated Poisson algebra is a fixed one.  We consider the
algebra of functions on the nilpotent cone of a semisimple Lie
algebra, and explain how to construct its filtered quantization by
central reductions of the universal enveloping Lie algebra.  We
explain how the geometry of the nilpotent cone encapsulates
representations of Lie algebras, via the Beilinson-Bernstein
theorem. Finally, we conclude by discussing an important example of
deformations of a noncommutative algebra, namely the preprojective
algebra of a quiver; this also allows us to refer to quivers in later
examples in the text.

\subsection{Preliminaries}
A central object of study for us is a graded associative algebra. First
we define a graded vector space:
\begin{definition}
A $\bZ$-graded vector space is a vector space $V = \bigoplus_{m \in \bZ} V_m$.
A homogeneous element is an element of $V_m$ for some $m \in \bZ$. For
$v \in V_m$ we write $|v|=m$.
\end{definition}
\begin{remark}
  There are two conventions for indicating the grading:
  either using a subscript, or a superscript. We will switch to the
  latter in later sections when dealing with dg algebras. See also
  Remark \ref{r:hom-cohom}.
\end{remark}
\begin{definition}
  A $\bZ$-graded associative algebra is an algebra $A= \bigoplus_{m
    \in \bZ} A_m$ with $A_m A_n \subseteq A_{m+n}$ for all $m,n \in
  \bZ$. 
%
\end{definition}
A graded associative algebra is in particular a graded vector space,
so we still write $|a|=n$ when $a \in A_n$.

The simplest (but very important) example of a graded associative algebra
is a tensor algebra $TV$ for $V$ a vector space:
\begin{definition}
  The tensor algebra $TV$ is defined as $TV := \bigoplus_{m \geq 0}
  T^m V$, with $T^m V := V^{\otimes m}$. The multiplication is the
  tensor product, and the grading is tensor degree: $(TV)_m := T^m V$.  
\end{definition}
Note that $(TV)_0 = \bk$.

The next example, which is commutative,
 is a symmetric algebra $\Sym V$ for $V$ a vector space:
\begin{definition}
  The symmetric algebra $\Sym V$ is defined by $\Sym V := TV / (xy-yx
  \mid x,y \in V)$.  It is graded again by tensor degree: $(\Sym V)_m$
  is defined to be the image of $(TV)_m = T^m V$, also denoted $\Sym^m
  V$.  The exterior algebra $\wedge V := TV / (x^2 \mid x \in V)$
  is defined similarly, with $(\wedge V)_m=\wedge^m V$.
\end{definition}
In other words, $(\Sym V)_m$ is spanned by monomials of length $m$:
$v_1 \cdots v_m$ for $v_1, \ldots, v_m \in V$.  Note that $(\Sym V)_m
= \Sym^m V := T^m V / S_m$, the quotient of the vector space $T^m V$
by the action of the symmetric group $S_m$, which is called
the $m$-th symmetric power of $V$.

Given a commutative algebra $B$ (which is finitely generated over
$\bk$ and has no nilpotent elements), we can consider equivalently the
affine variety $\Spec B$.  We use the notation $\cO(\Spec B) := B$ for
affine varieties.

There is a geometric interpretation of the above, which the reader
unfamiliar with algebraic groups can safely skip: a grading on a
commutative algebra $B$ is the same as an action of the algebraic
group $\bG_m = \bk^\times$ (the group whose $\bk$-points are nonzero
elements of $\bk$ under multiplication) on $\Spec B$, with $B_m$ the
weight space of the character $\chi(z)=z^m$ of $\bG_m$.
\begin{remark}
  In this section we will deal with many such algebras satisfying the
  commutativity constraint $xy=yx$.  We note that there is a very
  important alternative commutativity constraint, called
  supercommutativity or graded commutativity, given by $xy = (-1)^{|x|
    |y|} yx$; this constraint will come up in later sections.  The
  choice of commutativity constraint is linked to how one views the
  grading: gradings where the former (usual) constraint is applied are
  often ``weight'' gradings, coming from an action of the
  multiplicative group $\bG_m$ (otherwise known as $\bk^\times$) as
  explained above, whereas gradings where the latter type of
  constraint are applied are often ``homological'' or
  ``cohomological'' gradings, coming up, e.g., in Hochschild
  cohomology of algebras.
\end{remark}


\subsection{Weyl algebras}
Let $\bk$ have characteristic zero.  Letting $\bk \langle x_1, \ldots,
x_n \rangle$ denote the noncommutative polynomial algebra in variables
$x_1, \ldots, x_n$, one can define the $n$-th Weyl algebra as
\[
\Weyl_n := \bk \langle x_1, \ldots, x_n, y_1, \ldots, y_n \rangle / ([x_i,
y_j] - \delta_{ij}, [x_i, x_j], [y_i, y_j]),
\]
denoting here $[a,b] := ab-ba$.
\begin{exercise}\label{exer:weyl-diff}
  (a) Show that, setting $y_i := - \partial_{i}$, one obtains a
  surjective homomorphism from $\Weyl_n$ to the algebra, $\caD(\bA^n)$, of
  differential operators on $\bk[x_1, \ldots, x_n] = \cO(\bA^n)$ with
  polynomial coefficients (this works for any characteristic).

  (b)(*) Show that, assuming $\bk$ has characteristic zero, this
  homomorphism is an isomorphism.
\end{exercise}
More invariantly, recall that a \emph{symplectic vector space} is a
vector space equipped with a skew-symmetric, nondegenerate bilinear
form (called the symplectic form).
 \begin{definition}
   Let $V$ be a symplectic vector space with symplectic form
   $(-,-)$. Then, the Weyl algebra $\Weyl(V)$ is defined by
\[
\Weyl(V) = TV / (x y - yx - (x,y)).
\]
\end{definition}
(The above definition makes sense even if $(-,-)$ is
degenerate, but typically one imposes nondegeneracy.)
\begin{exercise}\label{exer:weylv}
  For every $n$, we can consider the symplectic vector space $V$ of
  dimension $2n$ with basis $(x_1,\ldots,x_n,y_1,\ldots,y_n)$, and
  form $(x_i, y_j)=\delta_{ij}=-(y_j, x_i), (x_i,x_j)=0=(y_i,y_j)$,
  for all $1\leq i,j \leq n$.  Show that $\Weyl(V) = \Weyl_n$ (i.e., check
  that $\Weyl(V)$ is generated by the $x_i$ and $y_i$, with the same
  relations as $\Weyl_n$.)
\end{exercise}
The Weyl algebra deforms the algebra $\Sym V$ in the following sense:
\begin{definition}
  An (increasing) filtration on a vector space $V$ is a sequence of
  subspaces $V_{\leq m} \subseteq V$, such that $V_{\leq m} \subseteq
  V_{\leq n}$ for all $m \leq n$.  This is called nonnegative if
  $V_{\leq m} = 0$ whenever $m < 0$.  It is called exhaustive if $V =
  \bigcup_{m} V_{\leq m}$.  It is called Hausdorff if
  $\bigcap_m V_{\leq m} = 0$.  A vector space with a filtration is
  called a filtered vector space.
\end{definition}
We will only consider exhaustive, Hausdorff filtrations, so we omit those
terms.  Most of the time, we will only consider nonnegative filtrations
(which immediately implies Hausdorff).  We will also usually only use
increasing filtrations, so we often omit that term.
\begin{definition}
  An (increasing) filtration on an associative algebra $A$ is an
  increasing filtration $A_{\leq m}$ such that
\[
  A_{\leq m} \cdot A_{\leq n} \subseteq A_{\leq (m+n)}, \forall m,n
  \in \bZ.
\]
An algebra equipped with such a filtration is called a filtered algebra.
\end{definition}
\begin{definition}
  For a filtered algebra $A = \bigcup_{m \geq 0} A_{\leq m}$, the 
  associated graded algebra is  $\gr A := \bigoplus_{m} A_{\leq m} / A_{\leq (m-1)}$. Let $\gr_m A = (\gr A)_m = A_{\leq m} /
  A_{\leq (m-1)}$.
\end{definition}
We now return to the Weyl algebra. It is equipped with two different
nonnegative filtrations, the \emph{additive} or \emph{Bernstein}
filtration, and the \emph{geometric} one. We first consider the
additive filtration.  This is the filtration by degree of
noncommutative monomials in the $x_i$ and $y_i$, i.e.,
\[
\Weyl(V)_{\leq n} = \Span \{v_1 \cdots v_m \mid v_i \in V, m \leq n\}.
\]
The \emph{geometric} filtration, on the other hand, assigns the $x_i$
degree $0$ and only the $y_i$ degree one.  This filtration has the
advantage that it generalizes from $\Weyl(V)$ to the setting of
differential operators on arbitrary varieties, since there one obtains
the filtration by order of differential operators, but has the
disadvantage that the full symmetry group $\Sp(V)$ does not preserve
the filtration, but only the subgroup $\GL(U)$, where $U = \Span
\{ x_i \}$ (which acts on $U' = \Span\{ y_i \}$ by
the inverse transpose matrix of the one on $U$, in the bases of the
$y_i$ and $x_i$, respectively).
\begin{exercise} \label{exer:weyl-gr} Continuing to assume that $\bk$
  is in characteristic zero, show that, with either the additive or
  geometric filtration, $\gr \Weyl(V)$ is isomorphic to the symmetric
  algebra $\Sym V$, but that the induced gradings on $\Sym V$ are
  different!   With the additive filtration, $\Sym V \cong \gr
  \Weyl(V)$ with the grading placing $V$ in degree one (i.e., the
  usual grading on the symmetric algebra placing $\Sym^k V$ in degree
  $k$).  With the geometric filtration, show that $V = V_0 \oplus V_1$
  where $V_0$ is spanned by the $x_i$ and $V_1$ is spanned by the
  $y_i$.
  
  The easiest way to do this is to use Exercise \ref{exer:weyl-diff}
  and to show that a basis for $\caD(\bA^n)$ is given by monomials of
  the form $f(x_1, \ldots, x_n) g(y_1, \ldots, y_n)$, where $f$ and $g$
   are considered as commutative monomials in the commutative subalgebras $\bk[x_1,\ldots,x_n]$ and $\bk[y_1,\ldots,y_n]$, respectively.  
\end{exercise}
\begin{exercise} \label{exer:weyl-basis}
More difficult: the algebra $\Weyl(V)$ can actually
  be defined over an arbitrary field (or even commutative ring) $\bk$.
  In general, show that, for $V$ the vector space with basis the $x_i$
  and $y_i$ as above (or the free $\bk$-module in the case $\bk$ is a
  commutative ring), a $\bk$-linear basis of $\Weyl(V)$ can be
  obtained via monomials $f(x_1,\ldots,x_n) g(y_1,\ldots,y_n)$ as in
  the previous exercise, and conclude that the canonical homomorphism
  $\Sym(V) \to \gr \Weyl(V)$ is still an isomorphism.

  To do this, first note that $\Sym(V) \to \gr \Weyl(V)$ is obviously
  surjective; we just have to show injectivity.  Next note that, if $R
  \subseteq S$ are commutative rings, and $V$ an $R$-module with a
  skew-symmetric pairing, then $\Weyl(V)$ can be defined, and that
  $\Weyl(V \otimes_R S) = \Weyl(V) \otimes_R S$.  So the desired
  statement reduces to the case $\bk=\mathbf{Z}$, with $V$ the free
  module generated by the $x_i$ and $y_i$.  There one can see that the
  desired monomials are linearly independent by using their action by
  differential operators on $\mathbf{Z}[x_1,\ldots,x_n]$, i.e.,
  sending $f(x_1,\ldots,x_n) g(y_1,\ldots,y_n)$ to $f(x_1,\ldots,x_n)
  g(-\partial_1,\ldots,-\partial_n)$.  This implies injectivity of
  $\Sym(V) \to \gr \Weyl(V)$.
\end{exercise}
This motivates the following:
\begin{definition}\label{d:filt-def}
A filtered deformation of a graded algebra $B$ is a filtered algebra $A$
such that $\gr(A) \cong B$ (as graded algebras).
\end{definition}
In the case that $B$ is commutative and $A$ is noncommutative, we will
call such a deformation a \emph{filtered quantization}. We will give a
formal definition in \S \ref{ss:qpa}, once we discuss Poisson
algebras.

\subsubsection{Invariant subalgebras of Weyl algebras}\label{sss:iswa}
Recall that the symplectic group $\Sp(V)$ is the group of linear
transformations $V \to V$ preserving the symplectic form $(-,-)$,
i.e., $\Sp(V) = \{T: V \to V \mid (T(v),T(w))=(v,w), \forall v,w \in
V\}$.

Observe that $\Sp(V)$ acts by algebra automorphisms of $\Weyl(V)$, and
that it preserves the additive (but not the geometric) filtration.
Let $\Gamma < \Sp(V)$ be a finite subgroup of order relatively prime
to the characteristic of $\bk$.  Then, one can consider the invariant
subalgebras $\Weyl(V)^\Gamma$ and $\Sym(V)^\Gamma$.

The latter can also be viewed as the algebra of polynomial functions
on the singular quotient $V^*/\Gamma$.  Recall that an affine variety
$X$ is defined as the spectrum $\Spec \cO(X)$ of a commutative
$\bk$-algebra $\cO(X)$ (to be called a variety, $\cO(X)$ should be finitely
generated over $\bk$ and have no nilpotents).  Then, for $V$ a vector
space, the spectrum $\Spec \Sym V$ is identified with the dual vector
space $V^*$ via the correspondence between elements $\phi \in V^*$ and
maximal ideals $\mathfrak{m}_\phi \subseteq \Sym V$, of elements of
$\Sym V$ such that evaluating everything in $V$ against $\phi$
simultaneously yields zero, i.e., $\sum_i v_{i,1} \cdots v_{i,j_i} \in
\mathfrak{m}_\phi$ if and only if $\sum_i \phi(v_{i,1}) \cdots
\phi(v_{i,j_i}) = 0$, for $v_{i,k} \in V$ and some $j_i \geq 1$.

Next, recall that an action of a discrete group $\Gamma$ on an affine
variety $X = \Spec \cO(X)$ is the same as an action of $\Gamma$ on the
commutative algebra $\cO(X)$.  
 The following
definition is then standard:
\begin{definition}
  The quotient $X/\Gamma$ of an affine variety $X$ by a discrete group
  $\Gamma$ is the spectrum, $\Spec \cO(X)^\Gamma$, of the algebra of
  invariants, $\cO(X)^\Gamma$.
\end{definition}
Therefore, one identifies $V^*/\Gamma$ with $\Spec (\Sym V)^\Gamma$,
i.e., $(\Sym V)^\Gamma = \cO(V^*/\Gamma)$.

Later on we will also have use for the case where $\Gamma$ is replaced
by a non-discrete algebraic group, in which case the same definition
is made with a different notation:
\begin{definition}
  The (categorical) quotient $X/\!/G$ is $\Spec \cO(X)^G$.
\end{definition}
The reason for the (standard) double slash $/\!/$ and for
``categorical'' in parentheses is to distinguish from other types of
quotients of $X$ by $G$, most notably the stack and GIT (geometric invariant
theory) quotients, which we will not need in these notes.

Let us return to the setting of $\Gamma < \Sp(V)$ which is relatively
prime to the characteristic of $\bk$.
One easily sees that $\Weyl(V)^G$ is filtered (using the additive
filtration)\footnote{When $G$ preserves the span of the $x_i$, then
  it also preserves the geometric filtration, and in this case
  $\Weyl(V)^G$ is also filtered with the geometric filtration. More
  generally, if $G$ acts on an affine variety $X$, one can consider
  the $G$-invariant differential operators $\caD(X)^G$.} and that
\[
\gr(\Weyl(V)^\Gamma) \cong \Sym(V)^\Gamma =: \cO(V^*/\Gamma).
\]
In the exercises, we will see that the algebra $\Weyl(V)^G$ can be
further deformed; this yields the so-called spherical symplectic
reflection algebras, which are an important subject of
Bellamy's chapter.
\begin{example}\label{ex:duval2}
  The simplest case is already interesting: $V = \bk^2$ and $\Gamma =
  \{\pm\Id\} \cong \bZ/2$, with $\bk$ not having characteristic two.
  Then, $$\Sym(V)^\Gamma = \bk[x^2,xy,y^2] \cong \bk[u,v,w]/(v^2-uw),$$
  the algebra of functions on a singular quadric hypersurface in
  $\bA^3$.  The noncommutative deformation, $\Weyl(V)^\Gamma$, on the
  other hand, is \emph{homologically smooth} for all $V$ and $\Gamma <
  \Sp(V)$ finite, i.e., the algebra $A := \Weyl(V)^\Gamma$ has a
  finitely-generated projective $A$-bimodule resolution.  In fact, it
  is a \emph{Calabi-Yau algebra of dimension $\dim V$} (see Definition
  \ref{d:cya} below).
\end{example}

\subsection{Universal enveloping algebras of Lie algebras}
Next, we consider the enveloping algebra $U \mfg$, which
deforms the symmetric algebra $\Sym \mfg$.  This is the algebra
whose representations are the same as representations of the Lie algebra
$\mfg$ itself.

\subsubsection{The enveloping algebra}
Let $\mfg$ be a Lie algebra with Lie bracket $[-,-]$.  Then the
representation theory of $\mfg$ can be restated in terms of the
representation theory of its \emph{enveloping algebra},
\begin{equation}
U\mfg := T \mfg/ (xy-yx-[x,y] \mid x,y \in \mfg),
\end{equation}
where $T \mfg$ is the tensor algebra of $\mfg$ and $(-)$ denotes
the two-sided ideal generated by this relation.
\begin{proposition}
A representation $V$ of $\mfg$ is the same as a representation of $U \mfg$.
\end{proposition}
\begin{proof}
  If $V$ is a representation of $\mfg$, we define the action of $U
  \mfg$ by
\[
x_1 \cdots x_n(v) = x_1(x_2(\cdots(x_n(v))\cdots)).
\]
We only have to check the relation defining $U \mfg$:
\begin{equation}\label{e:ug-reln}
(xy-yx-[x,y])(v) = x(y(v))-y(x(v)) - [x,y](v),
\end{equation}
which is zero since the action of $\mfg$ was a Lie action.  

For the opposite direction, if $V$ is a representation of $U \mfg$, we
define the action of $\mfg$ by restriction from $U \mfg$ to $\mfg$.  This
defines a Lie action since the LHS of \eqref{e:ug-reln} is zero.
\end{proof}
\begin{remark}
More conceptually, the assignment $\mfg \mapsto U \mfg$ defines a
functor from Lie algebras to associative algebras. Then the above
statement says
\[
\Hom_{\Lie}(\mfg, \End(V)) = \Hom_{\Ass}(U \mfg, \End(V)),
\]
where $\Lie$ denotes Lie algebra homomorphisms, and $\Ass$ denotes
associative algebra homomorphisms. This statement is a consequence of the
statement that $\mfg \mapsto U \mfg$ is a functor from Lie algebras to
associative algebras which is left adjoint to the restriction functor
$A \mapsto A^- := (A, [a,b] := ab-ba)$.  That is, we can write the above
equivalently as
\[
\Hom_{\Lie}(\mfg, \End(V)^-) = \Hom_{\Ass}(U \mfg, \End(V)).
\]
\end{remark}
\subsubsection{$U\mfg$ as a filtered deformation of $\Sym \mfg$}
 There is a natural increasing filtration on
$U \mfg$  by
assigning $\mfg$ degree one: that is, we define
\[
(U \mfg)_{\leq m} = \langle x_1 \cdots x_j \mid x_1, \ldots, x_j \in \mfg,
j \leq m \rangle.
\]

There is a surjection of algebras,
\begin{equation}\label{e:pbw}
  \Sym \mfg \onto \gr (U \mfg), \quad (x_1 \cdots x_m) \mapsto \gr (x_1 \cdots x_m),
\end{equation}
which is well-defined since, in $\gr (U \mfg)$, one has $xy-yx=0$.
The Poincar\'e-Birkhoff-Witt theorem states that this is an
isomorphism:
\begin{theorem}\label{t:pbw}
  (PBW) The map \eqref{e:pbw} is an isomorphism.
\end{theorem}
The PBW is the \emph{key} property that says that the deformations
have been deformed in a \emph{flat} way, so that the algebra has not
gotten any smaller (the algebra cannot get bigger by a filtered
deformation of the relations, only smaller).  This is a very special
property of the deformed relations: see the following exercise.
\begin{exercise}\label{exer:def-reln}
  Suppose more generally that $B = TV / (R)$ for $R \subseteq TV$ a
  homogeneous subspace (i.e., $R$ is spanned by homogeneous
  elements). Suppose also that $E \subseteq TV$ is an arbitrary
  filtered deformation of the relations, i.e., $\gr E = R$.  Let
  $A := TV/(E)$.  Show that
  there is a canonical surjection
\[
B \onto \gr A.
\]
So by deforming relations, the algebra $A$ can only get smaller than
$B$ (by which we mean that the above surjection is not injective), and
cannot get larger.  In general, it can get smaller: see Exercise
\ref{exer:nonflat-def}  for an example where the
above surjection is not injective; in that example, $B$
is infinite-dimensional and $\gr A$ is one-dimensional. In the case when
the above surjection is injective, we call the deformation $A$
\emph{flat}, which is equivalent to saying that $A$ is a filtered
deformation in the sense of Definition \ref{d:filt-def}.
\end{exercise}

\subsection{Quantization of Poisson algebras}\label{ss:qpa}
As we will explain, the isomorphism \eqref{e:pbw} is compatible with
the natural Poisson structure on $\Sym \mfg$.
\subsubsection{Poisson algebras}
\begin{definition}
A Poisson algebra is a commutative algebra $B$ equipped with a Lie bracket
$\{-,-\}$ satisfying the Leibniz identity:
\[
\{ab,c\} = a\{b,c\} + b \{a,c\}.
\]
\end{definition}
Now, let $B := \Sym \mfg$ for $\mfg$ a Lie algebra with bracket
$[-,-]$. Then, $B$ has a canonical Poisson bracket which extends the
Lie bracket:
\begin{equation}\label{e:sym-g}
\{a_1 \cdots a_m, b_1 \cdots b_n\} = \sum_{i,j} [a_i, b_j]
a_1 \cdots \hat a_i \cdots a_m b_1 \cdots \hat b_j \cdots b_n.
\end{equation}
Here and in the sequel, the hat denotes that the corresponding term is
\emph{omitted} from the list of terms: so $a_i$ and $b_j$ do not
appear in the products, but all other $a_k$ and $b_\ell$ do (for $k
\neq i$ and $\ell \neq j$).

\subsubsection{Poisson structures on associated graded algebras}
Generally, let $A$ be an increasingly filtered
algebra such that $\gr A$ is commutative.  This implies that
there exists $d \geq 1$ such that
that
\begin{equation}\label{e:mnd}
[A_{\leq m}, A_{\leq n}] \subseteq A_{\leq (m+n-d)}, \forall m,n.
\end{equation}
One can always take $d = 1$, but in general we want to take $d$
maximal so that the above property is satisfied; see Exercise
\ref{exer:zeropbr} below. Fix a value of $d \geq 1$ satisfying
\eqref{e:mnd}.

We claim that $\gr A$ is canonically Poisson, with the bracket, for
$a \in A_{\leq m}$ and $b \in A_{\leq n}$,
\[
\{\gr_m a, \gr_n b\} := \gr_{m+n-d} (ab-ba).
\]
\begin{exercise}\label{exer:poissbr}
  Verify that the above is indeed a Poisson bracket, i.e., it
  satisfies the Lie and Leibniz identities.
\end{exercise}
\begin{exercise}\label{exer:zeropbr}
Suppose that $d \geq 1$ is \emph{not} the maximum possible value such
that \eqref{e:mnd} is satisfied.  Show that the Poisson bracket on
$\gr A$ is zero.  This explains why we usually will take $d$ to be the
maximum possible.
\end{exercise}
We conclude that $\Sym \mfg$ is equipped with a Poisson bracket by
Theorem \ref{t:pbw}, i.e., by the isomorphism \eqref{e:pbw}, taking $d := 1$.
\begin{exercise}\label{exer:pb-lie}
Verify that the Poisson bracket on $\Sym \mfg$ obtained from \eqref{e:pbw}
is the same as the one of \eqref{e:sym-g}, for $d := 1$.
\end{exercise}
\begin{exercise}\label{exer:weyl-symv}
  Equip $\Sym V$ with the unique Poisson bracket such that
  $\{v,w\}=(v,w)$ for $v, w \in V$. This bracket has degree $-2$.
  Show that, with the additive filtration, $\gr \Weyl(V) \cong \Sym V$
  as Poisson algebras, where $d=2$ in \eqref{e:mnd}.  

  On the other hand, show that, with the geometric filtration, one can
  take $d=1$, and then one obtains an isomorphism $\gr \Weyl(V) \cong
  \Sym V'$ , where $V'$ is the same underlying vector space as $V$,
  but placing the $x_i$ in degree zero and the $y_i$ in degree one (so
  $V' = (V')_0 \oplus (V')_1$ with $\dim (V')_0 = \dim (V')_1 =
  \frac{1}{2} \dim V$). The Poisson bracket on $\Sym V'$ is given by
  the same formula as for $\Sym V$.
\end{exercise}

\subsubsection{Filtered quantizations}
The preceding example motivates the definition of a filtered quantization:
\begin{definition}
  Let $B$ be a graded Poisson algebra, such that the Poisson bracket
  has negative degree $-d$. Then a \emph{filtered quantization} of $B$
  is a filtered associative algebra $A$ such that \eqref{e:mnd} is
  satisfied, and such that $\gr A \cong B$ as Poisson algebras.
\end{definition}
Again, the key property here is that $\gr A \cong B$.  Since, in the
case where $A = U \mathfrak{g}$, this property is the PBW theorem, one
often refers to this property in general as the "PBW property." In
many examples of $B$ and $A$ of interest, proving that this property
holds is an important theorem, which is often called a "PBW theorem."

\subsection{Algebras of differential operators}\label{ss:ado}
As mentioned above, when $\bk$ has characteristic zero, the Weyl
algebra $\Weyl_n = \Weyl(\bk^{2n})$ is isomorphic to the algebra of
differential operators on $\bk^n$ with polynomial coefficients.

More generally, we can define
\begin{definition}[Grothendieck]
  Let $B$ be a commutative $\bk$-algebra.  We define the space
  $\Diff_{\leq m}(B)$ of differential operators \emph{of order $\leq
    m$} inductively on $m$.  For $a \in B$ and $\phi \in \End_\bk(B)$,
  let $[\phi,a] \in \End_\bk(B)$ be the linear operator
\[
[\phi,a](b) := \phi(ab) - a\phi(b), \forall b \in B.
\]
Then we define
\begin{gather}
  \Diff_{\leq 0}(B) := \{\phi \in \End_{\bk}(B) \mid [\phi, a] = 0,
  \forall a \in B\}
  = \End_B(B) \cong B; \\
  \Diff_{\leq m}(B) := \{\phi \in \End_{\bk}(B) \mid [\phi, a] \in
  \Diff_{\leq (m-1)}(B), \forall a \in B\}.
\end{gather}
Let $\Diff(B) := \bigcup_{m \geq 0} \Diff_{\leq m}(B)$.
\end{definition}
\begin{exercise}\label{exer:diffb}
  Verify that $\Diff(B)$ is a nonnegatively filtered associative
  algebra whose associated graded algebra is commutative.
\end{exercise}
Now, suppose that $B$ is finitely generated commutative and $X :=
\Spec(B)$. Then the global vector fields on $X$ are the same
as $\bk$-algebra derivations of the algebra $B$,
\[ \begin{array}{rl}
\Vect(X) &:=\ \Der_{\bk}(\cO(X),\cO(X)) \\ \noalign{\vskip 5pt}
&:=\ \{\phi \in \Hom_\bk(\cO(X),\cO(X)) \mid
\phi(fg) = \phi(f) g + f \phi(g), \forall f,g \in \cO(X)\}.
\end{array}
\]
This is naturally a $\cO(X)$-module.  Note that it can also be viewed
as global sections of the tangent sheaf, denoted by $T_X$ (which is in
general defined so as to have sections on open affine subsets given by
derivations as above).

Recall that $T^* X$, the total space of the cotangent bundle, can be
defined as $\Spec \Sym_{\cO(X)} \Vect(X)$.  Points of $T^*X$ are in
bijection with pairs $(x,p)$ where $x \in X$ and $p \in T^*_x X =
\mathfrak{m}_x/\mathfrak{m}_x^2$, where $\mathfrak{m}_x \subseteq
\cO(X)$ is the maximal ideal of functions vanishing at $x$. In the
case where $X$ is a smooth affine complex variety, this is the same as
a pair of a point of the complex manifold $X$ and a cotangent vector
at $x$.





Now we can state the result that, for smooth affine varieties in
characteristic zero, $\Diff(\cO(X))$ quantizes the cotangent bundle:
\begin{proposition}\label{p:diff-quant}
  If $X$ is a smooth (affine) variety and $\bk$ has characteristic
  zero, then as Poisson algebras,
\[
\gr \Diff(\cO(X)) \cong \Sym_{\cO(X)} \Vect(X) \cong \cO(T^* X).
\]
\end{proposition}
\begin{remark}
  When $X$ is not affine (but still smooth), the above generalizes if
  we replace the algebras $\cO(X)$ and $\cO(T^* X)$ by the sheaves
  $\cO_X$ and $\cO_{T^* X}$ on $X$, and the global vector fields
  $\Vect(X)$ by the tangent sheaf $T_X$. The material in the remainder
  of this section also generalizes similarly to the smooth nonaffine
  context.
\end{remark}

The proposition requires some clarifications.  First, since $\Vect(X)$ is
actually a Lie algebra, we obtain a Poisson structure on $\Sym_{\cO(X)}
\Vect(X)$ by the formula
\[
\{\xi_1 \cdots \xi_m, \eta_1 \cdots \eta_n\} = \sum_{i,j} [\xi_i,
\eta_j] \cdot \xi_1 \cdots \hat \xi_i \cdots \xi_m \eta_1 \cdots \hat
\eta_j \cdots \eta_n.
\]

Proposition \ref{p:diff-quant} says that, for smooth affine
varieties, the algebra $\Diff(\cO(X))$ quantizes $\cO(T^*X)$. To prove
this result, we will want to have an alternative construction
 of $\Diff(\cO(X))$:
\begin{definition}
The universal enveloping algebroid $U_{\cO(X)}(\Vect(X))$ is the quotient of
the usual enveloping algebra
$U(\Vect(X)) = U_{\bk}(\Vect(X))$ by the relations
\begin{equation}
  f \cdot \xi = f\xi,  \quad \xi \cdot f = \xi(f) + f\xi, 
  \quad f \in \cO(X), \quad \xi \in \Vect(X).
\end{equation}
\end{definition}
\begin{remark}\label{r:lie-algebroid}
  More generally, the above definition extends to the setting when we
  replace $\Vect(X)$ by (global sections of) an arbitrary Lie
  algebroid $L$ over $X$.  Namely, such an $L$ is a
  Lie algebra which is an $\cO(X)$-module together with a
  $\cO(X)$-linear map $a: L \to T_X$ satisfying the Leibniz rule,
\[ [\xi, f \eta] = a(\xi)(f) \eta + f [\xi,\eta], \quad \forall
\xi,\eta \in L, \quad \forall f \in \cO(X).
\]
This implies that the anchor map $a$ is a homomorphism of Lie algebras.
Note that a simple example of such an $L$ is $\Vect(X)$ itself, with
$a = \Id$.  One then defines the universal enveloping algebroid of $L$
as $U_{\cO(X)} L := U L / (f \cdot \xi - f\xi, \xi \cdot f - a(\xi)(f)
- f\xi)$.
\end{remark}
\begin{remark}
  The universal enveloping algebroid can also be defined in a way
  which generalizes the definition of the usual enveloping
  algebra. Namely, if $L$ is as in the previous remark,
  define a new $\cO(X)$-bimodule structure on $L$ by giving the usual
  left action, and defining the right multiplication by $\xi \cdot f =
  a(\xi)(f) + f\xi$.  Given any bimodule $M$ over a (not-necessarily
  commutative) associative algebra $A$, we can define the tensor
  algebra $T^{\text{nc}}_A M := \bigoplus_{m \geq 0} M^{\otimes_m A}$.
  Then, $U_{\cO(X)} L := T^{\text{nc}}_{\cO(X)} L / (\xi \cdot \eta -
  \eta \cdot \xi - [\xi,\eta])$.  Then, in the case that $\cO(X)=\bk$
  itself, this recovers the usual definition of the enveloping algebra
  $UL = U_{\bk} L$.
\end{remark}
\begin{proof}[Sketch of proof of Proposition \ref{p:diff-quant}]
  Filter  $U_{\cO(X)} \Vect(X)$  by making 
  $(U_{\cO(X)} \Vect(X))_{\leq m}$  the image of $T^{\leq m}(\Vect(X)) =
  \bigoplus_{0 \leq i \leq m} \Vect(X)^{\otimes i}$ under the defining
  quotient.  Then we have a natural map $\Sym_{\cO(X)} \Vect(X) \to \gr
  U_{\cO(X)} \Vect(X)$.

  Next, if $\Vect(X)$ is free, or more generally projective
  (equivalently, $T_X$ is a locally free sheaf), as will be true when
  $X$ is smooth, then the PBW theorem generalizes to show that the
  natural map $\Sym_{\cO(X)} \Vect(X) \to \gr U_{\cO(X)} \Vect(X)$ is
  an isomorphism.

So, the proposition reduces to showing that, in the case $X$ is smooth,
the canonical map
\begin{equation} \label{e:udiff}
U_{\cO(X)} \Vect(X) \to \Diff(\cO(X))
\end{equation}
is an isomorphism.

To prove this, it suffices to show that the associated graded
homomorphism, $\Sym_{\cO(X)} \Vect(X) \to \gr \Diff(\cO(X))$, is an
isomorphism. This statement can be checked locally, in the formal
neighborhood of each point $x \in X$, which is isomorphic to a formal
neighborhood of affine space of the same dimension at the origin.
More precisely, we can replace $\cO(X)$ by a formal power series ring
$\bk[\![x_1,\ldots,x_n]\!]$ and require that all derivations and
differential operators are continuous in the adic topology (i.e., if
$\xi$ is a derivation, we require that
$\xi(\sum a_{r_1, \ldots, r_n} x_1^{r_1} \cdots x_n^{r_n}) = \sum
a_{r_1,\ldots,r_n} \xi(x_1^{r_1} \cdots x_n^{r_n})$.)
But then $\Diff(\cO(X))$ becomes
$\bk[\![x_1,\ldots,x_n]\!][\partial_1, \ldots, \partial_n]$, a
completion of the Weyl algebra, and the statement follows as in
Exercises \ref{exer:weyl-diff} and \ref{exer:weyl-gr}.
(Alternatively, instead of using formal power series, one can use
ordinary localization at $x$; since $\dim \mathfrak{m}_x /
\mathfrak{m}_x^2 = \dim X$, Nakayama's lemma implies that a basis for
this vector space lifts to a collection of $\dim X$ algebra generators
of the local ring $\cO_{X,x}$, and then differential operators of
order $\leq m$ are determined by their action on products of $\leq m$
generators, and all possible actions are given by $\Sym^{\leq
  m}_{\cO_{X,x}} T_{X,x}$.)
\end{proof}
\begin{remark} Whenever $L$ is (the global sections of) a Lie
  algebroid over $\cO(X)$ as in Remark \ref{r:lie-algebroid}, then it
  follows just as in the case $L=\Vect(X)$ that $\Sym_{\cO(X)} L$ is a
  Poisson algebra. Then as before we have a natural map $\Sym_{\cO(X)}
  L \to \gr(U_{\cO(X)} L)$ and in the case $L$ is locally free, the
  PBW theorem generalizes to show that this is an isomorphism.  This
  applies even in the case $X$ is not smooth (and hence $\Vect(X)$
  itself is not locally free), since often one can nonetheless define
  interesting locally free Lie algebroids (and the same applies when
  $X$ need not be affine, replacing $\cO(X)$ by $\cO_X$ and $\Vect(X)$
  by $T_X$).  In fact, this is the setting of a large body of
  interesting recent work, such as Calaque-Van den Bergh's analogues
  of Kontsevich's formality theorem \cite{CV-HcAc,CV-GfGl} and their
  proof of C\u ald\u araru's conjecture \cite{CRV-CcTf} on the
  compatibility of the former with cap products with Hochschild
  homology.
\end{remark}
\subsubsection{Preview: $\caD$-modules}
Since Proposition \ref{p:diff-quant} does not apply to singular
varieties, it is less clear how to treat algebras of differential
operators and their modules in the singular case (again with
$\bk$ having characteristic zero).

One solution, discovered by Kashiwara, is to take a singular
variety $X$ and embeds it into a smooth variety $V$, and define the
category of right $\caD$-modules on $X$ to be the category of right
$\caD(V)$-modules supported on $X$, i.e., right modules $M$ over the
ring $\caD(V)$ of differential operators on $V$, with the property
that, for all $m \in M$, there exists $N \geq 1$ such that $m \cdot
I_X^N = 0$, where $I_X$ is the ideal corresponding to $X$.

This circumvents the problem that the ring $\caD(X)$ of differential
operators is not well-behaved, and one obtains a very useful
theory. An equivalent definition to Kashiwara's goes under the name of
\emph{(right) crystals}.  Under some restrictions (when the $\caD$-modules are
holonomic with regular singularities), one can also replace
$\caD$-modules by \emph{perverse sheaves}, which are, roughly
speaking, gluings of local systems on subvarieties, or more precisely,
complexes of such gluings with certain properties.

\subsection{Invariant differential operators}
It is clear that the group of automorphisms $\Aut(X)$ of the variety
$X$ acts by filtered automorphisms also on $\Diff(X)$ and also by
graded automorphisms of $\gr \Diff(X) = \Sym_{\cO(X)} \Vect(X)$. Let us continue
to assume $\bk$ has characteristic zero.

Now, suppose that $G<\Aut(X)$ is a finite subgroup of automorphisms of
$X$. Then one can form the algebras $\Diff(X)^G$ and $(\Sym_{\cO(X)}
\Vect(X))^G$.  By Proposition \ref{p:diff-quant}, we conclude that $\gr
\Diff(X)^G$ is a quantization of $(\Sym_{\cO(X)} \Vect(X))^G$.

One example of this is when $X = \bA^n$ and $G < \GL(n) <
\Aut(\bA^n)$.  Then we obtain that $\Weyl(\bA^{2n})$ is a quantization
of $\cO(T^* \bA^n) = \cO(\bA^{2n})$, which is the special case of
the example of \S \ref{sss:iswa} where $G < \GL(n)$ (note that $\GL(n)
< \Sp(2n)$, where explicitly, a matrix $A$ acts by the block matrix
$\begin{pmatrix} A & 0 \\ 0 & (A^t)^{-1} \end{pmatrix}$, with $A^t$
denoting the transpose of $A$.)
\begin{remark} By deforming $\Diff(X)^G$, one obtains global analogues
  of the spherical rational Cherednik algebras \cite{Eti-chavfga}; see
  Example \ref{ex:gca} below.
\end{remark}
\begin{example}Let $X = \bA^1 \setminus \{0\}$ (note that this is affine)
 and let $G = \{1,g\}$
  where $g(x)=x^{-1}$.  Then $(T^*X)/\!/G$ is a ``global'' or
  ``multiplicative'' version of the variety $\bA^2/(\bZ/2)$ of Example
  \ref{ex:duval2}. A quantization is $\caD(X)^G$, and one has, by the
  above,
\[
\gr \caD(\bA^1 \setminus \{0\})^{\bZ/2} \cong \cO(T^* (\bA^1 \setminus \{0\}))^{\bZ/2}.
\]
Explicitly, the action on tangent vectors is $g(\partial_x) = -x^2 \partial_x$, i.e., so that, applying the operator $g(\partial_x)$ to $g(x)$, we get $g(\partial_x)( g(x) )= 1 = -x^2\partial_x(x^{-1})$.  Thus, setting $y := \gr \partial_x$, we have $g(y) = -x^2 y$. So
$\cO(T^* X /\!/ G) = \bC[x+x^{-1},y-x^2y,x^2y^2]$ and $\caD(X)^G = \bC[x+x^{-1}, \partial_x - x^2 \partial_x, (x\partial_x)^2] \subseteq \caD(X)^G$.
\end{example}

\subsection{Quantization of the nilpotent cone}
\subsubsection{Central reductions}
Suppose, generally, that $A$ is a filtered quantization of $B = \gr
A$.  Suppose in addition that there is a central filtered subalgebra
$Z \subseteq A$.  Then $\gr Z$ is Poisson central in $B$:
\begin{definition}
  The center, $Z(B)$, of a Poisson algebra, $B$, is the subalgebra of
  elements $z \in B$ such that $\{z,b\}=0$ for all $b \in B$.  An
  element is called Poisson central if it is in $Z(B)$. A subalgebra
$C \subseteq B$ is called central if $C \subseteq Z(B)$.
\end{definition}
Next, for every character $\eta: Z \to \bk$, we obtain \emph{central reductions}
\[
A^\eta := A / \ker(\eta) A, \quad B^\eta := B / \gr(\ker \eta) B.
\]
Note here that $\ker(\eta) A$ is actually a two-sided ideal since $Z$
is central.  In the case that $B_0 = \bk$ (or more generally $(\gr
Z)_0 = \bk$), which will be the case for us, note that $B^\eta$ does
not actually depend on $\eta$, and we obtain $B^\eta = B / (\gr Z)_+
B$ for all $\eta$, where $(\gr Z)_+ \subseteq \gr Z$ is the
augmentation ideal (the ideal of positively-graded elements).

Then, the category $\Rep(A^\eta)$ can be identified with the category
of representations $V$ of $A$ such that $Z(A)$ acts by the character
$\eta$, i.e., for all $v \in V$ and all $z \in Z(A)$, we have $z \cdot
v = \eta(z) v$.
\begin{remark}
  The subcategories $\Rep(A^\eta) \subseteq \Rep(A)$ are all
  \emph{orthogonal} for distinct $\eta$, which means that there are no
  nontrivial homomorphisms or extensions between representations $V
  \in \Rep(A^\eta), W \in \Rep(A^{\xi})$ with distinct central
  characters $\eta \neq \xi$. This is because the ideal generated by
  $(z-\eta(z))$ maps to the unit ideal in $A^{\xi}$ for $\xi \neq
  \eta$, i.e., there is an element $z \in Z(A)$ which acts by one on
  all representations in $\Rep(A^{\xi})$ and by zero on all
  representations in $\Rep(A^\eta)$, and this allows one to
  canonically split any extension of representations in the two
  categories $\Rep(A^\xi)$ and $\Rep(A^\eta)$: if $V$ is such an
  extension, then $V = z V \oplus (1-z)(V)$.
\end{remark}
\subsubsection{The nilpotent cone}
Now let us restrict to our situation of $A = U \mfg$ and $B = \Sym
\mfg$ with $\bk$ of characteristic zero. Let us suppose moreover that
$\mfg$ is finite-dimensional semisimple (see, for example,
\cite{Hum-ilart}; for the reader who is not familiar with this, one
can restrict to the most important examples, such as the Lie algebra
$\mathfrak{sl}_n(\bk)$ of trace-zero $n \times n$ matrices, or the Lie
algebra $\mathfrak{so}_n(\bk)$ of skew-symmetric $n \times n$
matrices).  Then, the structure of the center $Z(A)$ is well-known.
\begin{example}\label{ex:sl2}
  Let $\mfg = \mfsl_2$ with basis $(e, f, h)$.  Then $Z(U\mfg) =
  \bk[C]$, where the element $C$ is the Casimir element, $C = ef + fe
  + \frac{1}{2} h^2$.  In this case, the central reduction
  $(U\mfg)^\eta$ describes those representations on which $C$ acts by
  a fixed scalar $\eta(C)$.  For example, there exists a
  finite-dimensional representation of $(U\mfg)^\eta$ if and only if
  $\eta(C) \in \{m + \frac{1}{2} m^2 \mid m \geq 0\}$, since $\eta(C)$
  acts on a highest-weight vector $v$ of $h$ of weight $\lambda$,
  i.e., a vector such that $ev=0$ and $hv=\lambda v$, by $C \cdot v =
  (h+\frac{1}{2}h^2)v = (\lambda + \frac{1}{2} \lambda^2)v$.  In
  particular, there are only countably many such characters $\eta$
  that admit a finite-dimensional representation.

  Moreover, when $\eta(C) \in \{m + \frac{1}{2} m^2 \mid m \geq 0\}$,
  there is \emph{exactly one} finite-dimensional representation of
  $(U\mfg)^\eta$: the one with highest weight $m$.  So these
  quantizations $(U\mfg)^\eta$ of $(\Sym \mfg)/(\gr C)$
  have at most one finite-dimensional representation, and only
  countably many have this finite-dimensional representation.
\end{example}
More generally, if $\mfg$ is finite-dimensional semisimple (still with
$\bk$ of characteristic zero), it turns out that $Z(U \mfg)$ is a
polynomial algebra, and $\gr Z(U\mfg) \to Z(\Sym \mfg)$ is an
isomorphism of polynomial algebras.  
\begin{definition}
  Given a Lie algebra $\mfg$ and a representation $V$, the invariants
  $V^\mfg$ are defined as $V^{\mfg} := \{v \in V \mid x \cdot v = 0,
  \forall x \in \mfg\}$.
\end{definition}
Note that, if $B$ is an algebra with a $\mfg$-action, i.e., $B$ is an
$\mfg$ representation and an algebra such that the multiplication map
$B \otimes B \to B$ is a map of $\mfg$-representations, then $B^\mfg
\subseteq B$ is a subalgebra.  In particular, one has $(\Sym
\mfg)^\mfg \subseteq \Sym \mfg$ which is an algebra.  In fact, this is
the Poisson center of $\Sym \mfg$, since $\{z,f\}=0$ for all $f \in
\Sym \mfg$ if and only if $\{z,x\}=0$ for all $x \in \mfg$: i.e.,
$Z(\Sym \mfg) = (\Sym \mfg)^\mfg$.

We then have the following extremely important result, whose history
is discussed in more detail in Remark \ref{r:t-z-ss-hist}:
\begin{theorem}[H.~Cartan, Chevalley \cite{Che-appLg,Chevalley},
  Coxeter, Harish-Chandra \cite{HC-saueasLa}, Koszul, Shephard and Todd
  \cite{ST}, Weil] \label{t:z-ss} Let $\mfg$ be finite-dimensional
  semisimple and $\bk$ of characteristic zero. Then $Z(U \mfg) \cong
  \bk[x_1, \ldots, x_r]$ is a polynomial algebra, with $r$ equal to
  the semisimple rank of $\mfg$.  Moreover, the polynomial algebra
  $\gr Z(U \mfg)$ equals the Poisson center $Z(\Sym \mfg) = (\Sym
  \mfg)^\mfg$ of $\gr U \mfg$, and there is a canonical algebra isomorphism
 $(\Sym \mfg)^\mfg \to Z(U \mfg)$.
\end{theorem}
The final isomorphism $(\Sym \mfg)^\mfg \to Z(U\mfg)$ is called
the \emph{Harish-Chandra isomorphism}; it actually generalizes to
an isomorphism which is defined 
for \emph{arbitrary} finite-dimensional $\mfg$ and is due to Kirillov
and Duflo; we will explain how this latter isomorphism also follows
from Kontsevich's formality theorem in Corollary \ref{c:hphh} (and in
fact, Kontsevich's result implies that it holds even for
finite-dimensional Lie superalgebras).

The degrees $d_i$ of the generators $\gr x_i$ are known as the
\emph{fundamental degrees}.  (In fact, they satisfy $d_i = m_i + 1$ where
$m_i$ are the Coxeter exponents of the associated root system,
cf.~e.g., \cite{Hum-rgcg}; we will not need to know anything about
Coxeter exponents or precisely what the $m_i$ are below.)

By the theorem, for every character $\eta: Z(U\mfg) \to \bk$, one
obtains an algebra $(U \mfg)^\eta$ which quantizes $(\Sym \mfg) /
((\Sym \mfg)^\mfg_+)$.  Here $(\Sym \mfg)^\mfg_+$ is the augmentation
ideal of $\Sym \mfg$, which equals $\gr (\ker \eta)$ since $(\Sym
\mfg)_0 = \bk$ (cf.~the comments above).

Recall that the dual vector space $\mfg^*$ is canonically a
representation of $\mfg$, with action, called the coadjoint action,
given by $(x \cdot \phi)(y) := -\phi([x,y])$ for $x,y\in \mfg$ and
$\phi \in \mfg^*$.  We denote $x \cdot \phi$ also by $\ad(x)(\phi)$.
\begin{definition}\label{d:nilcone}
The \emph{nilpotent cone} $\Nil \mfg \subseteq \mfg^*$ is the set
of elements $\phi \in \mfg^*$ such that, for some $x \in \mfg$, we have
$\ad(x)\phi = \phi$.  
\end{definition}
\begin{remark}
  Note that, in the case where $\mfg = \Lie G$ for $G$ a connected Lie
  (or algebraic) group, then $\Nil \mfg$ is the set of elements $\phi$
  such that the coadjoint orbit $G \cdot \phi$ contains the line
  $\bk^\times \cdot \phi$ (in the Lie group case, $\bk$ should be
  $\bR$ or $\bC$).
\end{remark}
\begin{remark}\label{r:ad-nil}
  In the case when $\mfg$ is finite-dimensional and semisimple, it is
  well-known that the Killing form isomorphism $\mfg^* \cong \mfg$
  takes $\Nil(\mfg)$ to the cone of elements which are ad-nilpotent
  (i.e., $(\ad x)^N = 0$ for some $N \geq 1$, where $\ad x(y) :=
  [x,y]$), which explains the terminology.  It is perhaps more
  standard to define the nilpotent cone as the latter cone inside
  $\mfg$, but for us it is more natural to use the above definition.
\end{remark}


\begin{proposition}
  Let $\mfg$ be finite-dimensional semisimple.
  Then, the algebra $B^0$ is the algebra of functions on the nilpotent
  cone $\Nil \mfg \subseteq \mfg^*$.
\end{proposition}
We give a proof modulo a number of facts about semisimple Lie algebras
and groups (and algebraic groups), so the reader not familiar with
them may skip it.
\begin{proof}  
We may assume that $\bk$ is algebraically closed (otherwise let $\bar \bk$ be an algebraic closure and replace $\mathfrak{g}$ by $\mathfrak{g} \otimes_{\bk} \bar \bk$). 

Let $\mathfrak{h} \subseteq \mathfrak{g}$ be a Cartan subalgebra and
$W$ be the Weyl group.  The Chevalley isomorphism states that $(\Sym
\mfg)^{\mfg} \cong \cO(\mfh^*/W)$.  This isomorphism maps the
augmentation ideal $(\Sym \mfg)^{\mfg}_+$ to the augmentation ideal of
$\cO(\mfh^*/W)_+$, which is the ideal of the zero element $0 \in
\mfh^*$.

Now, let $G$ be a connected algebraic group such that $\mathfrak{g} =
\Lie G$ (this exists by the well-known classification of semisimple
Lie algebras, see, e.g., \cite{Hum-ilart}, and by the existence
theorem for reductive algebraic groups given root data, see, e.g.,
\cite[16.5]{Spr-lag}); in the case $\bk = \bC$, one can let $G$ be a
complex Lie group such that $\Lie G = \mathfrak{g}$. Then, $(\Sym
\mfg)^{\mfg} = (\Sym \mfg)^G = \cO(\mfg^*/\!/G)$, and the ideal $
\cO(\mfh^*/W)_+ \cdot \Sym \mfg$ thus defines those elements $\phi
\in \mfg^*$ such that $\overline{G \cdot \phi} \cap \mfh^* = \{0\}$.
This set is stable under dilation, so that for all such nonzero
$\phi$, the coadjoint orbit $G \cdot \phi$ intersects a neighborhood
of $\phi$ in the line $\bk \cdot \phi$, and is hence in the nilpotent
cone; conversely, if $\phi$ is in the nilpotent cone, the ideal vanishes
on $\phi$.
\end{proof}
\begin{corollary}
  For $\mfg$ finite-dimensional semisimple, the central reductions
  $(U\mfg)^\eta = U\mfg / (\ker \eta) \cdot U\mfg$ quantize $\cO(\Nil
    \mfg)$.
\end{corollary}
\begin{example}\label{ex:sl2-2}
  In the case $\mfg = \mfsl_2$, the quantizations $(U \mfg)^\eta$,
  whose representations are those $\mfsl_2$-representations on which
  the Casimir $C$ acts by a fixed scalar, all quantize the cone of
  nilpotent $2 \times 2$ matrices,
\[
\Nil(\mfsl_2) = \Bigl \{ \begin{pmatrix} a & b \\ c &
  -a \end{pmatrix} \Big| a^2+bc=0 \Bigr \},
\]
which are identified with functions on matrices via the trace pairing:
$X(Y) := \tr(XY)$.
\end{example}
\begin{remark}\label{r:McKay}
  The quadric of Example \ref{ex:sl2-2} is isomorphic to the one
  $v^2=uw$ of Example \ref{ex:duval2}, i.e., $\bA^2/(\bZ/2) \cong
  \Nil(\mfsl_2)$. Thus we have given \emph{two quantizations} of the
  same variety: one by the invariant Weyl algebra,
  $\Weyl_1^{\bZ/2}$, and the other a family of quantizations
  given by the central reductions $(U\mfsl_2)^\eta = U\mfsl_2 /
  (C-\eta(C))$.  In fact, the latter family is a universal family of
  quantizations (i.e., all quantizations are isomorphic to one of
  these via a filtered isomorphism whose associated graded homomorphism is
  the identity), and one can see that $\Weyl_1^{\bZ/2} \cong (U
  \mfsl_2)^\eta$ where $\eta(C) = -\frac{3}{8}$ (see Exercise \ref{exer:weyl-usl2}).

  This coincidence is the first case of a part of the \emph{McKay
    correspondence} \cite{McKay,Brisessag,Slflsgs,Slsssag}, which
  identifies, for every finite subgroup $\Gamma < \SL_2(\bC)$, the
  quotient $\bC^2/\Gamma$ with a certain two-dimensional ``Slodowy''
  slice of $\Nil(\mfg)$, where $\mfg$ is the Lie algebra whose Dynkin
  diagram has vertices labeled by the irreducible representations of
  $\Gamma$, and a single edge from $V$ to $W$ if and only if $V
  \otimes W$ contains a copy of the standard representation $\bC^2$. See
  Wemyss's chapter,
  \S 5.
\end{remark}

\begin{remark}\label{r:t-z-ss-hist}
  Theorem \ref{t:z-ss} is stated anachronistically and deserves some
  explanation.  Let us restrict for simplicity to the case $\bk=\bC$
  (although everything below applies to the case $\bk$ is
  algebraically closed of characteristic zero, and to deal with the
  non-algebraically closed case, one can tensor by an algebraic
  closure).  Coxeter originally observed that $(\Sym \mfh)^W$ is a
  polynomial algebra for $W$ a Weyl (or Coxeter) group.  In the more
  general situation where $W$ is a complex reflection group, the
  degrees $d_i$ were computed by Shephard and Todd \cite{ST} in a
  case-by-case study, and shortly after this Chevalley
  \cite{Chevalley} gave a uniform proof that $(\Sym \mfh)^W$ is a
  polynomial algebra if and only if $W$ is a complex reflection group.

  Chevalley also observed that $(\Sym \mfg)^{\mfg}\cong (\Sym \mfh)^W$
  where $\mfh \subseteq \mfg$ is a Cartan subalgebra and $W$ is the
  Weyl group (the Chevalley restriction theorem) (this mostly follows
  from the more general property of conjugacy of Cartan subgroups and
  its proof, in \cite{Che-appLg}).  Note here that $(\Sym \mfg)^{\mfg}
  = \cO(\mfg^* /\!/ G)$, the functions on coadjoint orbits in
  $\mfg^*$. The latter are identified with adjoint orbits of $G$ in
  $\mfg$ by the Killing form $\mfg^* \cong \mfg$. The closed adjoint
  orbits all contain points of $\mfh$, and their intersections with
  $\mfh$ are exactly the $W$-orbits.
  
  Harish-Chandra constructed in \cite{HC-saueasLa} an explicit
  isomorphism $HC: Z(U \mfg) \iso \Sym(\mfh)^W$ (such an isomorphism
  was, according to Godement's review on Mathematical Reviews (MR0044515) of
  this article, independently a consequence of results of H.~Cartan,
  Chevalley, Koszul, and Weil in algebraic topology).
  Harish-Chandra's isomorphism is defined by the property that, for
  every highest-weight representation $V_\lambda$ of $\mfg$ with
  highest weight $\lambda \in \mfh^*$ and (nonzero) highest weight
  vector $v \in V_\lambda$,
\[
z \cdot v = HC(z)(\lambda) \cdot v,
\]
viewing $HC(z)$ as a polynomial function on $\mfh^*$.  The
Harish-Chandra isomorphism is nontrivial: indeed, the target of $HC$
equips $\mfh^*$ not with the usual action of $W$, but the
\emph{affine} action, defined by $w \cdot \lambda :=
w(\lambda+\delta)-\delta$, where the RHS uses the usual action of $W$
on $\mfh$, and $\delta$ is the sum of the fundamental weights. This
shifting phenomenon is common for the center of a quantization.  In
this case one can explicitly see why the shift occurs: the center must
act by the same character on highest weight representations
$V_\lambda$ and $V_{w \cdot \lambda}$, and computing this character
(done in the $\mfsl_2$ case in Example \ref{ex:sl2}) yields the invariance
under the affine action. To see why the center acts by the same
character on $V_\lambda$ and $V_{w \cdot \lambda}$, in the case where
$\lambda$ is dominant (i.e., $\lambda(\alpha) \geq 0$ for all positive
roots $\alpha$), and these are the Verma modules (i.e., $V_\lambda =
\Ind_{\mathfrak{b}}^{\mathfrak{g}} \chi_\lambda$ where $\chi_\lambda$
is the corresponding character of a Borel subalgebra $\mathfrak{b}$
containing $\mathfrak{h}$, and the same is true for $V_{w \cdot
  \lambda}$), one can explicitly check that $V_{w \cdot \lambda}
\subseteq V_\lambda$.  The general case follows from this one. See
\cite[\S 23]{Hum-ilart} for a detailed proof.
\end{remark}

\subsection{Beilinson-Bernstein localization theorem and global
  quantization}
The purpose of this section is to explain a deep property of the
quantization of the nilpotent cone of a semisimple Lie algebra
discussed above: this \emph{resolves} to a global (nonaffine)
symplectic quantization of a cotangent bundle (namely, the cotangent
bundle of the flag variety). We still assume that $\bk$ has
characteristic zero.

Return to the example $\Nil(\mfsl_2) \cong \bA^2 / \{\pm \Id\}$.
There is another way to view this Poisson algebra, by the
\emph{Springer resolution}.  Namely, let us view $\Nil(\mfsl_2)$ as
the locus of nilpotent elements in $\mfsl_2$, i.e., the nilpotent
two-by-two matrices (this is consistent with Definition
\ref{d:nilcone} if we identify $\mfsl_2 \cong \mfsl_2^*$ via the trace
pairing $(x,y)=\tr(xy)$, cf.~Remark \ref{r:ad-nil}).  Then, a nonzero
nilpotent element $x \in \mfsl_2$, up to scaling, is uniquely
determined by the line $\ker(x) = \im(x)$ in $\bk^2$. Similarly, a
nonzero element $\phi \in \Nil(\mfsl_2) \subseteq \mfg^*$ is uniquely
determined, up to scaling, by the line $\ell \subseteq \bk^2$ such
that $\phi(x) = 0$ whenever $\im(x) \subseteq \ell$. With a slight
abuse of notation we will also refer to this line as $\ker(\phi)$.

Consider the locus of pairs
\[
X:= \{(\ell, x) \in \bP^1 \times \Nil(\mfsl_2) \mid \ker(x) \subseteq
\ell\} \subseteq \bP^1 \times \Nil(\mfsl_2).
\]
This projects to $\Nil(\mfsl_2)$.  Moreover, the fiber over $x \neq 0$
is evidently a single point, since $\ker(x)$ determines $\ell$. Only
over the singular point $0 \in \Nil(\mfsl_2)$ is there a larger fiber,
namely $\bP^1$ itself.
\begin{lemma} $X \cong T^* \bP^1$.
\end{lemma}
\begin{proof}
  Fix $\ell \in \bP^1$. Note that $T_\ell \bP^1$ is naturally
  $\Hom(\ell, \bk^2/\ell)$. On the other hand, the locus of $x$ such
  that $\ker(x) \subseteq \ell$ naturally acts linearly on $T_\ell
  \bP^1$: given such an $x$ and given $\phi \in \Hom(\ell,
  \bk^2/\ell)$, we can take $x \circ \phi \in \Hom(\ell, \ell) \cong
  \bk$.  Since this locus of $x$ is a one-dimensional vector space, we
  deduce that it is $T^*_\ell \bP^1$, as desired.
\end{proof}
Thus, we obtain a resolution of singularities
\begin{equation}\label{e:spr-sl2}
\rho: T^* \bP^1 \onto \Nil(\mfsl_2) = \bA^2 / \{\pm \Id\}.
\end{equation}
Note that $\Nil(\mfsl_2)$ is affine, so that, since the map
$\rho$ is birational (as all resolutions must be),
$\cO(\Nil(\mfsl_2))$ is the algebra of global sections $\Gamma(T^*
\bP^1, \cO_{T^* \bP^1})$ of functions on $T^* \bP^1$.

Moreover, equipping $T^* \bP^1$ with its standard symplectic
structure, $\rho$ is a Poisson map, i.e., $\rho^* \pi_{\Nil(\mfsl_2)}
= \pi_{T^* \bP^1}$ for $\pi_{\Nil(\mfsl_2)}$ and $\pi_{T^* \bP^1}$ the
Poisson bivectors on $\Nil(\mfsl_2)$ and $T^* \bP^1$, respectively.
Indeed, the Poisson structure on $\cO_{\Nil(\mfsl_2)}$ is obtained
from the one on $\cO_{T^* \bP^1}$ by taking global sections.

In fact, the map $\rho$ can be \emph{quantized}: let $\caD_{\bP^1}$ be
the sheaf of differential operators with polynomial coefficients on
$\bP^1$. This quantizes $\cO_{T^* \bP^1}$, as we explained, since
$\bP^1$ is smooth (this fact works for nonaffine varieties as well, if
one uses sheaves of algebras).  Then, there is the deep

\begin{theorem}(Beilinson-Bernstein for $\mfsl_2$) \label{thm:BB}
\begin{enumerate}
\item[(i)] There is an isomorphism of algebras $(U \mfsl_2)^{\eta_0}
  \iso \Gamma(\bP^1, \caD_{\bP^1})$, where $\eta_0(C) = 0$;
\item[(ii)]  Taking global sections yields an equivalence of abelian categories
\[
\caD_{\bP^1}\text{-mod} \iso (U \mfsl_2)^{\eta_0}\text{-mod}.
\]
\end{enumerate}
\end{theorem}
This quantizes the Springer resolution \eqref{e:spr-sl2}.  
\begin{remark}
  This implies that the quantization $\caD_{\bP^1}$ of $\cO_{T^*
    \bP^1}$ is, in a sense, \textbf{affine}, since its category of
  representations is equivalent to the category of representations of
  its global sections.  The analogous property holds for the sheaf of
  algebras $\cO(X)$ on an arbitrary variety $X$ if and only if $X$ is
  affine.  But, even though $\bP^1$ is projective (the opposite of
  affine), the noncommutative algebra $\caD_{\bP^1}$ is still affine,
  in this sense.

  There is a longstanding conjecture that says that, if there is an
  isomorphism $\caD_X\text{-mod} \iso \Gamma(X,\caD_X)\text{-mod}$
  (i.e., $X$ is ``$\caD$-affine''), and $X$ is smooth, projective, and
  connected, then $X$ is of the form $X \cong G/P$ where $P < G$ is a
  parabolic subgroup of a connected semisimple algebraic group.  (The
  converse is a theorem of Beilinson-Bernstein, which generalizes
  Theorem \ref{t:spr-borel} below to the parabolic case $G/P$ instead
  of $G/B$.)  Similarly, there is also an (even more famous)
  ``associated graded'' version of the conjecture, which says that if $X$ is a smooth
  projective variety and $T^* X \onto Y$ is a symplectic resolution
  with $Y$ affine (i.e., $Y = \Spec \Gamma(T^* X, \cO(T^*X))$), then
  $X \cong G/P$ as before.
\end{remark}

In fact, this whole story generalizes to arbitrary connected
semisimple algebraic groups $G$ with $\mfg := \Lie G$
the associated finite-dimensional semisimple Lie algebra.  Let
$\mathcal{B}$ denote the flag variety of $G$, which can be defined as
the symmetric space $G/B$ for $B < G$ a Borel subgroup. Then, $\Nil(\mfg) \subseteq \mfg^*$ consists of the union of $\mathfrak{b}^\perp := \{\phi \mid \phi(x) = 0, \forall x \in \mathfrak{b}\}$ over all Borels.  
\begin{theorem}\label{t:spr-borel}
\begin{enumerate}
\item[(i)] (Springer resolution) There is a symplectic resolution $T^*
  \mathcal{B} \to \Nil(\mfg)$, which is the composition
\[
T^* \mathcal{B} = \{(\mathfrak{b}, x) \mid x \in \mathfrak{b}^\perp\}
\subseteq (\mathcal{B} \times \mfg^*) \onto \Nil(\mfg),
\]
where the last map is the second projection;
\item[(ii)] (Beilinson-Bernstein) There is an isomorphism
  $(U\mfg)^{\eta_0} \iso \Gamma(\mathcal{B}, \caD_{\mathcal{B}})$,
  where $\eta_0$ is the augmentation character, i.e., 
$\ker(\eta_0)$ acts by zero on the trivial representation of
  $\mfg$;
\item[(iii)] (Beilinson-Bernstein) Taking global sections yields an
  equivalence of abelian categories, $\caD_{\mathcal{B}}\text{-mod} \iso
  (U\mfg)^{\eta_0}\text{-mod}$.
\end{enumerate}
\end{theorem}
\begin{remark}
The theorem generalizes in order to replace the augmentation character
$\eta_0$ by arbitrary characters $\eta$, at the price of replacing the category
of $D$-modules $\mathcal{D}_B\text{-mod}$ by the category of twisted $D$-modules,
with twisting corresponding to the character $\eta$.
\end{remark}
\subsection{Quivers and preprojective algebras}\label{ss:quiver}
There is a very important generalization of tensor algebras, and hence
also of finitely presented algebras, that replaces a set of variables
$x_1, \ldots, x_n$ by a directed graph, which is typically called a
\emph{quiver}; the variables $x_1, \ldots, x_n$ then correspond to the
directed edges, which are called \emph{arrows}.  Note that it may seem
odd to have a special name for a completely ordinary object, the
directed graph, but we have P.~Gabriel to thank for this suggestive
renaming: essentially, whenever you see the word ``quiver,'' you
should realize this is merely a directed graph, except that one is
probably interested in representations of the quiver, or equivalently
of its path algebra, as defined below.
\begin{definition} A quiver is a directed graph, whose directed edges
  are called arrows. Loops (arrows from a vertex to itself) and
  multiple edges (multiple arrows with the same endpoints) are
  allowed.
\end{definition}
Typically, a quiver $Q$ has its set of vertices denoted by $Q_0$ and
its set of arrows denoted by $Q_1$.
\begin{definition}
  A representation $(\rho,(V_i)_{i \in Q_0})$ 
 of a quiver $Q$ is an assignment to each
  vertex $i \in Q_0$ a vector space $V_i$, and to each arrow $a: i \to
  j$ a linear map $\rho(a): V_i \to V_j$.  The dimension vector of
  $\rho$ is defined as $d(\rho) := (\dim V_i)_{i \in Q_0} \in
  \bZ_{\geq 0}^{Q_0}$.
\end{definition}
\begin{definition}
For every dimension vector $d \in \bZ_{\geq 0}^{Q_0}$, let
$\Rep_d(Q)$ be the set of representations of the form $(\rho, (\bk^{d_i}))$, i.e.,
with $V_i = \bk^{d_i}$ for all $i$.
\end{definition}
\begin{definition}
  The path algebra $\bk Q$ of a quiver $Q$ is defined as the vector
  space with basis the set of paths in the quiver $Q$ (allowing paths
  of length zero at each vertex), with multiplication of paths given
  by reverse concatenation: if $p: i \to j$ is a path from $i$ to $j$
  and $q: j \to \ell$ is a path from $j$ to $\ell$, then $qp: i \to
  \ell$ is the concatenated path from $i$ to $\ell$.  The product of
  two paths that cannot be concatenated (because their endpoints don't
  match up) is zero.
\end{definition}
Given an arrow $a \in Q_1$, let $a_h$ be its head (the incident vertex
it points to) and $a_t$ be its tail (the incident vertex the arrow
points away from), so that $a: a_t \to a_h$. Note that, viewing $a_t$
and $a_h$ as zero-length paths, we have $a = a_h a a_t$ (because we
are using reverse concatenation).  We further observe that every
zero-length path defines an idempotent in the path algebra: in particular,
$a_h$ and $a_t$ are idempotents.
\begin{exercise} \label{exer:quiver1}
\begin{enumerate}
  \item[(i)] Show that a representation of $Q$ is the same as a
    representation of the algebra $\bk Q$.
  \item[(ii)] Show that the set of representations $\Rep_{d}(Q)$ is
    canonically isomorphic to the vector space $\bigoplus_{a \in Q_1}
    \Hom(\bk^{d(a_t)}, \bk^{d(a_h)})$.
\end{enumerate}
\end{exercise}
\begin{exercise} \label{exer:quiver2}
Suppose that $Q$ has only one vertex, so that $Q_1$
  consists entirely of loops from the vertex to itself. Then show that
  $\bk Q$ is the tensor algebra over the vector space with basis
  $Q_1$.
\end{exercise}
\begin{definition}
  Given a quiver $Q$, the double quiver $\bar Q$ is defined as the
  quiver with the same vertex set, $\bar Q_0 := Q_0$, and with double
  the number of arrows, $\bar Q_1 := Q_1 \sqcup Q_1^*$, obtained by
  adding, for each arrow $a \in Q_1$, a reverse arrow, $a^* \in Q_1$,
  such that if $a: i\to j$ has endpoints $i$ and $j$, then $a^*: j \to
  i$ has the same endpoints but with the opposite orientation; thus
  $Q_1^* := \{a^* \mid a \in Q_1\}$.
\end{definition}
\begin{exercise}\label{exer:quiver3}
  Show that there is a canonical isomorphism $\Rep_{d}(\bar Q) \cong
  T^* \Rep_d(Q)$.
\end{exercise}
\begin{definition}
Let $\lambda \in \bk^{Q_0}$.  Then the \emph{deformed preprojective algebra} 
$\Pi_\lambda(Q)$ (defined in \cite{CrawleyBoeveyHolland}) is defined as 
\begin{equation}
\Pi_\lambda(Q) := \bk \bar Q / (\sum_{a \in Q_1} aa^*-a^* a - \sum_{i \in Q_0} \lambda_i).
\end{equation}
The (undeformed) preprojective algebra is $\Pi_0(Q)$ (defined in \cite{GP}).
\end{definition}
The algebra $\Pi_\lambda(Q)$ is filtered by the length of paths:
$(\Pi_\lambda(Q))_{\leq m}$ is the span of paths of length $\leq m$.
\begin{exercise}\label{exer:quiver4}
\begin{enumerate}
\item[(i)] Show that $\Pi_0(Q)$ is actually graded by path length.
\item[(ii)] Show that there is a canonical surjection $\Pi_0(Q) \to \gr \Pi_\lambda(Q)$ (Hint: observe that the relations are deformed, and refer to Exercise
\ref{exer:def-reln}).
\item[(iii)] Give an example to show that this surjection need not be
  flat in general (i.e., $\Pi_\lambda(Q)$ is not a (flat) filtered
  deformation in general). Hint: Try a quiver with one arrow and two
  vertices.
\begin{remark}\label{r:pil-flat}
  It is a deep fact that, whenever $Q$ is not Dynkin (i.e., the
  underlying graph forgetting orientation is not Dynkin), then
  $\Pi_\lambda(Q)$ is always a (flat) filtered deformation of
  $\Pi_0(Q)$, but we will not prove this here (it follows from the
  fact that $\Pi_0(Q)$ satisfies a quiver analogue of the Koszul
  property: it is Koszul over the semisimple ring $\bk Q_0$).
\end{remark}
\end{enumerate}
\end{exercise}
\begin{definition} For $d = (d_i) \in \bZ_{\geq 0}^{Q_0}$,
  let $\mfg(d) := \prod_{i \in Q_0} \mathfrak{gl}(d_i)$.  The
  \emph{moment map} on $\Rep_d(\bar Q)$ is the map $\mu: \Rep_d(\bar
  Q) \to \mfg(d)$ defined by
\begin{equation}
  \mu(\rho) = \sum_{a \in Q_1} \rho(a) \rho(a^*) - \rho(a^*) \rho(a).
\end{equation}
\end{definition}
Let $\Rep_d(\Pi_\lambda(Q))$ be
the set of representations of $\Pi_\lambda(Q)$ on $(\bk^{d_i})$, i.e., the set of
$\bk Q_0$-algebra homomorphisms 
\[
\Pi_\lambda(Q) \to \End(\bigoplus_{i \in Q_0} \bk^{d_i}),
\]
equipping the RHS with the $\bk Q_0$-algebra structure where each
vertex $i \in Q_0$ is the projection to $\bk^{d_i}$. In other words,
this is the set of representations on $\bigoplus_i \bk^{d_i}$ with
each arrow $a \in \bar Q_1$ given by a map from the $a_t$-factor
$\bk^{d_{a_t}}$ to the $a_h$-factor $\bk^{d_{a_h}}$.
\begin{exercise}\label{exer:quiver5}
  Show that there is a canonical identification
  $\Rep_d(\Pi_\lambda(Q)) = \mu^{-1}(\lambda \cdot \Id)$.
\end{exercise}
Note that, since $\Pi_0(Q)$ is not commutative, the deformation
$\Pi_\lambda(Q)$ (in the case it is flat) does \emph{not} give
$\Pi_0(Q)$ a Poisson structure (this would not make sense, since for
us Poisson algebras are by definition commutative).
\begin{remark}
  The above is closely related to an important example of
  quantization, namely the quantized quiver varieties (this is known
  since the origin of preprojective algebras, but see, e.g.,
  \cite{BL-Ecqqv} for a recent paper on the subject, where the algebra
  is denoted $A_\lambda(v)$, setting $v=d$ and $w=0$).  These are
  defined as follows. Letting $G(d) := \prod_{i \in Q_0} \GL(d_i)$, we
  can consider the variety $\Rep_d(\Pi_\lambda(Q))/\!/G(d)$
  parameterizing representations of $\Pi_\lambda(Q)$ of dimension
  vector $d$ \emph{up to isomorphism}.  This turns out to be Poisson,
  i.e., its algebra of functions,
  $\cO(\Rep_d(\Pi_\lambda(Q)))^{G(d)}$, is Poisson. To quantize it,
  one replaces $T^*\Rep_d(Q)$ by its quantization $\caD(\Rep_d(Q))$,
  and hence replaces the Poisson algebra
  $\cO(\Rep_d(\Pi_0(Q)))^{G(d)}$, which is a quotient of
  $\cO(T^*\Rep_d(Q)){G(d)}$, by the corresponding quotient of
  $\caD(\Rep_d(Q))^{G(d)}$: Let $\tr: \mfg(d) \to \bk$ be the sum of
  the trace functions $\GL(V_i) \to \bk$. Then we define
\[
A_\mu(d) := \caD(\Rep_d(Q))^{G(d)} / (\tr(x (\lambda \Id - \sum_{a
  \in Q_1} \rho(a) \rho(a^*) - \rho(a^*) \rho(a)))_{x \in \mfg(d)}),
\]
where the ideal we quotient by is two-sided (although actually equals
the same as the one-sided ideal on either side with the same
generators), and we consider the trace to be a function of $\rho$ for
every $x \in \mfg(d)$, and hence an element of $\cO(\Rep_d(Q))
\subseteq \caD(\Rep_d(Q))$, which is evidently $G(d)$-invariant.
\end{remark}
\subsubsection{Dynkin and extended Dynkin quivers}
Finally, as motivation for further study,
we briefly explain how the behavior of quivers and
preprojective algebras falls into three distinct classes, depending on
whether the quiver is Dynkin, extended Dynkin, or otherwise. Readers
not familiar with Dynkin diagrams can safely skip this subsection.  We
restrict to the case of a connected quiver, i.e., one whose underyling
graph is connected.  A quiver is Dynkin if its underlying graph is a
type $ADE$ Dynkin diagram, and it is extended Dynkin if its underlying
graph is a simply-laced extended Dynkin diagram, i.e., of types
$\tilde A_n$, $\tilde D_n$, or $\tilde E_n$ for some $n$.    Note that we already
saw one way in which the behavior changes between the Dynkin and non-Dynkin
case, in Remark \ref{r:pil-flat}.

Let us assume $\bk$ is algebraically closed of characteristic zero.
\begin{theorem}\label{t:quiv}
Let $Q$ be a connected quiver. Let $\lambda \in \bk^{Q_0}$ 
be arbitrary.
\begin{enumerate}
\item The following are equivalent: (a) $Q$ is Dynkin;
(b) The algebra $\Pi_\lambda(Q)$
  is finite-dimensional; and (c) the quiver $Q$ has finitely many
  indecomposable representations up to isomorphism;
\item The following are equivalent: (a) the quiver $Q$ is extended Dynkin;
(b) the algebra $\Pi_0(Q)$ has an
infinite-dimensional center; (c) the center of $\Pi_0(Q)$ is
of the form $\cO(\bA^2)^\Gamma = \cO(\bA^2/\Gamma)$ for a finite
subgroup $\Gamma < \SL_2(\bk)$.
\end{enumerate}
\end{theorem}
The correspondence between extended Dynkin quivers $Q$ and finite
subgroups $\Gamma < \SL_2(\bk)$ is called the \emph{McKay
  correspondence}: when $Q$ is of type $\tilde A_n$, then $\Gamma =
\bZ/n$ is cyclic; when $Q$ is of type $\tilde D_n$, then $\Gamma$ is a
double cover of a dihedral group; and when $Q$ is of type $\tilde
E_n$, then it is a double cover of the tetrahedral, octahedral, or
icosahedral rotation group. See also Remark \ref{r:McKay} and
Wemyss's chapter,
\S 5.
\begin{remark}
  The various parts of the theorem appeared as follows: Equivalence
  (1).(a) $\Leftrightarrow$ (1).(b) is due to Gelfand and Ponomarev
  \cite{GP} (at least for $\lambda = 0$) and the equivalence (1).(a)
  $\Leftrightarrow$ (1).(c) is due to Peter Gabriel \cite{Gab-ud1}.
  Equivalences (2).(a) $\Leftrightarrow$ (2).(b) and (2).(a)
  $\Leftrightarrow$ (2).(c) follow from the stronger statements that,
  when $Q$ is neither Dynkin nor extended Dynkin, then the center of
  $\Pi_0(Q)$ is just $\bk$ (\cite[Proposition 8.2.2]{CBEG}; see also
  \cite[Theorem 1.3.1]{EGncci}), whereas when $Q$ is extended Dynkin,
  then the center is $\cO(\bA^2)^\Gamma$ \cite{CrawleyBoeveyHolland}.  The theorem also
  extends to characteristic $p$ without much modification (only
  (2).(c) needs to be modified); see \cite[Theorem 10.1.1]{Shochv1}.
\end{remark}
In the next remark and later on, we will need to use skew product algebras:
\begin{definition}\label{d:skew-prod}
Let $\Gamma$ be a finite (or discrete) group acting by automorphisms on an algebra $A$. The skew (or smash) product $A \rtimes \Gamma$ is defined as the
algebra which, as a vector space, is the tensor product $A \otimes \bk[\Gamma]$,
with the multiplication 
\[
(a_1 \otimes g_1)(a_2 \otimes g_2) := a_1 g_1(a_2) \otimes g_1 g_2.
\]
\end{definition}
\begin{remark}
  In the extended Dynkin case, when $\lambda$ lies in a particular
  hyperplane in $\bk^{Q_0}$, then $\Pi_\lambda$ also has an infinite
  center, and this center $Z(\Pi_\lambda)$ gives a commutative
  deformation of $\cO(\bA^2/\Gamma)$ which is the versal deformation.
  For general $\lambda$, $\Pi_\lambda$ is closely related to a
  symplectic reflection algebra $H_{t,c}(\Gamma)$ deforming
  $\cO(\bA^2) \rtimes \Gamma$ (see Bellamy's chapter)
  or Example
  \ref{ex:sra} below for the definition of this algebra): there is an
  idempotent $f \in \bC[\Gamma]$ such that $\Pi_\lambda = f
  H_{t,c}(\Gamma) f$ for $t,c$ determined by $\lambda$, and this makes
  $\Pi_\lambda$ Morita equivalent to the symplectic reflection
  algebra; this follows from \cite{CrawleyBoeveyHolland}.  Similarly,
  $\cO(\bA^2/\Gamma) = e\Pi_0 e$ for $e \in \bk Q$ the idempotent
  corresponding to the extending vertex, and $e \Pi_\lambda e$ yields
  the spherical symplectic reflection algebra, $e H_{t,c}(\Gamma)e$.
\end{remark}
\begin{remark}
  There is also an analogue of (1).(c) for part two (due to
  \cite{Naz-rqit} and \cite{DF-rtfgaa}, see also \cite{DR-rga}):
  Continue to assume that $\bk$ has characteristic zero (or, it
  suffices for this statement for it to be infinite). Then, a quiver
  $Q$ is extended Dynkin if and only if, for each dimension vector,
  the isomorphism classes of representations can be parameterized by
  finitely many curves and points, i.e., the variety of representations
  modulo equivalence has dimension $\leq 1$, and there exists a
  dimension vector for which this has dimension one (i.e., it does not
  have finitely many indecomposable representations for every
  dimension vector, which would imply it is Dynkin by the theorem).
  In fact,  the only dimension vectors for which there are
  infinitely many indecomposable representations are the imaginary
  roots (the dimension vectors that are in the kernel of the Cartan
  matrix associated to the extended Dynkin diagram), and for these all
  but finitely many indecomposable representations are parameterized
  by the projective line $\bP^1$.
\end{remark}

\subsection{Exercises}

Exercises from the notes: \ref{exer:weyl-diff}, \ref{exer:weylv},
\ref{exer:weyl-gr}, \ref{exer:weyl-basis}, \ref{exer:def-reln},
\ref{exer:poissbr}, \ref{exer:zeropbr}, \ref{exer:pb-lie},
\ref{exer:weyl-symv}, and \ref{exer:diffb}; and the
exercises on quivers: \ref{exer:quiver1}, \ref{exer:quiver2},
\ref{exer:quiver3}, \ref{exer:quiver4}, and \ref{exer:quiver5}.

Additional exercises:
\begin{exercise}\label{exer:nonflat-def} We elaborate on the final point of
  Exercise \ref{exer:def-reln}, giving an example where a filtered
  deformation of homogeneous relations can yield an algebra of smaller
  dimension. Consider the quadratic algebra $B = \bk\langle
  x,y \rangle / (xy,yx)$ (by quadratic, we mean presented by quadratic
  relations, i.e., homogeneous relations of degree two). First show
  that $B$ has a basis consisting of monomials in either $x$ or $y$
  but not both, and hence it is infinite-dimensional. Now consider the
  family of deformations parameterized by pairs $(a,b) \in \bk^2$,
  given by
\[
A_{a,b} := \bk\langle x,y \rangle / (xy-a,yx-b),
\] 
i.e., this is a family of filtered algebras obtained by deforming the
relations of $B$.  Show that, for $a \neq b$, we get $A_{a,b} = \{0\}$, 
the zero ring.  Hence $A_{a,b}$ is not a (flat)
filtered deformation of $A_{0,0}$ for $a \neq b$;
equivalently, the surjection of
Exercise \ref{exer:def-reln} is not an isomorphism for $a \neq b$.

On the other hand, show that, for $a = b \neq 0$, then $A_{a,a} \cong
\bk[x,x^{-1}]$, and verify that the basis we obtained for $B$
(monomials in $x$ or $y$ but not both) gives also a basis for
$A_{a,a}$. So the family $A_{a,b}$ \emph{is} flat along the diagonal
$\{(a,a)\} \subseteq \bk^{2}$.
\end{exercise}

\begin{exercise}\label{exer:weyl-usl2}
(a) Verify, using Singular, that $\Weyl_1^{\bZ/2} \cong (U \mfsl_2)^\eta$
where $\eta(C)=\eta(ef+fe+\frac{1}{2}h^2)=-\frac{3}{8}$.

How to do this: Type the commands
\begin{verbatim}
LIB "nctools.lib";
def a = makeWeyl(1);
setring a;
a;
\end{verbatim}
Now you can play with the Weyl algebra with variables $D$ and $x$;
$D$ corresponds to $\partial_x$.  Now you need to figure out some
polynomials $e(x,D), f(x,D)$, and $h(x,D)$ so that
\[
[e,f] = h, \quad [h,e] = 2e, \quad [h,f] = -2f.
\]
Then, once you have done this, compute the value of the Casimir,
\[
C = ef+fe+\frac{1}{2} h^2,
\]
and verify it is $-\frac{3}{8}$.

Hint: the polynomials $e(x,D), f(x,D)$, and $h(x,D)$
 should be homogeneous quadratic polynomials (that way
the bracket is linear).

For example, try first
\begin{verbatim}
def e=D^2;
def f=x^2;
def h=e*f-f*e;
h*e-e*h;
2*e;
h*f-f*h;
2*f;
\end{verbatim}
and you see that, defining $e$ and $f$ as above and $h$ to be $ef-fe$
as needed, we don't quite get $he-eh=2e$ or $hf-fh=2f$. But you can
correct this...

(b) Show that the highest weight of a highest weight
  representation of $(U\mfsl_2)^\eta$, for this $\eta$, is either
  $-\frac{3}{2}$ or $-\frac{1}{2}$. Here, a highest weight
  representation is one generated by a vector $v$ such that $e \cdot v
  = 0$ and $h \cdot v = \mu v$ for some $\mu \in \bk$.  Then, $\mu$ is
  called its highest weight.

  The value $-\frac{1}{2}$ is halfway between the value $0$ of the
  highest weight of the trivial representation of $\mfsl_2$ and the
  value $-1$ of the highest weight of the Verma module for the unique
  $\chi$ such that $(U \mfsl_2)^\chi$ has infinite Hochschild
  dimension (see Remark \ref{r:hd} for the definition of Hochschild
  dimension).
\begin{remark}
  Here, a Verma module of $\mfsl_2$ is a module of the form $U \mfsl_2
  / (e, h-\mu)$ for some $\mu$; this Verma module of highest weight
  $-1$ is also the unique one that has different central character
  from all other Verma modules, since the action of the Casimir on a
  Verma with highest weight $\mu$ is by $\frac{1}{2}\mu^2 + \mu$, so
  $C$ acts by multiplication by $-\frac{1}{2}$ on this Verma and no
  others.  Therefore, the Bernstein-Gelfand-Gelfand \cite{BGG}
  category $\cO$ of modules for $U \mfsl_2$ (these are defined as
  modules which are finitely generated, have semisimple action of the
  Cartan subalgebra, and are such that, for every vector $v$, $e^N v =
  0$ for some $N \geq 0$) which factor through the central quotient $U
  \mfsl_2 / (C+\frac{1}{2})$ is equivalent to the category of vector
  spaces, with this Verma as the unique simple object. This is the
  only central quotient $U \mfsl_2 / (C-(\frac{1}{2}\mu^2+\mu))$, with
  $\mu$ integral, having this property.  Also, the corresponding
  Cherednik algebra $H_{1,c}(\bZ/2)$ considered in
  Bellamy's chapter,
  of which $U \mfsl_2/(C+\frac{1}{2})$ is the
  spherical subalgebra, has semisimple category $\cO$ as defined there
  (this algebra has two simple Verma modules: one of them is killed by
  symmetrization $V \mapsto eVe$, so only one yields a module over the
  spherical subalgebra).
\end{remark}
(c) Identify the representation $\bk[x^2]$ over $\Weyl_1^{\bZ/2}$ with a highest weight representation of $(U\mfsl_2)^\eta$ (using an appropriate choice
of $e, h$, and $f$ in $\Weyl_1^{\bZ/2}$).

(d) Use the isomorphism $\Weyl_1^{\bZ/2} \cong (U\mfsl_2)^\eta$ and
representation theory of $\mfsl_2$ to show that $\Weyl_1^{\bZ/2}$
admits no finite-dimensional irreducible representations.  Hence it
admits no finite-dimensional representations at all.  (Note: in
Exercise \ref{exer:weyl-simple} below, we will see that
$\Weyl_1^{\bZ/2}$ is in fact simple, which strengthens this, and we
will generalize it to $\Weyl(V)^{\bZ/2}$ for arbitrary $V$ and $\bZ/2$ acting
by $\pm \Id$.)
\end{exercise}

We will often need the notion of \emph{Hilbert series}:
\begin{definition}\label{d:hs}
  Let $B$ be a graded algebra, $B = \bigoplus_{m \in \bZ} B_m$, with
  all $B_m$ finite-dimensional.  Then the Hilbert series is defined as
  $h(B;t) = \sum_{m \in \bZ} \dim B_m t^m$.
\end{definition}

\begin{exercise}\label{exer:flat-hilb}
  In this exercise, we will learn another characterization of flatness
  of deformations via Hilbert series.
Let $A$ and $B$ be as in Exercise \ref{exer:def-reln}.  Suppose first
that $B$ is finite-dimensional.  Show that $A$ is a (flat) filtered
deformation of $B$, i.e., the surjection $B \to \gr(A)$ is an
isomorphism, if and only if $\dim A = \dim B$.  More generally,
suppose that $B$ has finite-dimensional weight spaces. Show then that
$\gr(A)$ automatically also has finite-dimensional weight spaces, and
that $A$ is a (flat) filtered deformation if and only if $h(\gr A;t)
:= h(B;t)$.
\end{exercise}

\begin{exercise}\label{exer:grob}
Now we use GAP to try to play with flat deformations. In these
  examples, you may take the definition of a flat filtered deformation
  of $B$ to be an algebra $A$ given by a filtered deformation of the relations of $B$ such that $h(\gr A;t) = h(B;t)$.

  One main technique to use is that of Gr\"obner bases. See
  Rogalski's chapter
  for the necessary background on these.

  Here is code to get you started for the universal enveloping algebra
  of $\mfsl_2$:
\begin{verbatim}
LoadPackage("GBNP");
A:=FreeAssociativeAlgebraWithOne(Rationals, "e", "f", "h");
e:=A.e;; f:=A.f;; h:=A.h;; o:=One(A);;
uerels:=[f*e-e*f+h,h*e-e*h-2*e,h*f-f*h+2*f];
uerelsNP:=GP2NPList(uerels);;
PrintNPList(uerelsNP);
GBNP.ConfigPrint(A);
GB:=SGrobner(uerelsNP);;
PrintNPList(GB);
\end{verbatim}
This computes the Gr\"obner basis for the ideal generated by the relations.
You can also get explicit information about how each Gr\"obner
basis element is obtained
from the original relations:
\begin{verbatim}
GB:=SGrobnerTrace(uerelsNP);;
PrintNPListTrace(GB);
PrintTracePol(GB[1]);
\end{verbatim}
Here the second line gives you the list of Gr\"obner basis elements,
and for each element $GB[i]$, the third line will tell you how it is
expressed in terms of the original relations.


(a) Verify that, with the relations, $[x,y]=z, [y,z]=x,
[z,x]=y$, one gets a flat filtered deformation of $\bk[x,y,z]$, by
looking at the Gr\"obner basis. Observe that this is the enveloping
algebra of a three-dimensional Lie algebra isomorphic to $\mfsl_2$.

  (b) Now play with modifying those relations and see that, for most
  choices of filtered deformations, one need not get a flat
  deformation.

  (c) Now play with the simplest Cherednik algebra: the deformation of
  $\Weyl_1 \rtimes \bZ/2$ (cf.~Definition \ref{d:skew-prod}). The
  algebra $\Weyl_1 \rtimes \bZ/2$ itself is defined by relations
  $x*y-y*x-1$ (the Weyl algebra relation), together with $z^2-1$ (so
  $z$ generates $\bZ/2$) and $x*z+z*x, y*z+z*y$, so $z$ anticommutes
  with $x$ and $y$.  Show that this is a flat deformation of
  $\Sym(\bk^2) \rtimes \bZ/2$, i.e., $\bk[x,y] \rtimes \bZ/2$.
  Equivalently, show that the latter is the associated graded algebra
  of $\Weyl_1 \rtimes \bZ/2$.

  Then, the deformation of $\Weyl_1 \rtimes \bZ/2$ 
  is given by replacing the relation $x*y-y*x-1$
  with the relation $x*y-y*x-1-\lambda \cdot z$, for $\lambda \in \bk$
  a parameter. Show that this is also a flat deformation of $\bk[x,y] \rtimes \bZ/2$ for all $\lambda$
 (again you can just compute the Gr\"obner basis).

  (d) Now modify the action of $\bZ/2$ so as to not preserve the
  symplectic form: change $y*z+z*y$ into $y*z-z*y$.  What happens to
  the algebra defined by these relations now?
\end{exercise}

\begin{exercise}\label{exer:weyl-simple}
  Prove that, over a field $\bk$ of characteristic zero, $V$ a
  symplectic vector space, and $\bZ/2 = \{\pm \Id\} < \Sp(V)$, the
  skew product $\Weyl(V) \rtimes \bZ/2$ is a simple algebra, and
  similarly show the same for $(\Weyl(V))^{\bZ/2}$.  Hint: For the
  latter, you can use the former, together with the symmetrizer
  element $e = \frac{1}{2}(1+\sigma)$, for $\bZ/2=\{1,\sigma\}$,
  satisfying $e^2 = e$ and $e(\Weyl(V) \rtimes \bZ/2)e =
  \Weyl(V)^{\bZ/2}$.

  (In fact, one can show that the same is true for $\Weyl(V) \rtimes
  \Gamma$ for arbitrary finite $\Gamma < \Sp(V)$: 
this is a much more difficult exercise you can
  try if you are ambitious. From this one can deduce that
  $(\Weyl(V))^\Gamma$ is also simple by the following Morita equivalence:
  if $e \in \bk[\Gamma]$ is the symmetrizer element, then $\Weyl(V)^\Gamma =
  e(\Weyl(V) \rtimes \Gamma)e$ and $(\Weyl(V) \rtimes \Gamma)e (\Weyl(V) \rtimes
  \Gamma) = \Weyl(V) \rtimes \Gamma$.)
\end{exercise}




For the next exercises, we need to recall the notions of graded
and filtered modules.

Recall that for a graded algebra $B = \bigoplus_m B_m$, a graded
module is a module $M = \bigoplus_m M_m$ such that $B_i \cdot M_j
\subseteq M_{i+j}$ for all $i,j \in \bZ$.  As before, we will call
these gradings the weight gradings. We say that $\phi: M \to
N$ has weight $k$ if $\phi(M_m) \subseteq N_{m+k}$ for all $m$.
If $M$ is a graded right $B$-module and $N$ is a graded
left $B$-module, then $M \otimes_B N$ is a graded vector space,
placing $M_m \otimes_B N_n$ in weight $m+n$.

This induces gradings also on the corresponding derived functors,
i.e., $M \otimes_B N$ becomes canonically graded and $\Ext_B^i(M,N)$ has
a notion of weight $k$ elements.
Explicitly, if $M$ is a graded left $B$-module and $P_\bullet
\twoheadrightarrow M$ is a graded projective resolution, then 
weight $k$ cocycles of the complex $\Hom_B(P_\bullet, N)$ induce
weight $k$ elements of $\Ext_B^\bullet(M,N)$.  Similarly if $M$ is a
graded right $B$-module and $P_\bullet
\twoheadrightarrow M$ is a graded projective resolution, then the
complex $P_\bullet \otimes_B N$ computing $\Tor_\bullet(M,N)$ is a
graded complex.

Given a nonnegatively filtered algebra $A$, a (good) nonnegatively
filtered module $M$ is one such that $M_{\leq -1} = 0 \subseteq
M_{\leq 0} \subseteq \cdots$, with $M = \bigcup_m M_{\leq m}$, and
$A_{\leq m} M_{\leq n} \subseteq M_{\leq (m+n)}$. (Note that
everything generalizes to the case where $B$ is $\bZ$-graded and $M$
is $\bZ$-filtered, if we replace this condition by $\bigcap_m M_{\leq m} =
\{0\}$.)

\begin{exercise}
\label{exer:hoch-res} Given an algebra $B$, to understand its
  Hochschild cohomology (which governs deformations of associative
  algebras!) as well as many things, we need to construct resolutions.
  A main tool then is \emph{deformations} of resolutions.

Suppose that $B$ is nonnegatively graded, i.e., $B = \bigoplus_{m \geq
  0} B_m$ for $B_m$ the degree $m$ part of $B$.  Let $A$ be a filtered
deformation (i.e., $\gr A = B$) and $M$ be a nonnegatively filtered
$A$-module. (We could also work for this problem in the $\bZ$-graded
and filtered setting, but do not do so for simplicity.)

Suppose that $Q_\bullet \to M$ is a nonnegatively filtered complex 
with $Q_i$ projective 
such that $\gr Q_\bullet \to \gr M$ is a projective resolution of
$\gr M$.  Show that $Q_\bullet \to M$ is a projective resolution of
$M$.


Hint: Show that $\dim H_i(\gr C_\bullet) \geq \dim H_i(C_\bullet)$ for
an arbitrary filtered complex $C_\bullet$, because $\ker (\gr C_i \to
\gr C_{i-1}) \supseteq \gr \ker(C_{i} \to C_{i-1})$ and $\im(\gr C_{i+1}
\to \gr C_i) \subseteq \gr \im (C_{i+1} \to C_i)$. Conclude, for an
arbitrary filtered complex, that exactness of $\gr(C_\bullet)$ implies
exactness of $C_\bullet$ (the converse is not true).
\end{exercise}


\begin{proposition}\label{p:def-res}
  Suppose that $B = TV/(R)$ for some homogeneous relations $R
  \subseteq R$, and that $B$ has a nonnegatively graded
  free $B$-bimodule
  resolution, $P_\bullet \onto B$, which is finite-dimensional in each
  weighted degree. If $E \subseteq TV$ is a filtered
  deformation of the homogeneous relations $R$, i.e., $\gr E = R$,
  then the following are equivalent:
\begin{itemize}
\item[(i)] $A := TV / (E)$ is a flat filtered deformation of $B$,
  i.e., the canonical surjection $B \onto \gr A$ is an isomorphism;
\item[(ii)] The resolution $P_\bullet \to B$ deforms to a filtered free
  resolution $Q_\bullet \to A$.
\item[(ii')] The resolution $P_\bullet$ deforms to a filtered complex
  $Q_\bullet \to A$.
\end{itemize}
In this case, the deformed complex in (ii) is a free resolution of
$A$.
\end{proposition}
We remark that the existence of $P_\bullet$ is actually automatic,
since $B$ always admits a unique (up to isomorphism) minimal free
resolution (where minimal means that, for each $i$, $P_i = B_i \otimes
V_i$ for $V_i$ a nonnegatively graded vector space with minimum
possible Hilbert series, i.e., for any other choice $V_i'$,
we have that $h(V_i';t)-h(V_i;t)$ has nonnegative coefficients).
 For the minimal resolution, each $P_i$ is in
degrees $\geq i$ for all $i$, and is finite-dimensional in each
degree, which implies that $\bigoplus_i P_i$ is also
finite-dimensional in each degree.
\begin{proof}[Sketch of Proof]
  It is immediate that (ii) implies (ii'), and also that (ii') implies
  (i), since part of the statement is that $\gr(A) \cong B$.  Also, by
  the exercise, (ii') implies (ii).

  To show (i) implies (ii) (it admittedly does not really help to try
  to show only (ii'), inductively construct a deformation of
  $P_\bullet$ to a filtered free resolution of $A$.  It suffices to
  let $P_\bullet$ be the minimal resolution as above, since any other
  resolution is obtained from this one by summing with a split exact
  complex of free graded modules. First of all, we know that $P_0 \onto B$
  deforms to $Q_0 \onto A$, since we can arbitrarily lift the map $P_0
  = B^{r_0} \to B$ to a filtered map $Q_0:=A^{r_0} \to A$, and this
  must be surjective since $B \onto \gr A$ is surjective.  Since in
  fact it is an isomorphism, the Hilbert series (Definition
  \ref{d:hs}) of the kernels are the same, and so we can arbitrarily
  lift the surjection $P_1 =B^{r_1} \onto \ker(P_0 \to B)$ to a
  filtered map $Q_1 := A^{r_1} \onto \ker(Q_0 \to A)$, etc. Fill in
  the details!
\end{proof}

In the case $B$ is Koszul, we can do better using the form of the
Koszul complex. This is divided over the next three exercises. In the
first exercise we explain augmented algebras. In the next exercise, we
explain quadratic and Koszul algebras. In the third exercise, we
improve the above result in the case $B$ is Koszul.

We begin with the definition and an exercise on augmented algebras:
\begin{definition}
An augmented algebra is an associative $\bk$-algebra $B$ together
with an algebra homomorphism $B \to \bk$. The \emph{augmentation ideal}
is the kernel, $B_+ \subseteq B$.
\end{definition}
\begin{exercise}\label{exer:aug}
(a) Show that an augmented algebra is equivalent to the data of
an algebra $B$ and a codimension-one ideal $B_+ \subseteq B$.

(b) Show that $TV$ is augmented with augmentation ideal $(TV)_+
  := \bigoplus_{m \geq 1}V^{\otimes m}$.

(c) Show that an augmented algebra $B$ is always of the
form $TV/(R)$ for $V$ a vector space and
 $R \subseteq (TV)_+$.

 (d) 
Suppose now that $R \subseteq ((TV)_+)^2$.\footnote{One can show that,
  if we complete $B$ and $TV$ with respect to the augmentation ideals,
  we can always assume this; for simplicity here we will not take the
  completion.}  Then we have an isomorphism $V \cong \Tor_1(\bk, \bk)$
(for $B = TV/(R)$).  Hint: one has a projective $B$-module resolution
of $\bk$ of the form
\begin{equation}\label{e:kr-2}
\cdots \to B \otimes R \to B \otimes V \to B \twoheadrightarrow \bk.
\end{equation}
Here the maps are given by
\[
b \otimes \sum_i (v_{i,1} \cdots v_{i,k_i}) \mapsto \sum_i
b (v_{i,1} \cdots v_{i,k_i-1}) \otimes v_{i,k_i}, \quad
b \otimes v \mapsto bv.
\]
In other words, the first map is obtained by restricting to $B \otimes R$
the splitting map
$B \otimes (TV)_+ \to B  \otimes V$, given by linearly
extending to all of $B \otimes (TV)_+$ the assignment
\[
b \otimes v_1 \cdots v_k \mapsto b v_1 \cdots v_{k-1} \otimes v_k.
\]
Then show that, when $R \subseteq ((TV)_+)^2$, then after applying the functor
$M \mapsto M \otimes_B \bk$ to the sequence, the last two differentials (before the
map $B \twoheadrightarrow \bk$) become zero.

(e) Continue to assume $R \subseteq ((TV)_+)^2$.  Assume that $R \cap
(RV+VR) = \{0\}$, where $(RV+VR)$ denotes
the ideal, i.e., $(TV)_+ R (TV) + (TV) R (TV)_+$; this is a minimality
condition.
 Show in this case that $\Tor_2(\bk, \bk) \cong R$.  In
particular, if $R$ is spanned by homogeneous elements,
 show that, replacing $R$ with a suitable subspace,
it satisfies this condition, and then $\Tor_2(\bk, \bk)
\cong R$.  Hint:
Show that the projective resolution \eqref{e:kr-2} can be
extended to
\begin{equation}\label{e:kr-3}
\cdots \to B \otimes S \to B \otimes R \to B \otimes V \to B \twoheadrightarrow \bk,
\end{equation}
where $S := (TV \cdot R) \cap (TV \cdot R \cdot TV_+)$ (or any
subspace which generates this as a left $TV$-module).  Show that the
multiplication map map $TV \otimes R \to TV \cdot R$ is injective.
Then, the map $B \otimes S \to B \otimes R$ is given by the
restriction of
\[
B \otimes (TV \otimes R) \to B \otimes R, \quad b \otimes (f \otimes r)
\mapsto bf \otimes r.
\]
Check that this gives the exact
sequence \eqref{e:kr-3} (beginning with $B \otimes S$).
\end{exercise}
We proceed to the definition and an exercise on quadratic algebras. 
\begin{definition}
  A quadratic algebra is an algebra of the form $B = TV/(R)$ where $R
  \subseteq T^2 V = V \otimes V$.
\end{definition}
In particular, such an algebra is nonnegatively graded, $B =
\bigoplus_{m \geq 0} B_m$, with $B_0 = \bk$.  For the rest of the
problem, whenever $B$ is graded, we will call the grading on $B$ the
\emph{weight grading}, to avoid confusion with the \emph{homological
  grading} on complexes of (graded) $B$-modules.

\begin{exercise}
\label{exer:quadratic}
 Let $B= \bigoplus_{m \geq 0} B_m$ be a nonnegatively graded
algebra with $B_0 = \bk$. Show that $B$ is quadratic if and only if
$\Tor_2(\bk, \bk)$ is concentrated in weight two.  In this case,
conclude that $B \cong TV/(R)$ where $V \cong \Tor_1(\bk,\bk)$ and $R
\cong \Tor_2(\bk,\bk)$.  
\end{exercise}
Now we define and study Koszul algebras:
\begin{definition} A Koszul algebra is a nonnegatively graded algebra
  $B$ with $B_0 = \bk$ such that $\Tor^B_i(\bk,\bk)$ is concentrated in
  weight $i$ for all $i \geq 1$.
\end{definition}
We remark that the above is often stated dually as:
$\Ext_B^i(\bk,\bk)$ is concentrated in weight $-i$.
\begin{exercise}\label{exer:koszul}
Show that the following are equivalent:
\begin{enumerate}
\item $B$ is Koszul.
\item $\bk$ has a projective resolution of the form
\begin{equation}\label{e:k}
  \cdots \to B \otimes V_i \to \cdots \to B \otimes V_1 \to B \twoheadrightarrow \bk,
\end{equation}
where $V_i$ are graded vector spaces placed in degree $i$ for all $i$,
and the differentials preserve the weight grading (or alternatively,
you can view $V_i$ as vector spaces in degree zero, and then the
differentials all increase weights by $1$).
\end{enumerate}
\end{exercise}
We now proceed to discuss a deep property of Koszul algebras which makes
them very useful in deformation theory, the \emph{Koszul deformation principle}, due to Drinfeld.
\begin{exercise}\label{exer:kdp-filt}
  (a) For a quadratic algebra $B = TV/(R)$, let $S := R\otimes V \cap
  V \otimes R$, taking the intersection in the tensor algebra.  Let $E
  \subseteq T^{\leq 2}$ be a deformation of $R$, in the sense that the
  composite map $E \to T^{\leq 2} V \twoheadrightarrow V^{\otimes 2}$ is an
  isomorphism onto $R$.  Let $A := TV / (E)$. Then we can consider $T
  := E \otimes V  \cap (V+\bk) \otimes E$, again taking the intersection in
  the tensor algebra (where $(V+\bk)$ does \emph{not} denote an ideal).

  Show first that the composition $T \to T^{\leq 3} V
  \twoheadrightarrow V^{\otimes 3}$ is an injection to $S$.

(b) Verify that the following are complexes:
\[
B \otimes S \to B \otimes R \to B \otimes V \to B \twoheadrightarrow
\bk,
\]
and
\[
A \otimes T \to A \otimes E \to A \otimes V \to A \twoheadrightarrow \bk.
\]
The maps in the above are all obtained by restriction from splitting
maps, $x \otimes \bigl(\sum_i v_i \otimes y_i) \mapsto x v_i \otimes
y_i$ for $x$ in either $A$ or $B$, $v_i \in V$, and $y_i$ in one of
the spaces $T, E, S$, or $R$; the map $A \otimes V \to A$ is the
multiplication map.

Show also that: they are exact at $B \otimes V$ and $B$, and at $A
\otimes V$ and $A$, respectively, and that the first complex is exact
in degrees $\leq 3$.

(c) Prove that, if $A$ is a flat deformation, then the map in (a) must
be an isomorphism, and hence the second sequence in (b) deforms the
first (in the sense that the first is obtained from the second by
taking the associated graded sequence, $\gr(T)= S$, $\gr(E)=R$).  Note
the similarity to Proposition \ref{p:def-res}.

(d)(*, but with the solution outlined) Finally, we state the Koszul
deformation principle, which is a converse of (c) in the Koszul
setting (see Theorem \ref{t:kdp} below for the usual version).  In
this case, by hypothesis, the first sequence in (b) is exact (in all
degrees).

\begin{theorem} \label{t:kdp-filt}
  \cite{Dri-qcrqc,PBW,BGS-Kdprt,PP-qa} (Koszul deformation principle:
  filtered version) Suppose $B = TV/(R)$ is a Koszul algebra and $E
  \subseteq T^{\leq 2}V$ is a deformation of $R$, i.e.,
  $\gr(E)=R$. Then $A = TV/(E)$ is a flat deformation of $B$ if and
  only if $\dim T = \dim S$, i.e., the injection $\gr(T) \to S$ is
  an isomorphism.
\end{theorem}
That is, in the Koszul case, $A$ is a flat deformation if and only if
$T$ is a flat deformation of $S$.  In this
case, the second sequence in (b) deforms the first one, and hence is
also exact.

We will give the more standard version of this using formal
deformations in Exercise \ref{exer:deduce-kdp-filt}.

\textbf{Prove that, in the situation of the theorem, the whole
  resolution \eqref{e:k} deforms to a resolution of $A$.}

We outline how to do this: First, you can form a minimal resolution of
$A$ analogously to \eqref{e:k}, of the form
\begin{equation}\label{e:k-def}
\cdots \to A \otimes W_i \to \cdots \to A \otimes W_1 \to A,
\end{equation}
where $W_i$ are now filtered vector spaces with $W_{\leq (i-1)} = 0$.
Note that $W_1 = V, W_2 = E$, and $W_3 = T$, and the assumption of the
theorem already gives that $\gr(W_i)=V_i$ for $i \leq 3$.

If we take the associated graded of \eqref{e:k-def}, since $\gr(A)=B$,
we obtain a complex which is exact beginning with $B \otimes S = B
\otimes V_3$. We want to show it is exact.  Suppose it is not, and
that the first nonzero homology is in degree $m \geq 3$, i.e., $\ker(B
\otimes \gr(W_{m}) \to B \otimes \gr(W_{m-1})) \neq \im(B \otimes
\gr(W_{m+1}) \to B \otimes \gr(W_m))$. 

First verify that, in this case, $\gr(W_i) = V_i$ for all $i \leq m$,
but $\gr_{m+1}(W_{m+1}) \subsetneq V_{m+1}$.

Next, taking the Hilbert series $h(\gr A;t) + \sum_{i \geq 1} (-1)^i
h(\gr(W_i);t) h(\gr A;t)$, we must get zero, since the sequence is
exact. The same fact says that $h(B;t) + \sum_{i \geq 1} (-1)^i
h(V_i;t) h(B;t) = 0$.  On the other hand, flatness is saying that $\gr
A = B$. Conclude that
\[
\sum_{i \geq 1} (-1)^i h(V_i;t) = \sum_{i \geq 1} (-1)^i h(\gr(W_i);
t).
\]
Finally, by construction, $(W_i)_{\leq (i-1)} = 0$ for all $i$, so
that $h(\gr(W_i);t)$ is zero in degrees $\leq (i-1)$.  Moreover,
$\gr(W_i) = V_i$ for $i \leq m$.  Taking the degree $m+1$ part of the
above equation, we then get $h(V_{m+1};t) = h(\gr_{m+1}(W_{m+1});t)$. This
contradicts the fact that $\gr_{m+1}(W_{m+1}) \subsetneq V_{m+1}$ (as
noted above by our assumption).
\end{exercise}

\section{Formal deformation theory and Kontsevich's
  theorem}\label{s:kont-state}
In this section, we state and study the question of deformation
quantization of Poisson structures on the polynomial algebra
$\bk[x_1,\ldots,x_n]=\cO(\bA^n)$, or more generally on $\cO(X)$ for
$X$ a smooth affine variety, with $\bk$ of characteristic zero.  We
defined the notion of Poisson structure in the previous section: such
a structure on $B:=\cO(X)$ means a Poisson bracket $\{-,-\}$ on
$B$. In this section we define the notion of deformation quantization,
which essentially means an associative product $\star$ on
$B[\![\hbar]\!]$ which, modulo $\hbar$, reduces to the usual
multiplication, and modulo $\hbar^2$, satisfies $[a,b] \equiv
\hbar\{a,b\}$, for $a,b \in \cO(X)$.  As proved by Kontsevich for
$\bk=\bR$, all Poisson structures on $X=\bA^n$ admit a quantization,
and his proof extends to the case of smooth affine (and some
nonaffine) varieties, as fleshed out by Yekutieli and others (and his
proof can even be extended to the case of general characteristic zero
fields, as shown recently in \cite{Dol-acekfqrrn}).  Moreover,
Kontsevich gives an explicit formula for the quantization in terms of
operators associated to graphs.  The coefficients of these operators
are given by certain, very interesting, explicit integrals over
configuration spaces of points in the upper-half plane, which
unfortunately cannot be explicitly evaluated in general.

Our goal is to develop enough of the definitions and background in
order to explain Kontsevich's answer, omitting the integral formulas
for the coefficients.  To our knowledge, this result itself is not
proved on its own in the literature: Kontsevich proves it as a
corollary of a more general result, his formality theorem, which led
to an explosion of literature on refinements, analogues, and related
results.

\subsection{Differential graded algebras}
In this section, we will often work with dg, i.e., differential graded
algebras.  These algebras have \emph{homological} grading, which means
that one should always think of the permutation of tensors $v \otimes
w \mapsto w \otimes v$ as carrying the additional sign
$(-1)^{|v||w|}$.  Precisely, this means the following:
\begin{remark}\label{r:hom-cohom}
  We will use superscripts for the grading on dg algebras since the
  differential increases degree by one (i.e., this is cohomological
  grading).  If one uses subscripts to indicate the grading
  (homological grading) then the differential should rather decrease
  degree by one. We will typically refer to this grading as
  ``homological'' even when it is really cohomological.
\end{remark}
\begin{definition}
  A dg vector space, or complex of vector spaces, is a graded vector
  space $V = \bigoplus_{m \in \bZ} V^m$ equipped with a linear
  differential $d: V^\bullet \to V^{\bullet + 1}$ satisfying $d(d(b))
  = 0$ for all $b \in V$.  A morphism of dg vector spaces is a linear
  map $\phi: V \to W$ such that $\phi(V^m) \subseteq W^m$ for all $m
  \in \bZ$ and $\phi \circ d = d \circ \phi$.
\end{definition}
\begin{definition}
  A dg associative algebra (or dg algebra) is a dg vector space $A =
  \bigoplus_{m \in \bZ} A^m$ which is also a graded associative
  algebra, i.e., equipped with an associative multiplication
  satisfying $A^m A^n \subseteq A^{m+n}$, such that the differential
  is a graded derivation:
\[
d(ab) = d(a)\cdot b +(-1)^{|a|} a \cdot d(b).
\]
A dg algebra morphism $A \to B$ is a homomorphism of associative algebras
which is also a morphism of dg vector spaces.
\end{definition}
\begin{definition}
  A dg commutative algebra $B$ is a dg associative algebra satisfying
  the graded commutativity rule, for homogeneous $a,b \in B$,
\begin{equation}\label{e:gr-comm}
a b = (-1)^{|a| |b|}ba.
\end{equation}
\end{definition}
A dg commutative algebra morphism is the same thing as a dg
(associative) algebra morphism (that is, commutativity adds no
constraint on the homomorphism).
\begin{definition}
  A dg Lie algebra (dgla) $L$ is a dg vector space equipped with a
  bracket $[-,-]: L \otimes L \to L$ which is a morphism of complexes
  (i.e., $[L^m, L^n] \subseteq L^{m+n}$ and
  $d[a,b]=[da,b]+(-1)^{|a|}[a,db]$) and satisfies the graded
  skew-symmetry and Jacobi identities:
\begin{gather}
[v,w] = -(-1)^{|v| |w|} [w,v], \\
[u,[v,w]] + (-1)^{|u|(|v|+|w|)} [v,[w,u]] + (-1)^{(|u|+|v|)|w|} [w,[u,v]] = 0.
\end{gather}
A dg Lie algebra morphism is a homomorphism of Lie algebras which is
also a morphism of dg vector spaces.
\end{definition}
\begin{remark}
  If you are comfortable with the idea of the category of dg vector
  spaces (i.e., complexes) as equipped with a tensor product
  (precisely, a symmetric monoidal category where the permutation is,
  as indicated above, the signed one $v \otimes w \mapsto (-1)^{|v|
    |w|} w \otimes v$), then a dg (associative, commutative, Lie)
  algebra is exactly an (associative, commutative, Lie) algebra in the
  category of dg vector spaces.
\end{remark}


\subsection{Definition of Hochschild (co)homology}
Let $A$ be an associative algebra. The Hochschild (co)homology is the
natural (co)homology theory attached to associative algebras.  We
give a convenient definition in terms of $\Ext$ and $\Tor$.  Let $A^e
:= A \otimes_\bk A^\op$, where $A^\op$ is the opposite algebra,
defined to be the same underlying vector space as $A$, but with the
opposite multiplication, $a \cdot b := ba$. Note that $A^e$-modules
are the same as $A$-bimodules (where, by definition, $\bk$ acts the
same on the right and the left, i.e., by the fixed $\bk$-vector space
structure on the bimodule).
\begin{definition} \label{d:hh-co} Define the Hochschild homology and
  cohomology, respectively, of $A$, with coefficients in an
  $A$-bimodule $M$, by
\[
\HH_i(A, M) := \Tor_i^{A^e}(A, M), \quad \HH^i(A, M) := \Ext^i_{A^e}(A,
M).
\]
Without specifying the bimodule $M$, we are referring to $M=A$, i.e.,
$\HH_i(A) := \HH_i(A,A)$ and $\HH^i(A):=\HH^i(A,A)$.
\end{definition}
In fact, $\HH^\bullet(A)$ is a \emph{ring} under the Yoneda product of
extensions: see \S \ref{ss:mod-hoch} below for more details (we will
not use this until then).  It moreover has a Gerstenhaber bracket
(which is a shifted version of a Poisson bracket), which we will
introduce and use beginning in \S \ref{ss:gerst} below.

The most important object above for us will be $\HH^2(A)$, in
accordance with the (imprecise) principle:
\begin{multline}
  \text{The space $\HH^2(A)$ parameterizes all (infinitesimal,
    filtered, or formal)} \\ 
  \text{deformations of $A$, up to obstructions (in $\HH^3(A)$) and
    equivalences (in $\HH^1(A)$).}
\end{multline}
We will give a first definition of infinitesimal and formal
deformations in the next subsection, leaving the general notion to \S
\ref{ss:fdef} below.

More generally, given any type of algebra structure, $\mathcal{P}$, on
an ordinary (not dg) vector space $B$:
\begin{multline*}
\text{The space $H_{\mathcal{P}}^2(B,B)$ parameterizes 
(infinitesimal, filtered, or formal)} \\
\text{deformations of
the $\mathcal{P}$-algebra structure on $B$,} \\
\text{up to obstructions
(in $H_{\mathcal{P}}^3(B,B)$) and equivalences (in $H_{\mathcal{P}}^1(B,B)$).}
\end{multline*}
When $\mathcal{P}$ indicates associative algebras,
$H_{\mathcal{P}}^\bullet(B,B) = \HH^\bullet(B,B)$, and for Lie,
commutative, and Poisson algebras, one gets Chevalley-Eilenberg,
Harrison, and Poisson cohomology, respectively.  For example, as we
will use later, when $\mathcal{P}$ indicates Poisson algebras, we
denote the Poisson cohomology by $\HP^\bullet(B,B)$, so that
$\HP^2(B,B)$ parameterizes Poisson deformations of a Poisson algebra
$B$, up to the aforementioned caveats.
\begin{remark}
  One way of making the above assertion precise is to make
  $\mathcal{P}$ an \emph{operad} and $B$ an algebra over this operad
  (see, e.g., \cite{MSS,LV-ao}): for example, there are commutative,
  Lie, Poisson, and associative operads, and algebras over each of
  these are commutative, Lie, Poisson, and associative algebras,
  respectively.  In some more detail, a $\bk$-linear \emph{operad}
  $\mathcal{P}$ is an algebraic structure which consists of, for every
  $m \geq 0$, a vector space $\mathcal{P}(m)$ equipped with an action
  of the symmetric group $S_m$, together with composition operations
  $\mathcal{P}(m) \times \bigl( \mathcal{P}(n_1) \times
  \mathcal{P}(n_2) \times \cdots \times \mathcal{P}(n_m) \bigr) \to
  \mathcal{P}(n_1 + n_2 + \cdots + n_m)$, for all $m \geq 1$ and $n_1,
  n_2, \ldots, n_m \geq 0$, together with a unit $1 \in
  \mathcal{P}(1)$, satisfying certain associativity and unitality
  conditions.  Then, a $\mathcal{P}$-algebra is a vector space $V$
  together with, for all $m \geq 0$ and $\xi \in \mathcal{P}(m)$, an
  $m$-ary operation $\mu_\xi: V^{\otimes m} \to V$, satisfying certain
  associativity and unitality conditions.  
  Then, for any operad $\mathcal{P}$ and any $\mathcal{P}$-algebra
  $A$, there is a natural notion of the $\mathcal{P}$-(co)homology of
  $A$, which recovers (in degrees at least two) in the
  associative case the Hochschild (co)homology; in the Lie case over
  characteristic zero, the Chevalley-Eilenberg (co)homology; and in
  the commutative case over characteristic zero, the Harrison or
  Andr\'e-Quillen (co)homology.  Moreover, one has the notion of an
  $A$-module $M$, which in the aforementioned cases recover bimodules,
  Lie modules, and modules, respectively, and one has the notion of
  $\mathcal{P}$-(co)homology of $A$ valued in $M$, which recovers (at
  least in positive degrees) the Hochschild, Chevalley-Eilenberg, and
  Harrison (co)homology valued in $M$. 

  In the case of Lie algebras $\mfg$, one can explicitly describe the
  Chevalley-Eilenberg cohomology of $\mfg$, owing to the fact that the
  category of $\mfg$-modules
  is equivalent to the category of modules over the associative
  algebra $U\mfg$. Let $H^\bullet_{\Lie}(\mfg,-)$ denote the Lie (or
  Chevalley-Eilenberg) cohomology of $\mfg$ with coefficients in Lie
  modules.  One then has (cf.~\cite[Exercise 7.3.5]{Weibel})
  $\Ext^\bullet_{U\mfg}(M, N) \cong H^\bullet_{\Lie}(\mfg, \Hom_\bk(M,
  N))$. So $H^\bullet_{\Lie}(\mfg, N) \cong \Ext^\bullet_{U\mfg}(\bk,
  N)$, the extensions of $\bk$ by $N$ as $\mfg$-modules; this differs
  from the associative case where one considers the bimodule
  extensions of $A$, rather than the Lie module extensions of $\bk$.
  (Note that, since $A$ need not be augmented, one does not
  necessarily have a module $\bk$; indeed one need not have a
  one-dimensional module at all, such as in the case where $A$ is the
  algebra of $n$ by $n$ matrices for $n \geq 2$, or when $A$ is the
  Weyl algebra in characteristic zero, which has no finite-dimensional
  modules).
%
\end{remark}
\subsection{Formality theorems}
We will also state results for $C^\infty$ manifolds, both for added
generality, and also to help build intuition. In the $C^\infty$ case
we will take $\bk = \bC$ and let $\cO(X)$ be the algebra of smooth
complex-valued functions on $X$; there is, however, one technicality:
we restrict in this case to the ``local'' part of Hochschild
cohomology as defined in Remark \ref{r:lhc} (this difference can be
glossed over in a first reading). In the affine case, in this section,
we will assume that $\bk$ is a field of characteristic zero.
\begin{theorem}[Hochschild-Kostant-Rosenberg] \label{t:hkr}
  Let $X$ be a either smooth affine variety over a field $\bk$ of
  characteristic zero or a $C^\infty$ manifold. Then, the Hochschild
  cohomology ring $\HH^\bullet(\cO(X),\cO(X)) \cong
  \wedge^\bullet_{\cO(X)} \Vect(X)$ is the ring of \emph{polyvector fields}
  on $X$.
\end{theorem}
Explicitly, a polyvector field of degree $d$ is a sum of elements of
the form
\[
\xi_1 \wedge \cdots \wedge \xi_d,
\]
where each $\xi_i \in \Vect(X)$ is a vector field on $X$, i.e., when
$X$ is an affine algebraic variety,
$\Vect(X) = \Der(\cO(X),\cO(X))$.

Thus, the deformations are classified by certain \emph{bivector
  fields} (those whose obstructions vanish). As we will explain, the
first obstruction is that the bivector field be Poisson.  Moreover,
deforming along this direction is the same as quantizing the Poisson
structure.  So the question reduces to: which Poisson structures on
$\cO(X)$ can be quantized?  It turns out, by a very deep result of
Kontsevich, that \emph{they all can.}

To make this precise, we need to generalize the notion of
quantization: we spoke about quantizing graded $\cO(X)$, but usually
this is not graded when $X$ is smooth. (Indeed, a grading would mean
that $X$ has a $\bG_m$-action, and in general the fixed point(s) will
be singular.  For example, if $\cO(X)$ is nonnegatively graded with
$(\cO(X))_0 = \bk$, i.e., the action is contracting with a single fixed
point, then $X$ is singular unless $X = \bA^n$ is an affine space.)

Instead, we introduce a formal parameter $\hbar$.\footnote{The use of
  $\hbar$ to denote the formal parameter originates from quantum
  physics, where it is actually a constant, given by Planck's
  constant divided by $2\pi$.}  Given a vector space $V$, let
$V[\![\hbar]\!] := \{\sum_{m \geq 0} v \hbar^m\}$ be the vector space
of formal power series with coefficients in $V$. If $V=B$ is an
algebra, we obtain an algebra $B[\![\hbar]\!]$. We will be interested
in assigning this a deformed associative multiplication $\star$, i.e.,
a $\bk[\![\hbar]\!]$-bilinear associative multiplication $\star:
B[\![\hbar]\!] \otimes B[\![\hbar]\!] \to B[\![\hbar]\!]$.  
\begin{exercise}\label{exer:cont-fd}
Show that such a multiplication is automatically 
continuous in the $\hbar$-adic
topology, i.e., that, for all $b_m, c_n \in B$,
\begin{equation}\label{e:star-cont}
  \bigl( \sum_{m \geq 0} b_m \hbar^m\bigr) \star \bigl( \sum_{n \geq 0}
  c_n \hbar^n \bigr) = \sum_{m,n \geq 0} (b_m \star c_n) \hbar^{m+n}.
\end{equation}
\end{exercise}
\begin{definition} \label{d:opfd}
A \emph{(one-parameter) formal deformation} of an
  associative algebra $B$ is an associative algebra $A_\hbar =
  (B[\![\hbar]\!], \star)$ such that
\begin{equation}\label{e:form-def}
 a \star b \equiv ab  \pmod \hbar, \forall a,b \in B.
\end{equation}
We require that $\star$ be associative and $\bk[\![\hbar]\!]$-bilinear.
\end{definition}
The general definition of formal deformation will be given in \S
\ref{ss:fdef} below; in contexts where it is clear we are speaking
about one-parameter formal deformations, we may omit
``one-parameter.''
\begin{remark}\label{r:sp-iso} Equivalent to the above, and often found in the
  literature, is the following alternative formulation: A
  one-parameter formal deformation is a $\bk[\![\hbar]\!]$-algebra
  $A_\hbar$, 
  together
  with an algebra isomorphism $A_\hbar/\hbar A_\hbar \iso B$, such
  that $A_\hbar$ is topologically free (i.e., isomorphic to
  $V[\![\hbar]\!]$ for some $V$, which in this case we can take to be
  any section of the map $A_\hbar \onto B$). The equivalence is given
  by taking any vector space section $\tilde B \subseteq A_\hbar$ of
  $A_\hbar \onto B$, and writing $A_\hbar = (\tilde B[\![\hbar]\!],
  \star)$ for a unique binary operation $\star$, which one can check
  must be 
  bilinear and satisfy \eqref{e:form-def}.
  Conversely, given $A_\hbar=(\tilde B[\![\hbar]\!],\star)$, one has a
  canonical isomorphism $A_\hbar/\hbar A_\hbar =B$.
\end{remark}
Modulo $\hbar^2$, we get the notion of infinitesimal deformation:
\begin{definition} \label{d:infinitesimal-def}
An \emph{infinitesimal deformation} of an associative algebra
is an algebra $(B[\varepsilon]/(\varepsilon^2), \star)$ such that $a \star b\equiv ab \pmod \varepsilon$.   
\end{definition}
Now we let $B$ be a Poisson algebra (which is automatically commutative).
\begin{definition}
  A \emph{deformation quantization} is a one-parameter formal
  deformation of a Poisson algebra $B$ which satisfies the
 identity
\begin{equation}\label{e:dq}
  a \star b - b \star a \equiv \hbar \{a,b\} \pmod{\hbar^2}, \forall a,b
  \in B.
\end{equation}
\end{definition}

We will often use the fact that, for $X$ a smooth affine variety in
odd characteristic or a smooth manifold, a Poisson structure on
$\cO(X)$ is the same as a bivector field
$\pi \in \wedge^2_{\cO(X)} \Vect(X)$, via
\[
\{f,g\} = i_\pi(df \wedge dg) := \pi(f \otimes g),
\]
satisfying the Jacobi identity (note here and in the sequel that
$i_\eta(\alpha)$ denotes the contraction of a polyvector field $\eta$
with a differential form $\alpha$).  When $X$ is affine space $\bA^n$
or a smooth manifold, this is clear.  For the general case of a smooth
affine variety, this can be shown as follows (the reader uncomfortable
with the necessary algebraic geometry can skip it and take $X=\bA^n$):
One needs to show, more generally, that $\wedge^2_{\cO(X)} \Vect(X)$
is canonically isomorphic to the vector space,
$\text{SkewBiDer}(\cO(X))$, of skew-symmetric biderivations $\cO(X)
\otimes_\bk \cO(X) \to \cO(X)$, i.e., skew-symmetric brackets
satisfying the Leibniz identity (but not necessarily the Jacobi
identity). Then, one has a canonical map $\wedge^2 \Vect(X) \to
\text{SkewBiDer}(\cO(X))$, and this is a map of $\cO(X)$-modules.
Then, the fact that it is an isomorphism is a local statement.
However, if $X$ is smooth, then $\Vect(X)$ is a projective
$\cO(X)$-module, i.e., the tangent sheaf $T_X$ is locally free.  On an
open affine subset $U \subseteq X$ such that $\Vect(U)$ is free as a
$\cO(U)$-module, it is clear that the canonical map is an isomorphism.
This implies the statement.

In terms of $\pi$, the Jacobi identity says $[\pi, \pi] = 0$, using
the Schouten-Nijenhuis bracket (see Proposition \ref{p:pb-sn}), defined as follows:
\begin{definition}\label{d:sn}
  The Schouten-Nijenhuis Lie bracket on $\wedge_{\cO(X)}^\bullet \Vect(X)$ is given
  by the formula
\begin{equation} \label{e:gerst}
[\xi_1 \wedge \cdots \wedge \xi_m, \eta_1 \wedge \cdots \wedge
\eta_n] = \sum_{i,j} (-1)^{i+j+m-1} [\xi_i, \eta_j] \wedge \xi_1
\wedge \cdots \hat \xi_i \cdots \wedge\xi_m \wedge \eta_1 \wedge
\cdots \hat \eta_j \cdots \wedge \eta_n.
\end{equation}
\end{definition}
As before, the hat indicates that the given terms are
\emph{omitted} from the product.
\begin{remark}
  Alternatively, the Schouten-Nijenhuis bracket is the Lie bracket
  uniquely determined by the conditions that $[\xi, \eta]$ is the
  ordinary Lie bracket for $\xi,\eta \in \Vect(X)$, that $[\xi,f] =
  \xi(f)$ for $\xi \in \Vect(X)$ and $f \in \cO(X)$, and such that the
  graded Leibniz identity is satisfied, for all homogeneous
  $\theta_1,\theta_2,\theta_3 \in \wedge_{\cO(X)}^\bullet \Vect(X)$ (see also
  Definition \ref{d:gerst-br} below):
\begin{equation}\label{e:gerst2}
[\theta_1, \theta_2 \wedge \theta_3] =
  [\theta_1, \theta_2] \wedge \theta_3 + (-1)^{|\theta_2||\theta_3|}
  [\theta_1, \theta_3] \wedge \theta_2.
\end{equation}
\end{remark}
\begin{example}\label{ex:moyal}
  The simplest example is the case $X=\bA^{n}$ with a constant Poisson
  bivector field, which can always be written up to choice of
  coordinates $(x_1, \ldots, x_m, y_1, \ldots, y_m, z_{1}, \ldots,
  z_{n-2m})$, as
\[
\pi = \sum_{i=1}^m \partial_{x_i} \wedge \partial_{y_i}, \quad
\text{i.e.,} \quad \{f, g\}= \sum_{i=1}^m \frac{\partial f}{\partial x_i}  \frac{\partial g}{\partial y_i} - \frac{\partial f}{\partial y_i} \frac{\partial g}{\partial x_i},
\]
for a unique $m \leq n/2$.
Then, for $\bk$ of characteristic zero, 
there is a well-known deformation quantization, called the
Moyal-Weyl star product:
\[
f \star g = \mu \circ e^{\frac{1}{2} \hbar \pi}(f \otimes g), \quad
\mu(a \otimes b):=ab.\footnote{In physics, over $\bk=\bC$, often one
  sees an $i = \sqrt{-1}$ also in the exponent, so that $a \star b - b
  \star a \equiv i\hbar \{a,b\} \pmod{\hbar^2}$, but according to our
  definition, which works over arbitrary $\bk$, we don't have it.}
\]
When $2m=n$, so that the Poisson structure is symplectic,
this is actually isomorphic to the usual Weyl quantization: see the next
exercise.  
\end{example}
\begin{exercise}\label{exer:moyal} We consider the Moyal-Weyl star product 
  for $X = \bA^n$ with coordinates $(x_1, \ldots, x_m, y_1, \ldots,
  y_m, z_{1}, \ldots, z_{n-2m})$ as above, for $\bk$ of characteristic zero.

  (a) Show that, for the Moyal-Weyl star product on
  $\cO(X)[\![\hbar]\!]$ with $X = \bA^n$ as above, in the above basis,
\[
x_i \star y_j - y_j \star x_i = \hbar\delta_{ij}, \quad x_i \star x_j
= x_j \star x_i, \quad y_i \star y_j = y_j \star y_i.
\]
whereas $z_1, \ldots, z_{n-2m}$ are central: $z_i \star f = f \star
z_i$ for all $1 \leq i \leq n-2m$ and all $f \in \cO(X)$.

(b) Show that this star product is actually defined over polynomials in
$\hbar$, i.e., on $\cO(X)[\hbar]$.  That is, if $f, g \in
\cO(X)[\hbar]$, so is $f \star g$, and hence we get an associative
algebra $(\cO(X)[\hbar],\star)$.  Note that this is homogeneous with
respect to the grading where $|x_i|=|y_i|=1$ and $|\hbar|=2$.

(c) Note that $(\hbar-1)$ is an ideal in $\bk[\hbar]$ and hence in
$(\cO(X)[\hbar], \star)$. Taking the quotient, get a
filtered quantization $(\cO(X)[\hbar], \star)/(\hbar-1)$, which is in
other words obtained by setting $\hbar=1$ above.

(d) Now we get to the goal of the exercise: to relate the quantization
$(\cO(X)[\hbar], \star)/(\hbar-1)$ to the Weyl algebra $\Weyl_m$.
Show first that there is a unique isomorphism of algebras,
\[
\Weyl_m \otimes \bk[z_1,\ldots,z_{n-2m}] \to (\cO(X)[\hbar],\star)/(\hbar-1),
\]
satisfying $x_i \mapsto x_i, y_i \mapsto y_i$, and $z_i \mapsto z_i$.

(e)(*) The main point of the exercise is to give the explicit
\emph{inverse} of (d).  To begin, for any vector space $V$, with $\bk$
of characteristic zero, we can consider the symmetrization map, $\Sym
V \to TV$, given by
\[
v_1 \cdots v_k \mapsto \frac{1}{k!} \sum_{\sigma \in S_k}
v_{\sigma(1)} \otimes \cdots \otimes v_{\sigma(k)}
\]
where $v_i \in V$ for all $i$. (Caution: this is \emph{not} an algebra
homomorphism.)

Now let $V = \Span
\{x_1,\ldots,x_m,y_1,\ldots,y_m,z_1,\ldots,z_{n-2m}\}$, and consider
the composition of the above linear map $\Sym V = \cO(X) \to TV$ with
the obvious quotient $TV \to \Weyl_m \otimes
\bk[z_1,\ldots,z_{n-2m}]$.  Show that the result yields an algebra
homomorphism
\[
(\cO(X)[\hbar],\star)/(\hbar-1) \to \Weyl_m \otimes \bk[z_1,\ldots,z_{n-2m}],
\]
which inverts the homomorphism of (d).
\end{exercise}
\begin{exercise}\label{exer:moyal-weyl-unique}
  In this exercise, we explain a uniqueness statement for the
  Moyal-Weyl quantization.  Suppose that $\star'$ is any other star
  product formula of the form
\begin{equation}\label{e:cont-star-graphs}
f \star' g = \mu \circ F(\pi) (f \otimes g),
\end{equation}
for $F = 1 + \frac{1}{2} \hbar \pi + \sum_{i \geq 2} \hbar^i
F_i(\pi)$, with each $F_i$ a polynomial in $\pi$.  Then, if
$(\cO(\bA^n)[\![\hbar]\!], \star')$ quantizes the Poisson bracket
$\{-,-\}$ given by $\pi$, show that we still have the relations of
part (a) of the previous exercise, for $\star'$ instead of $\star$.
Similarly to part (d) above, conclude that we have an isomorphism
\[
(\cO(\bA^n)[\![\hbar]\!], \star) \to (\cO(\bA^n)[\![\hbar]\!], \star'),
\]
given uniquely by
\[
v_1 \star \cdots \star v_m \mapsto v_1 \star' \cdots \star' v_m,
\]
for all linear functions $v_1,\ldots,v_m \in \Span
\{x_1,\ldots,x_m,y_1,\ldots,y_m,z_1,\ldots,z_{n-2m}\}$.  Moreover,
show that this isomorphism is the identity modulo $\hbar$.

(It is a much deeper fact that, when $n=2m$, we can drop the
assumption on the formula \eqref{e:cont-star-graphs}.)
\end{exercise}


We will also need the notion of a \emph{formal Poisson deformation} of
a Poisson algebra $\cO(X)$: 
\begin{definition}
A formal Poisson deformation of a Poisson algebra $\cO(X)$
 is a 
$\bk[\![\hbar]\!]$-linear Poisson
bracket on $\cO(X)[\![\hbar]\!]$ which reduces modulo $\hbar$ to the
original Poisson bracket on $\cO(X)$.
\end{definition}
\begin{theorem}\label{t:ps-quant}
  \cite{Kform,Kon-dqav,Yek-dqag,DTT-hgahcraf} (among others) Every
  Poisson structure on a smooth affine variety over a field of
  characteristic zero, or on a $C^\infty$ manifold, admits a canonical
  deformation quantization. In particular, every graded Poisson
  algebra with Poisson bracket of degree $-d < 0$ admits a filtered
  quantization.

Moreover, there is a canonical bijection, up to isomorphisms equal to
the identity modulo $\hbar$, between deformation quantizations and
formal Poisson deformations.
\end{theorem}
\begin{remark} \label{r:mathieu}
As shown by O. Mathieu \cite{Mat-haps}, in general
  there are obstructions to the existence of a quantization. We
  explain this following \S 1.4 of
  \url{www.math.jussieu.fr/~keller/emalca.pdf}, which more generally
  forms a really nice reference for much of the material discussed in
  these notes!

  Let $\mfg$ be a Lie algebra over a field $\bk$ of
  characteristic zero such that $\mfg \otimes_\bk \bar \bk$ is simple
  and not isomorphic to $\mfsl_n(\bar \bk)$ for any $n$, where $\bar
  \bk$ is the algebraic closure of $\bk$.  (In the cited references,
  one takes $\bk = \bR$ and thus $\bar \bk = \bC$, but this assumption
  is not needed.) For instance, one could take $\mfg = \mathfrak{so}(n)$
  for $n = 5$ or $n \geq 7$.

  One then considers the Poisson algebra $B := \Sym \mfg / (\mfg^2)$,
  equipped with the Lie bracket on $\mfg$. In other words, this is the
  quotient of $\cO(\mfg^*)$ by the square of the augmentation ideal,
  i.e., the maximal ideal of the origin.  Then, $\Spec B$, which is
  known as the first infinitesimal neighborhood of the origin in
  $\mfg^*$, is nonreduced and set-theoretically a point, albeit with a
  nontrivial Poisson structure.

  We claim that $B$ does not admit a deformation quantization. For a
  contradiction, suppose it did admit one, $B_\hbar :=
  (B[\![\hbar]\!], \star)$.  Since $B = \bk \oplus \mfg$ has the
  property that $B \otimes_{\bk} \bar \bk \cong (\bar \bk \oplus
  (\mfg \otimes_{\bk} \bar \bk))$ with $\mfg \otimes_{\bk} \bar \bk$ a
  semisimple Lie algebra, it follows from basic Lie cohomology
  (e.g., \cite[\S 7]{Weibel}) that
  $H^2_{\Lie}(B \otimes_{\bk} \bar \bk, B \otimes_{\bk} \bar \bk) = 0$
  and hence also $H^2_{\Lie}(B,B)=0$, i.e., $B$ has no deformations as
  a Lie algebra.  Hence, $B_\hbar \cong B[\![\hbar]\!]$ as a Lie
  algebra; in fact there is a 
  $\bk[\![\hbar]\!]$-linear
  isomorphism of Lie algebras which is the identity modulo $\hbar$.
  Now let $K := \overline{\bk(\!(\hbar)\!)}$, an algebraically closed
  field, and set $\tilde B := B_\hbar \otimes_{\bk[\![\hbar]\!]}
  K$. Since this is finite-dimensional over the algebraically closed
  field $K$, Wedderburn theory implies that, if $J \subseteq \tilde B$
  is the nilradical, then $M:= \tilde B /J$ is a product of matrix
  algebras over $K$ of various sizes. Then $J$ is also a Lie ideal in
  $\tilde B$ which is, as a Lie algebra, a sum
  $K \oplus (\mfg \otimes_{\bk} K)$ (using that $\mfg$ is
  finite-dimensional). As $J$ cannot be the unit ideal and $J$ is
  nilpotent, it must be zero, so $\tilde B$ itself is a direct sum of
  matrix algebras. But this would imply that $\mfg \otimes_{\bk}
  K \cong \mfsl_n(K)$, and hence
  $\mfg \otimes_{\bk} \bar \bk \cong \mfsl_n(\bar \bk)$, contradicting
  our assumptions.
\end{remark}
\begin{example}
  In the case $X = \bA^{n}$ with a constant Poisson bivector field,
  Kontsevich's star product coincides with the Moyal-Weyl one.
\end{example}
\begin{example}\label{ex:symg-sp}
  Next, let $\{-,-\}$ be a linear Poisson bracket on $\mfg^*=\bA^n$,
  i.e., a Lie bracket on the vector space $\mfg = \bk^n$ of linear
  functions. As explained in Exercise \ref{exer:kont-ug-iso} below
  (following \cite{Kform}), if $(\cO(\bA^n)[\![\hbar]\!], \star)$ is
  Kontsevich's canonical quantization, then
\[
x \star y - y \star x = \hbar [x,y],
\]
so that, as in Exercise \ref{exer:moyal}, the map which is the
identity on linear functions yields an isomorphism
\[
U_\hbar(\mfg) \to (\cO(\mfg^*)[\![\hbar]\!], \star). 
\]
Modulo $\hbar$, the inverse to this is the symmetrization map of
Exercise \ref{exer:moyal}. Thus, if we apply a gauge equivalence to
Kontsevich's star-product, then the inverse really is the
symmetrization map.  This is explained in detail in \cite{Dit-kspdla},
where the gauge equivalence is also explicitly computed.  The
resulting star product on $\cO(\mfg^*)$ is called the Gutt product and
dates to \cite{Gut-espctLg}; for a description of this product, see,
e.g., \cite[(13)]{Dit-kspdla}. Moreover, as first noticed in
\cite{Arn-pskdaln} (see also \cite{Dit-kspdla}), there is no gauge
equivalence required when $\mfg$ is nilpotent, i.e., in this case the
Kontsevich star-product equals the Gutt product (this is essentially
because nothing else can happen in this case).
\end{example}
These theorems all rest on the basic statement that the Hochschild
cohomology of a smooth affine variety or $C^\infty$ manifold is
\emph{formal}, i.e., the Hochschild cochain complex (which computes its
cohomology), is equivalent to its cohomology not merely as a vector
space, but as dg Lie algebras, up to homotopy.  (In fact, this
statement can be made to be equivalent to the bijection of Theorem
\ref{t:ps-quant} if one extends to deformations over dg commutative
rings rather than merely formal power series.)

We will explain what formality of dg Lie algebras means precisely
later: we note only that it is a completely different use of the term
``formal'' than we have used it before for formal deformations.  For
now, we only give some motivation via the analogous concept of
formality for vector spaces, modules, etc.  First of all, note that
all \emph{complexes of vector spaces are always equivalent to their
  cohomology}, i.e., they are all formal.  This follows because, given
a complex $C^\bullet$, one can always find an \emph{isomorphism of
  complexes}
\[
C^\bullet \iso H^\bullet(C^\bullet) \oplus S^\bullet,
\]
where $H^\bullet$ is the homology of $C^\bullet$, and $S^\bullet$ is a
contractible complex.  Here, contractible means that there exists a
linear map $h: C^\bullet \to C^{\bullet-1}$, decreasing degree by one,
so that $dh+hd = \Id$; such an $h$ is called a \emph{contracting
  homotopy}.  In particular, a contractible complex is acyclic (and
for complexes of vector spaces, the converse is also true). 

But if we consider modules over a more general ring that is not a
field, it is no longer true that all complexes are isomorphic to a
direct sum of their homology and a contractible complex.  Consider,
for example, the complex
\[
0 \to \bZ \mathop{\to}^{\cdot 2} \bZ \to 0.
\]
The homology is $\bZ/2$, but it is impossible to write the complex as
a direct sum of $\bZ/2$ with a contractible complex, since $\bZ$ has
no torsion.  That is, the above complex is \emph{not} formal.

Now, the subtlety with the formality of Hochschild cohomology is that,
even though the underlying Hochschild cochain complex is automatically
formal as a \emph{complex of vector spaces}, it is \emph{not}
automatically formal as a \emph{dg Lie algebra}.  For example, it may
not necessarily be isomorphic to a direct sum of its cohomology and
another dg Lie algebra (although being formal does not require this,
but only that the dg Lie algebra be ``homotopy equivalent'' to its
cohomology).

The statement that the Hochschild cochain complex is formal is
\emph{stronger} than merely the existence of deformation
quantizations.  It implies, for example:
\begin{theorem}\label{t:fdef-symp}
  If $(X,\omega)$ is either an affine symplectic variety over a field
  of characteristic zero or a symplectic $C^\infty$ manifold, and
  $A=(\cO(X)[\![\hbar]\!], \star)$ is a deformation quantization of
  $X$, then $\HH^\bullet(A[\hbar^{-1}]) \cong
  H_{DR}^\bullet(X,\bk(\!(\hbar)\!))$, and there is a versal formal
  deformation of $A[\hbar^{-1}]$ over the base $\hat \cO(H_{DR}^2(X))$
  (the completion of $\Sym (H_{DR}^2(X))^*$).
\end{theorem}
Here $H_{DR}^\bullet$ denotes the algebraic de Rham cohomology, which
in the case $\bk=\bC$ coincides with the ordinary topological de Rham
cohomology.  

We will define the notion of versal formal deformation
more precisely later; roughly speaking, such a deformation $U$ over a
base $\bk[\![t_1,\ldots,t_n]\!]$ is one so that every formal
deformation $(A[\hbar^{-1}][\![t]\!], \star)$ of $A[\hbar^{-1}]$ with
deformation parameter $t$ is isomorphic, as a $\bk[\![t]\!]$-algebra,
to $U
\otimes_{\bk[\![t_1,\ldots,t_n]\!]} \bk[\![t]\!]$ for some continuous
homomorphism $\bk[\![t_1,\ldots,t_n]\!] \to \bk[\![t]\!]$ (i.e., some
assignment of each $t_i$ to a power series in $t$ without constant
term) and that this isomorphism is the identity modulo $t$.  In the
situation of the theorem, $t_1,\ldots,t_n$ should be a basis of
$(H_{DR}^2)^*$.  We will sketch how the theorem follows from
Kontsevich's theorem in \S \ref{ss:tw-linf} below.

  

\subsection{Description of Kontsevich's deformation quantization for
  $\bR^d$}\label{ss:dq-fla}
It is worth explaining the general form of the star-products given by
Kontsevich's theorem for $X=\bR^d$, considered either as a smooth
manifold or an affine algebraic variety over $\bk=\bR$.  This is taken
from \cite[\S 2]{Kform}; the interested reader is recommended to look
there for more details.

Suppose we are given a Poisson bivector $\pi \in \wedge_{\cO(X)}^2 \Vect(X)$.  Then
Kontsevich's star product $f \star g$ is a linear combination of all
possible ways of applying $\pi$ multiple times to $f$ and $g$, but
with very sophisticated weights.

The possible ways of applying $\pi$ are easy to describe using
directed graphs.  Namely, the graphs we need to consider are placed in
the closed upper-half plane $\{(x,y) \mid y \geq 0\} \subseteq \bR^2$,
satisfying the following properties:
\begin{enumerate}
\item[(1)] There are exactly two vertices along the $x$-axis, labeled
  by $L$ and $R$.  The other vertices are labeled $1, 2, \ldots, m$.
\item[(2)] $L$ and $R$ are sinks, and all other vertices have exactly
two outgoing edges.
\item[(3)] At every vertex $j \in \{1,2,\ldots,m\}$, the two outgoing edges should be
  labeled by the symbols $e_1^j$ and $e_2^j$. That is, we fix an
  ordering of the two edges and denote them by $e_1^j$ and $e_2^j$.
\end{enumerate}
Examples of such graphs are given in Figures \ref{fig:pb-graph} and
\ref{fig:graph2}.

\begin{figure}\label{fig:pb-graph}

\begin{tikzpicture}

\draw (-5,0) -- (5,0);
\node [] (L) at (-3,0) {};
\draw [fill] (L) circle (0.05cm);
\node [above left] at (L) {$L$};
\node [above right] at (3,0) {$R$};
\node [] (R) at (3,0) {};
\draw [fill] (R) circle (0.05cm);

\draw [fill] (0,5) circle (0.05cm);
\node (1) [right] at (0,5) {$1$};

\draw [->] (1) -- (L);
\node [left] at (-1.5,2.5) {$e^1_1$};
\draw [->] (1) -- (R);
\node [right] at (1.6,2.5) {$e^1_2$};

\end{tikzpicture}
\caption{The graph corresponding to $f \otimes g \mapsto \{f,g\}$}
\end{figure}
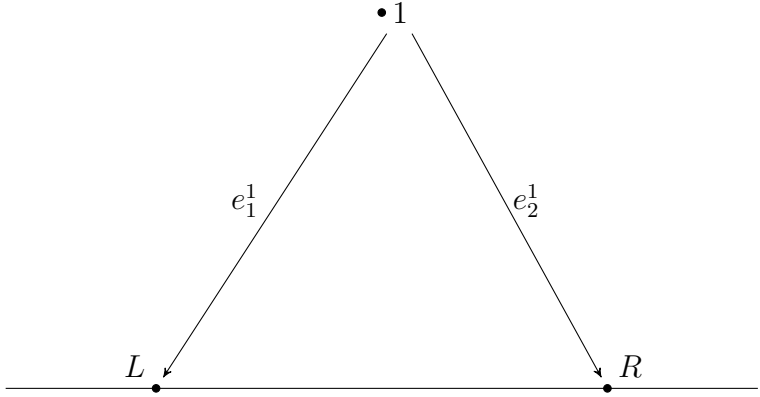

\begin{figure}\label{fig:graph2}

\begin{tikzpicture}

\draw (-5,0) -- (5,0);
\node [] (L) at (-3,0) {};
\draw [fill] (L) circle (0.05cm);
\node [above left] at (L) {$L$};
\node [above right] at (3,0) {$R$};
\node [] (R) at (3,0) {};
\draw [fill] (R) circle (0.05cm);

\draw [fill] (0,3) circle (0.05cm);
\node (2) [right] at (0,3) {$2$};

\draw [fill] (0,6) circle (0.05cm);
\node(1) [right] at (0,6) {$1$};

\draw [->] (2) -- (L);
\draw [->] (2) -- (R);
\node [left] at (-1.5,1.5) {$e^2_1$};
\node [right] at (1.6,1.5) {$e^2_2$};
\draw [->] (1) -- (L);
\draw [->] (1) -- (2);
\node [left] at (-1.5,3) {$e^1_1$};
\node [right] at (0.1,4.5) {$e^1_2$};

\end{tikzpicture}
\caption{The graph corresponding to $f \otimes g \mapsto
  \sum_{i,j,k,\ell}
  \pi^{i,j} \partial_j(\pi^{k,\ell}) \partial_i \partial_k(f) \partial_{\ell}(g)$}
\end{figure}
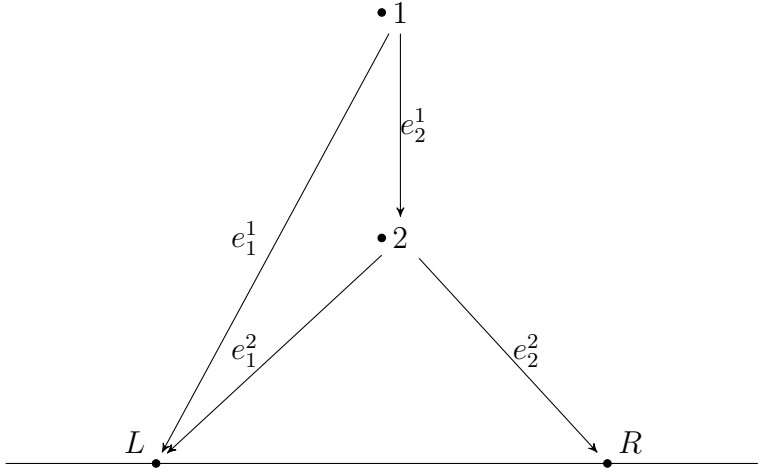

Write our Poisson bivector $\pi$ in coordinates as
\[
\pi = \sum_{i<j} \pi^{i,j} \partial_i \wedge \partial_j.
\]
Let $\pi^{j,i} := -\pi^{i,j}$.  We attach to the graph $\Gamma$ a
bilinear differential operator $B_{\Gamma,\pi}: \cO(X)^{\otimes 2} \to
\cO(X)$, as follows. Let $E_\Gamma = \{e_1^1,e_1^2, \ldots, e_m^1,
e_m^2\}$ be the set of edges. Given an edge $e \in E_\Gamma$, let
$t(e)$ denote the target vertex of $e$.\footnote{This corresponds to
  what we called the \emph{head}, $e_h$ in the quiver context: we
  caution that there $e_t$ was the tail, which is \emph{not} the same
  as the target vertex $t(e)$ here!} Then
\begin{multline*}
  B_{\Gamma,\pi}(f \otimes g) := \sum_{I: E_\Gamma \to \{1,\ldots,d\}}
  \biggl[ \prod_{i=1}^m \Bigl( \prod_{e \in E_\Gamma \mid t(e) =
    i} \partial_{I(e)} \Bigr) \pi^{I(e_i^1),I(e_i^2)} \biggr] \\ \cdot
  \Bigl(\prod_{e \in E_\Gamma \mid t(e)=L}
 \partial_{I(e)} \Bigr)(f) \cdot \Bigl(\prod_{e \in E_\Gamma \mid t(e)=R}
 \partial_{I(e)} \Bigr)(g).
\end{multline*}
Now, let $V_\Gamma$ denote the set of vertices of $\Gamma$ other than $L$ and $R$, so that in
the above formula, $m = |V_\Gamma|$.  Then the star product is given by
\[
f \star g = \sum_\Gamma \hbar^{|V_\Gamma|} w_\Gamma B_{\Gamma,\pi}(f
\otimes g),
\]
where we sum over isomorphism classes of graphs satisfying conditions
(1)--(3) above (we only need to include one for each isomorphism class
forgetting the labeling, since the operator is the same up to a sign).
The $w_\Gamma \in \bR$ are weights given by very explicit integrals
(which in general are impossible to evaluate).  Note that, in order to
have $f \star g \equiv fg \pmod{\hbar}$ (i.e., \eqref{e:form-def}),
then if $\Gamma_0$ is the graph
with no edges (and thus only vertices $L$ and $R$), $w_{\Gamma_0} =
1$. Similarly, if we let $\Gamma_1$ be the graph in Figure \ref{fig:pb-graph} with
only three vertices $L,R$, and $1$ and we include this graph in the
sum but not the isomorphic graph obtained by swapping the labels of
the edges, then the relation $f \star g - g \star
f \equiv \hbar \{f,g\} \pmod{\hbar^2}$ (i.e., \eqref{e:dq})
forces $w_{\Gamma_1}=1/2$ as well.

The \textbf{main observation} which motivates the operators
$B_{\Gamma,\pi}$ is the following: linear combinations of operators
$B_{\Gamma,\pi}$ as above are exactly all of bilinear operators
obtainable by contracting tensor powers of $\pi$ with $f$ and $g$.

\subsection{Formal deformations of algebras}\label{ss:fdef}
Now we define more precisely formal deformations.  These generalize
star products to deformations of arbitary associative algebras,
and also to the setting where more than one deformation parameter is
allowed. Still more generally, we are interested in deformations
over a base commutative augmented $\bk$-algebra
$R \supseteq R_+$ such that $R$ is complete with respect to the
$R_+$-adic topology, i.e., such that $R = \lim_{m \to \infty}
R/R_+^m$, taking the inverse limit under the system of surjections
$\cdots \to R/R_+^m \to R/R_+^{m-1} \to \cdots \to R/R_+ = \bk$.  We
call such rings \emph{complete augmented} commutative $\bk$-algebras. We will
need the completed tensor product,
\begin{equation}\label{e:hatotimes}
A \hat \otimes R := \lim_{m \to \infty} A \otimes R/R_+^m.
\end{equation}
\begin{example} When
$R$ is a formal power series ring
$R = \bk[\![t_1,\ldots,t_n]\!]$,
\[
A \hat \otimes R = A[\![t_1, \ldots, t_n]\!] = \{\sum_{i_1,\ldots,i_n
  \geq 0} a_{i_1, \ldots, i_n} t_1^{i_1} \cdots t_n^{i_n}\mid a_{i_1,
  \ldots, i_n} \in A\}.
\]
\end{example}
\begin{definition}
  A formal deformation of $A$ over a commutative complete augmented
  $\bk$-algebra $R$ is an $R$-algebra $A'$ isomorphic to $A \hat \otimes_\bk R$
  as an $R$-module such that $A' \otimes_R (R/R_+) \cong A$ as a
  $\bk$-algebra. 
\end{definition}
Equivalently, a formal deformation is an $R$-algebra $(A \hat \otimes_\bk
R, \star)$ such that $a \star b \equiv ab \pmod {R_+}$.
\begin{example}
  If $R = \bk[\![t_1,\ldots,t_n]\!]$, then a formal deformation of $A$
  over $R$ is the same as an algebra $(A[\![t_1,\ldots,t_n]\!],
  \star)$ such that $a \star b \equiv ab \pmod {t_1, \ldots, t_n}$.  In the case $n=1$, we often denote the parameter by $\hbar$, and then we recover
the notion of one-parameter formal deformations (Definition \ref{d:opfd}).
\end{example}

As a special case of formal deformations, we also have deformations
over augmented $\bk$-algebras where the augmentation
ideal is nilpotent: $R_+^n = 0$ for some $n \geq 1$.  Such
deformations are often simply called deformations, since one can also
think of $R$ as an ordinary (abstract) $\bk$-algebra without a
topology, and then $R$ is already complete: $R = \lim_{m \to \infty}
R/R_+^m$.  (Actually, it is enough to study only these, and this is
frequently done in some literature, since formal deformations are
always limits of such deformations; we are nonetheless interested in
formal deformations in these notes and Kontsevich's theorem as well as
many nice examples are tailored to them.)

In particular, such deformations include infinitesimal deformations
(Definition \ref{d:infinitesimal-def}) as well as $n$-th order
deformations:
\begin{definition}\label{d:nord-def}
  An $n$-th order deformation is a deformation over $R =
  \bk[\varepsilon]/(\varepsilon^{n+1})$, with $R_+ := \varepsilon R$.
\end{definition}
Note that a first-order deformation is the same thing as an
infinitesimal deformation.

\subsection{Formal vs. filtered deformations}
Main idea: if $A$ is a graded algebra, we can consider filtered
deformations on the same underlying filtered vector space $A$. These
are equivalent to homogeneous formal deformations of $A$ (over
$\bk[\![\hbar]\!]$), by replacing relations of degree $\leq m$,
$\sum_{i=0}^m p_i = 0$ with $|p_i|=i$, by $\sum_i \hbar^{m-i} p_i =
0$, which are now homogeneous with the sum of the grading on $A$ and
$|\hbar|=1$.  We can also do the same thing with $|\hbar| = d \geq 1$
and $d \mid m$, replacing $\sum_{i=0}^{m/d} p_{di}$ for $|p_{di}|=di$
by $\sum_i \hbar^{m-di} p_{di}$.  

We begin with two motivating examples:
\begin{example}\label{ex:hom-weyl}
We can form a homogenized version of the Weyl algebra with $|\hbar|=2$:
\[
\Weyl_\hbar(V) = TV[\![\hbar]\!] / (vw-wv-\hbar(v,w)).
\]
\end{example}
\begin{example}
  The homogenized universal enveloping algebra $U_\hbar \mfg$ is given
  by
\[
U_\hbar \mfg = T\mfg[\![\hbar]\!] / (xy-yx-\hbar [x,y]),
\]
with $|\hbar|=1$.
\end{example}
We now proceed to precise definitions and statements, for the case $|\hbar|=1$:
\begin{definition}
  Let $A$ be an increasingly filtered associative algebra. Then, the
  \emph{Rees algebra} of $A$ is the graded algebra
\[
RA := \bigoplus_{i \in \bZ} A_{\leq i} \cdot \hbar^i,
\]
with $|\hbar|=1$ and $|A|=1$, equipped with the multiplication
\[
(a \cdot \hbar^i) \cdot (b \cdot \hbar^j) = ab \cdot \hbar^{i+j}.
\]
Similarly, let the \emph{completed} Rees algebra, call it $\hat A$, be the
$\hbar$-adic completion of $RA$, $\hat A := \lim_{n \to \infty} RA/\hbar^n RA$. This can also be seen as the space of power series of the form $\sum_i f_i \hbar^i$ for $f_i \in A$, such that,
letting $|f_i| \in \bZ$ be the least integer such that $f_i \in A_{\leq |f_i|}$, we have $|f_i| \leq i$ and
$\lim_{i \to \infty} i - |f_i| = \infty$. This may seem a little odd, but it
is designed so that a topological $\bk[\![\hbar]\!]$-basis is given by any
lift of a basis of $\bigoplus_i \gr_i A \cdot \hbar^i$.
\end{definition}

The Rees algebra construction defines an equivalence between filtered
deformations and homogeneous formal deformations:
\begin{lemma}
  The functor $A \mapsto RA$ defines an equivalence of categories from
  increasingly filtered $\bk$-algebras to nonnegatively graded
  $\bk[\hbar]$-algebras (with $|\hbar|=1$) which are free as
  $\bk[\hbar]$-modules.  A quasi-inverse is given by $C \mapsto C /
  (\hbar - 1)$, assigning the filtration which lets
  $(C/(\hbar-1))_{\leq m}$ be the image of those elements of degree
  $\leq m$ in $C$.
\end{lemma}
The proof is left as an exercise. In other words, the equivalence
replaces $A_{\leq m}$ with the homogeneous part $(RA)_m$.
\begin{remark}
  The Rees algebra can be viewed as interpolating between $A$ and $\gr
  A$, in that $RA/(\hbar-1) = A$ and $RA/(\hbar) = \gr A$; actually
  for any $c \in \bk\setminus\{0\}$ we have a canonical isomorphism
  $RA/(\hbar-c) \cong A$ (sending $[a_{\leq i} \cdot \hbar^i]$, in the
  image of $A_{\leq i} \cdot \hbar^i \subseteq RA$, to $c^i a_{\leq i}
  \in A$).  Similarly, $\hat A[\hbar^{-1}] = A(\!(\hbar)\!)$ and $\hat
  A / (\hbar)=\gr A$.  So over $1 \in \Spec \bk[\hbar]$ (or any point
  other than the origin), or over the generic point
  of $\Spec \bk[\![\hbar]\!]$, we recover $A$ or $A(\!(\hbar)\!)$,
  respectively, and at the origin, or the special point of $\Spec
  \bk[\![\hbar]\!]$, we obtain $\gr A$.
\end{remark}
Given a graded algebra $A$, we canonically obtain a filtered algebra
by $A_{\leq n} = \bigoplus_{m \leq n} A_m$.  In this case, we can
consider filtered deformations of $A$ of the form $(A, \star_f)$,
where $\star_f$ reduces to the usual product on $\gr A = A$, i.e.,
$\gr (A, \star_f) = A$.

\begin{corollary}\label{cor2.5.6}
Let $A$ be a graded algebra. Then there is an equivalence
  between filtered deformations $(A, \star_f)$ and formal deformations
  $(A[\![\hbar]\!], \star)$ which are graded with respect to the sum
  of the grading on $A$ and $|\hbar|=1$, given by $(A, \star_f)
  \mapsto \widehat{(A, \star_f)}$.  For the opposite direction,
  because of the grading, $(A[\![\hbar]\!], \star) = (A[\hbar], \star)
  \hat \otimes_{\bk[\hbar]} \bk[\![\hbar]\!]$, so one can take $(A[\hbar],
  \star) / (\hbar-1)$.
\end{corollary}
\begin{remark}
  One can similarly give an analogue in the case $|\hbar|=d$: if one
  begins with a graded Poisson algebra $B$ with Poisson bracket of
  degree $-d$, then filtered quantizations $(B,\star_f)$ of the form
  $a \star_f b = ab + \sum_{i \geq 1} f_i(a,b)$ with $|f_i(a,b)| =
  |ab|-di$ are equivalent to formal deformations
  $(B[\![\hbar]\!],\star)$ which are homogeneous for $|\hbar|=d$.  In
  the case $d=2$, this includes the important Example
  \ref{ex:hom-weyl}, which we have also discussed earlier.
\end{remark}

\subsection{Universal deformations and gauge equivalence}
We begin with the proper notion of equivalence of two deformation
quantizations:
\begin{definition}\label{d:ge}
  Given two formal deformations $(A[\![\hbar]\!],\star)$ and
  $(A[\![\hbar]\!],\star')$, a gauge equivalence $\star \sim \star'$
  means a $\bk[\![\hbar]\!]$-linear automorphism $\Phi$ of
  $A[\![\hbar]\!]$ which is the identity modulo $\hbar$, such that
\begin{equation}\label{e:ge}
\Phi(a \star b) = \Phi(a) \star' \Phi(b).
\end{equation}
Similarly, given two formal deformations $(A \hat \otimes R, \star)$
and $(A \hat \otimes R, \star')$ over $R$, a gauge equivalence is an
$R$-linear automorphism of $A \hat \otimes R$ which is the identity modulo $R_+$,
such that \eqref{e:ge} holds.
\end{definition}
\begin{exercise}\label{exer:cont-fd2}
  $\Phi$ is automatically continuous 
in the $\hbar$-adic topology.
\end{exercise}
We are interested in finding a family which parameterizes all formal
deformations up to gauge equivalence.  To do so, we define two
notions: a versal deformation exhausts all formal deformations, and a
universal deformation exhausts the formal deformations uniquely.  
In general, one cannot expect to obtain a (nice) explicit (uni)versal
deformation, but there are cases where one can, as we explain in
Theorem \ref{t:hh3=0} below.
\begin{definition} Given a continuous homomorphism $p: R \to R'$ and a
  formal deformation $(A \hat \otimes R, \star)$, the formal
  deformation defined by base-change from $p$ is $(A \hat \otimes R',
  p(\star))$, with
\[
a\, p(\star)\, b := (\Id \otimes p)(a \star b).
\]
\end{definition}
\begin{definition}
  A \emph{versal} formal deformation $(A \hat \otimes R, \star)$, is
  one such that,
  for every other formal deformation $(A \hat \otimes R', \star')$,
  there exists a continuous homomorphism $p: R \to R'$ such that $(A
  \hat \otimes R', \star')$ is gauge-equivalent to the base-change
  deformation $(A \hat \otimes R', p(\star))$.

  The deformation is \emph{universal} if the homomorphism $p$ is
  always unique.
\end{definition}
In particular, if $R' = \bk[\![\hbar]\!]$, we see that, given a versal
deformation $(A \hat \otimes R, \star)$, all formal deformations
$(A[\![\hbar]\!],\star)$ are obtained from continuous homomorphisms $R
\to \bk[\![\hbar]\!]$.
\begin{remark}  
  The relationship between $(A[\![\hbar]\!], \star)$ and the
  (uni)versal $(A \hat \otimes R, \star_u)$ can be stated
  geometrically as follows: $p$ is a formal $\Spf
  \bk[\![\hbar]\!]$-point $p$ of $\Spf R$, and $(A[\![\hbar]\!],
  \star)$ is gauge-equivalent to the pullback of $(A \hat \otimes R,
  \star_u)$, namely, $(A[\![\hbar]\!], \star) = (A \hat \otimes R,
  \star_u) \otimes_{R} p$.
\end{remark}
\begin{remark}
  One can also define similarly the notion of versal and universal
  filtered deformations.  These are not actually filtered deformations
  of $A$, but rather filtered deformations of $A \otimes R$ for some
  graded $\bk$-algebra $R$, such that all filtered deformations are
  obtained by base-change by a character $R \to \bk$, up to
  $\bk$-linear filtered automorphisms of $A$ whose associated graded
  automorphism is the identity. We will not need this notion, although
  we will give some examples (see Examples \ref{ex:sra} and
  \ref{ex:gca} below).
\end{remark}
The principle that $\HH^3(A)$ classifies obstructions (which will be
made precise in \S \ref{ss:obst-3o} below) leads to the following
important result, which is provable using the Maurer-Cartan formalism
discussed in \S \ref{s:dgla}.  Given a finite-dimensional vector space
$V$, let $\hat \cO(V)$ be the completion of the algebra $\cO(V)$ at
the augmentation ideal. Explicitly, if $v_1,\ldots,v_n$ is a basis of
$V$ and $c_1, \ldots, c_n$ the dual basis of $V^*$, so that $\cO(V) =
\bk[f_1,\ldots,f_n]$, then we set $\hat \cO(V) :=
\bk[\![c_1,\ldots,c_n]\!]$, the formal power series algebra.
\begin{theorem}\label{t:hh3=0}
  If $\HH^3(A)=0$, then there exists a versal formal deformation of
  $A$ over $\hat \cO(\HH^2(A))$.  If, furthermore, $\HH^1(A)=0$, then
  this is a universal deformation.
\end{theorem}
At the end of the exercises from Section \ref{s:dgla}, we will include
an outline of a proof of a slightly weaker version of the first
statement of the theorem.

In the case that $\HH^3(A) \neq 0$, then in general there may not
exist a versal formal deformation over $\hat \cO(\HH^2(A))$;
see, however, Theorem \ref{t:fdef-symp} for another situation where this exists.




\begin{example}\label{ex:sra}
  If $A = \Weyl(V) \rtimes G$ for $G < \Sp(V)$ a finite subgroup
  (see Definition \ref{d:skew-prod}) then,
  by \cite{AFLS}, $\HH^i(A)$ is the space of
  conjugation-invariant functions on the set of group elements $g \in
  G$ such that $\rk(g - \Id) = i$. In particular, $\HH^i(A) = 0$ when
  $i$ is odd, and $\HH^2(A)$ is the space of conjugation-invariant
  functions on the set of \emph{symplectic reflections}: those
  elements fixing a codimension-two symplectic hyperplane.  Thus,
  there is a universal formal deformation over $\hat \cO(\HH^2(V))$.

   Moreover, $\HH^2(A) = \HH^2(A)_{\leq -2}$ is all in
  degree $-2$ (the degree also of the Poisson bracket on $\cO(V)$).
   This actually implies that 
  there is a universal filtered deformation of $A$ parameterized
  by elements $c \in \HH^2(A)$ as above, in the sense that every
  filtered deformation is isomorphic to one of these via a filtered
  isomorphism whose associated graded isomorphism is the
  identity.  (In fact, the universal formal deformation is obtained from this
  one by completing at the augmentation ideal of $\cO(\HH^2(V))$.)

  This universal filtered deformation admits an explicit description,
  known as the \emph{symplectic reflection algebra} \cite{EG}, first
  constructed by Drinfeld, which is defined as follows.  Let $S$ be
  the set of symplectic reflections, i.e., elements such that
  $\rk(g-\Id)=2$. Then $\HH^2(A) \cong \bk[S]^G$. Let $c \in
  \text{Fun}_{G}(S, \bk)$ be a conjugation-invariant $\bk$-valued
  function on $S$.  Then the corresponding symplectic reflection
  algebra $H_{1,c}(G)$ is presented as
\[
H_{1,c}(G) := TV / \Bigl(xy-yx-\omega(x,y)+2 \sum_{s \in S} c(s)
\omega_s(x,y) \cdot s\Bigr),
\]
where $\omega_s$ is the composition of $\omega$ with the projection to
the sum of the nontrivial eigenspaces of $s$ (which is a
two-dimensional symplectic vector space). In other words, $\omega_s$
is the restriction of the symplectic form $\omega$ to the
two-dimensional subspace orthogonal to the symplectic reflecting
hyperplane of $s$. The algebra $H_{1,c}(G)$ above is the case $t=1$ of
the more general symplectic reflection algebra $H_{t,c}(G)$ studied in
Bellamy's chapter,
which is obtained by replacing the
$\omega(x,y)$ by $t \cdot \omega(x,y)$ in the relation.
\end{example}
\begin{example}\label{ex:gca}
  If $X$ is a smooth affine variety or smooth $C^\infty$ manifold, and
  $G$ is a group acting by automorphisms on $X$, then by
  \cite{Eti-chavfga}, $\HH^2(\caD(X) \rtimes G) \cong H^2_{DR}(X)^G
  \oplus \bk[S]^G$, where $S$ is the set of pairs $(g,Y)$ where $g \in
  G$ and $Y \subseteq X^g$ is a connected (hence irreducible)
  subvariety of codimension one. Then $\bk[S]^G$ is the space of
  $\bk$-valued functions on $S$ which are invariant under the action
  of $G$, $h \cdot (g,Y) = (hgh^{-1}, h(Y))$.  Furthermore, by
  \cite{Eti-chavfga}, all deformations of $\caD(X) \rtimes G$ are
  \emph{unobstructed} and there exists a universal filtered
  deformation $H_{1,c,\omega}(X)$ parameterized by $c \in \bk[S]^G$
  and $\omega \in H^2_{DR}(X)^G$. More precisely, to each such
  parameters we have a filtered deformation of $\caD(X) \rtimes G$,
  and these exhaust all filtered deformations up to filtered
  isomorphism.
%
\end{example}

\subsection{Exercises}

Exercises from the notes: \ref{exer:cont-fd}, \ref{exer:moyal}, \ref{exer:moyal-weyl-unique}, and \ref{exer:cont-fd2}.

Additional exercises:
\begin{exercise}\label{exer:koszul-complexes}
We introduce Koszul complexes, which are one of the main
tools for computing Hochschild cohomology.

First consider, for the algebra $\Sym V$, the Koszul resolution of the
  augmentation module $\bk$:
\begin{gather}
0 \to \Sym V \otimes \wedge^{\dim V} V \to \Sym V\otimes \wedge^{\dim
V - 1} V \to \cdots \to \Sym V \otimes V \to \Sym V \onto \bk, \\
f \otimes (v_1 \wedge \cdots \wedge v_i) \mapsto \sum_{j=1}^i
(-1)^{j-1} (fv_j)
\otimes (v_1 \wedge \cdots \hat v_j \cdots \wedge v_i),
\end{gather}
where $\hat v_j$ means that $v_j$ was omitted from the wedge
product. We remark that this complex itself cannot deform to a complex
of $\Weyl(V)$-modules, since the $\bk$ itself does not deform (since
$\Weyl(V)$ is simple, it has no finite-dimensional modules).
Nonetheless, we will be able to deform a bimodule analogue of the
above.

(a) Construct analogously to the Koszul resolution a bimodule resolution of
$\Sym V$, of the form
\[
\Sym V \otimes \wedge^\bullet V \otimes \Sym V \onto \Sym V.
\]
Using this complex, show that $$\HH^i(\Sym V,\Sym V \otimes \Sym V) :=
\Ext^i_{(\Sym V)^e}(\Sym V, \Sym V \otimes \Sym V) \cong\Sym V[-\dim
V].$$ (As we will see in Exercise \ref{exer:kos-cy},
this implies that
$\Sym V$ is a Calabi-Yau algebra of dimension $\dim V$.)

(b) Now replace $\Sym V$ by the Weyl algebra, $\Weyl(V)$.
Show that the Koszul complex above deforms to give a complex whose zeroth
homology is $\Weyl(V)$:
\begin{multline}
0 \to \Weyl(V) \otimes \wedge^{\dim V} V \otimes \Weyl(V) \to
 \Weyl(V)\otimes \wedge^{\dim V - 1}V \otimes \Weyl(V)
 \\ \to \cdots \to \Weyl(V) \otimes V \otimes \Weyl(V)
\to \Weyl(V) \otimes \Weyl(V) \onto \Weyl(V),
\end{multline}
using the same formula.  Hint: You only need to show it is a complex,
by Exercise \ref{exer:hoch-res}.  Note also the fact that such a deformation
exists is a consequence of the corollaries of Exercise \ref{exer:hoch-res}.

Deduce that $\HH^\bullet(\Weyl(V), \Weyl(V) \otimes \Weyl(V)) \cong \Weyl
V[-d]$ (so, as we will see in Exercise \ref{exer:kos-cy}, it
is also Calabi-Yau of dimension $\dim V$).

(c) Suppose that $V = \mfg$ is a (finite-dimensional) Lie
algebra. Deform the complex to a resolution of the universal enveloping
algebra using the Chevalley-Eilenberg complex:
\[
U \mfg \otimes \wedge^{\dim \mfg} \mfg \otimes U \mfg
\to U \mfg \otimes \wedge^{\dim \mfg-1} \mfg \otimes U \mfg \to \cdots
\to U \mfg \otimes \mfg \otimes U\mfg \to U \mfg \otimes U \mfg
\onto U \mfg,
\]
where the differential is the sum of the preceding differential for
$\Sym V = \Sym \mfg$ and the additional term,
\begin{equation}\label{e:ce-diff}
x_1 \wedge \cdots \wedge x_k \mapsto
\sum_{i< j} [x_i, x_j] \wedge x_1 \wedge \cdots \hat x_i \cdots \hat x_j \cdots \wedge x_k.
\end{equation}
That is, verify this is a complex, and conclude from the (a) and the
preceding exercise that it must be a resolution.  Conclude as before
that $\HH^\bullet(U\mfg, U\mfg \otimes U\mfg) \cong U\mfg[-d]$.  (As
we will see in Exercise \ref{exer:ug-cy}, $U\mfg$ is twisted
Calabi-Yau, but not Calabi-Yau in general.)

(d) Recall the usual Chevalley-Eilenberg complexes computing Lie
algebra (co)homology with coefficients in a $\mfg$-module $M$: if
$C^{CE}_\bullet(\mfg) = \wedge^\bullet \mfg$ is the complex inside of
the two copies of $U\mfg$ above (with the differential as in
\eqref{e:ce-diff}), these are
\[
C^{CE}_\bullet(\mfg, M) := C^{CE}_\bullet(\mfg) \otimes M, \quad
C_{CE}^\bullet(\mfg, M) := \Hom_\bk(C^{CE}_\bullet(\mfg), M).
\]
Conclude that (cf.~\cite[Theorem 3.3.2]{L}), if $M$ is a $U\mfg$-bimodule,
\[
\HH^\bullet(U\mfg, M) \cong H_{CE}^\bullet(\mfg, M^{\ad}),
\quad \HH_\bullet(U\mfg, M) \cong H^{CE}_\bullet(\mfg, M^{\ad}),
\]
where $M^{\ad}$ is the $\mfg$-module obtained from $M$ by the adjoint action,
\[
{\ad}(x)(m) := xm-mx,
\]
where the LHS gives the Lie action of $x$ on $m$, and the RHS uses the
$U\mfg$-bimodule action.
  Conclude that $\HH^\bullet(U\mfg)
  \cong H_{CE}^\bullet(\mfg, U \mfg^{\ad})$. 

  Next, recall (or accept as a black box) that, when $\mfg$ is
  finite-dimensional semisimple (or more generally reductive) and $\bk$ has characteristic zero, then $U
  \mfg^{\ad}$ decomposes into a direct sum of finite-dimensional
  irreducible representations, and $H_{CE}^\bullet(\mfg, V) =
  \Ext_{U\mfg}^\bullet(\bk, V)= 0$ if $V$ is finite-dimensional
  irreducible and \emph{not} the trivial representation.

Conclude that, in this case, $\HH^\bullet(U\mfg) \cong \Ext_{U\mfg}^\bullet(\bk, U\mfg^{\ad}) \cong Z(U\mfg) \otimes \Ext_{U\mfg}^\bullet(\bk, \bk)$.  

\textbf{Bonus}: Compute that, still assuming that $\mfg$ is
finite-dimensional semisimple (or reductive) and $\bk$ has characteristic
zero, $\Ext_{U\mfg}(\bk,\bk)
\cong H_{CE}^\bullet(\mfg, \bk)$ is isomorphic to $(\wedge^\bullet
\mfg^*)^\mfg$.  To do so, 
show that the inclusion $C_{CE}^\bullet(\mfg, \bk)^\mfg\into
C_{CE}^\bullet(\mfg, \bk)$ is a quasi-isomorphism. This follows by
defining an operator $d^*$ with opposite degree of $d$, so that the
Laplacian $d d^* + d^* d$ is the quadratic Casimir element $C$.  Thus
this Laplacian operator defines a contracting homotopy onto the part
of the complex on which the quadratic Casimir acts by zero.  Since $C$
acts by a positive scalar on all nontrivial finite-dimensional
irreducible representations, this subcomplex is
$C_{CE}^\bullet(\mfg, \bk)^\mfg$.
But the latter is just
$(\wedge \mfg^*)^\mfg$ with zero differential.

Conclude that $\HH^\bullet(U\mfg) = Z(U \mfg) \otimes (\wedge^\bullet
\mfg^*)^\mfg$.  In particular, for $\mfg$ semisimple, $\HH^2(U\mfg) =
\HH^1(U\mfg) = 0$. As we will see, this implies that $U\mfg$ has no
nontrivial formal deformations.
\end{exercise}
\begin{remark}
  For $\mfg=\mfsl_2$, $U\mfsl_2$ is a Calabi-Yau algebra of dimension
  three, since it is a Calabi-Yau deformation of the Calabi-Yau
  algebra $\cO(\mfsl_2^*) \cong \bA^3$, which is Calabi-Yau with the
  usual volume form.  So $\HH^3(U\mfsl_2) \cong \HH_0(U\mfsl_2)$ by
  Van den Bergh duality, and the latter is $U\mfsl_2 / [U\mfsl_2,
  U\mfsl_2]$, which you should be able to prove is isomorphic to
  $(U\mfsl_2)_{\mfsl_2} \cong (U\mfsl_2)^{\mfsl_2}$, i.e., this is a
  rank-one free module over $\HH^0(U\mfsl_2)$.

  In other words, $H^3_{CE}(\mfg, \bk) \cong \bk$, which is also true
  for a general simple Lie algebra.  Since this group controls the
  \emph{tensor category deformations} of $\mfg$-mod, or the Hopf
  algebra deformations of $U\mfg$, this is saying that there is a
  one-parameter deformation of $U\mfg$ as a Hopf algebra which gives
  the quantum group $U_q \mfg$.  But since $\HH^2(U\mfg) = 0$, as
  an associative algebra, this deformation is trivial, i.e.,
  equivalent to the original algebra; this is \emph{not} true as a
  Hopf algebra!
\end{remark}

\begin{exercise}\label{exer:sra-flat}
  Using the Koszul deformation principle (Theorem \ref{t:kdp-filt}),
  show that the symplectic reflection algebra (Example \ref{ex:sra})
  is a flat filtered deformation of $\Sym(V) \rtimes G$.  Hint:
  This amounts to the ``Jacobi'' identity
\[ 
[\omega_s(x,y), z] + [\omega_s(y,z), x] + [\omega_s(z,x), y] = 0,
\]
for all symplectic reflections $s \in S$, and all $x, y, z \in V$, as
well as the fact that the relations are $G$-invariant.  Compare with
\cite[\S 2.3]{Eti-enadt}, where it is also shown that (the Rees
algebras of these) yield all formal deformations of $\Weyl(V) \rtimes
G$.
\end{exercise}

\begin{exercise}\label{exer:deduce-kdp-filt} 
In this exercise, we deduce the filtered version of the Koszul
  deformation principle, Theorem \ref{t:kdp-filt}, from the more
  standard deformation version (Theorem \ref{t:kdp}), using the
  correspondence between filtered deformations and homogeneous formal
  deformations.

  Let $R \subseteq TV$ and $B := TV/(R)$. Let $E \subseteq
  TV[\![\hbar]\!]$ be a subspace which lifts $R$ modulo $\hbar$, i.e.,
  such that the composition $E \to TV[\![\hbar]\!] \to TV$ is an
  isomorphism onto $R$.  Let $A := TV[\![\hbar]\!]/(E)$.  Then we have
  a canonical surjection $B \to A/\hbar A$.  We call $A$ a (flat)
  formal deformation of $B$ if this surjection is an isomorphism.

  (a) A priori, the above is a slightly stronger version of formal
  deformation than the usual one, because it requires a deformation
  $E$ of the relations.  Show, however, that such an $E$ always exists
  given a (one-parameter) formal deformation $A$ of $B := TV/(R)$, so
  the notion is equivalent to the usual one.  Hint: Recall that, by
  Remark \ref{r:sp-iso}, a one-parameter formal deformation $A$ can be alternatively
  viewed as a topologically free $\bk[\![\hbar]\!]$-algebra $A$
  together with an isomorphism $\phi: A/\hbar A \iso B$.
  Show that there exists 
  $E \subseteq V^{\otimes 2}[\![\hbar]\!]$ deforming $R$ such
  that $A=TV[\![\hbar]\!]/(E)$ and $\phi: A /\hbar A \to B$ is the
  identity on $V$.  

  (b)(*, but with the solution outlined) We are interested below in the case that $R \subseteq V^{\otimes
    2}$, i.e., $B$ is a quadratic algebra, and that $E \subseteq
  V^{\otimes 2}[\![\hbar]\!]$, i.e., $B$ is also quadratic (as a
  $\bk[\![\hbar]\!]$-algebra).  (The proof of (a) then adapts to show
  that such an $E$ exists for every graded formal deformation $A$ of
  $B$; moreover, the span $\bk[\![\hbar]\!] \cdot E \cong
  R[\![\hbar]\!]$ is in fact canonical, and independent of the choice
  of $E$ above, since it is the degree-two part of the kernel of
  $TV[\![\hbar]\!]  \to A$.)

The usual version of the Koszul deformation principle is as follows:
\begin{theorem}\cite{Dri-qcrqc,PBW,BGS-Kdprt,PP-qa}
 \label{t:kdp} (Koszul deformation principle) Assume that $B$
  is Koszul. If $E \subseteq V^{\otimes 2}[\![\hbar]\!]$ is a formal
  deformation, then $A := TV[\![\hbar]\!]/(E)$ is a (flat) formal
  deformation of $B$ if and only if it is flat in weight three, i.e.,
  the surjection $B_3 \to A_3/\hbar A_3$ is an isomorphism.
\end{theorem}
Use this to prove Theorem \ref{t:kdp-filt}, as follows.  Given a
Koszul algebra $B= TV/(R)$ and a filtered deformation $E \subseteq
T^{\leq 2} V$ of $R$, i.e., such that the composition $E \to T^{\leq
  2} \twoheadrightarrow V \otimes V$ is an isomorphism onto $R$, we
can form a corresponding formal deformation as follows. 
Let $A := TV/(E)$. First we homogenize $A$ by
forming the Rees algebra over the parameter $t$:
\[
R_{t}A := \bigoplus_{i \geq 0} A_{\leq i} t^i.
\]
This algebra is graded by degree in $t$, and deforms $B$ in the sense
that $R_t A / t R_t A \cong B$.  Thus $R_t A$ is a quadratic
algebra. Set $\tilde V := V \oplus \bk \cdot t$. Then we in fact have
$R_t A = T \tilde V / \tilde E$, where $\tilde E$ is the span of $t v
- v t$ for $v \in V$ together with the homogenized versions of
the elements of $E$: given a relation $r = r_2 + r_1 + r_0 \in E$ for
$r_i \in V^{\otimes i}$, we have an associated element $\tilde r \in
\tilde E$ given by $\tilde r = r_2 + r_1 t + r_0 t^2$.  This $\tilde E$
is spanned by the $\tilde r$ and the $tv-vt$.

Moreover, $R_t A$ can also be viewed as a filtered deformation of $B[t]$:
put the filtration on $R_t A$ by the filtration in $A$
(ignoring the $t$), i.e., $(R_t A)_{\leq m} = A_{\leq m} t^m \otimes
\bk[t]$ (thus, this is \emph{not} the filtration given by the grading
on $R_t A$).  Then $\gr(R_t A) = B[t]$.

Using the preceding filtration,
form  the completed Rees algebra in $\hbar$,
\[
\hat A := \widehat R_{\hbar} R_t A := \prod_{i \geq 0} (R_t A)_{\leq
  i} \cdot \hbar^i.
\]
The result is an algebra $\hat A$ such that $\hat A/\hbar \hat A \cong
B[t]$ (rather than $R_t A$), since we used the filtration on $R_t A$
coming from $A$ only. Moreover, $\hat A$ is a quadratic formal
deformation of $R_t A$, using the grading on $R_t A$, placing $\hbar$
in degree zero.

Finally, $A$ satisfies the assumption of Theorem \ref{t:kdp-filt},
i.e., that $\gr_3(E \otimes V \cap (V+\bk) \otimes E) = (R \otimes V \cap V
\otimes R)$ (where $\gr_3: T^{\leq 3}V \to V^{\otimes 3}$ is the
projection modulo $V^{\otimes \leq 2}$), if and only if $\hat A$
satisfies the condition that the natural map
\[
(R \otimes V \cap V \otimes R) \to (\tilde E \otimes \tilde V \cap
\tilde V \otimes \tilde E) / \hbar (\tilde E \otimes \tilde V \cap
\tilde V \otimes \tilde E)
\]
is an isomorphism.

Check that the latter condition is equivalent to the condition that,
in degree three, $\hat A$ is a flat deformation of $B[t]$, i.e., that
$(B[t])_3 \cong (\hat A)_3 / \hbar (\hat A)_3$.  (Note that $(B[t])_3
= B_3 \oplus B_2 t \oplus B_1 t^2 \oplus \bk t^3$, and the difficulty
is in dealing with the $B_3$ part).

Thus, Theorem \ref{t:kdp} applies and yields that, if $B$ is Koszul,
then for every filtered deformation $E$ of the relations $R$ of $B$,
the algebra $A=TV/(E)$ is a (flat) filtered deformation of $B$ if and
only if $\hat A$ is a (flat) formal graded deformation, which is true
if and only if the assumption of Theorem \ref{t:kdp-filt} is
satisfied, which proves that theorem.
\end{exercise}

\begin{exercise}\label{exer:koszul-def}
 Parallel to Exercise \ref{exer:koszul}, show
  that, if $A$ is a (flat) graded formal deformation of a Koszul
  algebra $B$ as in Theorem \ref{t:kdp}, then $A$ is also Koszul over
  $\bk[\![\hbar]\!]$ (meaning, $\Tor_i^A(\bk[\![\hbar]\!],\bk[\![\hbar]\!])$ is concentrated in
  degree $i$, or alternatively $A$ admits a free graded $A \otimes_{\bk[\![\hbar]\!]} A^{\op}$-module
  resolution $P_\bullet^\hbar \to A$ with $P_i^\hbar$ generated in
  degree $i$).
\end{exercise}


\section{Hochschild cohomology and infinitesimal
  deformations}\label{s:hoch}
In this section, we return to a more basic topic, that of Hochschild
(co)homology of algebras.  We first describe the explicit complexes
used to compute these, and how they arise from the bar resolution of
an algebra as a bimodule over itself.  Using this description, we give
the standard interpretation of the zeroth, first, second, and third
Hochschild cohomology groups as the center, the vector space of outer
derivations, the vector space of infinitesimal deformations up to
equivalence, and the vector space of obstructions to lifting an
infinitesimal deformation to a second-order deformation. We also
explain the definition of Calabi-Yau algebras, both to illustrate an
application of Hochschild cohomology, and to further understand our
running examples of Weyl algebras and universal enveloping algebras.

\subsection{The bar resolution}
To compute Hochschild (co)homology using Definition \ref{d:hh-co}, one
can resolve $A$ as an $A^e$-module, i.e., $A$-bimodule. The standard
way to do this is via the \emph{bar resolution}:
\begin{definition}
  The bar resolution, $C^{\obar}(A)$, of an associative algebra $A$
  over $\bk$ is the complex
\begin{equation}
\begin{tikzpicture}
\node (c1) at (-6, 0) {$\cdots$};
\node (c2) at (-3,0) {$A \otimes A \otimes A$};

\node (c3) at (0,0) {$A \otimes A$};
\node (c4) at (3,0) {$A,$};

\node (c5) at (-3,-1) {$a \otimes b \otimes c$};
\node (c6) at (0,-1) {$ab \otimes c - a \otimes bc$};
\node (c7) at (0,-2) {$a \otimes b$};
\node (c8) at (3,-2) {$ab$};

\draw[->] (c1) -- (c2);
\draw[->] (c2) -- (c3);
\draw[->] (c3) -- (c4);

\draw[|->] (c5) -- (c6);
\draw[|->] (c7) -- (c8);

\end{tikzpicture}
\end{equation}
More conceptually, we can define the complex as
\[
T_A (A \cdot \epsilon \cdot A), \quad d = \partial_\epsilon,
\]
viewing $A \cdot \epsilon \cdot A$ as an $A$-bimodule, assigning
degrees by $|\epsilon|=1, |A|=0$, and viewing the differential $d
= \partial_\epsilon$ is a \emph{graded derivation} of degree $-1$,
i.e.,
\[
\partial_\epsilon(x \cdot y) = \partial_\epsilon(x) \cdot y
+ (-1)^{|y|} x \cdot \partial_\epsilon(y).
\]
Finally, $\partial_\epsilon(\epsilon)=1$ and $\partial_\epsilon(A)=0$.
\end{definition}
\subsection{The Hochschild (co)homology complexes}
Using the bar resolution, we conclude that
\begin{equation}
  \HH_i(A,M) = \Tor_i^{A^e}(A,M) = H_i(C^\obar_\bullet(A)
  \otimes_{A^e} M) =
  H_i(C_\bullet(A,M)),
\end{equation}
where
\begin{equation}
  \begin{tikzpicture}
\node[left] (1) at (-4,0) {$C_\bullet(A,M)$};
\node[right] (1') at (-4,0) {$= (M \otimes_A C^\obar_\bullet(A)) /
      [A, M \otimes_A C^\obar_\bullet(A)]$};
\node (2)[right] at (-4,-1)  {$= \cdots$};
\node (3) at (-1,-1)   {$M \otimes A \otimes A$};
\node (4) at (3,-1)    {$M \otimes A$};
\node (5) at (6,-1)    {$M$};
\node (6) at (-1,-2)    {$(m \otimes b \otimes c)$};
\node (7) at (3,-2)    {$mb \otimes c -
      m \otimes bc + cm \otimes b$,};
\node (8) at (3,-3)    {$m \otimes a$};
\node (9) at (6,-3)    {$ma-am$.};
\draw[->] (2) -- (3);
\draw[->] (3) -- (4);
\draw[->] (4) -- (5);
\draw[|->] (6) -- (7);
\draw[|->] (8) -- (9);


\end{tikzpicture}
\end{equation}
Similarly, 
\begin{equation}
  \HH^i(A,M) = \Ext^i_{A^e}(A,M) = H^i(C^\bullet(A,M)),
\end{equation}
where
\begin{equation}\label{e:hoch-cocomplex}
\begin{tikzpicture}
\node[left] (1) at (-5,0) {$C^\bullet(A,M)$};
\node[right] (2) at (-5,0) {$=
\Hom_{A^e}(C^\obar_\bullet(A),
  M))$};

\node[right](3) at (-5,-1) {$= \cdots$};
\node (4) at (-1,-1) {$\Hom_\bk(A \otimes A, M)$};
\node (5) at (3,-1) {$\Hom_\bk(A, M)$};
\node (6) at (6,-1) {$M$,};

\node (7) at (-1,-2) {$(d\phi)(a \otimes b) = a\phi(b) - \phi(ab) + \phi(a) b$};
\node (8) at (3,-2) {$\phi$};
\node (9) at (3,-3) {$(dx)(a)=ax-xa$};
\node (10) at (6,-3) {$x.$};

\draw[->] (4) -- (3);
\draw[->] (5) -- (4);
\draw[->] (6) -- (5);
\draw[|->] (8) -- (7);
\draw[|->] (10) -- (9);
\end{tikzpicture}
\end{equation}
As for cohomology, we will use the notation $C^\bullet(A) := C^\bullet(A,A)$.
\begin{remark}\label{r:lhc}
 To extend Hochschild cohomology $\HH^\bullet(A)$ to the $C^\infty$ context,
  i.e., where $A =
  \cO(X) := C^\infty(X)$ for $X$ a $C^\infty$ manifold, we do not make
  the same definition as above (i.e., we do not treat
  $A$ as an ordinary associative algebra). Instead, we define
  the Hochschild cochain complex so as to place in degree $m$ the
 subspace of
 $\Hom_\bk(A^{\otimes m}, A)$ of \emph{smooth polydifferential
   operators}, i.e., linear maps spanned by tensor products of smooth
 (rather than algebraic) differential operators $A \to A$, or
 equivalently, operators which are spanned by polynomials in smooth
 vector fields (which are not the same as algebraic derivations in the
 smooth setting: a smooth vector field is always a derivation of $A$
 as an abstract algebra, but not conversely). This can be viewed as
 restricting to the ``local'' part of Hochschild cohomology 
 (because polydifferential operators are local in the sense that the
 result at a point $x$ only depends on the values in a neighborhood of
 $x$).
\end{remark}
\subsection{Zeroth Hochschild homology}
We see from the above that $\HH_0(A) = A/[A,A]$, where the quotient
is taken as a vector space.
\begin{example}
  Suppose that $A = TV$ is a free algebra on a vector space $V$.  Then
  $\HH_0(A) = TV/[TV,TV]$ is the vector space of \emph{cyclic words}
  in $V$.

  Similarly, if $A = \bk Q$ is the path algebra of a quiver $Q$ (see
  \S \ref{ss:quiver}), then $\HH_0(A) = \bk Q / [\bk Q, \bk Q]$ is the
  vector space of \emph{cyclic paths} in the quiver $Q$ (which by
  definition do not have an initial or terminal vertex).
\end{example}
\begin{example}
  If $A$ is commutative, then $\HH_0(A) = A$, since $[A,A]=0$.
\end{example}

\subsection{Zeroth Hochschild cohomology}

Note that $\HH^0(A) = \{a \in A \mid ab-ba=0, \forall b \in A\} = Z(A)$,
the center of the algebra $A$.
\begin{example}
  If $A = TV$ is a tensor algebra and $\dim V \geq 2$, then $\HH^0(A)
  = \bk$: only scalars are central.  The same is true if $A = \bk Q$
  for $Q$ a connected quiver which has more than one edge or more than
  one vertex (see \S \ref{ss:quiver}).
\end{example}
\begin{example}
  If $A$ is commutative, then $\HH^0(A) = A$, since every element is
  central.
\end{example}

\subsection{First Hochschild cohomology}
A Hochschild one-cocycle is an element $\phi \in \End_\bk(A)$ such
that $a\phi(b)+\phi(a)b = \phi(ab)$ for all $a, b \in A$. These are
the derivations of $A$.  A Hochschild one-coboundary is an element
$\phi = d(x), x \in A$, and this has the form $\phi(a)=ax-xa$ for all
$a \in A$. Therefore, these are the inner derivations. We conclude
that
\[
\HH^1(A)= \Out(A) := \Der(A) / \Inn(A),
\]
the vector space of outer derivations of $A$, which is by definition
the quotient of all derivations $\Der(A)$ by the inner derivations
$\Inn(A)$.  We remark that $\Inn(A) \cong A / Z(A) = A / \HH^0(A)$,
since an inner derivation $a \mapsto ax-xa$ is zero if and only if $x$
is central.
\begin{example}
  If $A = TV$ for $\dim V \geq 2$, then $\HH^1(A) = \Out(TV) =
  \Der(TV)/\Inn(TV) = \Der(TV) /\overline{TV}$, where $\overline{TV} =
  TV / \bk$, since $Z(TV)=\bk$.  Explicitly, derivations of $TV$ are
  uniquely determined by their restrictions to linear maps $V \to TV$,
  i.e., $\Der(TV) \cong \Hom_\bk(V, TV)$.  So we get $\HH^1(A) \cong
  \Hom_\bk(V, TV) / \overline{TV}$. (Note that the inclusion
  $\overline{TV} \to \Hom_{\bk}(V, TV)$ is explicitly given by $f
  \mapsto \ad(f)$, where $\ad(f)(v) = fv-vf$.)
\end{example}
\begin{example}
  If $A$ is commutative, then $\Inn(A) = 0$, so $\HH^1(A)=\Der(A)$.
  If, moreover, $A = \cO(X)$ is the commutative algebra of functions on
  an affine variety, then this is also known as the global vector
  fields $\Vect(X)$, as we discussed.
\end{example}
\begin{example}
In the case $A = C^\infty(X)$, restricting our Hochschild cochain complex
to differential operators in accordance with Remark \ref{r:lhc}, then
$\HH^1(A)$ is the space of smooth vector fields on $X$.
\end{example}
\begin{example}
  In the case that $V$ is smooth affine (or more generally normal),
  $\bk$ has characteristic zero, and $\Gamma$ is a finite group of
  automorphisms of $V$ which acts freely outside of a codimension-two
  subset of $V$, then all vector fields on the smooth locus of
  $V/\Gamma$ extend to derivations of $\cO(V/\Gamma)=\cO(V)^\Gamma$,
  since $V/\Gamma$ is normal.  Moreover, vector fields on the smooth
  locus of $V/\Gamma$ are the same as $\Gamma$-invariant vector fields
  on the locus, call it $U \subseteq V$, where $\Gamma$ acts freely.
  These all extend to all of $V$ since $V\setminus U$ has codimension
  two.  Thus we conclude that $\HH^1(\cO(V)^\Gamma) =
  \Der(\cO(V)^\Gamma) = \Vect(V)^\Gamma$ is the space of
  $\Gamma$-invariant vector fields on $V$ (the first equality here is
  because $\cO(V)^\Gamma$ is commutative).  In particular this
  includes the case where $V$ is a symplectic vector space and $\Gamma
  < \Sp(V)$, since the hyperplanes with nontrivial stabilizer group
  must be symplectic and hence of codimension at least two.
\end{example}

\subsection{Infinitesimal deformations and second Hochschild
  cohomology}
We now come to a key point.  A Hochschild two-cocycle is an element
$\gamma \in \Hom_\bk(A \otimes A, A)$ satisfying
\begin{equation}\label{e:h-tc}
a\gamma(b \otimes c) - \gamma(ab \otimes c) + \gamma(a \otimes bc) - \gamma(a
\otimes b)c = 0.
\end{equation}
This has a nice interpretation in terms of infinitesimal deformations
(Definition \ref{d:infinitesimal-def}).
Explicitly, an infinitesimal deformation is given by a linear map $\gamma: A
\otimes A \to A$, by the formula
\[
a \star_\gamma b = ab + \varepsilon \gamma(a \otimes b).
\]
Then, the associativity condition is exactly \eqref{e:h-tc}.

Moreover, we can interpret two-coboundaries as \emph{trivial}
infinitesimal deformations. More generally, we say that two
infinitesimal deformations $\gamma_1, \gamma_2$ are \emph{equivalent}
if there is a $\bk[\varepsilon]/(\varepsilon^2)$-module automorphism
of $A_\varepsilon$ which is the identity modulo $\varepsilon$ which
takes $\gamma_1$ to $\gamma_2$. Such a map has the form $\phi := \Id +
\varepsilon \cdot \phi_1$ for some linear map $\phi_1: A \to A$, i.e.,
$\phi_1 \in C^1(A)$.  We then compute that
\[
\phi^{-1} \bigl( \phi(a) \star_\gamma \phi(b) \bigr) = a
\star_{\gamma + d\phi_1} b.
\]
We conclude:
\begin{proposition}
  $\HH^2(A)$ is the vector space of equivalence classes of
  infinitesimal deformations of $A$.
\end{proposition}
\begin{example}
  For $A = TV$, a tensor algebra, we claim that $\HH^2(A) = 0$, so
  there are no nontrivial infinitesimal deformations.  Indeed, one can
  construct a short bimodule resolution of $A$,
\begin{gather*}
  0 \longrightarrow A \otimes V \otimes A \longrightarrow A \otimes A \onto A, \\
  a \otimes v \otimes b \mapsto av \otimes b - a \otimes vb, \quad a
  \otimes b \mapsto ab.
\end{gather*}
Since this is a projective resolution of length one, we conclude that
 $\Ext^2_{A^e}(A, M) = 0$ for all bimodules $M$, i.e., $\HH^2(A, M) = 0$
  for all bimodules $M$.
\end{example}
\begin{example} 
  If $A = \Weyl(V)$ for a symplectic vector space $V$, then
  $\HH^\bullet(A) = \bk$ (in degree zero), and hence $\HH^2(\Weyl(V))$ and $\HH^1(\Weyl(V))$ are zero,
 so that there are no nontrivial infinitesimal deformations nor
  outer derivations. Moreover, $\HH_\bullet(A) = \bk[-\dim V]$ (i.e.,
  $\bk$ in degree $\dim V$ and zero elsewhere). To see this, we can
  use the resolution of Exercise \ref{exer:koszul-complexes} (see also
  the solution).
\end{example}
\subsection{Remarks on Calabi-Yau algebras}
The Koszul resolutions above imply that $\Weyl(V)$ and $\Sym V$ are
\emph{Calabi-Yau algebras of dimension $\dim V$}, which we will define
as follows (and which was also mentioned in
Rogalski's chapter.
One
can also conclude from this that $\Weyl(V) \rtimes \Gamma$ and $\Sym V
\rtimes \Gamma$ are Calabi-Yau of dimension $\dim V$ for $\Gamma <
\Sp(V)$ finite, as we will do in the exercises.

\begin{definition}\label{defn:invbimod}
  An $A$-bimodule $U$ is right invertible if there exists an
  $A$-bimodule $R$ such that $U \otimes_A R \cong A$ as
  $A$-bimodules. It is left invertible if there exists an $A$-bimodule
  $L$ such that $L \otimes_A U \cong A$ as $A$-bimodules.  It is
  invertible if it is both left and right invertible.
\end{definition}
Note that, if $U$ is invertible, then any left inverse is a
right inverse (just as in the case of monoids), so this is also the
unique (two-sided) inverse of $U$ up to isomorphism.
\begin{question} Can you find an example of a left but not right invertible bimodule?
\end{question}
The following basic example will be our typical invertible bimodule:
\begin{example}\label{ex:asigma}
Given an algebra automorphism $\sigma: A \to A$, we can define
the $A$-bimodule $A^\sigma$, which as a vector space is merely $A$ itself, but with the bimodule action given by the formula, for $a, b \in A$ and $m \in A^{\sigma}$,
\[
a \cdot m \cdot b = (am\sigma(b)).
\]
\end{example}
\begin{exercise}\label{exer:asigma}
\begin{itemize}
\item[(a)]  Prove that, if $M$ is an $A$-bimodule which is free of rank
  one as a left and as a right module (i.e., isomorphic to $A$ as a
  left and as a right module) using the same generator $m \in M$, then
  $M \cong A^\sigma$ for some algebra automorphism $\sigma: A \to A$.
  Hint: Let $1 \in M$ be a left module generator, realizing $A \cong A
  \cdot 1 = M$.  Consider the map $\sigma: A \to A$ given by $1 \cdot
  h = \sigma(h) \cdot 1$.  Show that $\sigma$ is a $\bk$-algebra homomorphism,
  with inverse $\sigma^{-1}$ given by $h \cdot 1 = 1 \cdot
  \sigma^{-1}(h)$.


\item[(b)] Now assume $A$ is graded and $M$ is a graded bimodule
  obeying the assumptions of (a), with common generator $m$ is in
  degree zero.  Conclude that $M \cong A^\sigma$ as a graded bimodule,
  where $\sigma$ is a graded automorphism.

\end{itemize}
\end{exercise}
In order to define the notion of Calabi-Yau and twisted Calabi-Yau
algebras, we will need the following exercise:
\begin{exercise}\label{exer:hh-bimod}
  Verify that $\HH^\bullet(A, A\otimes A)$ is a canonically an
  $A$-bimodule using the \emph{inner} action, i.e., the action
  obtained from the $A$-bimodule structure on $M=A \otimes A$ given,
  for $x \otimes y \in M$ and $a,b \in A$, by the formula $a(x \otimes
  y)b = xb \otimes ay$.  The most explicit way to do this is to use
  \eqref{e:hoch-cocomplex}, where the action is inner multiplication
  on the output, i.e., $(a \cdot \phi \cdot b)(x_1 \otimes \cdots
  \otimes x_n) := a \cdot \phi(x_1 \otimes \cdots \otimes x_n) \cdot b$.
\end{exercise}
\begin{definition}\label{d:homsmooth} An algebra $A$ is \emph{homologically
    smooth} if $A$ has a finitely-generated projective $A$-bimodule
  resolution.
\end{definition}
Note that finitely-generated projective, for a complex of (bi)modules,
means that the complex is of finite length and consists of
finitely-generated projective (bi)modules.
\begin{remark}
  More generally, a complex of modules $M$ over an algebra $B$ is
  called perfect if it is a finite complex of finitely-generated
  projective modules.  If one works in the derived category $D^b(B)$
  of $B$-modules, then a complex $P_\bullet$ is quasi-isomorphic to a
  perfect complex if and only if it is \emph{compact}, i.e.,
  $\Hom(P_\bullet, -)$ commutes with arbitrary direct sums (it is
  automatic that $\Hom(M,-)$ commutes with finite direct sums for any
  $M$, but the condition to commute with arbitrary direct sums is much
  more subtle).  Compact objects of derived categories are extremely
  important.
\end{remark}
\begin{remark}\label{r:hd}
  In particular, homological smoothness implies that $A$ has finite
  Hochschild dimension: recall that this means that, for some $N \geq
  0$, we have $\HH^i(A,M) = 0$ for all $i > N$ and all bimodules $M$
  (the minimal such $N$ is then called the Hochschild dimension of
  $A$).  This is equivalent to the condition that $A$ have a bounded
  projective resolution, i.e., a resolution by a complex with finitely
  many nonzero terms, each of which are projective.  To be
  homologically smooth requires that these terms be finitely-generated
  as well (this is a subtle strengthening of the condition in
  general).  When $A$ is homologically smooth, one can show that it
  has a finitely-generated projective bimodule resolution whose length
  (i.e., number of terms) is the Hochschild dimension of $A$.
\end{remark}
\begin{definition}\label{d:cya}
  $A$ is a Calabi-Yau algebra of dimension $d$ if $A$ is homologically
  smooth and $\HH^\bullet(A, A \otimes A) = A[-d]$ as a graded
  $A$-bimodule (i.e., $\HH^i(A, A\otimes A) = 0$ for $i \neq d$, and
  $\HH^d(A, A\otimes A) = A$ as a graded $A$-bimodule).

  More generally, $A$ is \emph{twisted Calabi-Yau} of dimension $d$ if
  $A$ is homologically smooth and $\HH^\bullet(A, A \otimes A) =
  U[-d]$ as a graded $A$-bimodule, where $U$ is an invertible
  $A$-bimodule.
\end{definition}
Here, in $\HH^i(A, A \otimes A)$, $A \otimes A$ is considered as a
bimodule with the \emph{outer} bimodule structure, and the remaining
inner structure induces a bimodule structure on $A$, as shown in
Exercise \ref{exer:hh-bimod} above.
Typically, in the twisted Calabi-Yau case, $U \cong A^\sigma$ as in
Example \ref{ex:asigma}, which as we know from Exercise
\ref{exer:asigma}, is equivalent to saying that $\HH^d(A, A \otimes A)
\cong A$ as both left and right modules individually, with a common
generator (for example, $A$ could be graded, and the common generator
of $\HH^d(A, A \otimes A)$ could be the unique nonzero element up to
scaling in degree zero).

\begin{remark}
  By the Van den Bergh duality theorem (Theorem \ref{t:vdbdual}), if
  $A$ is Calabi-Yau of dimension $d$, then $A$ has Hochschild
  dimension $d$ (since $\HH^i(A, M) = 0$ for all $i > d$ and all
  bimodules $M$, and $\HH^d(A, A\otimes A) \neq 0$). It follows (as
  pointed out in Remark \ref{r:hd}) that $A$ has a finitely-generated
  projective bimodule resolution of length $d$.  Thus, we could have
  equivalently assumed the latter condition in Definition \ref{d:cya}
  rather than homological smoothness.
\end{remark}
\begin{exercise}\label{exer:kos-cy}
\begin{itemize}
\item[(a)] Prove from the Koszul resolutions from Exercise
  \ref{exer:koszul} that $\Sym V$ and $\Weyl(V)$ are Calabi-Yau of
  dimension $\dim V$.
\item[(b)] Let $\Gamma < \Sp(V)$ be finite and $\bk$ have
  characteristic zero.  Take the Koszul resolutions and apply $M
  \mapsto M \otimes \bk[\Gamma]$ to all terms, considered as bimodules
  over $A \rtimes \Gamma$ where $A$ is either $\Sym V$ or $\Weyl(V)$.
  This bimodule structure is given by
\[
(a \otimes g)(m \otimes h)(a' \otimes g') = (a \cdot g(m) \cdot
gh(a')) \otimes (ghg').
\]
Prove that the result are resolutions of $A \rtimes \Gamma$ as a bimodule
over itself.  Conclude that $A \rtimes \Gamma$ is also Calabi-Yau of
dimension $\dim V$.  In the case $A = \Sym V$, show that $A \rtimes \Gamma$
is actually Calabi-Yau given only that
 $\Gamma < \SL(V)$ (not $\Sp(V)$), still assuming $\Gamma$ is finite.
\end{itemize}
\end{exercise}
\begin{remark} \label{r:twcy-sd}
Similarly, one can show that, if $A$ is twisted
  Calabi-Yau, $\bk$ has characteristic zero, and $\Gamma$ is a finite
  group acting by automorphisms on $A$, then $A \rtimes \Gamma$ is
  also twisted Calabi-Yau.  In the case that $A$ is Calabi-Yau,
  $\Gamma$ must obey a unimodularity condition for $A \rtimes \Gamma$
  to also be Calabi-Yau.
\end{remark}

\subsection{Obstructions to second-order deformations and third
  Hochschild cohomology}\label{ss:obst-3o}
Suppose now that we have an infinitesimal deformation given by
$\gamma_1: A \otimes A \to  A$.  To extend this to a
second-order deformation, we require $\gamma_2: A \otimes A \to
 A$, such that
\[
a \star b := ab + \varepsilon \gamma_1(a \otimes b) + \varepsilon^2
\gamma_2(a \otimes b)
\]
defines an associative product on $A \otimes
\bk[\varepsilon]/(\varepsilon^3)$.

Looking at the new equation in second degree, this can be
written as
\[
a\gamma_2(b \otimes c) - \gamma_2(ab \otimes c) + \gamma_2(a \otimes bc) -
\gamma_2(a \otimes b)c = \gamma_1(\gamma_1(a \otimes b) \otimes c)
- \gamma_1(a \otimes \gamma_1(b \otimes c)).
\]
The LHS is $d \gamma_2(a \otimes b \otimes c)$, so the condition for
$\gamma_2$ to exist is exactly that the RHS is a Hochschild
coboundary.  Moreover, one can check that the RHS is always a
Hochschild three-cocycle (we will give a more conceptual explanation
when we discuss the Gerstenhaber bracket).  So the element on the RHS
defines a class of $\HH^3(A)$ which is the \emph{obstruction} to
extending the above infinitesimal deformation to a second-order
deformation:
\begin{corollary}
  $\HH^3(A)$ is the space of \emph{obstructions} to extending
  first-order deformations to second-order deformations.  If
  $\HH^3(A) = 0$, then all first-order deformations extend to
  second-order deformations.
\end{corollary}
We next consider general $n$-th order deformations.  By definition,
such a deformation is a deformation over   $\bk[\varepsilon]/(\varepsilon^{n+1})$.
\begin{exercise}\label{exer:obstr-nord}(*)
  Show that the obstruction to extending an $n$-th order deformation
  $\sum_{i=1}^n \varepsilon^i \gamma_i$ (where here
  $\varepsilon^{n+1}=0$) to an $(n+1)$-st order deformation
  $\sum_{i=1}^{n+1} \varepsilon^i \gamma_i$ (now setting
  $\varepsilon^{n+2}=0$), i.e., the existence of a $\gamma_{n+1}$ so
  that this defines an associative multiplication on $A \otimes
  \bk[\varepsilon] / (\varepsilon^{n+2})$, is also a class in
  $\HH^3(A)$.

  Moreover, if this class vanishes, show that two different choices of
  $\gamma_{n+1}$ differ by Hochschild two-cocycles, and that two are
  equivalent (by applying a
  $\bk[\varepsilon]/(\varepsilon^{n+2})$-module automorphism of $A
  \otimes \bk[\varepsilon] / (\varepsilon^{n+2})$ of the form $\Id +
  \varepsilon^{n+1} \cdot f$) if and only if the two choices of
  $\gamma_{n+1}$ differ by a Hochschild two-coboundary.  Hence, when
  the obstruction in $\HH^3(A)$ vanishes, the set of possible
  extensions to a $(n+1)$-st order deformation (modulo gauge
  transformations which are the identity modulo $\varepsilon^n$) form
  a set isomorphic to $\HH^2(A)$ (more precisely, it forms a
  \emph{torsor} over the vector space $\HH^2(A)$, i.e., an affine
  space modeled on $\HH^2(A)$ without a chosen zero element).  We will give a
  more conceptual explanation when we discuss formal deformations.
\end{exercise}
Note that, when $\HH^3(A) \neq 0$, it can still happen that 
infinitesimal deformations extend to all orders. For example, by
Theorem \ref{t:ps-quant}, this happens for Poisson structures on
smooth manifolds (a Poisson structure yields an infinitesimal
deformation by, e.g., $a \star b = ab + \frac{1}{2} \{a,b\} \cdot
\varepsilon$; this works for arbitrary skew-symmetric biderivations
$\{-,-\}$, but only the Poisson ones, i.e., those satisfying the
Jacobi identity, extend to all orders).

However, finding this quantization is \emph{nontrivial}: even though
Poisson bivector fields are those classes of $\HH^2(A)$ whose
obstruction in $\HH^3(A)$ to extending to second order vanishes, if
one does not pick the extension correctly, one \emph{can} obtain an
obstruction to continuing to extend to third order, etc.  In fact, the
proof of Theorem \ref{t:ps-quant} describes the space of \emph{all}
quantizations: as we will see, deformation quantizations are
equivalent to formal deformations of the Poisson structure.

\subsection{Deformations of modules and Hochschild cohomology}\label{ss:mod-hoch}
Let $A$ be an associative algebra and $M$ a module over $A$. Recall
that Hochschild (co)homology must take coefficients in an
\emph{$A$-bimodule}, not an $A$-module.  Given $M$, there is a
canonical associated bimodule, namely $\End_\bk(M)$ (this is an
$A$-bimodule whether $M$ is a left or right module; the same is true
for $\Hom_\bk(M,N)$ where $M$ and $N$ are both left modules, or
alternatively both right modules).
\begin{lemma} $\HH^i(A, \End_\bk(M)) \cong \Ext^i_A(M,M)$ for all $i \geq 0$.
More generally, $\HH^i(A, \Hom_\bk(M,N)) \cong \Ext^i_A(M,N)$ for all $A$-modules $M$ and $N$.
\end{lemma}
\begin{proof} We prove the second statement.
First of all, for $i=0$, 
\[
\HH^0(A, \Hom_\bk(M,N)) = \{\phi \in \Hom_\bk(M,N) \mid a\cdot \phi =
\phi \cdot a, \forall a \in A\} = \Hom_A(M, N).
\]
Then the statement for higher $i$ follows because they are the derived
functors of the same bifunctors $((A\text{-mod})^\op \times A\text{-mod})
\to \bk-mod$.

Explicitly, if $P_\bullet \onto A$ is a projective $A$-bimodule
resolution of $A$, then $P_\bullet \otimes_A M \onto M$ is a
projective $A$-module resolution of $M$, and
\[
\RHom_A^\bullet(M,N) \cong \Hom_A(P_\bullet \otimes_A M, N) = 
\Hom_{A^e}(P_\bullet, \Hom_\bk(M, N)) \cong
\RHom_{A^e}^\bullet(A, \Hom_\bk(M,N)),
\]
where for the middle equality, we used the adjunction $\Hom_B(X
\otimes_A Y, Z) = \Hom_{B \otimes A^{\op}}(X, \Hom_\bk(Y, Z))$, with
$X$ a $(B,A)$-bimodule, $Y$ a left $A$-module, and $Z$ a left
$B$-module.
\end{proof}
In particular, this gives the most natural interpretation of
$\HH^0(A, \End_\bk(M))$: this is just $\End_A(M)$.  For the higher
groups we recall the following standard descriptions of $\Ext_A^1(M,M)$
and $\Ext_A^2(M,M)$, which are convenient to see using Hochschild
cochains valued in $M$.
\begin{definition}
  A deformation of an $A$-module $M$ over an augmented commutative
  $\bk$ algebra $R = \bk \oplus R_+$ is an $A$-module
  structure on $M \otimes_\bk R$, commuting with the $R$ action,
  such that $(M \otimes_\bk R) \otimes_R (R/R_+) \cong M$ as an
  $A$-module.
\end{definition}
Let $M$ be an $A$-module and let
$\rho: A \to \End_\bk(M)$ be the original (undeformed)
  module structure.
\begin{proposition}\label{p:hoch-cc}
\begin{itemize}
\item[(i)] The space of Hochschild one-cocycles valued in $\End_\bk(M)$
 is the space of \emph{infinitesimal deformations} of the module $M$ over
$R=\bk[\varepsilon]/(\varepsilon^2)$;
\item[(ii)] Two such deformations are equivalent up to an $R$-module
automorphism of $M \otimes_\bk R$
which is the identity modulo $\varepsilon$ if and only if they differ by
a Hochschild one-coboundary. 

Thus $\HH^1(A, \End_\bk(M)) \cong \Ext_A^1(M,M)$ classifies infinitesimal
deformations of $M$.

\item[(iii)] The obstruction to extending an infinitesimal deformation
  with class $\gamma \in \Ext^1_A(M,M)$ to a second-order
  deformation, i.e., over $\bk[\varepsilon]/(\varepsilon^3)$, is the
  element
\[
\gamma \cup \gamma \in \Ext_A^2(M,M) \cong \HH^2(A, \End_\bk(M)),
\]
where $\cup$ is the Yoneda cup product of extensions.
\end{itemize}
\end{proposition}
Here and below we make use of the cup product on
$\HH^\bullet(A,\End_{\bk}(M))=\Ext_A^\bullet(M,M)$ and on
$\HH^\bullet(A, A)$.  We give an explicit description using the
standard cochain complexes \eqref{e:hoch-cocomplex}.  More generally,
let $N$ be an $A$-algebra, i.e., an algebra equipped with an algebra
homomorphism $A \to N$; this makes $N$ also an $A$-bimodule.
Then, if $\gamma_1 \in C^i(A,N)$ and $\gamma_2 \in C^j(A,N)$, set
$\gamma_1 \cup \gamma_2 \in C^{i+j}(A, N)$ by the formula
\[
(\gamma_1 \cup \gamma_2)(x_1,\ldots,x_{i+j}) = \gamma_1(x_1,\ldots,x_i) \gamma_2(x_{i+1},\ldots,x_{i+j}),
\]
and similarly define $\gamma_2 \cup \gamma_1$.  
Hence
also $\HH^{\bullet}(A,N)$ is an associative algebra.
(Note that, while $\HH^{\bullet}(A,A)$ is graded commutative as
explained in Exercise \ref{exer:circ-prod} below, this is \emph{not}
true for $C^\bullet(A,N)$ in general, and in particular for
$\Ext^\bullet(M,M)$, where $N=\End_\bk(M)$.)
\begin{proof}[Proof of Proposition \ref{p:hoch-cc}]
  (i) Hochschild one-cocycles are precisely $\gamma \in
  \Hom_{\bk}(A, \End_\bk(M))$ such that $\gamma(ab) = a \gamma(b) +
  \gamma(a) b$, which are also known as $A$-bimodule derivations valued in
  $\End_{\bk}(M)$. Infinitesimal deformations of the $A$-bimodule $M$
  are given by algebra homomorphisms $A
  \to \End_\bk(M)[\varepsilon]/(\varepsilon^2)$ which reduce to the
  usual action $\rho: A \to \End_\bk(M)$ modulo $\varepsilon$. Given a
  homomorphism $\rho + \varepsilon \phi$ of the latter type, we see
  that $\phi$ is an $A$-bimodule derivation valued in $M$, and
  conversely.

  (ii) If we apply an automorphism $\phi = \Id + \varepsilon \cdot
  \phi_1$ of $M \otimes_\bk R$, for $\phi_1 \in \End_\bk(M)$, then the
  infinitesimal deformation $\gamma$ is taken to $\gamma'$, where
\[
(\rho + \varepsilon \gamma')(a) = \phi \circ (\rho + \varepsilon
\gamma)(a) \circ \phi^{-1} =(\rho + \varepsilon \gamma)(a) +
\varepsilon (\phi_1 \circ \rho(a) - \rho(a) \circ \phi_1) = (\rho +
\varepsilon (\gamma + d \phi_1))(a).
\]
This proves that $\gamma'-\gamma = d\phi_1$, as desired. The converse is
similar and is left to the reader.
  
(iii) 
Working over $\tilde R :=
\bk[\varepsilon]/(\varepsilon^3)$, given a Hochschild one-cocycle
$\gamma_1$, and an arbitrary element $\gamma_2 \in
C^1(A, \End_\bk(M))$,
\begin{multline*}
  (\rho+\varepsilon \gamma_1 + \varepsilon^2 \gamma_2)(ab) -
  (\rho+\varepsilon \gamma_1 + \varepsilon^2 \gamma_2)(a) 
(\rho + \varepsilon \gamma_1 + \varepsilon^2 \gamma_2)(b)
\\=
  \varepsilon^2 \bigl( \gamma_2(ab) -\gamma_1(a) \gamma_1(b) -
  \gamma_2(a) \rho(b) - \rho(a) \gamma_2(b) \bigr),
\end{multline*}
and the last expression equals $-\varepsilon^2 \cdot (
\gamma_1 \cup \gamma_1 + d \gamma_2)(ab)$. Thus
the obstruction to extending the module structure
is the class $[\gamma_1 \cup \gamma_1] \in \HH^2(A, \End_\bk(M))$.
\end{proof}

Finally, we can study general formal deformations:
\begin{definition}
  Given an $A$-module $M$ and a formal deformation $A_R$ of $A$ over
  $R = \bk \oplus R_+$, a formal deformation of $M$ to an
  $A_R$-module is an $A_R$-module structure on 
  $M \hat \otimes_\bk R$ whose tensor product over $R$ with $R/R_+ = \bk$
  recovers $M$.

In the case that $A_R$ is the trivial deformation over $R$, we also call
this a formal deformation of the $A$-module $M$ over $R$.
\end{definition}
Recall in the above definition the notation $\hat \otimes$ from
\eqref{e:hatotimes}.

Analogously to the above, one can study (uni)versal formal
deformations of $M$; when the space of obstructions,
$\Ext_A^2(M,M)$,
is zero, parallel to Theorem \ref{t:hh3=0}, there exists a versal
formal deformation over the base $\hat \cO(\Ext_A^1(M,M))$,
and in the case that $\End_A(M,M)=\bk$ (or more generally, the map
$Z(A) =\HH^0(A,A) \to \HH^0(A, \End_\bk(M)) = \End_A(M,M)$ is
surjective), then this is universal.

More generally, consider a formal deformation $A_\hbar$ of $A$
and ask not for a formal deformation $M_\hbar$ of $M$ as an
$A$-module, but rather for an $A_\hbar$-module $M_\hbar$ deforming
$M$, i.e., satisfying $M_\hbar / \hbar M_\hbar \cong M$ as
$A$-modules, and such that $M_\hbar \cong M[\![\hbar]\!]$ as
$\bk[\![\hbar]\!]$-modules. This recovers formal deformations of
$A$-modules in the case $A_\hbar = A[\![\hbar]\!]$ is the trivial
deformation.  For general $A_\hbar$, however, $M_\hbar$ need not exist
(as one no longer has the trivial deformation $M[\![\hbar]\!]$), so one
can also ask when it does.  In this generality, the calculations of
Proposition \ref{p:hoch-cc} generalize to show that, if $\theta \in
\hbar \cdot C^2(A)[\![\hbar]\!]$ gives a formal deformation $A_\hbar$
of $A$, then the condition for $\gamma \in \hbar \cdot C^1(A,\End_\bk
M)[\![\hbar]\!]$ to give a formal deformation $M_\hbar$ of $M$ to a
module over $A_\hbar$ is
\begin{equation}\label{e:mod-def}
(\rho + \gamma) \circ \theta + d\gamma + \gamma \cup \gamma = 0,
\end{equation}
where here $(\gamma \cup \gamma)(a \otimes b) := \gamma(a) \gamma(b)$.
\begin{example}
  In the presence of a multiparameter formal deformation $(A[\![t_1,
  \ldots, t_n]\!], \star)$ of $A$, this can be used to show the
  existence of a deformation $M_\hbar$ over some restriction of the
  parameter space.  Let $U = \Span \{ t_1, \ldots, t_n \}$ and let
  $\eta: U \to \HH^2(A)$ be the map which gives the class of
  infinitesimal deformation of $A$. We will need the composition $\rho
  \circ \eta: U \to \HH^2(A,\End_\bk M)=\Ext^2_A(M,M)$. Then one can
  deduce from the above
  \begin{proposition}(see, e.g., \cite[Proposition 4.1]{EM-fdrsrawp})
    Suppose that the map $\rho \circ \eta$ is surjective with kernel
    $K$. Then there exists a formal deformation $M_S := (M\hat \otimes
    \cO(S), \rho_S)$ of $M$ over a formal subscheme $S$ of the formal
    neighborhood of the origin of $U$, with tangent space $K$ at the
    origin, which is a module over $(A \hat \otimes \cO(S),\star|_S)$.
    Moreover, if $\Ext^1_A(M,M)=0$, then $S$ is unique and $M_S$ is
    unique up to 
    $\cO(S)$-linear isomorphisms which are the
    identity modulo $\cO(S)_+$.
\end{proposition}
Note that the condition $\Ext^1_A(M,M)=0$ for uniqueness of the formal
definition is consistent with the case where $A_\hbar =
A[\![\hbar]\!]$ is the trivial deformation, since then, as above,
$\Ext^1_A(M,M)$ classifies infinitesimal (and ultimately formal)
deformations of $M$.

In \cite{EM-fdrsrawp}, this was used to show the existence of a unique
family of irreducible representations of wreath product Cherednik
algebras $H_{1,(k,c)}(\Gamma^n \rtimes S_n)$ for $\Gamma < \SL_2(\bC)$
finite deforming a module of the form $Y^{\otimes n} \otimes V$ for
$Y$ an irreducible finite-dimensional representation of
$H_{1,c_0}(\Gamma)$ and $V$ a particular irreducible representation of
$S_n$ (whose Young diagram is a rectangle). Here $c_0 \in
\text{Fun}_{\Gamma}(\mathcal{R},\bC)$ a conjugation-invariant function
on the set $\mathcal{R} \subseteq \Gamma$ of symplectic reflections (here $\mathcal{R} = \Gamma \setminus \{\Id\}$),
and $k \in \bC$, and there is a unique formal subscheme $S \subseteq
\bC \times \text{Fun}_{\Gamma}(\mathcal{R},\bC)$ containing $(0,c_0)$
such that $(k,c)$ is restricted to lie in $S$. Note here that $\bC
\times \text{Fun}_{\Gamma}(\mathcal{R},\bC)$ is viewed as
$\text{Fun}_{\Gamma^n \rtimes S_n}(\mathcal{R}', \bC)$ where
$\mathcal{R}' \subseteq (\Gamma^n \rtimes S_n)$ is the set of
symplectic reflections (they are all conjugate to reflections in
$\Gamma = (\Gamma \times \{1\}^{n-1}) \rtimes \{1\}$ except for the
conjugacy class of the transposition in $S_n$).
\end{example}
\subsection{Exercises}

Exercises from the notes: \ref{exer:asigma}, \ref{exer:hh-bimod},
\ref{exer:kos-cy}, and \ref{exer:obstr-nord}.

Additional exercises:
\begin{exercise}\label{exer:ug-cy}
  Prove that $U\mfg$ is twisted Calabi-Yau, with $\HH^d(U\mfg, U\mfg
  \otimes U\mfg) \cong U\mfg^\sigma$, where $\sigma(x) = x -
  \tr(\ad(x))$.  Therefore, it is Calabi-Yau if (and only if) $\mfg$
  is \emph{unimodular}, i.e., $\tr(\ad(x))=0$ for all $x$ (this 
  also follows from the Koszul resolutions from 
  Exercise \ref{exer:koszul-complexes}).  Observe that every semisimple Lie algebra, i.e., one
  satisfying $\mfg = [\mfg,\mfg]$, is unimodular (and hence the same
  is true for a reductive Lie algebra, i.e., one which is the direct
  sum of a semisimple and an abelian Lie algebra). The same is true for every nilpotent Lie algebra.
\end{exercise}
\begin{remark}\label{r:ug-cy}
  More generally, one can consider, for any finite group $\Gamma$ acting on
  the Lie algebra $\mfg$ by automorphisms, the skew product algebra
  $U\mfg \rtimes \Gamma$.  For $\bk$ of characteristic zero, generalizing
  the above (by tensoring the complexes with $\bk[\Gamma]$ and suitably
  modifying the differentials) one can show that $U\mfg \rtimes \Gamma$ is
  also twisted Calabi-Yau (cf.~Remark \ref{r:twcy-sd}). Then,
  \cite{HVOZ-cCYHad} computes that this is Calabi-Yau if and only if
  $\Gamma < \SL(V)$ and $\mfg$ is unimodular. (This extension should not be
  too surprising, since the skew-product algebra $\cO(V) \rtimes \Gamma$
  itself is Calabi-Yau if and only if $\Gamma < \SL(V)$, which is the
  condition for $\Gamma$ to preserve the volume form giving the Calabi-Yau
  structure on $V$.)
\end{remark}

\begin{exercise}\label{exer:hoch-sg}(*, but with many hints)
  In this exercise we compute the Hochschild (co)homology of a skew
  group ring.

  Let $A$ be an associative algebra over a field $\bk$ of
  characteristic zero, and $\Gamma$ a finite group acting on $A$ by
  automorphisms.  Form the algebra $A \rtimes \Gamma$, which as a
  vector space is $A \otimes \bk[\Gamma]$, with the multiplication
\[
(a_1 \otimes g_1)(a_2 \otimes g_2) = (a_1 g_1(a_2) \otimes g_1 g_2).
\]
Next, given any $\Gamma$-module $N$, let $N^\Gamma := \{n \in N \mid g \cdot n = n \text{ for all } g \in \Gamma\}$, and $N_\Gamma := N / \{n - g \cdot n \mid n \in N, g \in \Gamma\}$ be the invariants and coinvariants, respectively.
\begin{itemize}
\item[(a)] Let $M$ be an $A \rtimes \Gamma$-bimodule. Prove that
\[
\HH^\bullet(A \rtimes \Gamma, M) \cong \HH^\bullet(A, M)^\Gamma, \quad
\HH_\bullet(A \rtimes \Gamma, M) \cong \HH_\bullet(A, M)_\Gamma,
\]
where in the RHS, 
$\Gamma$ acts on $A$ and $M$ via the adjoint action, $g \cdot_{\text{Ad}} m = (gmg^{-1})$.

Hint: Write the first one as $\Ext_{A^e \rtimes (\Gamma \times
  \Gamma)}^\bullet(A \rtimes \Gamma, M)$, using that $\bk[\Gamma]
\cong \bk[\Gamma^\op]$ via the map $g \mapsto g^{-1}$.  Notice that $A
\rtimes \Gamma = \Ind_{A^e \rtimes \Gamma_\Delta}^{A^e \rtimes (\Gamma
  \times \Gamma)} A$, where $\Gamma_\Delta := \{(g,g) \mid g \in
\Gamma\} \subseteq \Gamma \times \Gamma$ is the diagonal subgroup.
Then, there is a general fact called Shapiro's lemma, for $H < K$ a
subgroup,
\[
\Ext^\bullet_{\bk[K]}(\Ind_H^K M, N) \cong \Ext^\bullet_{\bk[H]}(M, N).
\]
Similarly, we have $\Ext^\bullet_{A \rtimes K}(\Ind_H^K M, N) \cong \Ext^\bullet_{A \rtimes H}(M,N)$.
Using the latter isomorphism, show that
\begin{equation}\label{e:hoch-sp}
\Ext_{A^e \rtimes (\Gamma \times \Gamma)}^\bullet(A \rtimes \Gamma, M)
\cong \Ext_{A^e \rtimes \Gamma_\Delta}^\bullet(A, M) =
\Ext_{A^e \rtimes \Gamma}^\bullet(A, M^{\text{Ad}}),
\end{equation}
where $M^{\text{Ad}}$ means that $\Gamma$ acts by the adjoint action
from the $A \rtimes \Gamma$-bimodule structure.  Since taking
$\Gamma$-invariants is an exact functor (as $\bk$ has characteristic
zero and $\Gamma$ is finite), this says that the RHS above is
isomorphic to
\[
\Ext_{A^e}^\bullet(A, M^{\text{Ad}})^\Gamma = \HH^\bullet(A, M)^\Gamma.
\]
The proof for Hochschild homology is essentially the same, using $\Tor$.

\item[(b)] Now we apply the formula in part (a) to the special case
$M = A \rtimes \Gamma$ itself.

Let $C$ be a set of representatives of the conjugacy classes of
$\Gamma$: that is, $C \subseteq \Gamma$ and for every element $g \in
\Gamma$, there exists a unique $h \in C$ such that $g$ is conjugate to
$h$. For $g \in \Gamma$, let $Z_g(\Gamma) < \Gamma$ be the centralizer
of $g$, i.e., the collection of elements that commute with $g$. Prove
that
\[
\HH^\bullet(A \rtimes \Gamma) \cong \bigoplus_{h\in C}\HH^\bullet(A, A
\cdot h)^{Z_h(\Gamma)}, \quad
\]
Here, the bimodule action of $A$ on $A \cdot h$ is by
\[
a (b \cdot h) = ab \cdot h, \quad (b \cdot h)a = (bh(a)) \cdot h,
\]
and $Z_h(\Gamma)$ acts by the adjoint action.
\item[(c)] Now specialize to the case that $A = \Sym V$ and $\Gamma < \GL(V)$.
We will prove here that
\[
\HH^\bullet(A \rtimes G) \cong \bigoplus_{h \in C}
\Biggl( \Bigl( \wedge_{\Sym V^h} \Vect((V^h)^*) \Bigr) \otimes
\Bigl( \wedge^{\dim (V^h)^\perp} \langle \partial_\phi \rangle_{\phi \in (V^h)^\perp} \Bigr) \Biggr)^{Z_h(\Gamma)}.
\]
The perpendicular space $(V^h)^\perp$ here is the subspace of $V^*$
annihilating $V^h$.  The degree $\bullet$ on the LHS is the total
degree of polyvector field on the RHS, i.e., the sum of the degree in
the first exterior algebra with $\dim (V^h)^\perp = \dim V - \dim
V^h$.

The same argument shows that (for $\Omega^\bullet(X)$ denoting
the algebraic differential forms on $X$),
\[
\HH_\bullet(A \rtimes G) \cong \bigoplus_{h \in C} \Biggl( \Omega^\bullet((V^h)^*)
\otimes \Bigl( \wedge^{\dim (V^h)^\perp} d(((V^*)^h)^\perp)\Bigr)
\Biggr)_{Z_h(\Gamma)}.
\]

Hints: first, up to conjugation, we can always assume $h$ is diagonal
(since $\Gamma$ is finite).  Suppose that $\lambda_1, \ldots,
\lambda_n$ are the eigenvalues of $h$ on the diagonal.  Then let $h_1,
\ldots, h_n \in \GL_1$ be the one-by-one matrices $h_i = (\lambda_i)$.
Show that
\[
A = \bk[x_1] \otimes \cdots \otimes \bk[x_n], \quad
A \cdot h = (\bk[x_1] \cdot h_1) \otimes \cdots \otimes (\bk[x_n] \cdot h_n).
\]
Conclude using the K\"unneth formula, $\HH^\bullet(A \otimes B, M
\otimes N) = \HH^\bullet(A, M) \otimes \HH^\bullet(B, N)$ (for $M$ an
$A$-bimodule and $N$ a $B$-bimodule, under suitable hypotheses on $A,
B, M$, and $N$ which hold here), that
\begin{equation}\label{e:hh-kunn}
\HH^\bullet(A, A \cdot h)^{Z_h(\Gamma)}
\cong \bigotimes_{i=1}^n \HH^\bullet(\bk[x_i], \bk[x_i] \cdot h_i)^{Z_{h_i}(\Gamma)}.
\end{equation}
Since $\bk[x]$ has Hochschild dimension one (as it has a projective
bimodule resolution of length one, cf.~Remark \ref{r:hd}), conclude
that $\HH^j(\bk[x], \bk[x] \cdot h) = 0$ unless $j \leq 1$. Using the
explicit description as center and outer derivations of the module,
show that, if $h \in \GL_1$ is not the identity,
\[
\HH^0(\bk[x], \bk[x] \cdot h) = 0, \quad \HH^1(\bk[x], \bk[x] \cdot h)
= \bk.
\]
Note for the second equality that you must remember to mod by inner derivations.

On the other hand, recall that
\[
\HH^0(\bk[x], \bk[x]) = \bk[x], \HH^1(\bk[x], \bk[x]) = \bk[x],
\]
since $\HH^\bullet(\bk[x]) = \wedge^\bullet_{\bk[x]}\Der(\bk[x])$.  

Now suppose in \eqref{e:hh-kunn} that $h_i \neq \Id$ for $1 \leq i \leq j$, and that $h_i = \Id$ for $i > j$ (otherwise we can conjugate everything by a permutation matrix).
Conclude that \eqref{e:hh-kunn} implies
\begin{equation}
\HH^\bullet(A, A \cdot h)^{Z_h(\Gamma)}
\cong \Bigl( 
(\partial_{x_1} \wedge \cdots \wedge \partial_{x_j})
\otimes \wedge_{\Sym (V^h)} \Der(\Sym(V^h)) \Bigr)^{Z_h(\Gamma)}.
\end{equation}
Note that, without having to reorder the $x_i$, we could write
\[
\partial_{x_1} \wedge \cdots \wedge \partial_{x_j}
= \wedge^{\dim (V^h)^\perp} \langle \partial_{\phi} \rangle_{\phi \in (V^h)^\perp}.
\]
Put together, we get the statement. A similar argument works for Hochschild
homology.
\item[(d)] Use the same method to prove the main result of \cite{AFLS}
for $V$ symplectic and $\Gamma < \Sp(V)$ finite:
\[
\HH^i(\Weyl(V) \rtimes \Gamma)\cong \bk[S_i]^\Gamma, \quad \HH_i(\Weyl(V) \rtimes \Gamma)\cong \bk[S_{\dim V - i}]^\Gamma
\]
where
\[
S_i := \{g \in \Gamma \mid \rk(g - \Id)= i\}.
\]
Observe also that $S_i = \emptyset$ if $i$ is odd, and $S_2=$
the set of symplectic reflections (as defined in Example \ref{ex:sra}).  

Hint: Apply the result of part (b) and the method of part (c).  This
reduces the result to the case $\dim V = 2$, and to computing
$\HH^\bullet(\Weyl_1, \Weyl_1)$ and $\HH^{\bullet}(\Weyl_1, \Weyl_1
\cdot g)$ for nontrivial $g \in \SL_2(\bk)$. Then you can see from the Koszul
complexes that the first is $\bk$ (or, this can be deduced from
Theorem \ref{t:fdef-symp} in the special case $X=\bA^2$, or you can
explicitly compute it using the Van den Bergh duality
$\HH^2(\Weyl_1,\Weyl_1) \cong \HH_0(\Weyl_1, \Weyl_1)$ since $\Weyl_1$
is Calabi-Yau).  The second you can see must be $\bk[-2]$ since this
is already true for $\HH^\bullet(\bk[x,y], \bk[x,y] \cdot g)$, and
this surjects to $\gr \HH^\bullet(\Weyl_1, \Weyl_1 \cdot g)$.
\end{itemize}
\end{exercise}


\section{Dglas, the Maurer-Cartan formalism, and proof of formality
  theorems}\label{s:dgla}
Now the distinction between dg objects and ungraded objects becomes
important (especially for the purpose of signs): we will recall in
particular the notion of dg Lie algebras (dglas), which have
homological grading, and hence parity (even or odd degree).

\subsection{The Gerstenhaber bracket on Hochschild cochains}\label{ss:gerst}
We turn first to a promised fundamental structure of Hochschild
cochains: the Lie bracket, which is called its \emph{Gerstenhaber
  bracket}:
\begin{definition}
  The \emph{circle product} of Hochschild cochains $\gamma \in C^m(A),
  \eta \in C^n(A)$ is the element $\gamma \circ \eta \in
  C^{m+n-1}(A)$ given by
\begin{equation}
  \gamma \circ \eta(a_1 \otimes \cdots \otimes a_{m+n-1}) :=
  \sum_{i=1}^m (-1)^{(i-1)(n+1)}\gamma(a_1 \otimes \cdots \otimes a_{i-1} \otimes \eta(a_i \otimes \cdots \otimes a_{i+n-1}) \otimes a_{i+n} \otimes \cdots \otimes a_{m+n-1}).
\end{equation}
\end{definition}
\begin{definition}\label{d:gerst-br}
  The \emph{Gerstenhaber bracket} $[\gamma, \eta]$ of $\gamma \in
  C^m(A), \eta \in C^n(A)$ is
  \[ [\gamma,\eta] := \gamma \circ \eta - (-1)^{(m+1)(n+1)} \eta\circ
  \gamma.
\]
\end{definition}
\begin{definition}
  Given a cochain complex $C$, let $C[m]$ denote the shifted complex,
  so $(C[m])^i = C^{i+m}$.
\end{definition}
In other words, letting $C^m$ denote the ordinary vector space
obtained as the degree $m$ part of $C^\bullet$, so $C^m$ by definition
is a graded vector space in degree zero, we have
\[
C = \bigoplus_{m \in \bZ} C^m[-m].
\]
\begin{remark}\label{r:pre-lie}
  The circle product also defines a natural structure on $\mfg :=
  C^\bullet(A)[1]$ viewed as a graded vector space with zero differential:
 that of a graded \emph{right pre-Lie} algebra. This means that it satisfies
  the graded pre-Lie identity
\[
\gamma \circ (\eta \circ \theta) - (\gamma \circ \eta) \circ \theta =
(-1)^{|\theta| |\eta|} \bigl(\gamma \circ (\theta \circ \eta) - (\gamma
\circ \theta) \circ \eta\bigr).
\]
Given any graded (right) pre-Lie algebra, the obtained bracket
\[
[x, y] = x \circ y - (-1)^{|x| |y|} y \circ x
\]
defines a graded Lie algebra structure.  
\end{remark}
\begin{exercise}\label{exer:pre-lie}(*)
Verify the assertions of the remark! (The second assertion is easy, but
the first is a long computation.)
\end{exercise}
\begin{exercise}\label{exer:circ-prod}
  In fact, the circle product was originally defined by Gerstenhaber
  in order to prove that the cup product is graded commutative on
  cohomology.  Prove the following identity of Gerstenhaber, for
  $\gamma_1, \gamma_2 \in C^\bullet(A)$:
\[
\gamma_1 \cup \gamma_2 - (-1)^{|\gamma_1| |\gamma_2|} \gamma_2 \cup
\gamma_1 = d(\gamma_1 \circ \gamma_2) - ((d\gamma_1) \circ \gamma_2) -
(-1)^{|\gamma_1|}(\gamma_1 \circ (d \gamma_2)).
\]
Conclude from this identity that (a) the cup product is graded
commutative, and (b) the Gerstenhaber bracket is compatible with the
differential, i.e., it is a morphism of complexes $C^\bullet(A)[1]
\otimes C^\bullet(A)[1] \to C^{\bullet}(A)[1]$.
\end{exercise}
The remark and exercises immediately imply
\begin{proposition}
  The Gerstenhaber bracket defines a dg Lie algebra structure on the
  shifted complex $\mfg := C^\bullet(A)[1]$.
\end{proposition}

\subsection{The Maurer-Cartan equation}
We now come to the key description of formal deformations:
\begin{definition}
  Let $\mfg$ be a dgla over a field of characteristic not equal to
  two.  The Maurer-Cartan equation is
\begin{equation}\label{e:mc}
d\xi + \frac{1}{2} [\xi,\xi] = 0, \quad \xi \in \mfg^1.
\end{equation}
A solution of this equation is called a \emph{Maurer-Cartan element}. Denote the space of solutions by $\MCE(\mfg)$.
\end{definition}
The equation can be written suggestively as $[d + \xi, d +\xi] = 0$,
 if one defines $[d,d]=d^2=0$ and $[d,\xi] := d\xi$. In this form
the equation is saying that the ``connection'' $d+\xi$ is flat:
\begin{example}\label{ex:flatconn}
  Here is one of the original instances and motivation of the
  Maurer-Cartan equation.  Let $\mfg$ be a Lie algebra and $X$ a
  manifold or affine algebraic variety $X$.  Then we can consider the
  dg Lie algebra $(\Omega^\bullet(X,\mfg),d) := (\Omega^\bullet(X)
  \otimes \mfg,d)$, which is the de Rham complex of $X$ tensored with
  the Lie algebra $\mfg$.  
  The grading is given by the de Rham grading, with $|\mfg|=0$.  Then,
  associated to this is a notion of \emph{connection}, defined as a
  formal expression $\nabla^\alpha := d+\alpha$, where $\alpha \in
  \Omega^1(X,\mfg)$; thus connections are in bijection with
  $\mfg$-valued one forms.  (We will explain below for the
  relationship with the standard notion of connections on principal
  bundles.)  Associated to $\nabla^\alpha$ is the endomorphism of
  $\Omega^\bullet(X,\mfg)$, given by $\beta \mapsto d\beta + [\alpha,
  \beta]$.

  The curvature of $\nabla^\alpha$, denoted $(\nabla^{\alpha})^2$ or
  $\frac{1}{2}[\nabla^\alpha,\nabla^{\alpha}]$, is formally defined as
 \begin{equation}\label{e:curv}
(\nabla^\alpha)^2 = (d + \alpha)^2 = d\alpha + \frac{1}{2}[\alpha, \alpha].
\end{equation}
Then the Maurer-Cartan equation for $\alpha$ says that this is zero.
This is clearly equivalent to the assertion that the corresponding
endomorphism to $\nabla^\alpha$ has square zero, i.e., it is a
differential on $\Omega^\bullet(X,\mfg)$.  In other words,
Maurer-Cartan elements give \emph{deformations of the differential} on
the de Rham complex valued in $\mfg$ (where $\alpha$ acts via the Lie
bracket). In general, this is a good way to think about the
Maurer-Cartan equation, as we will formalize following this example.

We explain the relationship with the standard terminology: If $G$ is
an algebraic or Lie group such that $\mfg = \Lie G$, then
$\nabla^\alpha$ as above is equivalent to a connection on the trivial
principal $G$-bundle on $X$.  Precisely, the connection on $\pi: G
\times X \to X$ associated to $\nabla^\alpha$ is the one-form
$\omega+\pi^*\alpha \in \Omega^1(G \times X, \mfg)$, with $\omega$ the
canonical connection on the trivial bundle, and the curvature of
$\omega+\pi^* \alpha$ is the pullback $\pi^*(d\alpha + \frac{1}{2}
[\alpha, \alpha])$ of the curvature as defined above.
\end{example}
  Closely related to Example \ref{ex:flatconn} is the
  following very important observation. 
\begin{proposition}\label{p:mc-twist}   Suppose $\xi \in \MCE(\mfg)$.
\begin{enumerate}
\item[(i)]
The map
  $d^\xi: y \mapsto dy + [\xi,y]$ defines a new differential on
  $\mfg$.  Moreover, $(\mfg, d^\xi, [-,-])$ is also a dgla.
\item[(ii)] Maurer-Cartan elements
of $\mfg$ are in bijection with those of the twist $\mfg^\xi$ by the
correspondence
\[
\xi+\eta \in \mfg \leftrightarrow \eta \in \mfg^\xi.
\]
\end{enumerate}
\end{proposition}
\begin{definition}\label{d:mc-twist}
  We call the dg Lie algebra
$(\mfg, d^\xi, [-,-])$ given by the above proposition
the \emph{twist by $\xi$}, and denote
  it by $\mfg^\xi$.  
\end{definition}
\begin{proof}[Proof of Proposition \ref{p:mc-twist}]
(i) This is an explicit verification: $(d^\xi)^2(y) =
[\xi,dy] + [\xi,[\xi,y]] + d[\xi,y] = [d\xi+ \frac{1}{2}[\xi,\xi],
y]$, and
\[
d^\xi[x,y] - [d^\xi x, y] - (-1)^{|x|}[x, d^\xi y] = [\xi,[x,y]] -
[[\xi,x], y] - (-1)^{|x|} [x, [\xi,y]] = 0,
\]
where the first equality uses that $d$ is a (graded) derivation for
$[-,-]$, and the second equality uses the (graded) Jacobi identity for
$[-,-]$.

(ii) One immediately sees that $d^\xi(\eta) + \frac{1}{2} [\eta, \eta]
= d(\xi+\eta) + \frac{1}{2} [\xi+\eta, \xi+\eta]$, using that $d\xi +
\frac{1}{2} [\xi,\xi] = 0$.
\end{proof}

\subsection{General deformations of algebras}
\begin{proposition}
  One-parameter formal deformations $(A[\![\hbar]\!], \star)$
  of an associative algebra $A$ are in bijection
  with Maurer-Cartan elements
of the dgla $\mfg := \hbar \cdot (C^\bullet(A)[1]) [\![\hbar]\!]$.
\end{proposition}
\begin{proof} Let $\gamma := \sum_{m \geq 1} \hbar^m \gamma_m \in
  \mfg^1$.  Here $\gamma_m \in C^2(A)$ for all $m$, since $\mfg$ is
  shifted.

  To $\gamma \in \mfg^1$ we associate the star product $f \star g = fg
  + \sum_{m \geq 1} \hbar^m \gamma_m(f \otimes g)$. We need to show
  that $\star$ is associative if and only if $\gamma$ satisfies the
  Maurer-Cartan equation. This follows from a direct computation (see
  Remark \ref{r:mc-assoc} for a more conceptual explanation):
\begin{multline}
  f \star (g \star h)-(f \star g) \star h = \sum_{m \geq 1} \hbar^m
  \cdot \bigl( f \gamma_m(g \otimes h) - \gamma_m(fg \otimes h) +
  \gamma_m(f \otimes gh) -\gamma_m(f \otimes g)h \bigr)
  \\
  + \sum_{m,n \geq 1} \hbar^{m+n} \bigl( \gamma_m(f \otimes \gamma_n(g
  \otimes h)) - \gamma_m(\gamma_n(f \otimes g) \otimes h)\bigr)
  \\
  = d \gamma + \gamma \circ \gamma = d \gamma + \frac{1}{2} [\gamma,
  \gamma]. \qedhere
\end{multline}
\end{proof}
\begin{remark}\label{r:mc-assoc} For a more conceptual explanation of
  the proof, note that, if we let $A_0$ be an algebra with the zero
  multiplication, so that $C^\bullet(A_0)$ is a dgla with zero
  differential, then associative multiplications are the same as
  elements $\mu \in C^2(A_0) = \mfg^1$ satisfying
  $\frac{1}{2}[\mu,\mu] = 0$, where $\mfg = C^\bullet(A_0)[1]$ as
  before. (This is the Maurer-Cartan equation for $\mfg$.)  If we now
  take an arbitrary algebra $A$, we can set
  $A_0$ to be $A$ but viewed as an algebra with
 the \emph{zero} multiplication.
  Let $\mu \in C^2(A_0)$ represent the multiplication on $A$, hence $[\mu,\mu]=0$ by associativity.
Then, given $\gamma := \sum_{m \geq 1} \hbar^m
  \gamma_m \in \hbar \mfg^1 = \hbar C^2(A)[\![\hbar]\!]$, the 
 product $\mu + \gamma$ is associative if and only if, working in $(C^\bullet(A_0)[1])[\![\hbar]\!]$, we have
\[
0 = [\mu + \gamma, \mu + \gamma] = [\mu,\mu]+2[\mu,\gamma]+[\gamma,\gamma].
\]

Now, $[\mu,\gamma] = d_{A}(\gamma)$, with $d_A$ the (Hochschild)
differential on $(C^\bullet(A)[1])[\![\hbar]\!]$. 
More conceptually, this is saying that
$C^\bullet(A)[1] = C^\bullet(A_0)[1]^\mu$, the twist by $\mu$;
cf.~Proposition \ref{p:mc-twist} and Definition \ref{d:mc-twist}, as
well as Lemma \ref{l:zero-mult} below.
Then, by Proposition \ref{p:mc-twist},  Maurer-Cartan
elements $\xi \in \MCE((C^\bullet(A)[1])\![\hbar]\!])$ are the same
as associative multiplications $\mu+\xi$. They are $\mu$ modulo
$\hbar$ if and only if $\xi \in \hbar C^\bullet(A)[1]\![\hbar]\!]$.
\end{remark}
\begin{remark}
  The above formalism works, with the same proof, for formal
  deformations over arbitrary complete augmented commutative
  $\bk$-algebras.
  Namely, associative multiplications on $A \hat \otimes_\bk R$
  deforming the associative multiplication $\mu$ on $A$ are the same
  as Maurer-Cartan elements of the dgla $C(A)[1] \hat \otimes_\bk R_+$.
\end{remark}

\subsection{Gauge equivalence}\label{ss:gauge}
Recall from Example \ref{ex:flatconn} the example of flat connections
with values in $\mfg = \Lie G$ as solutions of the Maurer-Cartan
equation.  In that situation, one has a clear notion of equivalence of
connections, namely gauge equivalence: for $\gamma: X \to G$ a map,
and $\iota: G \to G$ the inversion map,
\[
\nabla \mapsto (\Ad \gamma)(\nabla); 
(d+\alpha) \mapsto d + (\Ad \gamma)(\alpha) +
\gamma \cdot d(\iota \circ \gamma).
\]
Here, the meaning of $\gamma \cdot d(\iota \circ \gamma)$ is as
follows: the derivative $d(\iota \circ \gamma)$ is defined, at each $x
\in X$, as a map $d(\iota \circ \gamma)|_x: T_xX \to
T_{\gamma(x)^{-1}} G$, and then we apply the derivative of the left
multiplication by $\gamma(x)$, $dL_{\gamma(x)}: T_{\gamma(x)^{-1}} G
\to T_e G$, to obtain an operator $\gamma \cdot d(\iota \circ
\gamma)|_x: T_xX \to T_e G = \mfg$.  We obtain in this way a one-form
$\gamma \cdot d(\iota \circ \gamma) \in \Omega^1(X, \mfg)$.  (We may
think of $d+\gamma \cdot d(\iota \circ \gamma)$ formally as
$\Ad(\gamma)(d)$; see also below for the case $\gamma=\exp(\beta)$.)

Now, restrict to the case $\gamma = \exp(\beta)$ for $\beta \in \cO(X)
\otimes \mfg$, assuming that $\bk=\bR$ or $\bC$ so $\exp$ is the usual
exponential map (if we restrict to the case where $G$ is connected,
then such elements $\gamma$ generate $G$, so generate all gauge
equivalences).  By taking a faithful representation, we may even
assume without loss of generality that $G < \GL_n$ and $\mfg <
\mfgl_n$, so $\gamma$ and $\beta$ are matrix-valued functions.  We can
then rewrite the above formula in a way not requiring $G$ or the
definition of $\gamma \cdot d(\iota \circ \gamma)$ as:
\begin{equation}\label{e:gauge-exp}
  \alpha \mapsto 
  \exp(\ad \beta)(\alpha) + \frac{1 - \exp(\ad \beta)}{\ad \beta} (d\beta),
\end{equation}
where $(\ad \beta)(\alpha) := [\beta, \alpha]$, using the Lie bracket on
$\mfg$.  The last term above can be thought of as $\exp(\ad
\beta)(d)  - d$, where we set $[d, \beta]=d(\beta)$, as explained in the
following exercise:
\begin{exercise}\label{exer:gauge}
  Verify \eqref{e:gauge-exp}. To do so, replace
  $\exp: \mfg \to G$ by its Taylor series, and
  use the standard identity $\Ad(\exp(\beta))= \exp(\ad \beta) :=
  \sum_{m \geq 0} (\ad \beta)^m$, which holds formally (setting
  $\Ad(\exp(\beta))(f)=\exp(\beta) \cdot f \cdot \exp(-\beta)$ and
  $\ad(\beta)(f) = \beta \cdot f - f \cdot \beta$), and follows from
  the basic theory of Lie groups.  Note also that, for $\alpha \in
  \cO(X) \otimes \mfg$ arbitrary, we have $([d,\beta]) \alpha :=
  d(\beta \alpha) - \beta d \alpha = (d \beta) \alpha$, so we can
  formally write $[d,\beta] = d(\beta)$ as above. Then, use all of
  this to expand and simplify $\gamma \cdot d(\iota \circ \gamma)$.
  Hint: write the latter, formally, as $\Ad(\exp(\beta))(d)-d$, then
  apply all of the above.
\end{exercise}
The above discussion motivates the following general definition, where
now $\mfg$ can be an arbitrary dgla (no longer a finite-dimensional
Lie algebra as above). \textbf{From now until the end of Section
  \ref{s:dgla}, $\bk$ should be a characteristic zero field.}
\begin{definition}
  Two Maurer-Cartan elements $\alpha, \alpha' \in \mathfrak{g}^1$ of a
  dgla are called gauge equivalent by an element $\beta \in
  \mathfrak{g}^0$ if $\alpha' = \exp(\ad \beta)(\alpha) + \frac{1 -
    \exp(\ad \beta)}{\beta} (d\beta)$, when this formula makes sense:
  for us, we take either (a) $\bk = \bR$ or $\bC$ and $\mfg$ is
  finite-dimensional as above; (b) $\mfg = \mfh \otimes_{\bk} R_+$
  with $\mfh$ an arbitrary dgla and $R$ a complete augmented (dg)
  commutative algebra; or (c) With $R$ as in (b), we can also take
  $\mfg = \mfh \otimes_{\bk} R$ and $\beta \in (R_+ \mfg)^0$.
\end{definition}
This definition is consistent with Definition \ref{d:ge}:
\begin{proposition}\label{p:star-ge}
 Two formal
  deformations $(A[\![\hbar]\!], \star)$ and $(A[\![\hbar]\!],
  \star')$ are gauge equivalent, i.e., isomorphic via a 
  $\bk[\![h]\!]$-linear automorphism of $A[\![\hbar]\!]$ which is the
  identity modulo $\hbar$, if and only if the corresponding
  Maurer-Cartan elements of $\mfg = \hbar \cdot C(A)[\![\hbar]\!]$ are
  gauge equivalent.
\end{proposition}
Here, the automorphism of $A[\![\hbar]\!]$ does not respect the
algebra structure on $A$: it is just a 
$\bk[\![\hbar]\!]$-linear automorphism.  Being the identity modulo
$\hbar$ means that the automorphism $\Phi$ satisfies the property that
$\Phi - \Id$ is a multiple of $\hbar$ as an endomorphism of the vector
space $A[\![\hbar]\!]$.
\begin{proof}
  This is an explicit verification: Let $\phi$ be an 
  automorphism of $A[\![\hbar]\!]$ which is the identity modulo
  $\hbar$. We can write $\phi = \exp(\alpha)$ where $\alpha \in
  \hbar \End_\bk(A)[\![\hbar]\!]$; one can check that $\exp(\alpha) =
  1 + \alpha + \alpha^2/2! + \cdots$ makes sense since we are using
  power series in $\hbar$.
  Let $\gamma, \gamma' \in \mfg^1$ be the Maurer-Cartan elements
  corresponding to $\star$ and $\star'$. Let $\mu: A \otimes A \to A$
  be the undeformed multiplication. Then
\begin{multline}
  \exp(\alpha) (\exp(-\alpha)(a) \star \exp(-\alpha)(b)) = \exp(\ad
  \alpha)(\mu + \gamma)(a \otimes b) \\
  = \bigl(\mu + \exp(\ad \alpha)(\gamma) + \frac{1-\exp(\ad
    \alpha)}{\ad \alpha}(d\alpha)\bigr)(a \otimes b),
\end{multline}
where the final equality follows because $[\mu, \alpha] = d\alpha$.
\end{proof}

\subsection{The dgla of polyvector fields, Poisson deformations, and
  Gerstenhaber algebra structures}
Let $X$ again be a smooth affine algebraic variety over a
characteristic zero field or a $C^\infty$ manifold.  By the
Hochschild-Kostant-Rosenberg theorem (Theorem \ref{t:hkr}), the
Hochschild cohomology $\HH^\bullet(\cO(X))$ is isomorphic to the
algebra of polyvector fields, $\wedge^\bullet_{\cO(X)} \Vect(X)$.
Since, as we now know, $C^\bullet(\cO(X))[1]$ is a dgla, one concludes
that $\wedge^\bullet_{\cO(X)} \Vect(X)[1]$ is also a dg Lie algebra
(with zero differential). In fact, this structure coincides with the
Schouten-Nijenhuis bracket:
\begin{proposition}
  The Lie bracket on $\wedge^\bullet_{\cO(X)} \Vect(X)[1]$ induced by
  the Gerstenhaber bracket is the Schouten-Nijenhuis bracket, as
  defined in Definition \ref{d:sn}.
\end{proposition}
Such a structure is called a \emph{Gerstenhaber algebra}:
\begin{definition}
  A (dg) Gerstenhaber algebra is a dg commutative algebra $B$ equipped with
  a dg Lie algebra structure on the shift $B[1]$, such that
  \eqref{e:gerst2} is satisfied.
\end{definition}
Note that, by definition, a Gerstenhaber algebra has to be
(homologically) graded; sometimes when the adjective ``dg'' is omitted
one means a dg Gerstenhaber algebra with zero differential. This is the
case for $\wedge^\bullet_{\cO(X)} \Vect(X)$.
\begin{remark}
  Note that the definition of a Gerstenhaber algebra is very similar
  to that of a Poisson algebra: the difference is that the Lie bracket
  on a Gerstenhaber algebra is \emph{odd}: it has homological
  degree $-1$.
\end{remark}
We easily observe:
\begin{proposition}\label{p:pb-sn}
A bivector field $\pi \in \wedge^2 \Vect(X)$ defines a Poisson bracket if and 
only if $[\pi,\pi]=0$.  That is, \emph{Poisson bivectors $\pi$ are
solutions of the Maurer-Cartan equation in $\wedge^\bullet_{\cO(X)} \Vect(X)[1]$.}
\end{proposition}
\begin{exercise}\label{exer:pb-sn}
Prove Proposition \ref{p:pb-sn}!
\end{exercise}
The same proof implies:
\begin{corollary}\label{c:fpoiss-mc}
  Formal Poisson structures in $\hbar \cdot \wedge^2_{\cO(X)} \Vect(X)
  [\![\hbar]\!]$ are the same as Maurer-Cartan elements of the dgla
  $\hbar\cdot (\wedge^\bullet_{\cO(X)} \Vect(X)[1])[\![\hbar]\!]$.
\end{corollary}

\subsection{Kontsevich's formality and quantization theorems}\label{subsect:formal}
We can now make a precise statement of Kontsevich's formality
theorem. As before, we need $\bk$ to be a characteristic zero field
for the remainder of Section \ref{s:dgla}.

\begin{remark}
  Kontsevich proved this result for $\bR^n$ or smooth $C^\infty$
  manifolds; for the general smooth affine setting, when $\bk$
  contains $\bR$, one can extract this result from \cite{Kon-dqav};
  for more details see \cite{Yek-dqag}, and also, e.g.,
  \cite{VdB-gdqac}. These proofs also yield a sheaf-level version of
  the statement for the nonaffine algebraic setting. For a simpler
  proof in the affine algebraic setting, which works over arbitrary
  fields of characteristic zero, see \cite{DTT-hgahcraf}.  We remark
  also that, recently in \cite{Dol-acekfqrrn}, Dolgushev showed that
  there actually exists a ``correction'' of Kontsevich's formulas
  which involve only rational weights, which replaces Kontsevich's
  proof by one that works over $\bQ$.
\end{remark}
The one parameter version of the theorem is
\begin{theorem}\cite{Kform,Kon-dqav,Yek-dqag,DTT-hgahcraf}
There is a map
\[
\text{Formal Poisson bivectors in $\hbar \cdot \wedge^2
  \Vect(X)[\![\hbar]\!]$} \rightarrow \text{Formal deformations of
  $\cO(X)$}
\]
which induces a bijection modulo 
automorphisms of
$\cO(X)[\![\hbar]\!]$ which are the identity modulo $\hbar$, and sends
a formal Poisson structure $\hbar \pi_\hbar$ to a deformation quantization
of the ordinary Poisson structure $\pi \equiv \pi_\hbar \pmod
{\hbar}$.
\end{theorem}
\begin{remark}
  By dividing the formal Poisson structure by $\hbar$, we also get a
  bijection modulo gauge equivalence from \emph{all} formal Poisson
  structures to formal deformations, now sending $\pi_\hbar$ to a
  deformation quantization of $\pi$; the way it is stated above
  generalizes better to the full (multiparameter) version below.
\end{remark}
We can state the full version of the theorem as follows:
\begin{theorem}\cite{Kform,Kon-dqav,Yek-dqag,DTT-hgahcraf} \label{t:form-mc}
  There is a map, functorial in dg commutative complete augmented
  $\bk$-algebras $R = \bk \oplus R_+$,
\[
\mathcal{U}: \text{Poisson bivectors in
 $\wedge^2_{\cO(X)} \Vect(X) \hat
  \otimes_\bk R_+$} \rightarrow \text{Formal deformations $(\cO(X)
  \hat \otimes_\bk R, \star)$}
\]
which induces a bijection modulo 
$R$-linear automorphisms of $\cO(X)
\hat \otimes_\bk R$ which are the identity modulo $R_+$.
Moreover, modulo $R_+^2$, this reduces to the
identity on bivectors valued in $R_+/R_+^2$.
\end{theorem}

To explain what we mean by ``the identity'' in the end of the theorem,
we note that, working modulo $R_+^2$, the Jacobi and associativity
constraints become trivial. Similarly, formal deformations over $R /
R_+^2$ are given by (not necessarily skew-symmetric) biderivations
$\cO(X) \otimes_{\bk} \cO(X) \to \cO(X) \otimes R_+/R_+^2$. Just as in
the case where $R = \bk[\varepsilon]$, up to equivalence, these are
given by their skew-symmetrization, a bivector $\wedge^2_{\cO(X)}
\Vect(X) \otimes R_+/R_+^2$.  Thus, up to equivalence, both the domain
and target reduce modulo $R_+^2$ to bivectors valued in $R_+/R_+^2$,
and we can ask that the map reduce to the identity in this case.

\subsection{Restatement in terms of morphisms of dglas}
We would like to restate the theorems above without using coefficients
in $R$, just as a statement relating the two dglas in question.  Let
us name these: $T_\poly := \bigl(\wedge^\bullet_{\cO(X)} \Vect(X)\bigr)[1]$ is the dgla
of (shifted) polyvector fields on $X$, and $D_\poly :=
C^\bullet(\cO(X))[1]$ is the dgla of (shifted) Hochschild cochains on
$X$, which in the $C^\infty$ setting are required to be differential
operators.

These dglas are clearly not isomorphic on the nose, since $T_\poly$ has
zero differential and not $D_\poly$.  They have isomorphic cohomology,
by the Hochschild-Kostant-Rosenberg theorem.   In this section we will
explain how they are quasi-isomorphic, which is equivalent to the
functorial equivalence of Theorem \ref{t:form-mc}.

First, the Hochschild-Kostant-Rosenberg (HKR) theorem (Theorem
\ref{t:hkr}) in fact gives a quasi-isomorphism of complexes $\HKR:
T_\poly \to D_\poly$, defined by
\[
\HKR(\xi_1 \wedge \cdots \wedge \xi_m)(f_1 \otimes \cdots \otimes f_m)
= \frac{1}{m!}\sum_{\sigma \in S_m}\sign(\sigma)
\xi_{\sigma(1)}(f_1) \cdots \xi_{\sigma(m)}(f_m).
\]
This clearly sends $T_\poly^{m-1} = \wedge^{m}_{\cO(X)} \Vect(X)$ to
$D_\poly^{m-1} = C^{m}(\cO(X), \cO(X))$, since the target is an
$\cO(X)$-multilinear differential operator.  Moreover, it is easy to
see that the target is closed under the Hochschild differential.  By
the proof of the HKR theorem, one in fact sees that $\HKR$ is a
quasi-isomorphism of complexes.

However, $\HKR$ is \emph{not} a dgla morphism, since it does not
preserve the Lie bracket.  It does preserve it when restricted to
vector fields, but already does not on bivector fields (which would be
needed to apply it in order to take a Poisson bivector field and
produce a star product).  For example, $[\HKR(\xi_1 \wedge \xi_2),
\HKR(\eta_1 \wedge \eta_2)]$, for vector fields $\xi_1,\xi_2,\eta_1$,
and $\eta_2$, is not, in general, in the image of $\HKR$: it is not
skew-symmetric, as one can see by Definition \ref{d:gerst-br}.

The fundamental idea of Kontsevich was to correct this deficiency by
adding higher order terms to $\HKR$. The result will not be a morphism
of dglas (this cannot be done), but it will be a more general type of
morphism
called an $L_\infty$ morphism, which we introduce in the next
subsection.  The idea behind an $L_\infty$ morphism is as follows: If we know that $\phi: \mfg \to \mfh$ has the property that $\phi[a,b] - [\phi(a),\phi(b)]$ is a boundary, say equal to $dc$, then we try to incorporate
the data of the $c$ into the morphism, by defining a map $\phi_2: \mfg \wedge \mfg \to \mfh$ sending $a \wedge b \to c$, and more generally such that
$\phi[x,y] - [\phi(x),\phi(y)] = d\phi_2(x \wedge y)$ for all $x,y$.  Then, we also need to define $\phi_3: \wedge^3 \mfg \to \mfh$ as well, and so on.  A full $L_\infty$ morphism is then a sequence of linear maps
$\phi_m: \wedge^m \mfg \to \mfh$ satisfying certain axioms.

Kontsevich therefore constructs an explicit sequence of linear maps
\[
\mathcal{U}_m: \Sym^m (T_\poly[1]) \to D_\poly[1]
\]
which satisfy these axioms, and hence yield an $L_\infty$ morphism.
Kontsevich constructs the $\mathcal{U}_m$ using graphs as in \S
\ref{ss:dq-fla}, except that now we must allow an arbitrary number of
vertices on the real axis, not merely two (the number of vertices
corresponds to two more than the degree of the target in
$D_\poly[1]$), and the outgoing valence of vertices above the real
axis can be arbitrary as well. As before, the vertices on the real
axis are sinks.  Note that $\mathcal{U}_1 = \HKR$ is just the sum of
all graphs with a single vertex above the real axis, and all possible
numbers of vertices on the real axis.

Then, if we plug in a formal Poisson bivector $\pi_\hbar$, we obtain
the star product described in \S \ref{ss:dq-fla},
\[
f \star g = \sum_{m \geq 1} \frac{1}{m!} \mathcal{U}_m(\pi_\hbar^m)(f
\otimes g),
\]
i.e., the star product is $\mathcal{U}(\exp(\pi_\hbar))$, where
$\mathcal{U} = \sum_{m \geq 1} \mathcal{U}_m$.

\subsection{$L_\infty$ morphisms}
One way to motivate $L_\infty$ morphisms is to study what we require
to obtain a functor on Maurer-Cartan elements.  We will study this
generally for two arbitrary dglas, $\mfg$ and $\mfh$.  

Given two augmented algebras $(A,A_+)$ and $(B,B_+)$, an augmented
algebra morphism is an algebra morphism $\phi: A \to B$ such that
$\phi(A_+) \subseteq B_+$.  We will always require our maps
of augmented algebras be augmented algebra morphisms. The following
then follows from definitions:

\begin{proposition}\label{p:dgla-morph} 
Any dgla morphism $F: \mfg \to \mfh$ induces a functorial map in 
complete augmented dg commutative $\bk$-algebras $R = \bk \oplus R_+$,
\[
F: \MCE(\mfg \hat \otimes_\bk R_+) \to \MCE(\mfh \hat \otimes_\bk R_+).
\]
\end{proposition}
However, it is not true that all functorial maps are obtained from
dgla morphisms; in particular, if they were, then all functorial maps
as above would define functorial maps if we replace $R_+$ by arbitrary
(unital or nonunital) rings, by the remark below.  But this is not true:
with general coefficients the infinite sums in, e.g., \S \ref{ss:dq-fla}
need not converge.
\begin{remark}\label{r:dgla-nonnilp}
  In fact, dgla morphisms also induce functorial maps in ordinary 
  (not necessarily complete augmented or even augmented) dg
  commutative $\bk$-algebras $R$, taking the ordinary tensor product.
  However, the generalization to $L_\infty$ morphisms  below
  requires complete augmented $\bk$-algebras.
\end{remark}

It turns out that there is a complete augmented dg commutative
$\bk$-algebra $B$ which represents the functor $R \mapsto \MCE(\mfg
\hat \otimes_\bk R_+)$. This means that $R$-points of $\Spf B$, i.e.,
continuous augmented dg algebra morphisms $B \to R$, are functorially
in bijection with Maurer-Cartan elements of $\mfg \hat \otimes_\bk R$.
To see what $B$ is, first consider the case where $\mfg$ is abelian
with zero differential. If $R$ is concentrated in degree zero, then
Maurer-Cartan elements with coefficients in $R_+$ are elements of
$\mfg^1 \hat \otimes R$.  For any ungraded vector space $V$, the
algebra of polynomial functions on it is $\Sym(V^*)$; so here one can
consider $B=\Sym((\mfg^1)^*)$, and then elements in $\mfg^1$ are the
same as continuous algebra homomorphisms $B \to \bk$. Taking the
completion $\hat B$ of $B$ at the augmentation ideal $(V^*)$,
continuous augmented algebra homomorphisms $\hat B \to R$ are the same as
elements of $\mfg \hat \otimes R_+$.  But if we take coefficients in a
graded ring $R$, then we need to incorporate all of $\mfg$, not just
$\mfg^1$.  Observing that $\mfg^1 = (\mfg[1])^0$, the natural choice
of graded algebra is $\Sym((\mfg[1])^*)$.  We wanted a completed
algebra, so we take the completion $\hat S (\mfg[1])^*$, which has the
same continuous augmented maps to complete augmented rings $R$.  Then,
the points of $\hat S (\mfg[1])^*$ valued in an ordinary (non-dg)
complete augmented algebra $R$ are in bijection with elements of
$\mfg^1 \hat \otimes R$, as desired.

Thus, we consider the completed symmetric algebra $\hat S
(\mfg[1])^*$.  In the case $\mfg$ is nonabelian, we can account for
the Lie bracket by deforming the differential on $\hat S (\mfg[1])^*$,
so that the spectrum consists of Maurer-Cartan elements rather than
all of $H^1(\mfg)$.

The result is the \emph{Chevalley-Eilenberg complex of $\mfg$}, which
you may already know as the complex computing the Lie algebra
cohomology of $\mfg$.  
\begin{remark}\label{r:top-dual}
  We need to consider here $(\mfg[1])^*$ as the \emph{topological}
  dual to $\mfg[1]$. Since $\mfg$ is considered as discrete (hence
  $\mfg^i = \lim_{\underset{V \subseteq \mfg^i \text{
        f.~d.}}{\rightarrow}} V$ is an inductive limit of its
  finite-dimensional subspaces), the dual $\mfg^*$ is the topological
  dg vector space $\bigoplus_{m \in \bZ} (\mfg^m)^*[m]$, where each
  $(\mfg^m)^*$ is equipped with a not-necessarily discrete topology,
  given by the \emph{inverse} limit of the finite-dimensional
  \emph{quotients} $(\mfg^m)^* \onto V^*$, which are the duals of the
  finite-dimensional subspaces $V \subseteq \mfg^m$:
\[
(\mfg^m)^* := \lim_{\underset{(\mfg^m)^* \onto V^* \text{
      f.d.}}{\leftarrow}} V^*.
\]
This is an inverse limit of finite-dimensional vector
spaces.\footnote{Such a vector space is sometimes a
  \emph{pseudocompact vector space}.  Another term that appears in
  some literature is \emph{formal vector space}, not to be confused,
  however, with the notion of formality we are discussing as in
  Kontsevich's theorem!}  It is discrete if and only if $V$ is finite-dimensional.
\end{remark}
\begin{definition}\label{d:ce}
  The Chevalley-Eilenberg complex is the complete dg commutative
  algebra $C_{CE}(\mfg) := (\hat S (\mfg[1])^*, d_{CE})$, where $d$ is
  the derivation such that $-d_{CE}(x) = d_\mfg^*(x) + \frac{1}{2}
  \delta_\mfg(x)$, where the degree one map $\delta_\mfg: \mfg[1]^* \to \Sym^2
  \mfg[1]^*$ is the dual of the Lie bracket $\wedge^2 \mfg \to \mfg$.
\end{definition}
\begin{proposition}\label{p:mc-ce}
  Let $R = \bk \oplus R_+$ be a dg commutative complete augmented
  $\bk$-algebra.  Then there is a canonical bijection between
  continuous dg commutative augmented algebra morphisms $C_{CE}(\mfg)
  \to R$ and Maurer-Cartan elements of $\mfg \hat \otimes R$, given by
  restricting to $\mfg[1]^*$.
\end{proposition}
\begin{proof}
  It is clear that, if we do not consider the differential, continuous
  algebra homomorphisms $C_{CE}(\mfg) \to R$ are in bijection
  with continuous graded maps $\chi: \mfg[1]^* \to R_+$.  Such elements,
  because of the continuity requirement, are the same as elements
  $x_\chi \in \mfg[1] \hat \otimes R_+$. Then, $\chi$ commutes with the
  differential if and only if $\chi \circ d_\mfg^* + \frac{1}{2} (\chi
  \otimes \chi) \circ \delta_\mfg + d \circ \chi=0$, i.e., if and only if
  $d(x_\chi) + \frac{1}{2} [x_\chi,x_\chi] = 0$.
\end{proof}
\begin{corollary}\label{c:mc-bij}
There is a canonical bijection
\begin{multline}
  \{\text{Functorial in $R$ maps $F: MCE(\mfg \hat \otimes R_+) \to
    \MCE(\mfh \hat \otimes R_+)$}
  \} \\
  \leftrightarrow \{\text{continuous dg commutative augmented morphisms $F^*: C_{CE}(\mfh) \to
    C_{CE}(\mfg)$}\},
\end{multline}
where $R$ ranges over dg commutative complete augmented
$\bk$-algebras.
\end{corollary}
\begin{proof}
  This is a Yoneda type result: given a continuous dg commutative
  augmented morphism $C_{CE}(\mfh) \to C_{CE}(\mfg)$, the pullback
  defines a map $MCE(\mfg \hat \otimes R_+) \to \MCE(\mfh \hat \otimes
  R_+)$ for every $R$ as described, which is functorial in $R$.
  Conversely, given the functorial map $\MCE(\mfg \hat \otimes R_+)
  \to \MCE(\mfh \hat \otimes R_+)$, we apply it to $R = C_{CE}(\mfg)$
  itself.  Then, by Proposition \ref{p:mc-ce}, the identity map
  $C_{CE}(\mfg) \to R$ yields a Maurer-Cartan element $I \in \mfg \hat
  \otimes C_{CE}(\mfg)_+$ (the ``universal'' Maurer-Cartan
  element). Its image in $\MCE(\mfh \hat \otimes C_{CE}(\mfg)_+)$
  yields, by Proposition \ref{p:mc-ce}, a continuous dg commutative
  augmented morphism $C_{CE}(\mfh) \to C_{CE}(\mfg)$.  It is
  straightforward to check that these maps are inverse to each other.
\end{proof}
\begin{definition}\label{d:linf-dgca}
  An $L_\infty$ morphism $F: \mfg \to \mfh$ is a continuous dg
  commutative augmented morphism $F^*: C_{CE}(\mfh) \to C_{CE}(\mfg)$.
\end{definition}
\begin{remark}
  If we remove the requirement ``augmented,'' then one obtains
  so-called \emph{curved $L_\infty$} morphisms.
\end{remark}
\begin{exercise}\label{exer:linf-dgla}
  Show that a dgla morphism is an $L_\infty$ morphism. More precisely,
  show that a dgla morphism $\mfg \to \mfh$ induces a canonical
  continuous dg commutative morphism $C_{CE}(\mfh) \to C_{CE}(\mfg)$
  (note that one can define a canonical linear map owing to the dual
  in the definition of $C_{CE}$ (Definition \ref{d:ce}); you need to
  show it is actually a morphism of dg commutative algebras).
\end{exercise}

We will refer to $F^*$ as the \emph{pullback} of $F$.
Thus, Corollary \ref{c:mc-bij} can be alternatively stated as
\begin{multline}\label{e:mc-bij2}
  \{\text{Functorial in $R$ maps $F: \MCE(\mfg \hat \otimes R_+) \to
    \MCE(\mfh \hat \otimes R_+)$}
  \} \\
  \leftrightarrow \{\text{$L_\infty$ morphisms $F: \mfg \to \mfh$}\}.
\end{multline}
Finally, the above extends to describe quasi-isomorphisms:
\begin{definition}
  An $L_\infty$ quasi-isomorphism is a $L_\infty$ morphism which is an
  isomorphism on homology, i.e., a dg commutative augmented
  quasi-isomorphism $F^*: C_{CE}(\mfh) \to C_{CE}(\mfg)$.
\end{definition}
This can also be called a homotopy equivalence of dglas.
\begin{theorem}
If $F$ is a quasi-isomorphism, then the above functorial map is a bijection
on gauge equivalence classes.
\end{theorem}
The theorem is part of \cite[Theorem 4.6]{Kform}, but is older and
considered standard.  (For example, to see that quasi-isomorphisms
admit quasi-inverses, see \cite{Hin-haha};
this is the dg version of the statement that a map of formal
neighborhoods of the origin of two vector spaces is an isomorphism if
and only if it is an isomorphism on tangent spaces (the formal inverse function theorem).  Using this, the
statement reduces to showing that a quasi-isomorphism which is the
identity on cohomology is also the identity on gauge equivalence
classes.)
\begin{remark}
  In fact, the above theorem can be significantly strengthened: the
  Maurer-Cartan set $\MCE(\mfg \hat \otimes R_+)$ is not just a set
  with gauge equivalences, but in fact a simplicial complex. The
  statement that $\MCE(\mfg \hat \otimes R_+) \to \MCE(\mfh \hat
  \otimes R_+)$ is a bijection on gauge equivalences is the same as
  the statement that it induces an isomorphism on $\pi_0$.  In fact, a
  quasi-isomorphism $\mfg \to \mfh$ induces a homotopy equivalence
  $\MCE(\mfg \hat \otimes R_+) \simeq \MCE(\mfh\hat \otimes R_+)$
  \cite[Proposition 4.9]{Get-ltnla}, which is stronger.
\end{remark}
\begin{remark}\label{r:kd}
  In fact, everything we have discussed above is an instance of Koszul
  duality: the dgla $\mfg$ is (derived) Koszul dual to its
  Chevalley-Eilenberg complex $C_{CE}(\mfg)$, a dg commutative
  algebra; in general, if $A$ and $B$ are algebras of any type (e.g.,
  algebras over an operad) and $A^!$ is the (derived) Koszul dual of
  $A$ (an algebra of the Koszul dual type, e.g., an algebra over the
  Koszul dual operad), then Proposition \ref{p:mc-ce} generalizes to
  the statement: Homotopy (i.e., infinity) morphisms $A \to B$
  identify with Maurer-Cartan elements in $A^! \hat \otimes B$ (note
  that $A^! \hat \otimes B$, properly defined, always has a dgla or at
  least an $L_\infty$-algebra structure). Similarly, Corollary
  \ref{c:mc-bij} and \eqref{e:mc-bij2} generalize to: Homotopy
  morphisms $A \to B$ are in bijection with functorial maps $\MCE(A
  \hat \otimes C) \to \MCE(B \hat \otimes C)$ where $C$ is an algebra
  of the Koszul dual type.  (Note that, taking $C = A^!$, the element
  in $\MCE(B \hat \otimes A^!)$ corresponding to the original morphism
  is the image of the canonical element in $\MCE(A \hat \otimes A^!)$
  corresponding to the identity.)
\end{remark}

\subsection{Explicit definition of $L_\infty$ morphisms}
Let us write out explicitly what it means to be an $L_\infty$
morphism.  Let $\mfg$ and $\mfh$ be dglas.  Then an $L_\infty$
morphism is a continuous commutative dg algebra morphism $F^*:
C_{CE}(\mfh) \to C_{CE}(\mfg)$. Since $C_{CE}(\mfh)$ is the symmetric
algebra on $\mfh[1]^*$, this map is uniquely determined by its
restriction to $\mfh[1]^*$. We then obtain a sequence of maps
\[
F_m^*: \mfh[1]^* \to \Sym^m \mfg[1]^*,
\]
or dually,
\[
F_m: \Sym^m \mfg[1] \to \mfh[1].
\]
The $F_m^*$ are the Taylor coefficients of $F^*$, since they are the parts
of $F^*$ of polynomial degree $m$, i.e., the order-$m$ Taylor coefficients
of the map on Maurer-Cartan elements.

The condition that $F^*$ commute with the differential says that 
\[
F^*(d^*_{\mfh}(x)) + F^*(\frac{1}{2} \delta_{\mfh}(x)) =
d^*_{\mfg}F^*(x) + \frac{1}{2} \delta_{\mfg}F^*(x).
\]
Writing this in terms of the $F_m$, we obtain that, for all $m \geq 1$,
\begin{equation}\label{e:linf-expl}
  d_{\mfh} \circ F_m + \frac{1}{2}\sum_{i+j=m} [F_i, F_j]_{\mfh} =
  F_m \circ d_{CE,\mfg},
\end{equation}
where $d_{CE,\mfg}$ is the Chevalley-Eilenberg differential for $\mfg$.

Now, let us specialize to $\mfg = T_\poly$ and $\mfh = D_\poly$, with
$F_m = \mathcal{U}_m$ for all $m \geq 1$.  In terms of Kontsevich's
graphs, the second term of \eqref{e:linf-expl} involves a sum over all
ways of combining two graphs together by a single edge to get a larger
graph, multiplying the weights for those graphs. The last term (the
RHS) of \eqref{e:linf-expl} involves summing over all ways of
expanding a graph by adding a single edge.  The first term on the LHS
says to apply the Hochschild differential to the result of all graphs,
and this can be suppressed in exchange for adding to Kontsevich's map
$\mathcal{U}$ a term $\mathcal{U}_0: \bR = \Sym^0 (T_\poly[1]) \to
D_\poly[1]$ which sends $1 \in \bR$ to $\mu_A \in C^2(A)$.  (Recall
that $\bk=\bR$ for Kontsevich's construction; this is needed to define
the weights of the graphs.)

\subsection{Formality in terms of a $L_\infty$ quasi-isomorphism}
We deduce from the preceding material that the formality theorem, Theorem
\ref{t:form-mc}, can be restated as
\begin{theorem}\cite{Kform,Kon-dqav,Yek-dqag,DTT-hgahcraf}
\label{t:form-linf}
There is an $L_\infty$ quasi-isomorphism,
\[
T_\poly(X) \iso D_\poly(X).
\]
\end{theorem}
The proof is accomplished for $X=\bR^n$ in \cite{Kform} by finding
weights $w_\Gamma$ to attach to all of the graphs $\Gamma$ described
above, so that the explicit equations of the preceding section are
satisfied. These explicit equations are quadratic in the weights, and
are of the form, for certain graphs $\Gamma$, denoting by $|\Gamma|$
the number of edges of $\Gamma$,
\[
\sum_{|\Gamma_1|+|\Gamma_2|=|\Gamma|} c(\Gamma,\Gamma_1,\Gamma_2) w_{\Gamma_1} w_{\Gamma_2}
+ \sum_{|\Gamma'|=|\Gamma|-1} c(\Gamma,\Gamma') w_{\Gamma'} = 0,
\] 
for suitable coefficients $c(\Gamma,\Gamma_1,\Gamma_2)$ and
$c'(\Gamma,\Gamma')$.  In fact, the graphs that appear are all of the
following form: in the first summation, $\Gamma_1 \subseteq \Gamma$ is
a subgraph which is incident to the real line, and $\Gamma_2 =
\Gamma/\Gamma_1$ is obtained by contracting $\Gamma_1$ to a point on
the real line.  In the second summation, $\Gamma'$ is obtained by
contracting a single edge in $\Gamma$.  The coefficient
$c(\Gamma,\Gamma_1,\Gamma_2)$ is a signed sum of the ways of realizing
$\Gamma_1 \subseteq \Gamma$ such that $\Gamma_2=\Gamma/\Gamma_1$, and
the coefficient $c(\Gamma,\Gamma')$ is a signed sum over edges $e$ in
$\Gamma$ such that $\Gamma/e = \Gamma'$.

To find weights $w_\Gamma$ satisfying these equations, Kontsevich
defines the $w_\Gamma$ as certain integrals over partially
compactified configuration spaces of vertices of the graph, such that
the above sum follows from Stokes' theorem for the configuration space
$C_\Gamma$ associated to $\Gamma$, whose boundary strata are of the
form $C_{\Gamma_1 \times \Gamma_2}$ or $C_{\Gamma'}$ (the technical
part of the resulting proof involves showing that the terms from
Stokes' theorem not appearing in the desired identity vanish).
\begin{remark}\label{r:ginf}
  In fact, $T_\poly$ and $D_\poly$ are quasi-isomorphic not merely as
  $L_\infty$ algebras but in fact as homotopy Gerstenhaber
  (``$G_\infty$'') algebras, i.e., including the structure of cup
  product (and additional ``brace algebra'' structures in the case of
  $D_\poly$). Although, by Theorem \ref{t:hkr}, the HKR morphism is
  compatible with cup product on cohomology,\footnote{The
    analogous sheaf-theoretic statement for nonaffine smooth varieties
    is no longer true: see, e.g., \cite{CV-HcAc,CV-GfGl}} the
  preceding statement is much stronger than this. 
  In \cite{DTT-hgahcraf} the main result actually constructs a
  homotopy Gerstenhaber equivalence between $T_\poly(X)$ and
  $D_\poly(X)$, in the general smooth affine algebraic setting over
  characteristic zero; in \cite{Wil-nBiKSif}, Willwacher completes
  Kontsevich's $L_\infty$ quasi-isomorphism to a homotopy Gerstenhaber
  quasi-isomorphism (which requires adding additional Taylor series
  terms).  In fact, in \cite{Wil-nBiKSif}, Willwacher shows that
  Kontsevich's morphism lifts to a ``$KS_\infty$ quasi-isomorphism''
  of pairs 
\[
(\HH^\bullet(\cO(X)), \HH_\bullet(\cO(X))) \iso (C^\bullet(\cO(X)),
C_\bullet(\cO(X))),
\] where we equip the
  Hochschild homology with the natural operations by the contraction
  and Lie derivative operations from Hochschild cohomology (i.e., the
  calculus structure), and similarly equip Hochschild chains with the
  analogous natural operations by Hochschild cochains.
\end{remark}

\subsection{Twisting the $L_\infty$-morphism; Poisson and Hochschild
  cohomology}\label{ss:tw-linf}
Given a formal Poisson
structure $\pi_\hbar \in \MCE(\hbar \cdot T_\poly(X)[\![\hbar]\!])$
and its image star product $\star$, corresponding to the element
$\mathcal{U}(\pi_\hbar) \in \MCE(\hbar \cdot
D_\poly(X)[\![\hbar]\!])$, we obtain a quasi-isomorphism of twisted
dglas,
\begin{equation}\label{e:tw-kont-quism}
  T_\poly(X)(\!(\hbar)\!)^{\pi_\hbar} \iso
  D_\poly(X)(\!(\hbar)\!)^{\mathcal{U}(\pi_\hbar)}.
\end{equation}
This has an important meaning in terms of Poisson and Hochschild cohomology.
Namely, the RHS computes the Hochschild cohomology of the algebra
$(\cO(X)(\!(\hbar)\!), \mathcal{U}(\pi_\hbar))$, which follows from the
following result, cf.~Remark \ref{r:mc-assoc}:
\begin{lemma}\label{l:zero-mult}
Let $A_0$ be a vector space, viewed as an algebra
with the zero multiplication, and let $A=(A_0,\mu)$ be
an associative algebra with multiplication map $\mu \in C^2(A_0)$.
Then $C^\bullet(A) = C^\bullet(A_0)^\mu$.
\end{lemma}
The proof follows, as in Remark \ref{r:mc-assoc}, because $C^\bullet(A)$ and
$C^\bullet(A_0)$ have the same underlying graded vector space; the differential on $C^\bullet(A_0)$ is zero, and one checks that the differential on $C^\bullet(A)$ is the operation $[\mu,-]$ of taking the Gerstenhaber bracket with $\mu$.
\begin{corollary} \label{c:two-mults}
If $\mu, \mu' \in C^2(A_0)$ are two different associative
multiplications on $A_0$, then setting $A=(A_0,\mu)$ and
$A'=(A_0,\mu')$, we have $C^2(A')=C^2(A)^{\mu'-\mu}$.
\end{corollary}
Thus, applying the corollary to the situation where $A_0 = \cO(X)[\![\hbar]\!]$, with $\mu$ the undeformed multiplication 
and $\mu'$ the one corresponding to $\star$ (i.e., $\mu + \mathcal{U}(\pi_\hbar)$), we obtain the promised
\begin{corollary}
$\HH^*(\cO(X)[\![\hbar]\!],\star) \cong H^*(D_\poly(X)(\!(\hbar)\!)^{\mathcal{U}(\pi_\hbar)})$.
\end{corollary}
Similarly, we can interpret the first term as the Poisson cohomology
of the Poisson algebra $(\cO(X)(\!(\hbar)\!),\pi_\hbar)$:
\begin{definition}[\cite{Br}]
Let $X$ be a smooth affine variety with a Poisson bivector
$\pi$. Then, the Poisson cohomology of 
$X$ is the cohomology of the dgla $T_\poly(X)^{\pi}$.
\end{definition}
\begin{remark}
  The definition of Poisson cohomology in the general nonsmooth affine
  context is quite different, and is expressed as the cohomology of a
  canonical differential on the Lie algebra of derivations of the free
  Poisson algebra generated by $\cO(X)^*[1]$, taking the topological
  dual as in Remark \ref{r:top-dual}. Note that this is completely
  analogous to the definition of Hochschild cohomology, which can also
  be defined as the cohomology of a canonical differential on the Lie
  algebra of derivations of the free associative algebra generated by
  $A^*[1]$ (in general, the analogous cohomology for an algebra $A$ of
  any type is given as the cohomology of a free algebra of the Koszul
  dual type generated by $A^*[1]$ with a canonical differential,
  cf.~Remark \ref{r:kd}).  In the smooth affine case, one can show
  this coincides with the above definition.
\end{remark}
\begin{corollary} Working over $\bk(\!(\hbar)\!)$,
the Poisson cohomology of $(\cO(X)(\!(\hbar)\!),\pi_\hbar)$ is
$H^*(T_\poly(X)(\!(\hbar)\!)^{\pi_\hbar})$.
\end{corollary}
We conclude:
\begin{corollary}\label{c:hphh}
There is an isomorphism of graded vector spaces,
\begin{equation}\label{e:hphh}
  \HP^\bullet(\cO(X)(\!(\hbar)\!), \pi_\hbar) \iso 
  \HH^\bullet(\cO(X)(\!(\hbar)\!), \star).
\end{equation}
In particular, applied to degrees $0, 1$, and $2$, there
are canonical $\bk(\!(\hbar)\!)$-linear
isomorphisms:
\begin{enumerate}
\item[(i)] From the Poisson center of $(\cO(X)(\!(\hbar)\!),
  \pi_\hbar)$ to the center of $(\cO(X)(\!(\hbar)\!), \star)$;
\item[(ii)] From the outer derivations of $(\cO(X)(\!(\hbar)\!), \pi_\hbar)$ to the outer derivations of $(\cO(X)(\!(\hbar)\!), \star)$;
\item[(iii)] From infinitesimal Poisson deformations of
  $(\cO(X)(\!(\hbar)\!), \star)$ to infinitesimal algebra deformations
  of $(\cO(X)(\!(\hbar)\!), \star)$.
\end{enumerate}
\end{corollary}
\begin{remark}\label{r:ginf2}
  In fact, the above isomorphism \eqref{e:hphh} is an isomorphism of
  \emph{$\bk(\!(\hbar)\!)$-algebras}, and hence also the morphism (i)
  is an isomorphism of algebras, and the morphisms (ii)--(iii) are
  compatible with the module structures over the Poisson center on the
  LHS and the center of the quantization on the RHS.  This is highly
  nontrivial and was proved by Kontsevich in \cite[\S 8]{Kform}.  More
  conceptually, the reason why this holds is that Kontsevich's
  $L_\infty$ quasi-isomorphism lifts to a homotopy Gerstenhaber
  isomorphism, as we pointed out in Remark \ref{r:ginf}.
\end{remark}
We can now sketch a proof of Theorem \ref{t:fdef-symp}.  Suppose that
$X=(X,\omega)$ is an symplectic affine variety or $C^\infty$ manifold
equipped with its canonical Poisson bracket. Recall that the symplectic
condition means
that $\omega$ is a closed two-form such that the map $\xi \mapsto
i_\xi(\omega)$ defines an isomorphism $\omega^\sharp: \Vect(X) \to \Omega^1(X)$
between vector fields and one-forms.
  The corresponding Poisson structure is given by
$(\omega^\sharp)^{-1}: \Omega^1(X) \to \Vect(X)$, namely, $(\omega^\sharp)^{-1} =
\pi^\sharp: \alpha \mapsto i_\pi \alpha$. For ease of notation, we
will write $\pi = \omega^{-1}$.

Suppose that $X$ is symplectic as above.  Then, \cite{Br} shows that
the Poisson cohomology $\HP^\bullet(\cO(X))$ is isomorphic to the de
Rham cohomology of $X$.  Now, let $\pi_\hbar := \hbar \pi$ be the
obtained Poisson structure on $\cO(X)(\!(\hbar)\!)$; this corresponds
to the symplectic structure $\omega_\hbar := \hbar^{-1} \omega$.
Then, working over $\bk(\!(\hbar)\!)$, \cite{Br} implies the first
statement of the theorem.

For the second statement, we first caution that Theorem \ref{t:hh3=0}
\emph{cannot} be applied in general, since
$\HH^3(\cO(X)(\!(\hbar)\!),\star)\cong H_{DR}^3(X)(\!(\hbar)\!)$,
which need not be zero for general $X$. Nonetheless, a versal family
can be explicitly constructed: by \eqref{e:tw-kont-quism}, it suffices
to construct a versal family of deformations of the Poisson
structure on $(\cO(X)(\!(\hbar)\!),\pi_\hbar)$.  Let $\alpha_1,
\ldots, \alpha_k$ be closed two-forms on $X$ which map to a basis of
$H^2_{DR}(X)$.  Then we obtain the versal family of symplectic
structures $\hbar^{-1}(\omega + c_1 \alpha_1 + \cdots + c_k \alpha_k)$
over $\bk(\!(\hbar)\!)[\![c_1,\ldots,c_k]\!]$, and similarly the
versal family $\hbar (\omega + c_1 \alpha_1 + \cdots + c_k \alpha_k)^{-1}$
of Poisson deformations by inverting the elements of this family.

\subsection{Explicit twisting of $L_\infty$ morphisms}
We caution that, unlike in the untwisted case, \eqref{e:hphh} is
\textbf{not} obtained merely from the HKR morphism. Indeed, even in
degree zero, an element which is Poisson central for $\pi_\hbar$ need
not correspond in any obvious way to an element which is central in
the quantized algebra.  The only obvious statement one could make is
that we have a map modulo $\hbar$, where $\pi_\hbar = \hbar \cdot \pi
+ $ higher order terms, with $\pi$ an ordinary Poisson structure,
\[
Z(\cO(X), \pi) \leftarrow 
Z(\cO(X)[\![\hbar]\!], \star)/(\hbar) 
\]
which is quite different (and weaker) than the above statement.

To write the correct formula for the isomorphism \eqref{e:hphh}, we
need to discuss functoriality for twisting.  Namely,
given an $L_\infty$ morphism $F:
\mfg \to \mfh$
and a Maurer-Cartan
element $\xi \in \MCE(\mfg)$, we need to define an $L_\infty$ morphism
$F^\xi: \mfg^\xi \to \mfh^{F(\xi)}$. 
Given this, we
 obtain the twisted $L_\infty$ quasi-isomorphism
\eqref{e:tw-kont-quism}, then pass to cohomology to obtain
\eqref{e:hphh}.

For Kontsevich's morphism $\mathcal{U}$, which is explicitly defined
by $\mathcal{U}^*$, this produces the HKR isomorphism
$H^\bullet(T_\poly(X)) = \wedge^{\bullet+1}_{\cO(X)} \Vect(X) \to
H^\bullet(D_\poly(X)) \cong \HH^{\bullet+1}(\cO(X))$.  However, to
apply this to the twisted versions $\mathcal{U}^{\pi_\hbar}:
T_{\poly}(X)^{\pi_\hbar} \to D_{\poly}(X)^{\mathcal{U}(\pi_\hbar)}$,
we need an explicit formula for the pullback $(F^\xi)^*$. This is
nontrivial by the observation at the beginning of the subsection---on
cohomology one does \emph{not} obtain the HKR morphism.

By Proposition \ref{p:mc-twist}.(ii), there is a canonical map on Maurer-Cartan elements $F^{\xi}: \MCE(\mfg) \to \MCE(\mfh)$:
\begin{definition}
Set $F^{\xi}(\eta) := F(\eta+\xi)-F(\xi)$.  
\end{definition}
We extend this to a map $\MCE(\mfg \hat \otimes R_+) \to \MCE(\mfh \hat \otimes R_+)$ for all complete augmented $R$ functorially. 
Explicitly, this implies that the pullback $(F^\xi)^*$ can be defined
as follows:  $(F^{\xi})^* = T^*_{\xi} \circ F^* \circ T^*_{-F(\xi)}$, where
$T_\xi(\eta) = \eta+\xi$, so $T^*_\xi(x) = x-\xi(x)$ for $x \in \mfg[1]^*$ (and $\xi \in \MCE(\mfg) \subseteq \mfg^1$, so $\xi(x) = 0$ if $x \in (\mfg[1]^*)^m$ with $m \neq 0$).
This extends uniquely to a continuous augmented dg morphism.


Explicitly, the formula we obtain on Taylor coefficients $F^\xi_m$ is,
for $\eta_1, \ldots, \eta_m \in \mfg$,
\begin{gather}
  F^\xi_m(\eta_1 \wedge \cdots \wedge \eta_m) = \sum_{k \geq 0}
  {m+k \choose k} F^\xi_{m+k}(\xi^{\wedge k} \wedge \eta_1 \wedge \cdots \wedge
  \eta_m).
\end{gather}
Thus, the composition of $(F^\xi)^*$ with the projection yields the
map on cohomology,
\begin{equation}\label{e:kont-tgt}
  H^\bullet(\mfg^\xi) \to H^\bullet(\mfh^{F(\xi)}), \quad
  x \mapsto \sum_{m \geq 1} m F_m(x \wedge \xi^{\wedge(m-1)}).
\end{equation}
Note that this formula is not so obvious from the simple definition
above: one has to be careful to interpret that formula functorially so
as to be given by conjugating $F^*$ by a dg algebra isomorphism.

Let us now explain a geometric interpretation, which was used by
Kontsevich in \cite[\S 8]{Kform} (this also gives an alternative way
to derive \eqref{e:kont-tgt}).  He observes that the cohomology of the
twisted dgla $\mfg^\xi$ is the \emph{tangent space} in the moduli
space $\MCE(\mfg) / \sim$ to the Maurer-Cartan element $\xi$.  Here
$\sim$ denotes gauge equivalence.  This is because, if we fix $\xi$,
and differentiate the Maurer-Cartan equation $d(\xi+\eta)+\frac{1}{2}
[\xi+\eta, \xi+\eta]$ with respect to $\eta$, then the tangent space
is the kernel of $d+\ad \xi$ on $\mfg^1$, and two tangent vectors
$\eta$ and $\eta'$ are gauge equivalent if and only if they differ by
$(d+\ad \xi)(z)$ for $z \in \mfg^0$.  Functorially, this says that the
dg vector space $(\mfg, d+\ad \xi)$ is the dg tangent space to the
moduli space of Maurer-Cartan elements of $\mfg$, taken with
coefficients in arbitrary complete augmented dg commutative algebras:
this is because, if you use such a complete augmented dg commutative
algebra $R$ which is not concentrated in degree zero, then $\MCE(\mfg
\hat \otimes R_+)$ will detect cohomology of $\mfg$ which is not
merely in degree one. So its cohomology is the (cohomology of the)
tangent space.

Therefore, we will adopt Kontsevich's terminology and refer to
\eqref{e:kont-tgt} as the \emph{tangent map}
(in \cite[\S 8]{Kform}, it is denoted by $I_T$, at least in the
situation of \eqref{e:tw-kont-quism} with $X$ the dual to a
finite-dimensional Lie algebra, equipped with its standard Poisson
structure).

\subsection{The algebra isomorphism $(\Sym \mfg)^\mfg \iso Z(U\mfg)$
  and Duflo's isomorphism}\label{ss:duflo}
Now let $\mfg$ be a finite-dimensional Lie algebra and $X = \mfg^* =
\Spec \Sym \mfg$ the associated affine Poisson variety.  The Poisson
center $\HP^0(\cO(X)) = Z(\cO(X))$ is equal to $(\Sym \mfg)^\mfg$.  By
Remark \ref{r:ginf2}, Kontsevich's morphism induces an isomorphism of
algebras
\begin{equation}\label{e:pc-c-lie}
(\Sym \mfg)^\mfg[\![\hbar]\!] \iso Z(\Sym \mfg[\![\hbar]\!], \star).
\end{equation}
Moreover, by Example \ref{ex:symg-sp} and Exercise
\ref{exer:kont-ug-iso},
$(\Sym \mfg[\![\hbar]\!], \star) \cong U_\hbar \mfg$, so we obtain
from the above an isomorphism
\begin{equation}\label{e:pc-c-lie-u}
(\Sym \mfg)^\mfg[\![\hbar]\!] \iso Z(U_\hbar \mfg).
\end{equation}
We can check that this is actually defined over polynomials in
$\hbar$, since the Poisson bracket has polynomial degree $-1$ (i.e.,
$|\{f,h\}|=|f|+|h|-1$ for homogeneous $f$ and $h$), so that the target
of an element of $(\Sym \mfg)^\mfg$ is polynomial in $\hbar$.
Therefore we get an isomorphism $(\Sym \mfg)^\mfg[\hbar] \iso
Z(T\mfg[\hbar] / (xy-yx-\hbar\{x,y\})$, and further modding by
$(\hbar-1)$ we obtain an isomorphism
\[
(\Sym \mfg)^\mfg \iso Z(U\mfg).
\]
That such an isomorphism exists, for \emph{arbitrary}
finite-dimensional $\mfg$, is a significant generalization of the
Harish-Chandra isomorphism for the semisimple case; it was first
noticed by Kirillov and then proved by Duflo, using a highly
nontrivial formula.  By partially computing his isomorphism
\eqref{e:pc-c-lie}, Kontsevich was able to show that his isomorphism
\eqref{e:pc-c-lie-u} coincides with the Duflo-Kirillov
isomorphism. The latter isomorphism is given by the following explicit
formula. Let $\symm: \Sym \mfg \to U \mfg$ be the symmetrization map,
$x_1 \cdots x_m \mapsto \frac{1}{n!} \sum_{\sigma \in S_n}
x_{\sigma(i)} \cdots x_{\sigma(m)}$. This is a vector space
isomorphism (even an isomorphism of $\mfg$-representations via the
adjoint action) but \emph{not} an algebra isomorphism.  However, it
can be corrected to an isomorphism $I_{DK}: (\Sym \mfg)^\mfg \iso
Z(U\mfg)$ given by $I_{DK} = \symm \circ I_{\text{strange}}$, where
$I_{\text{strange}}: \Sym \mfg \to \Sym \mfg$ is given by an
(infinite-order) constant-coefficient differential operator. 
Such operators can be viewed as formal power series functions
of $\mfg$, i.e., elements of $\hat S \mfg^*$, via the inclusions
\[
\Sym^m \mfg^* \into \Hom_\bk(\Sym^{\bullet} \mfg, \Sym^{\bullet-m} \mfg). 
\]
In other words, if $\mfg$ has a basis $x_i$ with dual basis
$\partial_i \in \mfg^*$, then an element of $\Sym \mfg^*$ is a
polynomial in the $\partial_i$, i.e., a constant-coefficient
differential operator, and $\hat S \mfg^*$ is a power series in the $\partial_i$.

Then, the element $I_{\text{strange}} \in \hat S \mfg^*$, as a power
series function of $\mfg$, is expressed as
\begin{equation}\label{e:kd-fla}
I_{\text{strange}} =  \Bigl( x \mapsto \exp\Bigl( \sum_{k \geq 1} \frac{B_{2k}}{4k(2k!)} \tr(\ad(x)^{2k}) \Bigl) \Bigr),
\end{equation}
where $B_2, B_4, \ldots$ are Bernoulli numbers.  Unpacking all of this, we see that $I_{\text{strange}}(x^m)-x^m$ is a certain linear combination of elements of the form
\[
x^{m-2(i_1+\cdots+i_k)} \tr((\ad x)^{2i_1}) \cdots \tr((\ad x)^{2i_k}),
\] 
for $i_1, \ldots, i_k$ positive integers such that $2(i_1+\cdots+i_k) \leq m$
and $k \geq 1$.

Let us explain in more detail how to unpack Kontsevich's isomorphism
and see in the process why it might coincide with the Duflo-Kirillov
isomorphism.  As explained in the previous subsection,
\eqref{e:pc-c-lie} is a \emph{nontrivial} map, given by
\[
f \mapsto f + \sum_{m \geq 1} (m+1)\hbar^{m} \mathcal{U}_{m+1}(f \wedge
\pi^{\wedge m}).
\]
We are interested in the case that $f$ is an element of the Poisson
center of $\cO(X)[\![\hbar]\!]$, which means that $f$ corresponds to a
vertex (in the upper-half plane) with no outgoing edges (it is a
polyvector field of degree zero).  Since $\pi$ is a bivector field,
the term $\pi^{\wedge m}$ corresponds to $m$ vertices in the
upper-half plane with two outgoing edges.  Thus, the above sum
re-expresses as a sum over all graphs in the upper half plane $\{(x,y)
\mid y \geq 0\} \subseteq \bR^2$, up to isomorphism, which have a
single vertex labeled by $f$ in the upper half plane with no outgoing
edges and $m$ vertices labeled by $\pi$ with two outgoing edges each.
Moreover, we can discard all the graphs where the two outgoing edges
of a given vertex labeled by $\pi$ have the same target, i.e., we can
assume the graph has no \emph{multiple edges}, since $\pi$ is
skew-symmetric, and thus the resulting bilinear operation (cf.~\S
\ref{ss:dq-fla}) would be zero. 

As a result, we conclude by counting that each vertex labeled by $\pi$
has exactly one incoming edge, and the vertex labeled $f$ has $m$
incoming edges, one from each vertex labeled by $\pi$.  We can express
such a graph as a union of oriented $m$-gons labeled by $\pi$, with
each vertex of each $m$-gon pointing to a single additional vertex
labeled by $f$.  In the case where there is a single $m$-gon, we put
the vertex $f$ in the center of the $m$-gon, and the resulting graph
looks like a wheel, so is called a \emph{wheel}.  We consider a
general graph to be a union of wheels (where the union is taken by
gluing the vertices labeled $f$ together to a single vertex).

As we explain in Exercise \ref{exer:kont-duflo}.(ii), the differential
operator corresponding to an $m$-gon is $x \mapsto \tr((\ad x)^m)$,
and we explain in part (iii) that the differential operator attached
to a union of $m_i$-gons is the product of these differential
operators. Then, in part (iv), we deduce that Kontsevich's isomorphism
should be obtained from a linear map $\Sym \mfg \to U \mfg$ sending
$x^m$ to a polynomial in $x$ and $\tr((\ad(x))^j)$, as in the
Kirillov-Duflo isomorphism.

Kontsevich shows that the weight of a union of wheels is the product
of the weights of the wheels. Moreover, by a symmetry argument he uses
also elsewhere, he concludes that the weight is zero when a wheel has
an odd number of edges.  This shows (see Exercise
\ref{exer:kont-duflo}.(v) below) that
Kontsevich's isomorphism must be the composition of $\symm$ and
an operator of the form \eqref{e:kd-fla} except possibly with
different coefficients than $\frac{B_{2k}}{4k(2k!)}$.

Finally, Kontsevich shows that there can only be at most one
isomorphism of this form, so (without computing them!) his
coefficients must equal the $\frac{B_{2k}}{4k(2k!)}$.




\subsection{Exercises}

Exercises from the notes: \ref{exer:pre-lie}, \ref{exer:circ-prod}, \ref{exer:gauge}, \ref{exer:pb-sn}, and \ref{exer:linf-dgla}.

Additional exercises:
\begin{exercise}\label{exer:kont-ug-iso}
Let $\mfg$ be a finite-dimensional Lie algebra. Equip
  $\cO(\mfg^*)$ with its standard Poisson structure.  Show that, for
  Kontsevich's star product $\star$ on $\Sym \mfg$ associated to this
  Poisson structure,
one has
\[
v \star w - w \star v = \hbar [v,w], \quad \forall v,w \in \mfg.
\]
Hint: Show that only the graph corresponding to the Poisson bracket
can give a nonzero contribution to $v \star w - w \star v$ when $v,w
\in \mfg$.
\end{exercise}

\begin{exercise}(*, but with the proof outlined) \label{exer:hh3=0}
  Prove, following the outline below, the following slightly weaker
  version of the first statement of Theorem \ref{t:hh3=0}: there is a
  map from formal power series $\gamma = \sum_{m \geq 1} \hbar^m \cdot
  \gamma_m \in \hbar \cdot \HH^2(A)[\![\hbar]\!]$ to formal
  deformations of $A$ which exhausts all formal deformations up to
  gauge equivalence.  (With a bit more work, the proof below can be
  extended to give the first statement of the Theorem.) This is
  similar to Exercise \ref{exer:obstr-nord} and its proof.


Namely, use the Maurer-Cartan formalism, and the fact that, if
$C^\bullet$ is an arbitrary complex of vector spaces, there exists a
homotopy $H: C^\bullet \to C^{\bullet-1}$ such that
$\Id - (Hd+dH)$ is a projection of $C^\bullet$ onto a subspace of
$\ker(d)$ which maps isomorphically to $H^\bullet(C)$ (this is called
a Hodge decomposition; we remark that it always satisfies $dHd=d$).
In this case, let $i: H^\bullet(C) \cong \im(\Id-(Hd+dH)) \into
C^\bullet$ be the obtained inclusion; we have $C^\bullet =
i(H^\bullet(C)) \oplus (dH+Hd)(C^\bullet)$, with $(dH+Hd)(C^\bullet)$
a contractible complex.

In the case $\HH^3(A)=0$, let $H$ be a homotopy as above for
$C^\bullet := C^\bullet(A)$. Then, on Hochschild three-cocycles $Z^3(A)$,
$dH|_{Z^3(A)} = \Id$ is the identity.  Now, if $\gamma = \sum_{m \geq 1}
\hbar^m \cdot \gamma_m \in \hbar \cdot \HH^2(A)[\![\hbar]\!]$, we can
construct a corresponding solution $x$ of the Maurer-Cartan equation
$\MC(x) := dx + \frac{1}{2}[x,x]=0$ as follows.  Set $x^{(1)} := 
i(\gamma)$, so in particular $dx^{(1)}=0$. Then, also $\MC(x^{(1)})
\in \hbar^2 Z^3(A)$.  Set $x^{(2)} := x^{(1)} - H \circ \MC(x^{(1)})$.
  Then $\MC(x^{(2)}) \in \hbar^3 C^3(A)[\![\hbar]\!]$.  We
claim that $\MC(x^{(2)}) \in \hbar^3 Z^3 + \hbar^4 C^3$. 
\begin{lemma}
Let $\mfg$ be any dgla, and suppose that
$z \in \hbar \mfg^1[\![\hbar]\!]$ satisfies $\MC(z) \in \hbar^{n}
\mfg^2$.  Then, $\MC(z) \in \hbar^{n} Z^2(\mfg) + \hbar^{n+1}
\mfg^3$.
\end{lemma}
\begin{proof} We have
\[
d \MC(z) = [dz,z] \equiv -\frac{1}{2} [[z,z],z] \pmod {\hbar^{n+1}},
\]
but the RHS is zero by the Jacobi identity.
\end{proof}
Thus $\MC(x^{(2)}) \in \hbar^3 Z^3(A) + \hbar^4 C^3(A)[\![\hbar]\!]$.
Inductively for $n \geq 2$, suppose that $x^{(n)}$ is constructed and
$\MC(x^{(n)}) \in \hbar^{n+1} Z^3(A) + \hbar^{n+2}
C^3(A)[\![\hbar]\!]$. Then, we set $x^{(n+1)} := x^{(n)} - H \circ
\MC(x^{(n)})$, i.e., $x^{(n+1)} = (\Id - H \circ \MC)^n i(\gamma)$. 
It follows from construction that $\MC(x^{(n+1)}) \in
\hbar^{n+2} C^3(A)[\![\hbar]\!]$.  By the above lemma, it is in
$\hbar^{n+2} Z^3(A) + \hbar^{n+3} C^3(A)[\![\hbar]\!]$, completing the
inductive step.  By construction, $x^{(n+1)} \equiv x^{(n)} \pmod
{\hbar^{n+1}}$, so that $x:=\lim_{n \to \infty} x^{(n)}$ exists. Also
by construction, $x$ is a solution of the Maurer-Cartan equation. 

One can show that the map $x \mapsto \gamma$ above yields all possible
formal deformations of $A$ up to gauge equivalence, by showing
inductively on $m$ that it yields all $m$-th order deformations, i.e.,
deformations over $\bk[\hbar]/(\hbar^{m+1})$, for all $m \geq 1$. This
is because the tangent space to the space of extensions of an $m$-th
order deformation to $(m+1)$-st order deformations modulo gauge
equivalence is given by $\hbar^{m+1} \cdot \HH^2(A)$.
\end{exercise}

\begin{exercise}\label{exer:kont-duflo}(*) 
  Here, following \cite[\S 8]{Kform}, we complete the steps from \S
  \ref{ss:duflo} above, outlining why Kontsevich's isomorphism $(\Sym
  \mfg)^{\mfg} \to Z(U \mfg)$ must be of the form \eqref{e:kd-fla}
  except with possibly different coefficients than the
  $\frac{B_{2k}}{4k(2k!)}$ there.  In particular, we show here why it
  is at least a power series in $\tr(\ad(x)^i)$, and how the desired
  formula follows from certain properties of the weights.

  The description in \S \ref{ss:dq-fla} of polydifferential operators
  from graphs with two vertices on the real line generalizes in a
  straightforward way to define operators
\[
B_\Gamma(\gamma): \cO(X)^{\otimes_\bk n} \to \cO(X), \quad \gamma:
V_\Gamma \to T_\poly(X),
\]
where $\Gamma$ is an arbitrary graph in the upper-half plane with $n$
(not necessarily two) vertices on the real line, and $\gamma: V_\Gamma
\to T_\poly(X)$ is a function sending each vertex $v$ to an element of
degree equal to one less than the number of outgoing edges from $v$
(i.e., the corresponding degree in $\wedge_{\cO(X)}^\bullet \Vect(X)$ is equal to
the number of outgoing edges).  Define this generalization, or
alternatively see \cite{Kform}.

Now, let $X = \mfg^*$ equipped with its standard Poisson structure,
for $\mfg$ a finite-dimensional Lie algebra.  Recall the wheels from
\S \ref{ss:duflo}.  Let $W_m$ denote the wheel with $m$ vertices
labeled $\pi$ and one vertex labeled $f$.  We are interested in the
operators
\[
B_{W_m}(-,\pi,\ldots,\pi): \Sym \mfg \to \Sym \mfg,
\]
placing the vertex labeled $f$ first, where $\pi$ is the Poisson
bivector on $\mfg^*$.
\begin{itemize}
\item[(i)]
Show that $B_{W_m}(-,\pi,\ldots,\pi)$ is a constant-coefficient differential
operator of order $m$, i.e., that
\[
B_{W_m}(-,\pi,\ldots,\pi) \in \Sym^m \mfg^*.
\]
\item[(ii)] In terms of $\Sym^m \mfg^* = \cO(\mfg)$, show that
  $B_{W_m}(-,\pi,\ldots,\pi)$ corresponds to the polynomial function
\[
x \mapsto \tr ((\ad x)^m), x \in \mfg.
\]
\item[(iii)] Now suppose that $\Gamma$ is a graph which is a union
  (glued at the vertex labeled $f$) of $k$ wheels $W_{m_1}, \ldots,
  W_{m_k}$.  Show that, considered as polynomials in $\cO(\mfg) =
  \Sym^m \mfg^*$,
\[
B_{\Gamma}(-,\pi,\ldots,\pi) = x \mapsto \prod_{i=1}^k \tr((\ad x)^{m_i}).
\]
\item[(iv)] 
Conclude that Kontsevich's isomorphism 
\begin{equation}\label{e:expl-pcm-g}
(\Sym \mfg)^\mfg[\![\hbar]\!] \iso Z(\cO(\mfg^*)[\![\hbar]\!], \star)
\end{equation}
has the form, for some constants $c_{m_1,\ldots,m_k}$, now viewed as
an element of $\hat S \mfg^*$, i.e., a power series function contained
in the completion $\hat \cO(\mfg)$,
\[
x \mapsto \sum_{k; m_1, \ldots, m_k} c_{m_1,\ldots,m_k} \cdot \prod_{i=1}^k \tr((\ad x)^{m_i}),
\]
viewed as an element of $\hat S \mfg^*$, i.e., a formal sum of
differential operators (with finitely many summands of each
order). Note here that we allow $m_1, \ldots, m_k$ to all be
independent (and do not require, for example, $m_1 \leq \cdots \leq
m_k$).
\item[(v)] Now, it follows from Kontsevich's explicit definition of
  the weights $c_\Gamma$ associated to graphs $\Gamma$, which are the
  coefficients of $B_\Gamma$ in the definition of his $L_\infty$
  quasi-isomorphism $\mathcal{U}$, that
\begin{gather*}
  c_{m_1,\ldots,m_k} = \frac{1}{k!} \prod_{i=1}^k c_{m_i}, \text{ and }\\
  c_m = 0 \text{ if $m$ is odd.}
\end{gather*}
Using these identities (if you want to see why they are true, see
\cite{Kform}), conclude that \eqref{e:expl-pcm-g} is given by the
(completed) differential operator corresponding to the series
(cf.~\cite[Theorem 8.5]{Kform})
\[
x \mapsto \exp\bigl( \sum_{m \geq 1} c_{2m} \tr((\ad x)^{2m}) \bigr)
\in \hat S \mfg^* = \hat \cO(\mfg).
\]
\end{itemize}
\end{exercise}

\section{Calabi-Yau algebras and isolated hypersurface
  singularities}\label{s:cy}
The goal of this section, which may seem somewhat of an abrupt
departure from previous sections, is to introduce Calabi-Yau algebras
(first mathematically defined and studied in \cite{GinzCY}) and use them
to study deformation theory. Our motivation is twofold: many of the
deformations we have studied can be obtained and understood by
realizing algebras as quotients of Calabi-Yau algebras and deforming
the Calabi-Yau algebras, and second, because Calabi-Yau algebras are
a unifying theme appearing throughout this book.  We explain each of
these motivations in more detail (in \S \ref{ss:cy-mot1} and
\ref{ss:cy-mot2}), before discussing first the homological properties
of these algebras (\S \ref{ss:cy-hom}), then the remarkable phenomenon
that they are often defined by potentials (\S \ref{ss:cy-pot}) and how
that can be exploited to deform them (\S \ref{ss:cy-pot-def}), and
finally how to use them to study deformations of isolated hypersurface
singularities and del Pezzo surfaces (\S \ref{ss:eg}), following
Etingof and Ginzburg.

\subsection{Motivation: deformations of quotients of
Calabi-Yau algebras}\label{ss:cy-mot1}
We have seen that, in many cases, it is much easier to describe
deformations of algebras via generators and relations than via star
products (and note that, by Exercise \ref{exer:deduce-kdp-filt}, all
formal deformations can be obtained by deforming the relations).  For
example, this was the case for the Weyl and universal enveloping
algebras, as well as the symplectic reflection algebras, where it is
rather simple to write down the deformed relations, but the star
product is complicated.  More generally, given a Koszul algebra, one
can easily study the formal homogeneous (or filtered) deformations of
the relations that satisfy the PBW property, i.e., give a flat
deformation, by using the Koszul deformation principle (Theorem
\ref{t:kdp}, see also Theorem \ref{t:kdp-filt}), even though writing
down the corresponding star products could be difficult.

In our running example of $\bC[x,y]^{\bZ/2} = \bC[x^2,xy,y^2] =
\cO(\Nil \mfsl_2)$, one way we did this was not by directly finding a
noncommutative deformation of the singular ring $\bC[x,y]^{\bZ/2}$
itself, but rather by deforming either $\bC[x,y]$ to $\Weyl_1$ and
taking $\bZ/2$ invariants, or more generally deforming $\bC[x,y]
\rtimes \bZ/2$ to a symplectic reflection algebra $H_{1,c}(\bZ/2)$ and
then passing to the spherical subalgebra $e H_{1,c}(\bZ/2) e$.

Another way we did it was by realizing the ring as $\cO(\Nil
  \mfsl_2)$ for the subvariety $\Nil \mfsl_2 \subseteq \mfsl_2 \cong
\bA^3$ inside affine space, then first deforming $\cO(\mfsl_2) =
\cO(\mfsl_2^*)$ to the noncommutative ring $U \mfsl_2$, then finding
a central quotient that yields a quantization of $\cO(\Nil \mfsl_2)$.

We would like to generalize this approach to quantizing more general
hypersurfaces in $\bA^3$.  Suppose $f \in \bk[x,y,z]$ is a
hypersurface and $\bk$ has characteristic zero. Then we claim that
there is a canonical Poisson bivector field on $Z(f)$.  This comes
from the \emph{Calabi-Yau} structure on $\bA^3$,
i.e., \emph{everywhere nonvanishing volume form}.  Namely, $\bA^3$ is
equipped with the volume form
\[
dx \wedge dy \wedge dz.
\]
The inverse of this is the everywhere nonvanishing top polyvector field
\[
\partial_x \wedge \partial_y \wedge \partial_z,
\]
in the sense that the
contraction of the two is the constant function $1$.
Now, we can contract this with $df$ and obtain a bivector field,
\begin{equation}\label{e:pif}
\pi_f := (\partial_x \wedge \partial_y \wedge \partial_z) (df)
= \partial_x(f) \partial_y \wedge \partial_z
+ \partial_y(f) \partial_z \wedge \partial_x
+ \partial_z(f) \partial_x \wedge \partial_y.
\end{equation}
\begin{exercise}  \label{exer:hypsfc}
Show that this is Poisson.  Show also that $f$ is
  Poisson central, so that the quotient $\cO(Z(f)) = \cO(\bA^3) /
  (f)$ is Poisson, i.e., the surface $Z(f) \subseteq \bA^3$ is
  canonically equipped with a Poisson bivector from the Calabi-Yau
  structure on $\bA^3$.  

  Moreover, show that the Poisson bivector $\pi_f$ is
  \emph{unimodular}: for every Hamiltonian vector field $\xi_g :=
  i_\pi(dg) = \{g, -\}$ with respect to this Poisson bivector, we have
  $L_{\xi_g} (\vol) = 0$, where $\vol = dx \wedge dy \wedge dz$.
  Equivalently, $L_{\pi}(\vol) = 0$.

  More generally (but harder), show that, for an arbitrary complete
  intersection surface $Z(f_1, \ldots, f_{n-2})$ in $\bA^n$ with $\bk$ still
of characteristic zero, then
\[
(\partial_{1} \wedge \cdots \wedge \partial_n)(df_1 \wedge \cdots
\wedge df_{n-2})
\]
is a Poisson structure.  Show that $f_1, \ldots, f_{n-2}$ are Poisson
central, so that the surface $Z(f_1, \ldots, f_{n-2})$ is a closed
Poisson subvariety, and in particular has a canonical Poisson
structure.  Moreover, show that $\pi$ is \emph{unimodular}: for every
Hamiltonian vector field $\xi_g := i_\pi(dg)=\{g,-\}$, we have $L_{\xi_g} (\vol)
= 0$, where $\vol = dx_1 \wedge \cdots \wedge dx_n$.  Equivalently,
$L_{\pi}(\vol) = 0$.

From this, we can deduce that the same holds if $\bA^n$ is replaced by
an arbitrary $n$-dimensional Calabi-Yau variety $X$, i.e., a variety
$X$ equipped with a nonvanishing volume form, since the Jacobi
condition $[\pi, \pi] = 0$ can be checked in the formal neighborhood
of a point of $X$, which is isomorphic to the formal neighborhood of
the origin in $\bA^n$.
\end{exercise}
Now, our strategy, following \cite{EGdelpezzo}, for deforming $X =
Z(f) \subseteq \bA^3$ is as follows:
\begin{enumerate}
\item First, consider \emph{Calabi-Yau} deformations of $\bA^3$, i.e.,
  noncommutative deformations of $\bA^3$ as a Calabi-Yau algebra. We
  should consider these in the direction of the Poisson structure
  $\pi_f$ defined above.
\item Next, inside such a Calabi-Yau deformation $A$ of $\cO(\bA^3)$,
  we identify a central (or more generally normal) element $\Phi$
  corresponding to $f \in \cO(\bA^3)$.
\item Then, the quantization of $\cO(X)$ is $A / (\Phi)$.
\end{enumerate}
In order to carry out this program, we need to recall the notion of
\emph{(noncommutative) Calabi-Yau algebras} and their convenient
presentation by relations derived from a single \emph{potential}.
Then it turns out that deforming in the direction of $\pi_f$ is
obtained by deforming the potential of $\bA^3$ in the ``direction of
$f$.''  Since $f$ is Poisson central, by Kontsevich's theorem
(Corollary \ref{c:hphh}), $f$ deforms to a central element of the
quantization.

\subsection{Calabi-Yau algebras as a unifying theme}\label{ss:cy-mot2}
Recall Definition \ref{d:cya}.  These algebras are ubiquitous and in
fact they have appeared throughout the entire book:
\begin{enumerate}
\item A commutative Calabi-Yau algebra is the algebra of functions on
  a Calabi-Yau affine algebraic variety, i.e., a smooth affine
  algebraic variety equipped with an everywhere nonvanishing volume
  form;
\item Most deformations we have considered of Calabi-Yau algebras are
  still Calabi-Yau. This includes:
\begin{enumerate}
\item The universal enveloping algebra $U \mfg$ of a
  finite-dimensional Lie algebra $\mfg$ which is unimodular ($\tr(\ad
  x) = 0$ for all $x \in \mfg$): this is Calabi-Yau of dimension $\dim
  \mfg$ (and without the unimodular condition, is twisted Calabi-Yau
  of the same dimension);
\item The Weyl algebras $\Weyl(V)$ (Calabi-Yau of dimension $\dim V$);
\item The skew-group ring $\cO(V) \rtimes G$ for a vector space $V$ and
$G < \SL(V)$ finite;
\item All symplectic reflection algebras;
\end{enumerate}
\item The invariant subrings $\Weyl(V)^G$ for $G < \Sp(V)$ finite, and
  more generally, all homologically smooth spherical symplectic reflection algebras;
\item All NCCRs that resolve a Gorenstein singularity (as discussed
in Wemyss's chapter; 
\item All of the regular algebras discussed in Rogalski's chapter
  are either
  Calabi-Yau or at least twisted Calabi-Yau. In particular, the
  quantum versions of $\bA^n$, with
\[
x_i x_j = r_{ij} x_j x_i, \quad i < j,
\]
are Calabi-Yau if and only if, setting $r_{ji} := r_{ij}^{-1}$ for $i < j$,
\[
r_{ij} r_{ji} = 1, \forall i,j, \quad
\prod_{j \neq m} r_{mj} = 1, \forall m.
\]
For instance, in three variables, we have a single parameter $q =
r_{12} = r_{23} = r_{31}$ and then $q^{-1} = r_{21} = r_{32} =
r_{13}$.
\end{enumerate}
\subsection{Van den Bergh duality and the BV differential}\label{ss:cy-hom}
\begin{theorem} \cite{VdBdual} \label{t:vdbdual} Let $A$ be Calabi-Yau
  of dimension $d$.  Then, fixing an isomorphism $\HH^d(A, A\otimes A)
  \cong A$ yields a canonical isomorphism, for every $A$-bimodule $M$,
\begin{equation}\label{e:vdb}
\HH_i(A,M) \iso \HH^{d-i}(A,M).
\end{equation}

More generally, if $A$ is twisted Calabi-Yau with $\HH^d(A, A \otimes
A) \cong U$, for $U$ an invertible $A$-bimodule,
then
$\HH_i(A,U \otimes_A M) \iso \HH^{d-i}(A, M)$.
\end{theorem}
We actually only will require that $U$ be a projective right
$A$-module (which is implied if it is only left invertible, since then
the functor $U \otimes_A -: A\text{-mod} \to A\text{-mod}$ has a
quasi-inverse and hence preserves projectives).
\begin{remark}
  The above theorem is easy to prove, as we will show, but it is an
  extremely important observation. 
\end{remark}
\begin{proof}
  This is a direct computation. For the first statement, using that
  $\HH^\bullet(A,M) = H^i(\RHom_{A^e}(A, M))$ and similarly
  $\HH_\bullet(A, M) = A \otimes^{\operatorname{L}}_{A^e} M$,
\[
\RHom_{A^e}(A, M) \cong \RHom_{A^e}(A, A \otimes A)
\otimes^{\operatorname{L}}_{A^e} M \cong A[-d]
\otimes^{\operatorname{L}}_{A^e} M.
\]
The homology of the RHS identifies with $\HH_{d-\bullet}(A,M)$ (the
degrees were inverted here because $\HH_\bullet$ uses homological
grading and $\HH^\bullet$ uses cohomological grading).  For the second
statement, replacing $\RHom_{A^e}(A, A\otimes A)$ by $U[-d]$, we get
\[
\RHom^\bullet(A,M) \cong U[-d] \otimes^{\operatorname{L}}_{A^e} M
= A[-d] \otimes^{\operatorname{L}}_{A^e} (U \otimes_A^{\operatorname{L}} M),
\]
so that if $U$ is projective as a right $A$-module, hence
$U \otimes_A^{\operatorname{L}} M = U \otimes_A M$ is an ordinary $A$-bimodule,
the homology of the RHS is $\HH_{d-\bullet}(A, U \otimes_A M)$, as desired (and as remarked above, invertibility implies projectivity as a right $A$-module).
\end{proof}
Next, recall the HKR theorem: $\HH^\bullet(\cO(X)) \cong
\wedge^\bullet_{\cO(X)} \Vect(X)$ when $X$ is smooth affine. There is a
counterpart for Hochschild homology: $\HH_\bullet(\cO(X)) \cong
\Omega^\bullet(X)$, the algebraic de Rham differential forms.
Moreover, there is a \emph{homological} explanation of the de Rham
differential: this turns out to coincide with the \emph{Connes
  differential}, which is defined on the Hochschild homology of an
\emph{arbitrary} associative algebra.

In the presence of the Van den Bergh duality isomorphism
\eqref{e:vdb}, this differential $B$ yields a differential on the
Hochschild cohomology of a Calabi-Yau algebra,
\[
\Delta: \HH^\bullet(A) \to \HH^{\bullet-1}(A).
\]
This is often called the Batalin-Vilkovisky (BV) differential.
As pointed out in \cite[\S 2]{EGdelpezzo},
\begin{proposition}The infinitesimal deformations of a Calabi-Yau
  algebra A \emph{within the class of Calabi-Yau algebras} are
  parameterized by $\ker(\Delta): \HH^2(A) \to \HH^1(A)$.
\end{proposition}
This is not necessarily the right question to ask, however: if one
asks for the deformation space of $A$ \emph{together with its
  Calabi-Yau structure}, i.e., isomorphism $\HH^d(A, A\otimes A)\cong A$,
one obtains:
\begin{theorem}\cite{VdBV-CYdnch} The infinitesimal
  deformations of pairs $(A,\eta)$ where $A$ is a Calabi-Yau algebra
  and $\eta$ an $A$-bimodule quasi-isomorphism in the derived category
  $D(A^e)$,
\[
\eta: C^\bullet(A, A \otimes A) \iso A[-d],
\]
is given by the negative cyclic homology
$\operatorname{HC}^-_{d-2}(A)$. The obstruction to extending to second
order is given by the string bracket $[\gamma, \gamma] \in
\operatorname{HC}^-_{d-3}(A)$.
\end{theorem}
You should think of cyclic homology as a noncommutative analogue of de
Rham cohomology, i.e., the homology of $(\HH_\bullet(A), B)$ where $B$
is Connes' differential above (although this is not correct in
general, and moreover there are three flavors of cyclic homology:
ordinary, negative, and periodic; it is the periodic flavor that
coincides with the de Rham cohomology for algebras of functions on
smooth affine varieties).

\subsection{Algebras defined by a potential, and
Calabi-Yau algebras of dimension three}\label{ss:cy-pot}
It turns out that ``most'' Calabi-Yau algebras are presented by a
\emph{(super)potential} (as recalled in the introduction, a precise version
of the existence for general dimension is given in \cite{VdB-cyas}
when the algebra is complete; on the other hand counterexamples in the
general case are given in \cite{Dav-sam}).

In the case of dimension three, potentials have the following form:
\begin{definition}
  Let $V$ be a vector space. A \emph{potential} is an element $\Phi
  \in TV / [TV, TV]$.
\end{definition}
For brevity, from now on we write $TV_\cyc := TV / [TV, TV]$.
\begin{definition}
  Let $\xi \in \Der(TV)$ be a constant-coefficient vector field, e.g.,
  $\xi = \partial_i$. We define an action $\xi: TV_\cyc \to TV$ 
as
  follows: for cyclic words $[v_1 \cdots v_m]$ with $v_i \in V$,
\[
\xi[v_1 \cdots v_m] := \sum_{i=1}^m \xi(v_i) v_{i+1} \cdots v_m v_1
\cdots v_{i-1}.
\]
Then we extend this linearly to $TV_\cyc$.
\end{definition}
\begin{definition}\label{d:alg-pot}
  The algebra $A_\Phi$ defined by a potential $\Phi \in TV_\cyc$ is
\[
A_\Phi := TV / (\partial_1 \Phi, \ldots, \partial_n \Phi).
\]
\end{definition}
\begin{definition}
  A potential $\Phi$ is called a Calabi-Yau potential (of
  dimension three) if $A_\Phi$ is a Calabi-Yau algebra of dimension three.
\end{definition}
\begin{example} For $V = \bk^3$, $\bA^3$ is defined by the
  potential 
\begin{equation}\label{e:a3-pot}
\Phi = [xyz]-[xzy] \in TV_\cyc,
\end{equation}
since
\[
\partial_x \Phi = yz-zy, \quad \partial_y \Phi = zx-xz,
\quad \partial_z \Phi = xy-yx.
\]
\end{example}
\begin{example} The universal enveloping algebra $U \mfsl_2$ of a Lie
  algebra is defined by the potential
  \[ [efh]-[ehf]-\frac{1}{2}[h^2] -2[ef].
\]
\end{example}
\begin{example}\label{ex:skly-e6}
  The Sklyanin algebra with relations
\[
xy-tyx + cz^2 = 0, \quad yz-tzy+cx^2 = 0, \quad zx-xz+cy^2 = 0,
\]   
from Rogalski's chapter
is given by the potential
\[ [xyz] - t[xzy] + \frac{c}{3}[x^3+y^3+z^3].
\]
\end{example}
\begin{example} As mentioned above, NCCRs of Gorenstein singularities
  are Calabi-Yau; in Wemyss's chapter,
  comments are made throughout
  that the examples given can be presented by a potential.
  In fact, by \cite{VdB-cyas},
  any complete Calabi-Yau algebra of dimension three can be presented
  by a potential; this means that, for any of the examples in
  Wemyss's chapter,
  at least after completing at the ideal
  generated by all the arrows, the relations can be expressed by a
  potential. In most cases there, the singularity is either
  already complete (a quotient of $\bC[\![x_1,\ldots,x_n]\!]$), or else
  it is a cone (i.e., the relations are all homogeneous with respect
  to some assignment of each variable to a weight) and the
  potential is as well, hence defined without completing.
\end{example}
\subsection{Deformations of potentials and PBW theorems}\label{ss:cy-pot-def}
The first part of the following theorem was proved in the filtered case
in \cite{BerTai} and in the formal case in \cite{EGdelpezzo}. I have written
informal notes proving the converse (the second part). The theorem can also be
viewed as a consequence of the main result of \cite{VdB-cyas}.
\begin{theorem}
  Let $A_\Phi$ be a graded Calabi-Yau algebra defined by a
  (Calabi-Yau) potential $\Phi$.  Then for any filtered or formal
  deformation $\Phi + \Phi'$ of $\Phi$, the algebra $A_{\Phi+\Phi'}$
  is a filtered or formal Calabi-Yau deformation of $A_\Phi$.

Conversely, all filtered or formal deformations of $A_\Phi$ are
obtained in this way.
\end{theorem}
Now, let us return to the setting of $\cO(\bA^3) = A_\Phi$, which
is Calabi-Yau using the potential \eqref{e:a3-pot}.  Let $f \in \cO(\bA^3)$
be a hypersurface.  Then one obtains as above the Poisson structure
$\pi_f$ on $\bA^3$ for which $f$ is Poisson central.  As above, all
quantizations are given by Calabi-Yau deformations $\Phi + \Phi'$.  Moreover,
by Corollary \ref{c:hphh}, $f$ deforms to a central element, call it
$f_{\Phi'}$, of each such quantization.  Therefore, one obtains a
quantization of the hypersurface $(\cO(\bA^3)/(f), \pi_f)$, namely
$A_{\Phi+\Phi'} / (f_{\Phi'})$.

Next, what we would like to do is to extend this to \emph{graded}
deformations, in the case that $\Phi$ is homogeneous, i.e., cubic:
this will yield the ordinary three-dimensional Sklyanin algebras. More
generally, we will consider the \emph{quasihomogeneous} case, which
means that it is homogeneous if we assign each of $x$, $y$, and $z$
certain weights, which need not be equal. This will recover the
``weighted Sklyanin algebras.''  Note that producing actual graded
deformations is not an immediate consequence of the above theorem, and
they are not immediately provided by Kontsevich's theorem either.  In
fact, their existence is a special case (the nicest nontrivial one!)
of the following broad conjecture of Kontsevich:
\begin{conjecture}\cite[Conjecture 1]{Kon-dqav} \label{c:kont-conj}
  Suppose that $\pi$ is a quadratic Poisson bivector on $\bA^n$, i.e.,
  $\{-,-\}$ is homogeneous, and $\bk=\bC$.  Then, the star-product
  $\star_\hbar$ in the Kontsevich deformation quantization, up to
  a suitable formal gauge equivalence, actually
  converges for $\hbar$ in some complex neighborhood of zero,
  producing an actual deformation $(\cO(\bA^n), \star_\hbar)$
  parameterized by $|\hbar| < \epsilon$, for some $\epsilon > 0$.
\end{conjecture}
As we will see (and was already known, via more indirect
constructions), the conjecture holds for $\pi_f$ with $f \in
\cO(\bA^3)$ quasihomogeneous such that $Z(f)$ has an isolated
singularity at the origin.

\subsection{Etingof-Ginzburg's quantization of isolated
  quasihomogeneous surface singularities}\label{ss:eg}
In \cite{EGdelpezzo}, Etingof and Ginzburg constructed a (perhaps
universal) family of \emph{graded} quantizations of hypersurface
singularities in $\bA^3$. These essentially coincide with the Sklyanin
algebras associated to types $E_6, E_7$, and $E_8$ (the $E_6$ type is
the one of Example \ref{ex:skly-e6}).  Because these hypersurface
singularities also deform to affine surfaces $S \setminus E$ where $S$
is a (projective) del Pezzo surface and $E \subseteq S$ is an elliptic
curve, the family yields the universal family of quantizations of
these del Pezzo surfaces (in more detail, by \cite[Proposition
1.1.3]{EGdelpezzo}, the Rees algebra of $\cO(S \setminus E)$ is the
homogeneous coordinate ring of the anti-canonical embedding of $S$, so
the Rees algebras of filtered deformations of $\cO(S \setminus E)$ can be viewed
as deformations of the projective del Pezzo surface).  In this
subsection, we will stick to the hypersurface singularities and will
not mention del Pezzo surfaces any further.

Let $\bk$ have characteristic zero. Suppose that $x, y$, and $z$ are
assigned positive weights $a, b, c > 0$, not necessarily all one.  Let
$d := a+b+c$.  Suppose $f \in
\bk[x,y,z]$ is a polynomial which is weight-homogeneous of degree $m$.
Then the Poisson bracket $\pi_f$ \eqref{e:pif} has weight $m-d$. If we
want to deform in such a way as to get a \emph{graded}
quantization, we will need to have $m=d$, which is called the
\emph{elliptic} case.  For a filtered but not necessarily graded
quantization, we need only that $f$ be a sum of weight-homogeneous
monomials of degree $\leq d$.  The case where $f$ is
weight-homogeneous of degree strictly less than $d$ turns out to yield
all the \emph{du Val} (or Kleinian) singularities $\bA^2/\Gamma$ for
$\Gamma < \SL(2)$ finite, which we discussed in Remark \ref{r:McKay} and Theorem \ref{t:quiv}.

So, to look for graded quantizations, suppose $f$ is homogeneous of
degree $d$. We also need to require that $Z(f)$ has an isolated
singularity at the origin. In this case, it is well-known that the
possible choices of $f$, up to weight-graded automorphisms of
$\bk[x,y,z]$, fall into three possible families, each parameterized by
$\tau \in \bk$:
\[
\frac{1}{3}(x^3 + y^3 + z^3) + \tau \cdot xyz, \quad \frac{1}{4} x^4 +
\frac{1}{4} y^4 + \frac{1}{2} z^2 + \tau \cdot xyz, \quad \frac{1}{6}
x^6 + \frac{1}{3} y^3 + \frac{1}{2} z^2 + \tau \cdot xyz.
\]
Let $p$, $q$, and $r$ denote the exponents: so in the first case,
$p=q=r=3$, in the second case $p=q=4$ and $r=2$, and in the third
case, $p=6, q=3$, and $r=2$.  Note that $ap=bq=cr$; so in the first
case we could take $a=b=c=1$; in the second, $a=b=1, c=2$; in the
third, $a=1, b=2$, and $c=3$.  Let $TV_\cyc^{\leq m}$ and $TV_\cyc^m$
be the subspaces of $TV_\cyc$ of weighted degrees $\leq m$ and exactly
$m$, respectively. Similarly define $TV_\cyc^{< m} = TV_\cyc^{\leq
  m-1}$.  For a filtered algebra $A$, we similarly define $A^m,
A^{\leq m}$, and $A^{< m}$.

We define $\mu$ to be the Milnor number of the singularity of the
homogeneous $f$ above at the origin, i.e.,
\[
\mu = \frac{(a+b)(a+c)(b+c)}{abc} = p+q+r-1.
\]
Finally, we will consider more generally inhomogeneous polynomials
replacing $f$ above, of the form $P(x)+Q(y)+R(z)$, such that
$\gr(P(x)+Q(y)+R(z)) = f$ and $P(0)=Q(0)=R(0)=0$. Clearly such $f$
have $\mu$ parameters.

Similarly, we will consider potentials $\Phi_{P,Q,R}^{t,c}$ of the form
\begin{equation}
\Phi_{P,Q,R}^{t,c} := [xyz]-t[yxz] + c(P(x) + Q(y) + R(z)).
\end{equation}
If we set $t := e^{\tau \hbar}$, then working over $\bk[\![\hbar]\!]$, with
$c \in \bk$ and $P(x), Q(y), R(z)$ polynomials as above,
the algebra $A_{\Phi_{P,Q,R}^{t,c}}$ (recall Definition \ref{d:alg-pot}) is 
a deformation quantization of $\cO(\bA^3)$, equipped with the Poisson bracket
\[
\pi_f, \quad f = \tau \cdot xyz-c\bigl(P(x)+Q(y)+R(z)\bigr).
\]
Thus, for $\bk=\bC$, the graded algebras $A_{\Phi_{P,Q,R}^{t,c}}$
indeed give graded quantizations (for actual values of $\hbar$) of the
Poisson structures $\pi_f$ associated to all quasihomogeneous (when
$P$, $Q$, $R$ are quasihomogeneous), and more generally all
filtered polynomials $f$ of weighted degree $\leq d$. This
is actually true for arbitrary values of $a,b,c \geq 1$ with
$d=a+b+c$, without requiring that the homogeneous part of $f$
have an isolated singularity at zero.   This
confirms Kontsevich's conjecture \ref{c:kont-conj} in these cases,
with $\hbar \in \bC$ convergent everywhere, up to the fact that
$A_{\Phi_{P,Q,R}^{t,c}}$ need not a priori correspond
precisely to the Poisson structure $\hbar \cdot \pi_f$, but rather to
some formal Poisson deformation thereof.  

However, in this case, since the $f$ above are all possible
filtered polynomials of degree $\leq d$, we can conclude that,
for every $f_\hbar \in \hbar \cdot \cO(\bA^3)^{\leq d}[\![\hbar]\!]$,
letting $\star_{f_\hbar}$ be the Kontsevich quantization of
$\pi_{f_\hbar}$, there must exist $\Phi_{\hbar} \in TV_\cyc^{\leq
  d}[\![\hbar]\!]$ such that, via a continuous isomorphism which is the
identity modulo $\hbar$,
\begin{equation}\label{e:quants}
(\cO(\bA^3)[\![\hbar]\!], \star_{f_\hbar}) \cong A_{\Phi_{\hbar}}.
\end{equation}
\begin{question}
  What is the relation between the parameters $f_\hbar$ and
  $\Phi_{\hbar}$ above? In particular:
\begin{enumerate}
\item Does \eqref{e:quants} send star products that, up to a gauge
  equivalence, converge for $\hbar$ in some neighborhood of zero to
  potentials that, up to a continuous automorphism of
  $TV[\![\hbar]\!]$, also converge in some neighborhood of zero?
\item Is
  $A_{\Phi_{P,Q,R}^{t,c}}$ isomorphic (via a continuous $\bk[\![\hbar]\!]$-algebra isomorphism
  which is the identity modulo $\hbar$) to Kontsevich's
  quantization of $\hbar \cdot \pi_f$ (or some constant multiple), for
$f = \tau \cdot xyz - c\bigl(P(x)+Q(y)+R(z)\bigr)$ (and $t=e^{\tau \hbar}$)?
\end{enumerate}
\end{question}
As stated above, if the second question had a positive answer, this
would be enough to confirm Conjecture \ref{c:kont-conj} in this case.

Returning to the case that the degrees $(a,b,c)$ are in the set
$\{(1,1,1),(1,1,2),(1,2,3)\}$, i.e., that generic quasihomogeneous
polynomials of degree $d$ have an isolated singularity at the origin,
the main result of \cite{EGdelpezzo} shows that, 
for generic $A_{\Phi}$ with $\Phi$ filtered of degree $\leq d = a+b+c$,
the algebra $A_{\Phi}$ is in the family $A_{\Phi_{P,Q,R}^{t,c}}$
and the family provides a versal deformation; moreover, the center
is a polynomial algebra in a single generator, and quotienting by this
generator produces a quantization of the original isolated singularity, which
also restricts to a versal deformation generically:
\begin{theorem}\cite{EGdelpezzo} \begin{itemize}
  \item[(i)] Suppose that $\Phi \in TV_\cyc^d$ is a homogeneous Calabi-Yau
    potential of weighted degree $d=a+b+c$, where $|x|=a$, $|y|=b$, and $|z|=c$ (which can be arbitrary).
Then
    for any potential $\Phi' \in TV_\cyc^{< d}$ of degree strictly
    less than $d$, $\Phi + \Phi'$ is also a Calabi-Yau potential.  Moreover,
    the Hilbert series of the associated graded
 algebra of the Calabi-Yau algebra $A_{\Phi + \Phi'}$ is
\[
h(\gr(A_{\Phi+\Phi'});t) = \frac{1}{(1-t^a)(1-t^b)(1-t^c)}.
\]
\item[(ii)] There exists a nonscalar central element $\Psi \in
A_{\Phi + \Phi'}^{\leq d}$.
\item[(iii)] Now suppose that $(a,b,c) \in \{(1,1,1),(1,1,2),(1,2,3)\}$ as
above.  Then, if $\Phi$ is generic, then there exist parameters $P, Q, R, t$, $c$, as above, such that for all $\Phi'\in TV_\cyc^{< d}$,
\[
A_{\Phi+\Phi'} \cong A_{\Phi_{P,Q,R}^{t,c}}.
\]
Moreover, in this case, the center $Z(A_{\Phi+\Phi'}) = \bk[\Psi]$ is a polynomial
algebra in one variable, and the Poisson center $Z(\gr A_{\Phi +\Phi'}) = \bk[\gr \Psi]$.
\item[(iv)] Keep the assumption of (iii). The family
 $\{A_{\Phi_{P',Q',R'}^{t',c'}}\}$ restricts to a versal deformation of $A_{\Phi+\Phi'}$
in the formal neighborhood of $(P,Q,R,t,c)$, depending
on $\mu$ parameters. Moreover,
the family $\{A_{\Phi_{P',Q',R'}^{t',c'}} / \Psi_{P',Q',R'}^{t',c'}\}$ restricts
to a versal deformation of $A_{\Phi+\Phi'}/(\Psi)$ in the same formal neighborhood.
\end{itemize}
\end{theorem}
\begin{question}
  Note that we did \emph{not} say anything above about the family
  $A_{\Phi_{P,Q,R}^{t,c}} / (\Psi_{P,Q,R}^{t,c})$ restricting to a
  versal quantization of the original singularity $\cO(\bA^3) / (f)$
  for $f = \tau \cdot xyz - c\bigl(P(x)+Q(y)+R(z)\bigr))$.  More
  precisely, does the family of central reductions,
\[
A_{\Phi_{\hbar \cdot P_\hbar,\hbar \cdot Q_\hbar,\hbar \cdot R_\hbar}}^{t_\hbar,
\hbar \cdot c_\hbar} / (
\Psi_{\hbar \cdot P_\hbar,\hbar \cdot Q_\hbar,\hbar \cdot R_\hbar}^{t_\hbar,
\hbar \cdot c_\hbar}),
\]
produce a versal deformation quantization, for $P_\hbar(x), Q_\hbar(y), R_\hbar(z) \in \cO(\bA^3)[\![\hbar]\!]$ and $t_\hbar, c_\hbar \in \bk[\![\hbar]\!]$
such that $P=P_0, Q=Q_0, R=R_0, c=c_0$, and $t_\hbar \equiv e^{\tau \hbar} \pmod{\hbar^2}$?
\end{question}
In the case that $P, Q$, and $R$ are homogeneous, i.e., $f=P(x)+Q(y)+R(z)$,
we can explicitly write out the central elements $\Psi$ for
the first two possibilities for $f$ (\cite[(3.5.1)]{EGdelpezzo}).  For the
first one, $f=\frac{1}{3}(x^3+y^3+z^3) + \tau \cdot xyz$, we get
\begin{equation}
\Psi = c \cdot y^3 + \frac{t^3-c^3}{c^3+1}(yzx+c \cdot z^3) - t \cdot zyx.
\end{equation}
For the second one, $\frac{1}{4} x^4 + \frac{1}{4} y^4 + \frac{1}{2} z^2 + \tau \cdot xyz$, we get
\begin{equation}
\Psi = (t^2+1)xyxy - \frac{t^4+t^2+1}{t^2-c^4}(t \cdot xy^2x + c^2 \cdot y^4) + t \cdot y^2x^2.
\end{equation}
For the third case, the answer is too long, but you can look it up at

\centerline{\url{http://www-math.mit.edu/~etingof/delpezzocenter}.}  

\noindent Moreover,
there one can obtain the formulas for the central elements $\Psi$ in
the filtered cases as well.

\subsection{Exercises}

Exercise \ref{exer:hypsfc} from the notes.

\section{Solutions to exercises}

\begin{proof}[Solution to Exercise \ref{exer:weyl-diff}] 
  (a) Write $\Weyl_n := A/(R)$ where $A$ is the free algebra on
  $x_1,\ldots,x_n$ and $y_1, \ldots, y_n$, and $R$ is the span of the
  defining relations $x_i y_j - y_j x_i = \delta_{ij}$, $x_i x_j - x_j
  x_i$, and $y_i y_j - y_j y_i$. We then have a unique homomorphism
  $\phi: A \to \caD(\bA^n)$ by the assignment $\phi(x_i)=x_i$ and
  $\phi(y_i)=-\partial_i$.  We show that $\phi(R)=0$, and hence $\phi$
  factors through a homomorphism $\bar \phi: \Weyl_n \to \caD(\bA^n)$.
  This follows from a direct computation: to verify $x_i (-\partial_j)
  - (-\partial_j) x_i = \delta_{ij}$, we check that $x_i (-\partial_j)
  (f) - (-\partial_j)(x_if) = -(-\partial_j)(x_i) \cdot f =
  \delta_{ij} f$; then note that multiplication operators by $x_i$ and
  $x_j$ commute, as do partial derivative operators
  $\partial_i, \partial_j$.  It is clear that $\bar \phi$ is surjective, since
  it sends generators to generators.

  (b) We need to show that $\bar \phi$ is injective. We first claim
  that $\Weyl_n$ is spanned by the operators $f \cdot P$ where $f \in
  \bk[x_1,\ldots,x_n]$ and $P \in \bk[y_1, \ldots, y_n]$ are
  monomials. This is easy to see, because given any monomial in $x_i$
  and $y_j$, we can apply the relations to push the $y_j$ to the right
  and the $x_i$ to the left, resulting in a linear combination of
  elements of the desired form.  Next, we claim that the resulting
  elements $\bar \phi(f P)$ are linearly independent (again assuming
  $f$ and $P$ are monomials).  This implies not only that $\bar \phi$
  is an isomorphism, but in fact that these elements $fP$ form a basis
  of $\Weyl_n$ (note that, in Exercise \ref{exer:weyl-basis} below, we
  will actually see that this is a basis for arbitrary commutative
  rings $\bk$, and not merely for characteristic zero fields; but note
  that $\bar \phi$ will no longer be an isomorphism in this
  generality).

It remains to prove the final claim.  For a contradiction, suppose
that a nonzero linear combination $F := \sum_{f,P} \lambda_{f,P} \bar
\phi(fP)$ of such monomials is zero, with $\lambda_{f,P} \in \bk$. Let
$P$ be of maximal degree such that $\lambda_{f,P}\neq 0$.  Let $g \in
\bk[x_1,\ldots,x_n]$ be the monomial corresponding to $P$, i.e., such
that, viewing $g$ as a function in $n$ variables, $P =
g(\partial_1,\ldots,\partial_n)$. Then, applying the operator $F$ to
$g$, we get $F(g) = cf$ where $c \geq 1$ is a positive integer
(specifically, if $g = x_1^{r_1} \cdots x_n^{r_n}$, then $c = r_1!
r_2! \cdots r_n!$).  Since we assumed that $\bk$ had characteristic
zero, we have $cf \neq 0$, which is a contradiction, since $F$ was
assumed to be zero.  (Note that this proof works when $\bk$ is any
ring of characteristic zero, not necessarily a field.)
\end{proof}

\begin{proof}[Solution to Exercise \ref{exer:weylv}] It is clear from the definition
  that the map $\phi: \Weyl_n \to \Weyl(V)$ such that $\phi(x_i)=x_i$
  and $\phi(y_i)=y_i$ is a homomorphism. It is surjective since $V$ is
  spanned by the $x_i$ and $y_i$.  We only need to show it is
  injective.  The kernel is $(R)$ where $R$ is the span of the
  relations $xy - yx - (x,y)$, for $x,y \in V$.  We only need to show
  these are zero in $\Weyl_n$.  To do so, write $x$ and $y$ as a
  linear combination of the $x_i$ and $y_i$.  Then $xy-yx-(x,y)$
  decomposes as a linear combination of the defining relations of
  $\Weyl_n$ (where $x$ and $y$ are replaced by basis elements).
\end{proof}

\begin{proof}[Solution to Exercise \ref{exer:weyl-gr}] 
Following the hint, we showed
  above in the solution to Exercise \ref{exer:weyl-diff}.(b) that a basis
  for $\caD(\bA^n)$ is of the desired form.  Since the associated
  graded elements of this basis are clearly a basis for $\Sym V$,  it
  remains only to see that the degree of these elements matches the
  description: under the additive filtration, $x_i$ and $y_i$ are both
  in degree one, whereas under the geometric filtration, the $x_i$ are
  in degree zero and the $y_i$ are in degree one.
\end{proof}
\begin{proof}[Solution to Exercise \ref{exer:weyl-basis}]
Setting $\bk=\bZ$, the result
  actually follows from the solution given for Exercise
  \ref{exer:weyl-diff}.(b) above.  But then, applying $\otimes_{\bZ} \bk$
  gives the result for $\Weyl_n(\bk)$ defined over an arbitrary rings
  $\bk$, again by the same relations (since $\Weyl_n(\bk) =
  \Weyl_n(\bZ) \otimes_{\bZ} \bk$).  In the case that $\bk$ is a field
  and $V$ is a vector space, we then know that $\Weyl_n = \Weyl(V)$ by
  Exercise \ref{exer:weylv}.
\end{proof}
\begin{proof}[Solution to Exercise \ref{exer:def-reln}]
  Let $\phi: TV=\gr(TV) \to \gr A$ be the tautological surjection.  We
  show that $\phi(R)=0$, so that it descends to a homomorphism $\bar
  \phi: B \to \gr A$.  Indeed, $R=\gr(R)=\gr(E) \subseteq \gr(TV)=TV$,
  which implies that $\phi(R)=0$.
\end{proof}
\begin{proof}[Solution to Exercise \ref{exer:poissbr}]
   The skew-symmetry and Jacobi
  identities for $\{-,-\}$ follow immediately from those identities
  for the commutator $[a,b]=ab-ba$, by taking associated graded, using
  the identities $\{\gr_m a,\gr_n b\}=\gr_{m+n-d}[a,b]$ and $\{\gr_m
  a,\{\gr_n b,\gr_p c\}\} = \gr_{m+n+p-2d}[a,[b,c]]$, for $a \in
  A_{\leq m}, b \in A_{\leq n}$, and $c \in A_{\leq p}$ (which follow
  immediately).  The Leibniz rule is then equivalent to the statement
  that $\{\gr_m a,-\}$ is a derivation for all $a \in A_{\leq m}$ (and
  all $m \geq 0$). This follows because $[a,-]$ is a derivation
  already in $A$: $[a,bc] = [a,b]c + b[a,c]$; we then apply
  $\gr_{m+n+p-d}$, assuming $a \in A_{\leq m}, b \in A_{\leq n}$, and
  $c \in A_{\leq p}$.
\end{proof}
\begin{proof}[Solution to Exercise \ref{exer:zeropbr}] If $d$ is not maximal such that
  $\gr(A)$ is commutative, then $[A_{\leq m}, A_{\leq n}] \subseteq
  A_{\leq (m+n-d-1)}$ for all $m,n \geq 0$.  Thus $\gr_{m+n-d}[A_{\leq
    m}, A_{\leq n}] = 0$, and hence $\{\gr_m A, \gr_n A\}=0$.
\end{proof}
\begin{proof}[Solution to Exercise \ref{exer:pb-lie}]  
  For $x,y \in \mfg$, we have $\{\gr_1 x,
  \gr_1 y\}= \gr_1 (xy-yx)=\gr_1 \{x,y\}$ (denoting the Lie bracket on
  $\mfg$ also by $\{-,-\}$), which is the bracket of
  \eqref{e:sym-g}. Then, \eqref{e:sym-g} follows from the Leibniz
  identity (indeed, it is clear that the Leibniz identity uniquely
  determines the Poisson bracket from the Lie bracket on generators,
  so there is a unique formula for the LHS of \eqref{e:sym-g}, and it
  is easy to check this formula is the RHS).
\end{proof}
\begin{proof}[Solution to Exercise \ref{exer:weyl-symv}]
 By Exercise \ref{exer:weyl-gr}
  (solved above), we only need to identify the Poisson bracket on $\gr
  \Weyl(V)$ with the bracket on $\Sym V$. By the Leibniz rule, it
  suffices to check this on homogeneous generators. In the case of the
  additive filtration, $\gr V$ is homogeneous of degree one, and we
  have $uv-vw = (u,v)$ in $\Weyl(V)$, matching $\{u, v\}=(u,v)$ in
  $\Sym V$, as desired.  In the case of the geometric filtration, the
  same identity holds, requiring $u$ and $v$ to be homogeneous (of
  degree zero or one) in $V$.
\end{proof}
\begin{proof}[Solution to  Exercise \ref{exer:diffb}] 
  By definition, $\Diff$ is
  nonnegatively filtered.  We only need to show that its associated
  graded algebra is commutative, i.e., that $[\Diff_{\leq m}(B),
    \Diff_{\leq n}(B)] \subseteq \Diff_{\leq m+n-1}(B)$. We prove this
  inductively on the sum $m+n$.  Applying $\ad(b)$ for $b \in B$, we
  obtain by the Leibniz identity for commutators,
\begin{multline*}
\ad(b)([\Diff_{\leq m}(B), \Diff_{\leq n}(B)]) \subseteq
   [\ad(b)(\Diff_{\leq m}(B)), \Diff_{\leq n}(B)] + [\Diff_{\leq
       m}(B), \ad(b) \Diff_{\leq n}(B)] \\ \subseteq [\Diff_{\leq
       m-1}(B), \Diff_{\leq n}(B)] + [\Diff_{\leq m}(B), \Diff_{\leq
       n-1}(B)],
\end{multline*}
and both terms on the RHS lie in $\Diff_{m+n-1}(B)$ by the inductive
hypothesis.
\end{proof}
\begin{proof}[Solution to Exercise \ref{exer:quiver1}] 
(i) This is a matter of checking the definition.  Given a representation of $Q$, we let every path in $\bk Q$ act by the corresponding composite of linear transformations.  Given a representation of $\bk Q$, we restrict the representation to the subset $Q \subseteq \bk Q$ to obtain a representation of $Q$.

(ii)  More generally, if we fix $(V_i)$, then every representation of
the form $(\rho, (V_i))$ is clearly given by the choices of linear operators
$\rho(a): V_{a_t} \to V_{a_h}$ for all $a \in Q_1$.
\end{proof}

\begin{proof}[Solution to Exercise \ref{exer:quiver2}] 
If there is only one vertex, then paths
in the quiver are the same as words in the arrows, which form a basis
for the tensor algebra over $\bk Q_1$.
\end{proof}

\begin{proof}[Solution to Exercise \ref{exer:quiver3}] 
  For any vector space, $T^*V \cong V \oplus V^*$, canonically.  Thus
  the assertion follows from Exercise \ref{exer:quiver1}.(ii).
\end{proof}

\begin{proof}[Solution to Exercise \ref{exer:quiver4}] 
(i) This follows because the relations are quadratic (hence homogeneous).  

(ii) The associated graded space of the relations for $\Pi_\lambda(Q)$
identifies with the span of the relations for $\Pi_0(Q)$, so by Exercise
\ref{exer:def-reln} we have a surjection $\gr \Pi_\lambda(Q) \to
\Pi_0(Q)$. 

(iii) For the quiver $Q$ with two vertices and one arrow, from one
vertex to the other, we can consider $\lambda=(0,1)$.  Let $a, a^*$ be
the two arrows in $\bar Q$, with $\lambda_{a_t}=0$ and
$\lambda_{a_h}=1$.  Then $\Pi_\lambda(Q) = \bk \bar Q / (a^* a, a a^*
- a_h)$. Thus $a^* (a a^*) = a^* a_h = a^*$ in $\Pi_\lambda(Q)$, but
also $(a^* a) a^* = 0$. So $a^* = 0$ in $\Pi_\lambda(Q)$, hence also
in $\gr \Pi_\lambda(Q)$, but this is not true in $\Pi_0(Q)$.
\end{proof}

\begin{proof}[Solution to Exercise \ref{exer:quiver5}] 
By definition, $\Rep_d(\Pi_\lambda(Q))$ is
the subspace of the space of all representations of $\bar Q$ satisfying
the relation $\mu(\rho)=\lambda \cdot \Id$.
\end{proof}

\begin{proof}[Solution to Execise \ref{exer:nonflat-def}]
  The first statement follows because the monomials given are those in
  which $xy$ and $yx$ (which are zero) do not appear. The next
  statement, that $A_{a,b}=\{0\}$ for $a \neq b$, follows because in
  this case $xyx=ax=bx$ implies $x = 0$, and then $0=xy=a$ and
  $0=yx=b$ shows that, since one of $a$ and $b$ is nonzero, we obtain
  the relation $1=0$, i.e., $A_{a,b} = \{0\}$.

  The final statement, that $A_{a,a} \cong \bk[x,x^{-1}]$, follows
  because in this case the relations are equivalent to $y=ax^{-1}$.
  It is clear that a basis of $\bk[x,x^{-1}]$ is given by monomials in
  either $x$ or $x^{-1}=a^{-1} y$, but not both, proving the next
  statement.  Since the basis is the same as for the case $a=0$, we
  obtain that $A_{a,b}$ is flat along the diagonal, as desired.
\end{proof}
\begin{proof}[Solution to Exercise \ref{exer:weyl-usl2}]
(a) One way to correct this, that we will use below, is $e=-\frac{1}{2}D^2, f=\frac{1}{2}x^2$,
and $h = [e,f] = -\frac{1}{2} (xD + Dx)$.

(b) For this, we first compute that, in a highest weight
  representation of highest weight $m$, the element $C$ acts by
  $\frac{1}{2}m^2+m$.  Indeed, since $C$ is central, it suffices to
  compute this on the highest weight vector $v$, where
  $(ef+fe+\frac{1}{2} h^2)(v) = (h+\frac{1}{2} h^2)(v) = (\frac{1}{2}
  m^2 + m)v$.  Thus we only have to solve $\frac{1}{2} m^2 + m =
  -\frac{3}{8}$.  We obtain the solutions $-\frac{1}{2},
  -\frac{3}{2}$, as desired.

(c) It is clear that $\{1, x^2, x^4, \ldots\}$ is an eigenbasis under
  the operator $h = -\frac{1}{2} (x D + Dx)$, with eigenvalues
  $-\frac{1}{2}, -\frac{5}{2}, -\frac{9}{2}, \ldots$.  Moreover, it
  is clearly generated by the highest weight vector $1$. So this is a
  highest weight representation of highest weight $-\frac{1}{2}$.

(d)  We know that all
  finite-dimensional irreducible representations of $\mfsl_2$ are
  highest weight representations, with nonnegative highest
  weights. Part (b) above shows that $(U\mfsl_2)^\eta$ does not admit
  such highest weight representations, hence it admits no
  finite-dimensional irreducible representations.  By the isomorphism
  of (a), neither does $\Weyl(V)^G$.  For the final statement, note
  that any finite-dimensional representation has an irreducible
  quotient, so the existence of a finite-dimensional representation
  implies the existence of a finite-dimensional irreducible
  representation.
\end{proof}

\begin{proof}[Solution to Exercise \ref{exer:flat-hilb}]
  The first statement follows because a surjection of
  finite-dimensional spaces is an isomorphism if and only if the
  dimensions are equal.  The second statement follows for the same
  reason, applied to each weight space individually.
\end{proof}

\begin{proof}[Solution to Exercise \ref{exer:grob}]
(a) Notice that the defining relations form a Gr\"obner basis:
  $xy-yx-z, yz-zy-x$, and $zx-xz-y$.  Therefore, using the
  lexicographical ordering on monomials with $x < y < z$, we can
  uniquely reduce any element to a linear combination of monomials of
  the form $x^a y^b z^c$, just as for $\bk[x,y,z]$.  (That is, since
  we have a Gr\"obner basis for a filtered algebra $A$ whose
  associated graded set is a Gr\"obner basis for the associated graded
  algebra $B = \gr(A)$, we conclude that the canonical surjection $B
  \to \gr(A)$ of Exercise \ref{exer:def-reln} is an isomorphism.)

  It is clear that this algebra is the enveloping algebra of the
  three-dimensional Lie algebra with basis $x,y,z$ satisfying
  $[x,y]=z, [y,z]=x$, and $[z,x]=y$.  This must be isomorphic to
  $\mfsl_2$ since the latter is the unique three-dimensional Lie
  algebra $\mfg$ which is semisimple, i.e., satisfying
  $[\mfg,\mfg]=\mfg$.  Alternatively, we can explicitly write an
  isomorphism with $\mfsl_2$, by the assignment $x=\frac{1}{2}(e-f),
  y=\frac{1}{2}(e+f)$, and $z=\frac{1}{2}h$.

  (b) For example, one can get a deformation which is not flat by
  setting $[x,y]=x, [y,z]=0$, and $[z,x]=y$.  This does not satisfy
  the Jacobi identity: $[x,[y,z]]+[y,[z,x]]+[z,[x,y]]=0+0+[z,x]=y$.
  Thus it will not be flat: the element $y$ is generated by these
  relations, and hence so is $x$ since $[x,y]=x$.  Thus the quotient
  by these relations is $\bk[z]$.

  (c) A Gr\"obner basis is again given by the defining relations,
  whose associated graded relations are the defining relations for the
  skew product $B:=\Sym(\bk^2) \rtimes \bZ/2 = \bk[x,y] \rtimes
  \bZ/2$, independently of $\lambda$.  Thus, letting $A$ denote the
  Cherednik algebra we define here, we get an isomorphism $B \to
  \gr(A)$, which says (by our definition) that $A$ is a flat
  deformation of $B$ for all $\lambda$.

  (d) For the associated graded relations, we still get a semidirect
  product $\bk[x,y] \rtimes \bZ/2$, where now the action of $\bZ/2$ is
  by $\pm \Id$ on $x$, but is trivial on $y$.  But for the Cherednik
  algebra, we will get the relation
\[
0 = [z,1]=[z,[xy-yx]] = [z,x]y + x[z,y]-[z,y]x - y[z,x]
= 2z(xy-yx),
\]
and together with the relation $xy-yx=1$ we get $z=0$. Since $z^2=1$
we get $1=0$, so that the algebra defined by these relations is the
zero algebra.
\end{proof}

\begin{proof}[Solution to Exercise \ref{exer:weyl-simple}]
We identify $\Weyl(V)$ with $\Weyl_n$ for $\dim V = 2n$.
To see that $\Weyl_n \rtimes \bZ/2$ is simple, first note that
  if you have an ideal generated by an element of the form $f \in
  \Weyl_n$, then it is the unit ideal, because $\Weyl_n$ itself is
  simple (this is easy to verify: take commutators of $f$ sequentially
  with generators $x_i$ or $y_i$ until we get a nonzero element of
  $\bk$).  Writing $\bZ/2 = \{1,\sigma\}$, we also have $(f\sigma) =
  (f\sigma)\sigma = (f) = (1)$ for $\sigma \in \bZ/2$.  Now take an
  ideal of the form $(a+b\sigma)$ for nonzero $a,b \in \Weyl_n$.  Again by
  taking commutators with generators, we obtain that there is an
  element of the form $(1+c\sigma)$ for some $c \in \Weyl_n$. If
 $c = 0$, we are done. Otherwise,
  $[x, 1+c\sigma] = (xc + cx)\sigma$, which cannot be zero since
  $\gr(xc+cx)=\gr(xc) = \gr(x)\gr(c) \neq 0$.  We then get from above
  that $((xc+cx)\sigma)=(1)$ and hence our ideal is the unit ideal.

  For $\Weyl_n^{\bZ/2}$, we can also do an explicit computation;
  alternatively, we can follow the hint.  Namely, for $f \in
  \Weyl_n^{\bZ/2}$, note that $f=efe$.  Thus,
\begin{multline*}
\Weyl_n^{\bZ/2} \cdot f \cdot \Weyl_n^{\bZ/2} = e (\Weyl_n \rtimes
\bZ/2)(efe)(\Weyl_n \rtimes \bZ/2)e \\ = e (\Weyl_n \rtimes
\bZ/2)(f)(\Weyl_n \rtimes \bZ/2)e = e (\Weyl_n \rtimes \bZ/2) e =
\Weyl_n^{\bZ/2}. \qedhere
\end{multline*}
\end{proof}

\begin{proof}[Solution to Exercise \ref{exer:hoch-res}] Following the
  hint, we see that if $\gr(C_\bullet)$ is exact, then so must be
  $C_\bullet$.  Thus, in the situation at hand, since $\gr(Q_\bullet)
  \to \gr(M)$ is exact, so is $Q_\bullet \to M$.
\end{proof}

\begin{proof}[Solution to Exercise \ref{exer:aug}](a) Given a
  codimension-one ideal, we have the algebra morphism $B \to
  B/B_+=\bk$, and given the algebra morphism, we can take its kernel,
  which has codimension one.

(b) It is clear that $TV/(TV)_+ = \bk$, so by (a) this is an augmentation.

(c) Let $V \subseteq B_+$ be any generating subspace (e.g., simply
  $V=B_+$); we then have a canonical surjection of augmented algebras,
  $TV \to B$. Since this is compatible with the augmentation, the
  kernel must be in the augmentation ideal, $(TV)_+$, so we can let
  $R$ be this kernel (or any generating subspace of the kernel).

(d) We follow the hint.  The last two differentials become zero by construction, since they involve multiplying $V$ on $\bk$. Then the first homology is $V$.



(e) The second assertion (about the case that $R$ is spanned by
homogeneous elements) follows because we can assume that $R$ has
minimal Hilbert series among subspaces spanned by homogeneous elements
which generate the same ideal, call it $J$.  This is equivalent to the
given condition $R \cap (RV+VR) = \{0\}$, for $R$ spanned by
homogeneous elements (hence, $R$ a graded subspace of $TV$), which we
can see because, for all $m \geq 2$, $R_m$ is a complementary subspace
in $J_m$ to $J_m \cap (R_{\leq (m-1)})$.

We now consider the assertion $\Tor_A(\bk,\bk) \cong R$.
For the complex in the hint to be well-defined, we need to show that
the multiplication map $TV \otimes R \to TV \cdot R$ is injective
(hence an isomorphism). To see this, for a contradiction, suppose that
the kernel, call it $K$, is nonzero, and that $m \geq 0$ is the least
nonnegative integer such that the projection of the kernel to $T^m V
\otimes R$ is nonzero.  By the assumption $(RV+VR) \cap R = \{0\}$ (in
fact we only need $(TV)_+ R \cap R = \{0\}$), we must have $m \neq
0$. Take an element $f \in K$ projecting to a nonzero element of $T^m
V \otimes R$.  Let $(v_i)$ be a basis of $V$ and write $f = \sum_i v_i
f_i$ for $f_i \in TV \otimes R$.  Then we clearly have $f_i \in K$ as
well for all $i$; but some $f_i$ must have a nonzero projection to
$T^{m-1} V \otimes R$, a contradiction.

Let $K$ be the kernel of the map $B \otimes R \to B \otimes V$, and
consider its inverse image $\tilde K \subseteq TV \otimes R \subseteq
TV$.  Then it is clear that $\tilde K = (TV \cdot R) \cap (TV \cdot R
\cdot TV_+)$.  Thus we can let $S = \tilde K$ (or any left $TV$-module
generating subspace thereof), and we obtain the extension of the
resolution given in the hint.  Applying $\otimes_B \bk$, and using the
assumption $R \cap (RV+VR)=\{0\}$, we get a complex $S \to R \to V
\onto \bk$. We claim that the maps are zero. The map $R \to V$ is zero
by the assumption $R \subseteq ((TV)_+)^2$. The final step, showing
that $S \to R$ is zero, uses the full strength of the assumption $R
\cap (RV+VR)=\{0\}$.  Namely, we can assume $S = \tilde K = (TV \cdot
R) \cap (TV \cdot R \cdot TV_+)$.  Then the kernel of $S \to R$ is
$(TV_+ \cdot R) \cap (TV \cdot R \cdot TV_+)$.  This is all of $S$,
though, because if $\sum_i f_i r_i = \sum_i g_i r_i' h_i'$ for
$f_i,g_i \in TV$, $h_i \in TV_+$, and $r_i, r_i' \in R$, then we can
write $f_i = \lambda_i + f_i^+$ for $f_i^+ \in TV_+$, and we then have
$\sum_i \lambda_i r_i = -\sum_j f_j^+ r_j + \sum_i (g_i r_i' h_i') \in
R \cap (RV+VR)$, which must be zero.  Hence $f_i = f_i^+$, and $(TV_+
\cdot R) \cap (TV \cdot R \cdot TV_+) = (TV \cdot R) \cap (TV \cdot R
\cdot TV_+)$, so the map $S \to R$ is zero, as desired.
\end{proof}

\begin{proof}[Solution to Exercise \ref{exer:quadratic}]
  This follows from the previous exercise.  Namely, in the graded
  case, we can certainly assume $B = TV/(R)$ where $R$ is spanned by
  homogeneous elements in degrees $\geq 2$. It follows from part (d)
  that $V \cong \Tor_1(\bk,\bk)$. If $R$ is minimal, then it follows
  from part (e) that $R \cong \Tor_2(\bk, \bk)$. Hence, if
  $\Tor_2(\bk,\bk)$ is concentrated in degree two, $B$ is quadratic by
  taking minimal $R$. The converse is immediate from part (e) above.
\end{proof}

\begin{proof}[Solution to Exercise \ref{exer:koszul}]
  If we apply $\otimes_B \bk$ to
  \eqref{e:k}, we immediately get that (2) implies (1).  We
  show that (1) implies (2). Take a minimal graded
  projective resolution as in \eqref{e:k}, where by minimal we mean
  that the Hilbert series of the $V_i$ are minimal (which uniquely
  determines the resolution, inductively constructing $B \otimes V_i
  \to B \otimes V_{i-1}$ so that $h(V_i;t)$ is minimal). Then $V_i$ is
  in degrees $\geq i$ for all $i$. We show by induction that $V_i$ is in
  degree exactly $i$ for all $i$.  Indeed, if not, and $m$ is minimal
  such that $V_m$ is not only in degree $m$, then applying $\otimes_B \bk$
  to the minimal resolution, we obtain that $\Tor_m(\bk,\bk) \cong V_m$ is not
  concentrated in degree $m$, a contradiction.  
\end{proof}

\begin{proof}[Solution to Exercise \ref{exer:kdp-filt}]
(a) Taking associated graded, $\gr_3(E \otimes V \cap (V+\bk) \otimes E)
  \subseteq \gr_3(E \otimes V) \cap \gr_3((V+\bk) \otimes E) = R\otimes V
  \cap V \otimes R$.

  (b) The fact that the sequences are complexes is straightforward,
  and the exactness at $B \otimes V$, $B$, $A\otimes V$, and $A$ was
  already observed in Exercise \ref{exer:aug}.(d). Now, restrict the
  first complex (which is a complex of graded modules) to degrees
  $\leq 3$. Then, its exactness follows from Exercise
  \ref{exer:aug}.(e), since in this case the degree $\leq 3$ part of
  $(R)V \cap (TV)_+ \cdot R$ is just $RV \cap VR$, which is $S$ by
  definition.

  (c) Note that the maps $A \otimes T \to A \otimes E$ and $B \otimes
  S \to B \otimes R$ are injective in degrees $\leq 3$.  Restricting
  to degrees $\leq 3$, the second complex deforms the first except
  possibly at $A \otimes T$.  Moreover, in this situation, the
  associated graded of the kernel of $A \otimes E \to A \otimes V$
  must be $S$ in degree three; since $E$ deforms $R$, this implies
  that the kernel lies in $(V \oplus \bk) \otimes E$, and hence must
  be $(V+k)E \cap EV$ (here the parentheses do \emph{not} denote an
  ideal).


(d) The solution is outlined in the problem sheet.
\end{proof}


\begin{proof}[Solution to Exercise \ref{exer:cont-fd}]
  Here we note that, if $B$ is finite-dimensional, then
  $B[\![\hbar]\!] = B \otimes_{\bk} \bk[\![\hbar]\!]$, the ordinary
  tensor product (allowing only finite linear combinations).
  Explicitly, if $b_1,\ldots,b_n$ is a basis of $B$, then
  $B[\![\hbar]\!]$ is a free $\bk[\![\hbar]\!]$-module with basis
  $b_1,\ldots,b_n$.  Then, \eqref{e:star-cont} follows from the fact
  that $bf \star b'g = (b \star b') fg$ for $b,b' \in B$ and $f,g \in
  \bk[\![\hbar]\!]$.

  On the other hand, if $B$ is infinite-dimensional, then this
  argument does not apply: knowing $b \star b'$ only will determine
  $(b_1 f_1 + \cdots + b_m f_m) \star (c_1 g_1 + \cdots + c_n g_n)$ by
  $\bk[\![\hbar]\!]$-linearity, for $b_1,\ldots,b_m,c_1,\ldots,c_n \in
  B$ and $f_1,\ldots,f_m,g_1,\ldots,g_n \in \bk[\![\hbar]\!]$.  So
  multiplying two series with coefficients linearly independent in $B$
  will not be determined by $\bk[\![\hbar]\!]$-linearity from the star
  product on $B$, i.e., \eqref{e:star-cont} need not hold.  (We will
  not attempt to construct a counterexample, however.)
\end{proof}
\begin{proof}[Solution to Exercise \ref{exer:moyal}]
(a) This follows immediately from the formula: in particular,
  $x_i \star y_j = x_i y_j + \frac{1}{2} \hbar \delta_{ij}$, $y_j
  \star x_i = x_i y_j - \frac{1}{2} \hbar \delta_{ij}$, $x_i \star x_j
  = x_i x_j$, and $y_i \star y_j = y_i y_j$. Also, $f \star z_i = fz_i
  = z_i \star f$ for all $f$.

  (b) This follows because $e^{\frac{1}{2} \hbar \pi} = \sum_{n \geq
    0} \frac{1}{n!} \frac{1}{2^n} \hbar^n \pi^n$, and $\pi^n$
  decreases degree by $2n$, so the sum evaluated on any element $f
  \otimes g$ is finite.

  (c) Since we have a star product defined over $\bk[\hbar]$, i.e.,
  $(\cO(X)[\hbar],\star)$, we can quotient by the ideal $(\hbar-1)$
  (which is not the unit ideal) and get back a star product which is
  obtained from the above by setting $\hbar=1$; in particular it is
  clearly a filtered deformation of the undeformed product.

  (d) By part (a), the given map is a homomorphism, and it is clearly
  surjective.  To test injectivity, it suffices to show that the
  associated graded morphism is injective, but the latter morphism is
  the identity.  Uniqueness follows because the $x_i, y_i$, and $z_i$
  generate the source (they also generate the target, of course).

  (e) If we show that the map, call it $\Phi$, is an algebra morphism,
  then it will automatically invert the morphism of (d), since it
  sends the generators $x_i, y_i$, and $z_i$ to the generators $x_i,
  y_i$, and $z_i$.  We have to show that, for generators $f_i,g_j \in
  \{x_1, \ldots, x_m, y_1, \ldots, y_m, z_1, \ldots, z_{n-2m}\}$,
  that, with multiplication taken in the Weyl algebra,
\begin{equation}\label{e:moyal-id}
\Phi(f_1\cdots f_n)\Phi(g_1\cdots g_p) = 
\Phi(f_1 \cdots f_n\star g_1 \cdots g_p).
\end{equation}
The LHS can be expanded as 
\begin{equation}\label{e:moyal-id-lhs}
\frac{1}{n! p!} \sum_{\sigma \in S_n, \tau \in S_p}
f_{\sigma(1)} \cdots f_{\sigma(n)} g_{\tau(1)} \cdots g_{\tau(p)}.
\end{equation}
Now we can attempt to symmetrize this by applying the relations of the
Weyl algebra.  More precisely, recall that an $n,p$-shuffle is a
permutation $\theta \in S_{n+p}$ such that $\theta(1) < \cdots <
\theta(n)$ and $\theta(n+1) < \cdots < \theta(n+p)$.  Let
$\text{Sh}_{n,p} \subseteq S_{n+p}$ be the set of all such shuffles,
which has size $\frac{(n+p)!}{n!p!}$. For each $\theta \in
\text{Sh}_{n,p}$, and each summand $f_{\sigma(1)} \cdots f_{\sigma(n)}
g_{\tau(1)} \cdots g_{\tau(p)}$ of the above, we can rearrange the
terms according to the shuffle $\theta$, by moving first
$f_{\sigma(n)}$ to its proper place, then $f_{\sigma(n-1)}$, etc.;
each time we move an $f_i$ past a $g_j$, we apply the relation $f_i
g_j = g_j f_i + [g_j,f_i]$. Since the $g_j$ and $f_i$ are generators
of the Weyl algebra, $[g_j, f_i] \in \bk[\![\hbar]\!]$.  Doing this,
we obtain the following identity. Define the following index set we
will use for the summation that results:
\[
\text{Pairs}_{n,p}:=\{(I,J,\iota) \mid I \subseteq
\{1,\ldots,n\}, J \subseteq \{1,\ldots,p\}, \iota: I \to J \text{ bijective}\},
\]
which is the set of partial pairings between $\{1,\ldots,n\}$ and
  $\{1,\ldots,p\}$. Define the subset
\[
\text{Pairs}_{\theta} =\{(I,J,\iota) \in \text{Pairs}_{n,p} \mid
\iota(i) < \theta(i), \forall i \in I\}.
\]
For each $1 \leq i \leq
n$, let $[f_i]$ denote $f_i$ if $i \notin I$, and otherwise let it
denote $[f_i, g_{\iota(i)}]=\{f_i, g_{\iota(i)}\}$.  Let $\{g_j\}$
denote $g_j$ if $j \notin J$ and $1$ otherwise.  We then get
\[
f_1 \cdots f_n g_1 \cdots g_p 
= \sum_{(I,J,\iota) \in \text{Pairs}_{\theta}} 
 \theta\bigl(\prod_{i=1}^n [f_{i}] \prod_{j=1}^n \{g_{j}\}\bigr),
\]
where the $\theta$ rearranges the following $n+p$ terms according to
the permutation.  Call the RHS $R_\theta$.  We can then write the LHS
as $\frac{n!p!}{(n+p)!}\sum_{\theta \in \text{Sh}_{n,p}} R_\theta$, a
symmetrization.  Applying this to \eqref{e:moyal-id-lhs}, one can check
that we get the
RHS of \eqref{e:moyal-id} identically. Namely, both yield the following
expansion:
\[
\sum_{(I,J,\iota) \in \text{Pairs}_{n,p}} 
 \frac{1}{2^{|I}} 
\symm \bigl(\prod_{i=1}^n [f_{i}] \prod_{j=1}^n \{g_{j}\}\bigr),
\]
where now we take the product in the symmetric algebra and apply
$\symm: \Sym V \to \Weyl(V)$.  This is so because, by symmetry, both
sides have to be a sum of the above form but possibly with
coefficients $c_{(I,J,\iota)}$ replacing $\frac{1}{2^{|I|}}$; the LHS
of \eqref{e:moyal-id} has these coefficients being the probability
that, in a random ordering of $f_1,\ldots,f_n,g_1,\ldots,g_p$, we have
$f_i$ occurring before $g_{\iota(i)}$ (which is $\frac{1}{2^{|I|}}$),
and the RHS is $\frac{1}{2^{|I|}}$ by definition. See also \cite[\S
  4.3]{GSmoy} for a similar proof in the context of a Moyal product on
  algebras defined from quivers.
\end{proof}
\begin{proof}[Solution to Exercise \ref{exer:moyal-weyl-unique}]
  Since $\pi^2$ annihilates $v \otimes w$ when $v,w$ are generators,
  it is immediate that any quantization will satisfy the relations of
  part (a).  Therefore the map given is an isomorphism.  The fact that
  it is the identity modulo $\hbar$ is equivalent to the statement
  that the associated graded morphism is the identity endomorphism of
  $\cO(\bA^n)[\![\hbar]\!]$, which is immediate from the definition.
\end{proof}
\begin{proof}[Solution to Exercise \ref{exer:cont-fd2}]
  As above, if $A$ is finite-dimensional, then $A[\![\hbar]\!] = A
  \otimes_\bk \bk[\![\hbar]\!]$, so the space of
  $\bk[\![\hbar]\!]$-linear endomorphisms of $A$ is identified with
  $\End_{\bk}(A) \otimes \bk[\![\hbar]\!]$, and every element therein
  is continuous.  Note that, if $A$ is infinite-dimensional, then a
  $\bk[\![\hbar]\!]$-linear automorphism need not be continuous.
\end{proof}

\begin{proof}[Solution to Exercise \ref{exer:koszul-complexes}]
 (a) Here is a bimodule resolution of $\Sym V$:\footnote{It
    actually can be directly obtained from the given resolution of
    $\bk$,
    if one notices that $\Sym V$ is a Hopf algebra: to go from the
    above resolution to the bimodule one, one applies the functor of
    induction from $\Sym V$ to $(\Sym V)^{\otimes 2}$ (using the Hopf
    algebra structure); for the opposite direction, one applies the
    functor $\otimes_{\Sym V} \bk$ (which exists for arbitrary
    augmented algebras).}
\begin{gather}
  0 \to \Sym V \otimes \wedge^{\dim V} V \otimes \Sym V \to \Sym V
  \otimes \wedge^{\dim V - 1} V \otimes \Sym V \to \cdots \notag \\
  \to \Sym V \otimes
  V \otimes \Sym V \to \Sym V \otimes \Sym V \onto \Sym V, \\
  f \otimes (v_1 \wedge \cdots \wedge v_i) \otimes g \mapsto
  \sum_{j=1}^i (-1)^{j-1} (fv_j) \otimes (v_1 \wedge \cdots \hat v_j
  \cdots \wedge v_i) \otimes g \notag \\ +(-1)^{j} f \otimes (v_1
  \wedge \cdots \hat v_j \cdots \wedge v_i) \otimes (v_j g).
\end{gather}
Cutting off the $\Sym V$ term and applying $\Hom_{(\Sym V)^e} (-,\Sym
V \otimes \Sym V)$, we get the same complex, up to the isomorphism
$\wedge^{\dim V - i} V^* \cong \wedge^{i} V$, which comes from
contracting with a nonzero element of $\wedge^{\dim V} V$.  Thus, we
compute that $\HH^\bullet(\Sym V, \Sym V \otimes \Sym V) = \Sym
V[-\dim V]$, as desired.

(b) The latter complex deforms to give a resolution of $\Weyl(V)$,
given by the same formula as above:
\begin{gather}
  0 \to \Weyl(V) \otimes \wedge^{\dim V} V \to \Weyl(V)\otimes
  \wedge^{\dim V - 1}
  V \to \cdots \notag \\ \to \Weyl(V) \otimes \Weyl(V) \onto \Weyl(V), \\
  f \otimes (v_1 \wedge \cdots \wedge v_i) \otimes g \mapsto
  \sum_{j=1}^i (-1)^{j-1} (fv_j) \otimes (v_1 \wedge \cdots \hat v_j
  \cdots \wedge v_i) \otimes g +\notag\\(-1)^{j} f \otimes (v_1 \wedge
  \cdots \hat v_j \cdots \wedge v_i) \otimes (v_j g).
\end{gather}
As pointed out, to see it is a resolution, we only need to verify it
is a complex, which is straightforward. Now, if we cut off the last
term $\Weyl(V)$ and apply $\Hom_{\Weyl(V)^e}(-,\Weyl(V))$, just as
before we get the same complex as before applying this functor, so the
homology will again be concentrated in degree $d$:
$HH^\bullet(\Weyl(V), \Weyl(V) \otimes \Weyl(V)) = \Weyl(V)[-\dim V]$,
as desired.

(c) The fact that this forms a complex, call it $P_\bullet$ (cutting
off the $U\mfg$, setting $P_0=U \mfg \otimes U \mfg$), is a
straightforward verification.  Then, setting $Q^\bullet :=
\Hom_{U\mfg^e}(P_\bullet, U\mfg \otimes U\mfg)$, we may not any longer
get an isomorphic complex (this is related to the fact that $U\mfg$
is not Calabi-Yau in general, as we will see in the next exercise
sheet), but the associated graded
complex is still the resolution of $\Sym V$, and hence it is still a
deformation of the resolution of $\Sym V$.  Thus, by the same argument
as before, $Q^\bullet$ is a resolution of its $\dim V$-th
cohomology, $\HH^{\dim V}(U\mfg, U\mfg \otimes U\mfg)$, which is a
filtered deformation of the bimodule $\Sym V$ (i.e., it is a filtered
$U\mfg$-bimodule whose associated graded $\Sym V$-bimodule is $\Sym
V$). 


(d) The first paragraph is straightforward to verify explicitly
following the details given.

For the final assertion, the first isomorphism is immediate from the
first paragraph.  For the second, we can write $U\mfg^{ad} = Z(U\mfg)
\oplus W$ where the center, $Z(U\mfg)$, is the sum of all trivial
subrepresentations, and $W$ is the sum of all nontrivial irreducible
subrepresentations.  Thus, $H_{CE}^{\bullet}(\mfg,W)=0=H^{CE}_\bullet(\mfg,W)$,
and we obtain the final isomorphism since $Z(U\mfg) = Z(U\mfg)
\otimes_{\bk} \bk$ (regarding the latter as the vector space
$Z(U\mfg)$ tensored with the trivial representation $\bk$).
\end{proof}

\begin{proof}[Solution to Exercise \ref{exer:sra-flat}]
  Since the deformation of the quadratic relation is by a scalar term
  only, denoting by $E \subseteq V^{\otimes 2} + \bk$ the deformed
  relation, we have (referring to Exercise \ref{exer:kdp-filt}) that
  $T = E \otimes V \cap (V+\bk) \otimes E$ must equal $E \otimes V
  \cap V \otimes E$.  Next note that $S = R \otimes V \cap V \otimes
  R$ is spanned by half of the undeformed Jacobi identity, $[x,y]
  \otimes z + [y,z] \otimes x + [z,x] \otimes y$. Therefore, the map
  $\gr(T) \to S$ is an isomorphism if and only if the given (deformed)
  Jacobi identity is satisfied.
\end{proof}

\begin{proof}[Solution to Exercise \ref{exer:deduce-kdp-filt}]
(a) Given only $A$ together with $\phi: A/\hbar A \iso B$, $A$
  is generated by $\phi^{-1}(V)$ as a topological
  $\bk[\![\hbar]\!]$-module.  This follows because it is so generated
  modulo $\hbar$, and $A$ is $\hbar$-adically complete as it is a
  topologically free $\bk[\![\hbar]\!]$-module.  So we have a
  continuous surjection $q: TV[\![\hbar]\!] \to A$.  Now, for every
  element $r \in R$, we have $q(r) \in \hbar A$, so there exists a
  power series $r' \in \hbar TV[\![\hbar]\!]$ such that $q(r+\hbar r')
  = 0$.  Let $\{r_i\}$ be a basis of $R$ and let $\{r_i'\}$ be as
  before. Let $E$ be the span of the $r_i'$. Then $E \to
  TV[\![\hbar]\!] \to TV$ is an isomorphism onto $R$. We also have an
  obvious surjection $\psi: TV[\![\hbar]\!]/(E) \to A$. Moreover, the
  composition of this with the quotient to $A/\hbar \cong B$ clearly
  is the identity on $V$. We have only to show that $\psi$ is an
  isomorphism.  To see this, consider the $\hbar$-adic filtration.  We
  get the surjection $\gr(\psi): TV[\![\hbar]\!]/(R) \to
  \gr_{\hbar}(TV[\![\hbar]\!]/(E))=\gr_{\hbar}(A)$.  It suffices to
  show that this is an isomorphism. Composing with the isomorphism
  $\gr_{\hbar}(A) \cong B[\![\hbar]\!]$, we obtain a surjection
  $TV[\![\hbar]\!](R) \to B[\![\hbar]\!]$, which is nothing but the
  canonical isomorphism.  So $\gr(\psi)$ is an isomorphism, and hence
  so is $\psi$.

(b) Most of the details here are provided; we leave to the reader to
  fill in the missing ones.
\end{proof}

\begin{proof}[Solution to Exercise \ref{exer:koszul-def}]
Suppose $A$ is a graded formal deformation over
  $\bk[\![\hbar]\!]$ of a Koszul algebra $B$.  Since $B$ is Koszul, it
  admits a graded free bimodule resolution $P_\bullet \to B$, with
  each $P_i$ generated in degree $i$.  The same argument as in
  Exercise \ref{exer:koszul}.(d) (taking now the
  $\hbar$-adic filtration, instead of the weight filtration) shows
  that $P_\bullet \to B$ deforms to a graded bimodule resolution
  $P^{\hbar}_\bullet \to A$, with each $P^{\hbar}_i$ free over $A$ of
  the same rank as $P_i$ over $B$.  Moreover, $P^{\hbar}_i$ is also
  generated in degree $i$, so $A$ is Koszul over $\bk[\![\hbar]\!]$ in the sense given.
\end{proof}



\begin{proof}[Solution to Exercise \ref{exer:asigma}]
 (a) We follow the hint. To see
  $\sigma$ is a $\bk$-algebra homomorphism, note that $\sigma(ab)
  \cdot 1 = 1 \cdot ab = (1 \cdot a) \cdot b = \sigma(a) \cdot 1 \cdot
  b = \sigma(a) \sigma(b) \cdot 1$.  Then
  $\sigma(ab)=\sigma(a)\sigma(b)$ because $M$ is a free left module.
  It is clear that $\sigma(1)=1$, and easy to check that $\sigma$ is
  $\bk$-linear. Then, let $\tau$ be defined by $h \cdot 1 = 1 \cdot
  \tau(h)$.  Then it is clear that $\sigma \circ \tau = \tau \circ
  \sigma = \Id$.

  (b) This follows immediately from (a), noting that, as a graded
  module on either side, $M \cong A$ (since $M$ is generated in degree
  zero), and $\sigma$ must be a graded automorphism.
\end{proof}

\begin{proof}[Solution to Exercise \ref{exer:hh-bimod}]
  This follows from the hint: we only need to observe that the given
  action on \eqref{e:hoch-cocomplex} is indeed compatible with the
  differential and outer bimodule structure. More generally, if $M$ is
  a $A^e \otimes B$-module, then $\HH^\bullet(A, M)$ always has a
  canonical $B$-module structure; in this case we can view $M=A
  \otimes A$ as an $A^e \otimes A^e$-module, where the first $A^e$
  acts via the outer action and the second $A^e$ via the inner action.
\end{proof}

\begin{proof}[Solution to Exercise \ref{exer:kos-cy}]
  (a) We computed $\HH^\bullet(\Sym V, \Sym V \otimes \Sym V)$ and
  $\HH^\bullet(\Weyl(V), \Weyl(V) \otimes \Weyl(V))$ in the last
  exercise sheet, as vector spaces.  If we take care of the bimodule
  action, we see that $\HH^{\dim V}(\Sym V, \Sym V \otimes \Sym V)$
  really is $\Sym V$ as a bimodule, since the bimodule complex
  computing this is isomorphic to the original one computing $\Sym V$
  as a bimodule.  The same is true for $\Weyl(V)$.

(b) The resulting complex is still exact because $\bk[\Gamma]$ is
flat over $\bk$ (in fact, free), and we haven't changed the differential
by applying $\otimes_{\bk} \bk[\Gamma]$. We only need to check we get
a complex of $A \rtimes \Gamma$-bimodules, and this follows directly from
the definition of the bimodule structure.  For the final assertion, note
that, for general $\Gamma < \GL(V)$, we can still consider
$\Sym V \rtimes \Gamma$, and the above yields a bimodule resolution.
Applying $\Hom_{(\Sym V)^e}(-, \Sym V)$, we get a complex computing $\HH^\bullet(\Sym V, \Sym V \otimes \Sym V)$, and we get that this is $(\Sym V)^\sigma[-d]$,
where $\sigma$ is the automorphism $\sigma(v \otimes \gamma) =
\det(\gamma) (v \otimes \gamma)$.  So this is trivial if $\Gamma < \SL(V)$,
which yields that in this case $\Sym V \rtimes \Gamma$ is trivial.
\begin{remark}\label{r:asigma-inn}
  For a general algebra $B$ and automorphism $\sigma$, we have $B
  \cong B^\sigma$ as bimodules if and only if $\sigma$ is inner, i.e.,
  of the form $\sigma(a)=x^{-1}ax$ for some invertible $x \in B$.  In
  the above situation of $B=\Sym V \rtimes \Gamma$, and $\sigma(v
  \otimes \gamma)=\det(\gamma)(v \otimes \gamma)$, one can check that
  $\sigma$ is not inner when $\Gamma$ is not in $\SL(V)$, so that
  actually $\Sym V \rtimes \Gamma$ is Calabi-Yau (for $\Gamma$ finite)
  if and only if $\Gamma < \SL(V)$.
\end{remark} 
\end{proof}

\begin{proof}[Solution to Exercise \ref{exer:obstr-nord}]
 Let $\Gamma_n := \sum_{i=1}^n
  \varepsilon^i \gamma_i$.  Then, working in
  $\bk[\varepsilon]/(\varepsilon^{n+2})$, we have $\frac{1}{2}[\mu + \Gamma_n,
  \mu + \Gamma_n] = \delta \varepsilon^{n+1}$, for some $\delta \in
  C^2(A,A)$. Moreover, we claim that $\delta$ is a cocycle. Using
  $[\mu,-]=d(-)$, we have
\[
d(\frac{1}{2}[\mu+\Gamma_n, \mu+\Gamma_n]) = [d\Gamma_n, \mu + \Gamma_n] 
= [d\Gamma_n, \Gamma_n]=
[[\mu,\Gamma_n],\Gamma_n].
\]
Then, since $[\Gamma_n]$ is a multiple of $\varepsilon$ and
$[\mu,\Gamma_n] + \frac{1}{2}[\Gamma_n, \Gamma_n]$ is a multiple of
$\varepsilon^{n+1}$, the RHS equals (modulo $\varepsilon^{n+2}$):
\[
-[\frac{1}{2}[\Gamma_n, \Gamma_n], \Gamma_n] = 0,           .
\]
by the Jacobi identity (which holds on cochains identically). Thus
$\delta$ is indeed a three-cycle, and defines a cohomology class
$[\delta] \in \HH^3(A)$.

Now, to extend $\Gamma_n$ to an $(n+1)$-st order deformation
$\Gamma_{n+1} = \Gamma_n + \varepsilon^{n+1} \gamma_{n+1}$, we need to
solve the equation $\frac{1}{2}[\mu + \Gamma_{n+1}, \mu + \Gamma_{n+1}]=0$
(modulo $\varepsilon^{n+2}$), which simplifies to
\[
d(\gamma_{n+1}) + \delta = 0.
\]
Hence, the condition to extend $\Gamma_n$ to an $(n+1)$-st order
deformation is the condition that the cohomology class $[\delta] \in
\HH^3(A)$ is zero.

Finally, the space of choices of $\gamma_{n+1}$ is equal to
$d^{-1}(\delta)$, which is an affine space on the space of Hochschild
two-cocycles.  If we apply an automorphism of the form $\Phi = \Id +
\varepsilon^{n+1}(f)$, we get that $\Phi \circ (\mu+\Gamma_n) \circ
(\Phi^{-1} \otimes \Phi^{-1})$ is nothing but $\mu + \Gamma_n +
[\varepsilon^{n+1} f, \Gamma_n] = \mu + \Gamma_n -\varepsilon^{n+1}
(df)$.  So, up to these gauge equivalences, the space of extensions to
an $n+1$-st order deformation $\Gamma_{n+1}$ is an affine space on
$\HH^2(A)$ (provided it is nonempty, i.e., $\delta$ is a coboundary).
\end{proof}

\begin{proof}[Solution to Exercise \ref{exer:ug-cy}]
  We computed $\HH^\bullet(U\mfg, U\mfg \otimes U\mfg)$ as a vector
  space in the last exercise sheet.  A little more work computes the
  bimodule structure.  We find that $\HH^{\dim V}(U\mfg, U\mfg \otimes
  U\mfg) \cong U\mfg^{\sigma}$ for $\sigma$ the automorphism
  $\sigma(x)=x - \tr(\ad(x))$ because $\tr(ad(x))$ coincides with the
  action of $\ad(x)$ on $\wedge^{\dim V} V$.  Therefore $U \mfg$ is
  twisted Calabi-Yau, and actually Calabi-Yau when $\tr(ad(x))=0$ for
  all $x$ (i.e., $\mfg$ is unimodular).  For the final statement,
  note that, if $\mfg = [\mfg, \mfg]$, then for some $x_1^{(i)},
  x_2^{(i)} \in \mfg$, we have $x = \sum_i [x_1^{(i)}, x_2^{(i)}]$,
  and hence $\tr(\ad(x)) = \sum_i \tr([\ad(x_1^{(i)}),
  \ad(x_2^{(i)})]) = 0$. It is clear that the adjoint action, hence
  also its trace, is zero on an abelian Lie algebra, and we conclude
  that all reductive Lie algebras are unimodular.
\end{proof}
\begin{remark}
  If $\tr(\ad(x)) \neq 0$ for some $x$, then $U \mfg^{\sigma}$ is not
  isomorphic to $U \mfg$, since the automorphism $\sigma$ is not
  inner. In fact, the only invertible elements of $U\mfg$ are scalars,
  since $fg = 1$ implies $\gr(f) \gr(g) = 1$ as $\Sym \mfg$ has no
  zero divisors.  (We remark that we could alleviate this by
  completing $U \mfg$ at the augmentation ideal $(\mfg)$, in which
  case all elements not in the ideal are invertible; however, $\sigma$
  does not preserve this ideal so it still cannot be inner.) So
  $U\mfg$ is Calabi-Yau if and only if $\mfg$ is unimodular.
\end{remark}

\begin{proof}[Solution to Exercise \ref{exer:hoch-sg}]
 (a) The details for the first isomorphism (for Hochschild
  cohomology) are already given in the hint.  We omit the similar
  computation for Hochschild homology.

(b) Write $A \rtimes \Gamma = \bigoplus_{g \in G} A \cdot g$
as an $A$-bimodule. For every pair of elements $g,h \in \Gamma$,
we have an isomorphism by conjugation,
$\Ad(g): \HH^\bullet(A, A \cdot h) \iso \HH^\bullet(A, A \cdot (ghg^{-1}))$.
Therefore, letting $\mathcal{C}_h := \Ad(\Gamma) \cdot h$ denote
the conjugacy class, we obtain, as a $\Gamma$-representation,
\[
\HH^\bullet(A, A \cdot \mathcal{C}_h) \cong \Ind_{Z_h(\Gamma)}^\Gamma \HH^\bullet(A, A \cdot h),
\]
and taking $\Gamma$-invariants, we get $\HH^\bullet(A, A \cdot h)^{Z_h(\Gamma)}$.
Summing over all conjugacy classes yields the desired formula.

(c) The necessary details are in the hint (we omit the corresponding
ones for Hochschild homology).

(d) The necessary details here are also given; see \cite{AFLS} for the
full details (it is only a few pages).
\end{proof}


\begin{proof}[Solution to Exercise \ref{exer:pre-lie}]
First, we verify the final assertion, that the skew-symmetrization
of a dg (right) pre-Lie product is a Lie bracket.  This is because the
pre-Lie identity, upon skew-symmetrization, becomes a multiple of the Jacobi identity, and the compatibility with the differential carries over to the skew-symmetrization.  


Going back to the situation of $C^\bullet(A)[1]$ for $A$ an associative
algebra, we have to check that $\circ$ is a derivation and that it
satisfies the right pre-Lie identity. The derivation property follows
from the fact that $d(\gamma \circ \eta)$ and
$d(\gamma) \circ \eta +(-1)^{|\gamma|} \gamma \circ (d \eta)$, applied
to $(a_1 \otimes \cdots \otimes a_{m+n}$, are both
\begin{multline*}
\sum_{i=1}^m (-1)^{(i-1)(n+1)}\bigl( a_1\gamma(a_2 \otimes \cdots \otimes a_{i} \otimes \eta(a_{i+1} \otimes \cdots \otimes a_{i+n}) \otimes a_{i+n+1} \otimes \cdots \otimes a_{m+n}) + \\
\sum_{k=1}^{i-1}(-1)^k\gamma(a_1 \otimes \cdots a_k a_{k+1} \cdots \otimes a_{i} \otimes \eta(a_{i+1} \otimes \cdots \otimes a_{i+n}) \otimes a_{i+n+1} \otimes \cdots \otimes a_{m+n}) + \\
\sum_{k=i}^{i+n-1} (-1)^k \gamma(a_1 \otimes \cdots \otimes a_{i} \otimes \eta(a_{i+1} \otimes \cdots a_k a_{k+1} \cdots \otimes a_{i+n}) \otimes a_{i+n+1} \otimes \cdots \otimes a_{m+n}) + \\
\sum_{k=i+n}^{m+n-1}  (-1)^k \gamma(a_1 \otimes \cdots \otimes a_{i} \otimes \eta(a_{i+1} \otimes \cdots \otimes a_{i+n}) \otimes a_{i+n+1} \otimes \cdots 
a_k a_{k+1} \cdots \otimes a_{m+n}) + \\
 (-1)^{m+n} \gamma(a_1 \otimes \cdots \otimes a_{i} \otimes \eta(a_{i+1} \otimes \cdots \otimes a_{i+n}) \otimes a_{i+n+1} \otimes \cdots \otimes a_{m+n-1})a_{m+n} \bigr).
\end{multline*}
The basic point is that the terms on the LHS are all printed above:
they involve multiplying adjacent components either inside $\gamma$ or
inside $\eta$ with a sign.  The corresponding terms on the RHS all
occur with the same sign, and the RHS also has some terms which cancel
pairwise, of the form $\pm \gamma(a_1 \otimes \cdots \otimes a_i
\eta(a_{i+1} \otimes \cdots \otimes a_{i+n}) \otimes \cdots \otimes
a_{m+n})$ and similarly replacing the middle by $\eta(a_1 \otimes
\cdots \otimes a_{i+n-1})a_{i+n}$.  One could (and should) think of
this diagrammatically, where $\eta$ takes $n$ successive inputs to one
output, $\gamma$ takes $m$ successive inputs to one output, and the
multiplication map takes two successive inputs to one output, and then
both the LHS and RHS express as the same signed combination of
diagrams.

A similar, but more involved, direct computation shows that the
pre-Lie identity is satisfied (see \cite{Ge}).  The basic idea is
that, diagrammatically, both sides are the ways of applying $\eta$ to
some successive inputs and $\theta$ simultaneously to other successive
inputs, and finally applying $\gamma$ to the result.
\end{proof}

\begin{proof}[Solution to Exercise \ref{exer:circ-prod}]
(a) The identity is again a similar explicit computation; we refer
  to \cite{Ge} for details. Both sides diagrammatically give the
  result of applying $\gamma_1$ and $\gamma_2$ each to different
  blocks of successive inputs (so $\gamma_1$ to either the leftmost or
  rightmost $|\gamma_1|$ inputs, and $\gamma_2$ to the other inputs)
  and then multiplying the result, with an appropriate sign.  On the
  RHS, all other ways of applying $\gamma_1, \gamma_2$, and the
  multiplication cancel pairwise.

  One can similarly verify the identity that the cup product is
  compatible with the differential, so it descends to a binary
  operation on cohomology.  Using the identity, we see that the LHS is
  a coboundary, hence the cup product is symmetric on cohomology,
  as desired. 

  Note that, since the cup product is associative on cochains, it is
  also on cohomology, so we get a commutative graded algebra structure
  on Hochschild cohomology.

(b) Note that the Gerstenhaber bracket is the
skew-symmetrization of the circle product, and so $d[\gamma_1,
\gamma_2] - [d\gamma_1, \gamma_2] - (-1)^{|\gamma_1|}[\gamma_1,
d\gamma_2]$ is the skew-symmetrization of the LHS, which is zero.  So
the Gerstenhaber bracket is indeed compatible with the differential,
which implies we obtain a Gerstenhaber bracket on Hochschild
cohomology.

We remark that one can similarly directly verify the Leibniz identity
for the Gerstenhaber bracket on Hochschild cohomology (see \cite{Ge}).
\end{proof}

\begin{proof}[Solution to Exercise \ref{exer:gauge}]
Since we have assumed $G < \GL_n$ and $\mfg < \mfgl_n$ and
  $\bk=\bR$ or $\bC$, we can write
\[
\gamma \cdot d(\iota \circ \gamma) = \exp(\beta) \cdot
d(\exp^{-\beta}) = \Ad(\exp^{\beta})(d) = \exp(\ad(\beta))(d) =
\frac{\exp(\ad(\beta)) - 1}{\ad(\beta)} (-d \beta),
\]
where $\Ad(\exp^{\beta})(d)$ and $\exp(\ad(\beta))(d)$ are just two
formal expressions (standing for certain infinite linear combinations
of terms of the form $\beta^a d \beta^b$, which are equal because
$\Ad(\exp^{-}) = \exp(\ad(-))$ as formal series).  This implies
\eqref{e:gauge-exp}.
\end{proof}

\begin{proof}[Solution to Exercise \ref{exer:pb-sn}]
From $\pi \in \wedge^2 \Vect(X)$ we define
the bracket $\{f,g\}=\pi(f \otimes g)$.  The skew-symmetry
of the bracket is immediate from the skew-symmetry of $\pi$,
and the Leibniz rule follows from the fact that $\Vect(X)$
acts by derivations on $\cO(X)$.

We need to check that $[\pi, \pi]=0$ if and only if the Jacobi
identity is satisfied.  
Write $\pi = \sum_i \xi_i^{(1)} \otimes \xi_i^{(2)}$, a skew-symmetric element of $\Vect(X)^{\otimes 2}$
(which is identified with the image of
 $\sum_i \xi_i^{(1)} \wedge \xi_i^{(2)}$ 
under the skew-symmetrization map $\wedge^2 \Vect(X) \to \Vect(X)^{\otimes 2}$). 
Then we have
\[
[\pi,\pi] = 4 \sum_{i,j} [\xi_i^{(1)}, \xi_j^{(1)}] \wedge \xi_j^{(2)}
\wedge \xi_i^{(2)}.
\]
Let $\Phi = \sum_{i,j} [\xi_i^{(1)}, \xi_j^{(1)}] \otimes \xi_j^{(2)}
\otimes \xi_i^{(2)}$. Then
\[
\Phi(f \otimes g \otimes h) = \{\{f,g\},h\}-\{\{f,h\},g\}.
\]
Skew-symmetrizing and multiplying by four, we get
\[ [\pi,\pi](f \otimes g \otimes h) = -\frac{8}{3} (\{f,\{g,h\}\} +
\{g,\{h,f\}\} + \{h,\{f,g\}\}),
\]
which implies that $[\pi,\pi]=0$ if and only if the Jacobi identity
is satisfied.
\end{proof}

\begin{proof}[Solution to Exercise \ref{exer:linf-dgla}]
Given a dgla morphism $\phi: \mfg \to \mfh$, we need to show
that the induced map $\phi^*: \hat S (\mfh[1])^* \to
\hat S (\mfg[1])^*$ is a dg algebra morphism.  It is clear that it is
an algebra morphism, since this only requires $\phi$ to be linear.  
We need to show that $\phi^*$ is compatible with the differential.  This follows
from the fact that $\mfg \to \mfh$ is both compatible with the differential
and with the Lie bracket, since the differential $d_{CE}$ is the sum of two
terms, one corresponding to each.  Namely, for $\xi \in \mfh[1]^*$
and $a,b \in \mfg$, we have
\begin{gather*}
  \phi^* \circ d_{CE}(\xi)(a \otimes b) = \xi([\phi(a),\phi(b)]) =
  \xi(\phi([a,b])) = d_{CE} \circ \phi^*(\xi)(a \otimes b), \\ \phi^*
  \circ d_{CE}(\xi)(a) = \xi(d \phi(a)) = \xi(\phi(da)) = d_{CE} \circ
  \phi^*(\xi)(a). \qedhere
\end{gather*}
\end{proof}

\begin{proof}[Solution to Exercise \ref{exer:kont-ug-iso}]
  The Poisson bivector has degree $-1$, or equivalently, $|\pi(f
  \otimes g)| = |f|+|g|-1$ when $f$ and $g$ are homogeneous.  Hence,
  the only possible graphs that can be nonzero applied to $v \otimes
  w$ for $v,w \in \mfg$ are those corresponding to the product, $vw$,
  the Poisson bracket $\{v,w\} = \pi(v \otimes w)$, and finally
  $\pi^2(v \otimes w)$.  But, $\pi^2$ is symmetric (since $\pi$ is
  skew-symmetric), and hence $\pi^2(v\otimes w - w \otimes v)= 0$;
  similarly $vw-wv=0$.  Therefore the only graph that can contribute
  to $v \star w - w \star v$ is the one corresponding to the Poisson
  bracket, $\{v,w\}=[v,w]$. We conclude that $v \star w - w \star v =
  c \hbar [v,w]$ for some $c \in \bk$.  Then, the fact that we get a
  deformation quantization implies that $c=1$.
\end{proof}

\begin{proof}[Solution to Exercise \ref{exer:hh3=0}]
The necessary details are given already in the exercise sheet.
\end{proof}

\begin{proof}[Solution to Exercise \ref{exer:kont-duflo}]
The details below (and more) are all taken from \cite{Kform}.

(i) To see that $B_{W_m}(\pi,\ldots,\pi)$ has order $m$, note that
by construction it has $m$ arrows pointed at the vertex on labeled $f$
and each arrow corresponds to differentiation.  Moreover, since $\pi$ has
degree $-1$, this operator has degree $-m$; an operator of degree $-m$ and
order $m$ must be a constant-coefficient operator.

(ii) The formula is equivalent to the statement that, if we apply the
wheel to the function $x^m$ placed at the vertex $f$ (for $x \in \mfg$), we
obtain $\tr((\ad x)^m)$. Note that $\pi(x \otimes -) = \ad(x)(-)$.
Let $x_i$ be a basis of $\mfg$, and write $\ad(x) = \sum_{i,j} c_{ij}
x_i \partial_j$ for some $c_{ij} \in \bk$. Applying the wheel to $x^m$
yields
\[
\sum_{i_1, i_2, \ldots, i_m} c_{i_1 i_2} c_{i_2 i_3} \cdots c_{i_m i_1},
\]
but this is nothing but $\tr((\ad x)^m)$, as desired.

(iii) This follows immediately from the single wheel ($k=1$) case, by
the definition of $B_\Gamma$.

(iv) Kontsevich's isomorphism must be expressed as a formal sum of
graphs, and the graphs must have exactly one vertex labeled $f$
(since the underlying linear map of our isomorphism is the restriction
of a linear map from a single copy of $\Sym \mfg$ to $\Sym
\mfg[\![\hbar]\!]$).  In the graphs that appear, each vertex 
labeled $\pi$
cannot have both of its arrows pointing to the same
vertex, since $\pi$ is skew-symmetric. Also, each vertex 
labeled $\pi$
can be the target of at most one arrow, since $\pi$
is linear: in more detail, if we write $\pi = \sum_{j,k}
f_{jk} \partial_j \wedge \partial_k$ for $f_{jk} \in \mfg$, applying
more than one partial derivative to $f_{jk}$ would yield zero, hence
$B_\Gamma=0$ if $\Gamma$ is a graph with a vertex 
labeled $\pi$ which is
the target of multiple arrows. Since every vertex labeled $\pi$
 is the source of two arrows, the only possibility (to
have $B_\Gamma \neq 0$) is to have every vertex labeled $\pi$
 pointing to both the vertex labeled $f$
and one other vertex,
 such that each vertex labeled $\pi$
is the target of
exactly one arrow.  Such graphs are the union of wheels, so the result
follows from (iii).

(v) The stated results imply that the isomorphism is a sum of the form
\[
x \mapsto \sum_{m_1, \ldots, m_k \text{ even }}
\frac{1}{k!} c_{m_1} \cdots c_{m_k} \tr((\ad x)^{m_1}) \cdots \tr((\ad x)^{m_k}),
\]
which coincides with the given formula (since the $m_1, \ldots, m_k$
can occur in all possible orderings, e.g., $2,4$ and $4,2$ both occur).
\end{proof}


\begin{proof}[Solution to Exercise \ref{exer:hypsfc}] 
More generally, if $f_1, \ldots, f_k \in \cO(\bA^n)$
cut out a variety $X$ of dimension $n-k$,  then we can form
\[
\Xi_X:=i_{\partial_1 \wedge \cdots \wedge \partial_n}(df_1 \wedge
\cdots \wedge df_k) \in \wedge^{n-k}_{\cO(\bA^n)} \Vect(\bA^n),
\]
a $(n-k)$-polyvector field on $\bA^n$.  Note that, by construction,
$i_{\Xi_X}(df_i) = 0$ for all $i$.

We claim that $\Xi_X$ is unimodular in the sense that each vector
field of the form $\xi=i_{\Xi_X}(dh_1 \otimes \cdots \otimes
dh_{n-k-1}) \in \Vect(\bA^n)$ is divergence-free, i.e.,
$L_{\xi}(\vol)=0$ with $\vol=dx_1 \wedge \cdots \wedge dx_n$ the
standard volume form on $\bA^n$.  Equivalently, $i_{\partial_1 \wedge
  \cdots \wedge \partial_n}(\alpha)$ is unimodular whenever $\alpha$
is an exact $(n-1)$-form.  This is a standard local computation, and
it only uses that $\alpha$ is closed: writing $\alpha = \sum_i f_i
dx_1 \wedge \cdots \wedge \widehat{dx_i} \wedge \cdots \wedge dx_n$,
we have $i_{\partial_1 \wedge \cdots
  \wedge \partial_n}(d\alpha)=\sum_i \frac{d f_i}{dx_i}$, which is
the same as the divergence of $i_{\partial_1 \wedge \cdots
  \wedge \partial_n}(\alpha)=\sum_i f_i \partial_i$.

By construction, $\Xi_X$ is parallel to $X$ (since $f_1, \ldots, f_k$
vanish there); algebraically this is saying that we have a
well-defined map,
\[
\Xi_X|_X: \cO(X)^{\otimes (n-k)} \to \cO(X), \quad (g_1 \otimes \cdots
\otimes g_{n-k}) \mapsto i_{\Xi_X}(dg_1 \otimes \cdots \otimes
dg_{n-k}),
\]
which is skew-symmetric and a derivation in each component. As a
result, so is $[\Xi_X, \Xi_X]$, which on $X$ is obviously zero (in the
case $\dim X \geq 2$ this is because its degree is greater than $\dim
X$; for $\dim X = 1$ this is clear).  Thus $[\Xi_X, \Xi_X]=0$.  In the
case $\dim X = 2$, this implies that $\Xi_X$ is a Poisson bivector on
$\bA^n$ (by Proposition \ref{p:pb-sn}).  Finally, to see that $f_i$
are central in this bracket, we recall the identity $i_{\Xi_X}(df_i) =
0$ above, but $i_{\Xi_X}(df_i)$ is the Hamiltonian vector field $\xi_f
:= \{f,-\}$.
\end{proof}

\bibliographystyle{amsalpha}
\bibliography{local-bib}

\def\cprime{$'$} \def\cprime{$'$} \def\cprime{$'$} \def\cprime{$'$}
  \def\cprime{$'$} \def\cprime{$'$}
\providecommand{\bysame}{\leavevmode\hbox to3em{\hrulefill}\thinspace}
\providecommand{\MR}{\relax\ifhmode\unskip\space\fi MR }
\providecommand{\MRhref}[2]{%
  \href{http://www.ams.org/mathscinet-getitem?mr=#1}{#2}
}
\providecommand{\href}[2]{#2}
\begin{thebibliography}{VdBdTdV12}

\bibitem[AFLS00]{AFLS}
J.~Alev, M.~A. Farinati, T.~Lambre, and A.~L. Solotar, \emph{Homologie des
  invariants d'une alg{\`e}bre de {W}eyl sous l'action d'un groupe fini}, J.
  Algebra \textbf{232} (2000), 564--577.

\bibitem[Arn98]{Arn-pskdaln}
D.~Arnal, \emph{Le produit star de {K}ontsevich sur le dual d'une alg\`ebre de
  {L}ie nilpotente}, C. R. Acad. Sci. Paris S\'er. I Math. \textbf{327} (1998),
  no.~9, 823--826. \MR{1663738 (99k:58077)}

\bibitem[BG96]{PBW}
A.~Braverman and D.~Gaitsgory, \emph{Poincar\'e-{B}irkhoff-{W}itt theorem for
  quadratic algebras of {K}oszul type}, J. Algebra \textbf{181} (1996), no.~2,
  315--328. \MR{1383469 (96m:16012)}

\bibitem[BGG76]{BGG}
I.~N. Bern{\v{s}}te{\u\i}n, I.~M. Gel{\cprime}fand, and S.~I. Gel{\cprime}fand,
  \emph{A certain category of {${\mathfrak g}$}-modules}, Funkcional. Anal. i
  Prilo\v zen. \textbf{10} (1976), no.~2, 1--8. \MR{0407097 (53 \#10880)}

\bibitem[BGS96]{BGS-Kdprt}
A.~Beilinson, Victor Ginzburg, and W.~Soergel, \emph{Koszul duality patterns in
  representation theory}, J. Amer. Math. Soc. \textbf{9} (1996), no.~2,
  473--527. \MR{1322847 (96k:17010)}

\bibitem[BL13]{BL-Ecqqv}
R.~Bezrukavnikov and I.~Losev, \emph{Etingof conjecture for quantized quiver
  varieties}, arXiv:1309.1716, 2013.

\bibitem[Boc06]{Boc-gcyad3}
R.~Bocklandt, \emph{Graded {C}alabi-{Y}au algebras of dimension {$3$}},
  arXiv:math.RA/0603558, 2006.

\bibitem[Bri71]{Brisessag}
E.~Brieskorn, \emph{Singular elements of semi-simple algebraic groups}, Actes
  du {C}ongr\`es {I}nternational des {M}ath\'ematiciens ({N}ice, 1970), {T}ome
  2, Gauthier-Villars, Paris, 1971, pp.~279--284. \MR{MR0437798 (55 \#10720)}

\bibitem[Bry88]{Br}
J.-L. Brylinski, \emph{A differential complex for {P}oisson manifolds}, J.
  Differential Geom. \textbf{28} (1988), no.~1, 93--114.

\bibitem[BSW10]{BSW}
R.~Bocklandt, Travis Schedler, and M.~Wemyss, \emph{Superpotentials and higher
  order derivations}, J. Pure Appl. Algebra \textbf{214} (2010), no.~9,
  1501--1522, arXiv:0802.0162.

\bibitem[BT07]{BerTai}
R.~Berger and R.~Taillefer, \emph{Poincar{\'e}-{B}irkhoff-{W}itt deformations
  of {C}alabi-{Y}au algebras}, J. Noncommut. Geom. (2007), 241--270.

\bibitem[CBEG07]{CBEG}
W.~Crawley-Boevey, Pavel Etingof, and Victor Ginzburg, \emph{Noncommutative
  geometry and quiver algebras}, Adv. Math. \textbf{209} (2007), no.~1,
  274--336, arXiv:math.AG/0502301.

\bibitem[CBH98]{CrawleyBoeveyHolland}
W.~Crawley-Boevey and M.~P. Holland, \emph{Noncommutative deformations of
  {K}leinian singularities}, Duke Math. J. \textbf{92} (1998), no.~3, 605--635.

\bibitem[Che41]{Che-appLg}
C.~Chevalley, \emph{An algebraic proof of a property of {L}ie groups}, Amer. J.
  Math. \textbf{63} (1941), 785--793. \MR{0006544 (4,2g)}

\bibitem[Che55]{Chevalley}
\bysame, \emph{Invariants of finite groups generated by reflections}, Amer. J.
  Math. \textbf{77} (1955), 778--782. \MR{0072877 (17,345d)}

\bibitem[CRVdB12]{CRV-CcTf}
D.~Calaque, C.~A. Rossi, and M.~Van~den Bergh, \emph{C\u ald\u araru's
  conjecture and {T}sygan's formality}, Ann. of Math. (2) \textbf{176} (2012),
  no.~2, 865--923. \MR{2950766}

\bibitem[CVdB10a]{CV-GfGl}
D.~Calaque and M.~Van~den Bergh, \emph{Global formality at the
  {$G_\infty$}-level}, Mosc. Math. J. \textbf{10} (2010), no.~1, 31--64, 271.
  \MR{2668829 (2011j:18022)}

\bibitem[CVdB10b]{CV-HcAc}
\bysame, \emph{Hochschild cohomology and {A}tiyah classes}, Adv. Math.
  \textbf{224} (2010), no.~5, 1839--1889. \MR{2646112 (2011i:14037)}

\bibitem[CWWZ14]{CWWZ-Hafra}
K.~Chan, C.~Walton, Y.~H. Wang, and J.~J. Zhang, \emph{Hopf actions on filtered
  regular algebras}, J. Algebra \textbf{397} (2014), 68--90. \MR{3119216}

\bibitem[Dav12]{Dav-sam}
B.~Davison, \emph{Superpotential algebras and manifolds}, Adv. Math.
  \textbf{231} (2012), no.~2, 879--912. \MR{2955196}

\bibitem[DF73]{DF-rtfgaa}
P.~Donovan and M.~R. Freislich, \emph{The representation theory of finite
  graphs and associated algebras}, Carleton University, Ottawa, Ont., 1973,
  Carleton Mathematical Lecture Notes, No. 5. \MR{0357233 (50 \#9701)}

\bibitem[Dit99]{Dit-kspdla}
G.~Dito, \emph{Kontsevich star product on the dual of a {L}ie algebra}, Lett.
  Math. Phys. \textbf{48} (1999), no.~4, 307--322. \MR{1709042 (2000i:53126)}

\bibitem[Dol13]{Dol-acekfqrrn}
V.~Dolgushev, \emph{All coefficients entering kontsevich's formality
  quasi-isomorphism can be replaced by rational numbers}, arXiv:1306.6733,
  2013.

\bibitem[DR74]{DR-rga}
V.~Dlab and C.~M. Ringel, \emph{Representations of graphs and algebras},
  Department of Mathematics, Carleton University, Ottawa, Ont., 1974, Carleton
  Mathematical Lecture Notes, No. 8. \MR{0387350 (52 \#8193)}

\bibitem[Dri92]{Dri-qcrqc}
V.~G. Drinfel{\cprime}d, \emph{On quadratic commutation relations in the
  quasiclassical case [translation of {\it {m}athematical physics, functional
  analysis ({r}ussian)}, 25--34, 143, ``{N}aukova {D}umka'', {K}iev, 1986;
  {MR}0906075 (89c:58048)]}, Selecta Math. Soviet. \textbf{11} (1992), no.~4,
  317--326, Selected translations. \MR{1206296}

\bibitem[DTT07]{DTT-hgahcraf}
V.~Dolgushev, D.~Tamarkin, and B.~Tsygan, \emph{The homotopy {G}erstenhaber
  algebra of {H}ochschild cochains of a regular algebra is formal}, J.
  Noncommut. Geom. \textbf{1} (2007), no.~1, 1--25, arXiv:math/0605141.
  \MR{2294189 (2008c:18007)}

\bibitem[DWL83]{DWL-espfdPLaasm}
M.~De~Wilde and P.~B.~A. Lecomte, \emph{Existence of star-products and of
  formal deformations of the {P}oisson {L}ie algebra of arbitrary symplectic
  manifolds}, Lett. Math. Phys. \textbf{7} (1983), no.~6, 487--496. \MR{728644
  (85j:17021)}

\bibitem[EG02]{EG}
Pavel Etingof and Victor Ginzburg, \emph{Symplectic reflection algebras,
  {C}alogero-{M}oser space, and deformed {H}arish-{C}handra homomorphism},
  Invent. Math. \textbf{147} (2002), no.~2, 243--348. \MR{MR1881922
  (2003b:16021)}

\bibitem[EG07]{EGncci}
\bysame, \emph{Noncommutative complete intersections and matrix integrals},
  Pure Appl. Math. Q. \textbf{3} (2007), no.~1, 107--151,
  arXiv:math.AG/0603272.

\bibitem[EG10]{EGdelpezzo}
\bysame, \emph{Noncommutative del {P}ezzo surfaces and {C}alabi-{Y}au
  algebras}, J. Eur. Math. Soc. (JEMS) \textbf{12} (2010), no.~6, 1371--1416,
  arXiv:0709.3593v4. \MR{2734346}

\bibitem[EK96]{EK}
Pavel Etingof and D.~Kazhdan, \emph{Quantization of {L}ie bialgebras {I}},
  Selecta Math (N.S.) \textbf{2} (1996), no.~1, 1--41.

\bibitem[EM05]{EM-fdrsrawp}
Pavel Etingof and S.~Montarani, \emph{Finite dimensional representations of
  symplectic reflection algebras associated to wreath products}, Represent.
  Theory \textbf{9} (2005), 457--467 (electronic). \MR{2167902 (2007d:16024)}

\bibitem[Eti05]{Eti-enadt}
Pavel Etingof, \emph{Exploring noncommutative algebras via deformation theory},
  arXiv:math/0506144, 2005.

\bibitem[Eti17]{Eti-chavfga}
\bysame, \emph{Cherednik and {H}ecke algebras of varieties with a finite group
  action}, Mosc. Math. J. \textbf{17} (2017), no.~4, 635--666. \MR{3734656}

\bibitem[EW14]{EW-sHacd}
Pavel Etingof and C.~Walton, \emph{Semisimple {H}opf actions on commutative
  domains}, Adv. Math. \textbf{251} (2014), 47--61. \MR{3130334}

\bibitem[Fed94]{Fed-sgcdq}
B.~V. Fedosov, \emph{A simple geometrical construction of deformation
  quantization}, J. Differential Geom. \textbf{40} (1994), no.~2, 213--238.
  \MR{1293654 (95h:58062)}

\bibitem[Gab72]{Gab-ud1}
P.~Gabriel, \emph{Unzerlegbare {D}arstellungen. {I}}, Manuscripta Math.
  \textbf{6} (1972), 71--103; correction, ibid. 6 (1972), 309. \MR{0332887 (48
  \#11212)}

\bibitem[Ger63]{Ge}
M.~Gerstenhaber, \emph{The cohomology structure of an associative ring}, Ann.
  of Math. (2) \textbf{78} (1963), 267--288.

\bibitem[Get09]{Get-ltnla}
E.~Getzler, \emph{Lie theory for nilpotent {$L_\infty$}-algebras}, Ann. of
  Math. (2) \textbf{170} (2009), no.~1, 271--301. \MR{2521116 (2010g:17026)}

\bibitem[Gin06]{GinzCY}
Victor Ginzburg, \emph{{C}alabi-{Y}au algebras}, arXiv:math.AG/0612139, 2006.

\bibitem[GP79]{GP}
I.~M. Gelfand and V.~A. Ponomarev, \emph{Model algebras and representations of
  graphs}, Func. Anal. and Applic. \textbf{13} (1979), no.~3, 1--12.

\bibitem[GS06]{GSmoy}
Victor Ginzburg and Travis Schedler, \emph{{M}oyal quantization and stable
  homology of necklace {L}ie algebras}, Mosc. Math. J. \textbf{6} (2006),
  no.~3, 431--459, arXiv:math.QA/0605704.

\bibitem[Gut83]{Gut-espctLg}
S.~Gutt, \emph{An explicit {$^{\ast} $}-product on the cotangent bundle of a
  {L}ie group}, Lett. Math. Phys. \textbf{7} (1983), no.~3, 249--258.
  \MR{706215 (85g:58037)}

\bibitem[HC51]{HC-saueasLa}
Harish-Chandra, \emph{On some applications of the universal enveloping algebra
  of a semisimple {L}ie algebra}, Trans. Amer. Math. Soc. \textbf{70} (1951),
  28--96. \MR{0044515 (13,428c)}

\bibitem[Hin97]{Hin-haha}
V.~Hinich, \emph{Homological algebra of homotopy algebras}, Comm. Algebra
  \textbf{25} (1997), no.~10, 3291--3323. \MR{1465117 (99b:18017)}

\bibitem[Hum78]{Hum-ilart}
James~E. Humphreys, \emph{Introduction to {L}ie algebras and representation
  theory}, Graduate Texts in Mathematics, vol.~9, Springer-Verlag, New
  York-Berlin, 1978, Second printing, revised. \MR{499562 (81b:17007)}

\bibitem[Hum90]{Hum-rgcg}
J.~E. Humphreys, \emph{Reflection groups and {C}oxeter groups}, Cambridge
  Studies in Advanced Mathematics, vol.~29, Cambridge University Press,
  Cambridge, 1990. \MR{1066460 (92h:20002)}

\bibitem[HVOZ10]{HVOZ-cCYHad}
J.-W. He, F.~Van~Oystaeyen, and Y.~Zhang, \emph{Cocommutative {C}alabi-{Y}au
  {H}opf algebras and deformations}, J. Algebra \textbf{324} (2010), no.~8,
  1921--1939. \MR{2678828 (2011f:16079)}

\bibitem[Kon01]{Kon-dqav}
M.~Kontsevich, \emph{Deformation quantization of algebraic varieties}, Lett.
  Math. Phys. \textbf{56} (2001), no.~3, 271--294, EuroConf{\'e}rence Mosh{\'e}
  Flato 2000, Part III (Dijon). \MR{1855264 (2002j:53117)}

\bibitem[Kon03]{Kform}
\bysame, \emph{Deformation quantization of {P}oisson manifolds}, Lett. Math.
  Phys. \textbf{66} (2003), no.~3, 157--216.

\bibitem[Lod98]{L}
J.-L. Loday, \emph{Cyclic homology}, Grundlehren der Mathematischen
  Wissenschaften, vol. 301, Springer-Verlag, Berlin, 1998.

\bibitem[LV12]{LV-ao}
Jean-Louis Loday and Bruno Vallette, \emph{Algebraic operads}, Grundlehren der
  Mathematischen Wissenschaften [Fundamental Principles of Mathematical
  Sciences], vol. 346, Springer, Heidelberg, 2012. \MR{2954392}

\bibitem[Mat97]{Mat-haps}
O.~Mathieu, \emph{Homologies associated with {P}oisson structures}, Deformation
  theory and symplectic geometry ({A}scona, 1996), Math. Phys. Stud., vol.~20,
  Kluwer Acad. Publ., Dordrecht, 1997, pp.~177--199. \MR{1480723 (99b:58097)}

\bibitem[McK80]{McKay}
J.~McKay, \emph{Graphs, singularities, and finite groups}, The {S}anta {C}ruz
  {C}onference on {F}inite {G}roups ({U}niv. {C}alifornia, {S}anta {C}ruz,
  {C}alif., 1979), Proc. Sympos. Pure Math., vol.~37, Amer. Math. Soc.,
  Providence, R.I., 1980, pp.~183--186. \MR{MR604577 (82e:20014)}

\bibitem[MSS00]{MSS}
M.~Markl, S.~Shnider, and J.~Stasheff, \emph{Operads in algebra, topology and
  physics}, Mathematical surveys and monographs, vol.~96, American Mathematical
  Society, 2000.

\bibitem[Naz73]{Naz-rqit}
L.~A. Nazarova, \emph{Representations of quivers of infinite type}, Izv. Akad.
  Nauk SSSR Ser. Mat. \textbf{37} (1973), 752--791. \MR{0338018 (49 \#2785)}

\bibitem[PP05]{PP-qa}
A.~Polishchuk and L.~Positselski, \emph{Quadratic algebras}, University Lecture
  Series, vol.~37, American Mathematical Society, Providence, RI, 2005.
  \MR{2177131 (2006f:16043)}

\bibitem[Sch07]{Shochv1}
Travis Schedler, \emph{Hochschild homology of preprojective algebras over the
  integers}, arXiv:0704.3278v1; accepted to Adv. Math., 2007.

\bibitem[Slo80a]{Slflsgs}
P.~Slodowy, \emph{Four lectures on simple groups and singularities},
  Communications of the Mathematical Institute, Rijksuniversiteit Utrecht,
  vol.~11, Rijksuniversiteit Utrecht Mathematical Institute, Utrecht, 1980.
  \MR{MR563725 (82b:14002)}

\bibitem[Slo80b]{Slsssag}
\bysame, \emph{Simple singularities and simple algebraic groups}, Lecture Notes
  in Mathematics, vol. 815, Springer, Berlin, 1980. \MR{MR584445 (82g:14037)}

\bibitem[Spr09]{Spr-lag}
T.~A. Springer, \emph{Linear algebraic groups}, second ed., Modern Birkh\"auser
  Classics, Birkh\"auser Boston, Inc., Boston, MA, 2009. \MR{2458469
  (2009i:20089)}

\bibitem[ST54]{ST}
G.~C. Shephard and J.~A. Todd, \emph{Finite unitary reflection groups},
  Canadian J. Math. \textbf{6} (1954), 274--304. \MR{MR0059914 (15,600b)}

\bibitem[Tam02]{Tam-qLbfold}
Dimitri Tamarkin, \emph{Quantization of {L}ie bialgebras via the formality of
  the operad of little disks}, Deformation quantization ({S}trasbourg, 2001),
  IRMA Lect. Math. Theor. Phys., vol.~1, de Gruyter, Berlin, 2002,
  pp.~203--236. \MR{1914789 (2003d:53168)}

\bibitem[Tam03]{Tam-Fcold}
D.~E. Tamarkin, \emph{Formality of chain operad of little discs}, Lett. Math.
  Phys. \textbf{66} (2003), no.~1-2, 65--72. \MR{2064592 (2005j:18010)}

\bibitem[VdB98]{VdBdual}
M.~Van~den Bergh, \emph{A relation between {H}ochschild homology and cohomology
  for {G}orenstein rings}, Proc. Amer. Math. Soc. \textbf{126} (1998), no.~5,
  1345--1348. \MR{MR1443171 (99m:16013)}

\bibitem[VdB06]{VdB-gdqac}
\bysame, \emph{On global deformation quantization in the algebraic case},
  arXiv:math/0603200, 2006.

\bibitem[VdB10]{VdB-cyas}
\bysame, \emph{Calabi-{Y}au algebras and superpotentials}, arXiv:1008.0599,
  2010.

\bibitem[VdBdTdV12]{VdBV-CYdnch}
M.~Van~den Bergh and L.~de~Thanhoffer~de V{\"o}lcsey, \emph{{C}alabi-{Y}au
  deformations and negative cyclic homology}, arXiv:1201.1520, 2012.

\bibitem[Wei94]{Weibel}
C.~A. Weibel, \emph{An introduction to homological algebra}, Cambridge Studies
  in Advanced Mathematics, vol.~38, Cambridge University Press, Cambridge,
  1994. \MR{1269324 (95f:18001)}

\bibitem[Wil11]{Wil-nBiKSif}
T.~Willwacher, \emph{A note on {Br}-infinity and {KS}-infinity formality},
  arXiv:1109.3520, 2011.

\bibitem[Yek05]{Yek-dqag}
A.~Yekutieli, \emph{Deformation quantization in algebraic geometry}, Adv. Math.
  \textbf{198} (2005), no.~1, 383--432. \MR{2183259 (2006j:53131)}

\end{thebibliography}

\end{document}